\documentclass[11pt]{article}
\usepackage[margin=1in]{geometry}
\usepackage[T1]{fontenc}
\usepackage{mathptmx}
\usepackage{amsthm,amsmath,amsfonts,amssymb}
\usepackage{mathtools}
\usepackage{bm}
\usepackage{booktabs}
\usepackage{graphicx}
\usepackage{enumitem}
\usepackage{microtype}
\usepackage{needspace}
\usepackage{flafter}
\usepackage{etoolbox}
\BeforeBeginEnvironment{theorem}{\Needspace{8\baselineskip}}
\BeforeBeginEnvironment{definition}{\Needspace{8\baselineskip}}
\usepackage[authoryear]{natbib}
\setcitestyle{authoryear,round,semicolon}
\usepackage{xcolor}
\definecolor{linkblue}{rgb}{0.05,0.20,0.55}
\usepackage[colorlinks=true,linkcolor=linkblue,citecolor=linkblue,urlcolor=linkblue]{hyperref}
\hypersetup{pdftitle={Efficiency Optimality without Pathwise Differentiability: A Variational Theory for Marginal-Integral Functionals},pdfauthor={Shuoxun Xu and Xinzhou Guo}}
\providecommand{\texorpdfstring}[2]{#1}

\let\raluoriginalsection\section
\renewcommand{\section}{\Needspace{6\baselineskip}\raluoriginalsection}
\let\raluoriginalsubsection\subsection
\renewcommand{\subsection}{\Needspace{4\baselineskip}\raluoriginalsubsection}
\let\raluoriginalsubsubsection\subsubsection
\renewcommand{\subsubsection}{\Needspace{4\baselineskip}\raluoriginalsubsubsection}
\numberwithin{equation}{section}
\theoremstyle{plain}
\newtheorem{theorem}{Theorem}[section]
\newtheorem{proposition}[theorem]{Proposition}
\newtheorem{corollary}[theorem]{Corollary}
\newtheorem{lemma}[theorem]{Lemma}
\theoremstyle{definition}
\newtheorem{definition}[theorem]{Definition}
\newtheorem{assumption}[theorem]{Assumption}

\newtheorem{remark}[theorem]{Remark}

\newcommand{\R}{\mathbb{R}}
\newcommand{\N}{\mathbb{N}}
\newcommand{\cP}{\mathcal{P}}
\newcommand{\cX}{\mathcal{X}}
\newcommand{\cW}{\mathcal{W}}
\newcommand{\cU}{\mathcal{U}}
\newcommand{\cN}{\mathcal{N}}
\newcommand{\cK}{\mathcal{K}}
\newcommand{\cA}{\mathcal{A}}
\newcommand{\cF}{\mathcal{F}}
\newcommand{\cT}{\mathcal{T}}
\newcommand{\cM}{\mathcal{M}}
\newcommand{\cZ}{\mathcal{Z}}
\newcommand{\cB}{\mathcal{B}}
\newcommand{\cY}{\mathcal{Y}}

\newcommand{\E}{\mathbb{E}}
\newcommand{\Var}{\mathrm{Var}}
\newcommand{\Cov}{\mathrm{Cov}}
\newcommand{\1}{\mathbf{1}}
\DeclareMathOperator{\argmax}{argmax}
\DeclareMathOperator{\argmin}{argmin}
\DeclareMathOperator{\supp}{supp}

\newcommand{\eqdef}{\coloneqq}
\newcommand{\Wext}{\mathcal{W}_{\mathrm{aux}}}
\newcommand{\Vlow}{V_{\mathrm{low}}}
\newcommand{\Vmin}{V^{\min}}
\newcommand{\TRL}{\cT_{\mathrm{RL}}}
\newcommand{\IFRL}{\cF_{\mathrm B}}

\newcommand{\psib}{\bar{\psi}}

\title{Efficiency Optimality without Pathwise Differentiability:\\ A Variational Theory for Marginal-Integral Functionals}
\author{Shuoxun Xu\thanks{Division of Biostatistics, School of Public Health, University of California, Berkeley.  E-mail: \texttt{shuoxunxu\_ucb@berkeley.edu}.}
\and
Xinzhou Guo\thanks{Department of Mathematics, The Hong Kong University of Science and Technology.  E-mail: \texttt{xinzhoug@ust.hk}.  Corresponding author.}}
\date{September 2026}

\begin{document}
\maketitle

\begin{abstract}
In this work, we provide a new perspective on semiparametric efficiency theory. In particular, we reformulate the questions of optimal efficiency and its attainment as a variational problem: minimize variance over estimating functions subject to robust unbiasedness constraints. We develop this theory for marginal-integral functionals without requiring pathwise differentiability. We study estimating functions whose expectations remain equal to the target when any one specified nuisance component is misspecified and the others are correct. We characterize the infimum of their variances through conditional variance minimization. For maxima of affine functions of treatment-specific conditional means, we obtain explicit optimal weights and construct cross-fitted estimators that attain the bound under conditions on nuisance estimation and the probability of near ties. For analogous maxima based on jointly observed quantities, the optimal weights use the full conditional covariance matrix. We also identify conditions under which the variance bound agrees with a classical convolution bound for parametric perturbations that preserve ties to first order. Examples include optimal policy values, $L^1$ calibration error, Balke--Pearl bounds, and mediation parameters. In an application to the National Longitudinal Survey of Young Men, covariance weighting reduces the median estimated variance of the cross-fitted estimating function relative to equal weighting by $12.9\%$ for the lower Balke--Pearl endpoint and $16.5\%$ for the upper, with improvements in all $20$ repeated cross-fitting splits.
\end{abstract}


\section{Introduction}\label{sec:intro}

Many statistical targets are obtained by applying a known function to conditional means, probabilities, or distributions and then averaging over covariates. An optimal treatment value averages the largest conditional mean outcome across treatments. Calibration error averages an absolute difference between a prediction and the conditional mean outcome given that prediction. Bounds on partially identified parameters often average a maximum or minimum of functions of conditional probabilities. We refer to this general class as \emph{marginal-integral functionals}. It includes both smooth parameters and parameters involving maxima, minima, or absolute values; Section~\ref{sec:mi} gives the formal definition.

Classical semiparametric efficiency theory uses pathwise differentiability to represent changes in the target along regular parametric submodels by a canonical gradient. The convolution theorem then bounds the asymptotic variance of regular estimators \citep{BKRW1993}. For an optimal treatment value, differentiability can fail when several treatments have the same largest conditional mean on a set of positive probability. Similar difficulties arise for other targets involving maxima or absolute values. Asymptotically linear and asymptotically normal estimators may nevertheless exist. The efficiency question is how to compare their asymptotic variances and determine the best precision allowed by specified robustness requirements.

Our central contribution is to reformulate the questions of optimal efficiency and its attainment as a variational problem. We design an estimating function whose expectation equals the target and whose centered version gives an asymptotic linear representation. Its variance determines asymptotic precision, while robust unbiasedness specifies the constraints under which that variance is minimized.

Let $O$ denote an observation, $\cP$ the statistical model, and $\Psi(Q)$ the target under a law $Q\in\cP$. Write $\eta(Q)=(\eta_1(Q),\ldots,\eta_d(Q))$ for the nuisance functions subject to robustness, and $\kappa(Q)$ for target-specific information held at its true value in the unbiasedness requirement. Inside $f$, the nuisance functions are evaluated at the observed covariates. At a law $P\in\cP$, we seek a solution to
\begin{equation}\label{eq:intro-variational}
\begin{aligned}
f^\star &\in \underset{f}{\operatorname{arg\,min}}\;
\Var_P\!\left[f\{O;\eta(P),\kappa(P)\}\right]\\
\text{subject to}\quad
&\E_Q\!\left[f\{O;\eta',\kappa(Q)\}\right]=\Psi(Q),\\
&\text{whenever }\eta'_{-j}=\eta_{-j}(Q),
\quad Q\in\cP,\quad j=1,\ldots,d.
\end{aligned}
\end{equation}
Here $\eta'$ ranges over working nuisance values arising from distributions in $\cP$, and the subscript $-j$ denotes all components except $j$. The equality $\eta'_{-j}=\eta_{-j}(Q)$ means that those components are correct, up to $Q$-null sets. Thus the same estimating function must remain unbiased throughout the model when any one nuisance component is misspecified and the others are correct. Analyst-chosen working quantities are included in the choice of $f$.

The variance infimum in~\eqref{eq:intro-variational} gives the efficiency bound for this class. A minimizing function, when it exists, gives the optimal influence representation after centering at the target. We establish conditions for this minimum to exist and for cross-fitted estimators to attain it. We call the resulting estimator class robust asymptotically linear unbiased, or RALU. For optimal treatment values, $\eta$ consists of the outcome regressions and propensity scores, while $\kappa$ records the set of optimal treatments. The constraint is then unbiasedness when either nuisance collection is correct, with that set held at its true value; a feasible estimator also estimates the optimal set. Section~\ref{sec:ralu} gives the formal class definition.

The softmax approach to valid inference for optimal policy values was initiated by \citet{Chen2025Softmax}. \citet{XuGuo} addressed the efficiency question for the binary optimal-treatment value. Their local smoothing construction has two tuning parameters: one controls the approximation error, while the other determines the relative weights assigned to tied treatments. The optimal weights are proportional to each arm's treatment probability divided by its conditional outcome variance; under homoscedasticity they reduce to the treatment probabilities. Xu and Guo formalized this optimization through the RALU class, derived a sharp asymptotic-variance lower bound, constructed an adaptive-smoothing estimator that attains it, and developed a multi-stage extension in their supplement.

The binary problem illustrates why these weights matter. Let $\mu_a(x)$, $\pi_a(x)$, and $\sigma_a^2(x)$ denote the arm-$a$ outcome regression, propensity, and conditional outcome variance, and let $T=\{x:\mu_1(x)=\mu_0(x)\}$ be the tie set. For the optimal-treatment value $\E[\max_a\mu_a(X)]$, the optimal RALU and equal-weight conditional residual variances on $T$ are, respectively,
\begin{equation}\label{eq:intro-binary}
\begin{aligned}
V_{\mathrm{opt}}(x)&=\Bigl\{\frac{\pi_1(x)}{\sigma_1^2(x)}+\frac{\pi_0(x)}{\sigma_0^2(x)}\Bigr\}^{-1},\\
V_{\mathrm{eq}}(x)&=\frac14\Bigl\{\frac{\sigma_1^2(x)}{\pi_1(x)}+\frac{\sigma_0^2(x)}{\pi_0(x)}\Bigr\}.
\end{aligned}
\end{equation}
They coincide exactly when the two arms have equal conditional information, $\pi_1/\sigma_1^2=\pi_0/\sigma_0^2$. Off $T$, both equal $\sigma_{a^\ast}^2(x)/\pi_{a^\ast}(x)$ for the unique optimal arm. Their total variances add $\Var\{\max_a\mu_a(X)\}$ to the respective integrated conditional residual variances.

Both weighting rules estimate the same optimal value: assigning different probabilities to equally effective treatments leaves the expected outcome unchanged. Their precision differs because the information available about the tied treatments can differ. The minimum-variance rule weights the arms in proportion to $\pi_a/\sigma_a^2$, assigning greater weight to treatments with more information. Theorem~\ref{thm:kink} derives the corresponding formula for any finite number of treatments, and Theorem~\ref{thm:cov-kink} shows how correlations change the optimal weights when several quantities are observed together.

\citet{Whitehouse2025Softmax} developed softmax inference for nonparametric policy values and marginal expectations of finite maxima of affine functions of unknown regressions. Their nonparametric static-policy estimator is semiparametrically efficient when the optimal treatment is unique almost surely; at treatment ties, its asymptotic representation weights the optimal arms equally. Their discussion identifies the need for ``an analogue of efficiency theory for this irregular setting'' that permits direct comparison of irregular estimators. Our theory addresses that question. For static multi-arm optimal-treatment values, the optimal RALU variance is no larger than that of their nonparametric softmax estimator. When total variances are finite, the gain is strict if arms with unequal conditional information are tied on a set of positive probability. Section~\ref{sec:numerics} compares the two estimators in designs exhibiting this difference.

The lower-bound argument varies the covariate distribution while holding the conditional distributions fixed. Unbiasedness under these changes fixes the conditional expectation of the estimating function at the true nuisance values. The law of total variance therefore separates a common contribution from variation in the conditional target and a conditional residual variance that can be optimized. We show that the infimum of total variance equals this common contribution plus the average conditional variance infimum. Additional conditions ensure that a minimizing estimating function exists. For maxima of affine transformations of treatment-specific conditional means, cross-fitting yields estimators that attain the bound when the nuisance estimates converge sufficiently quickly and the probability of near ties satisfies the stated rate conditions.

The same unbiasedness identities explain the scope of the theory. Under the perturbation conditions of Section~\ref{sec:kink}, if one transformed conditional mean uniquely determines the maximum, its transformation must be affine over the relevant interval of conditional means. Nonlinear transformations can remain compatible with unbiasedness when several transformed means remain tied under the allowed perturbations. The most explicit formulas thus concern affine functions, while the general variance minimization principle applies to the broader class of marginal-integral functionals.

We also relate the variance bound to classical semiparametric efficiency. We consider regular parametric perturbations under which tied conditional means have the same first-order change. We identify conditions under which the optimal influence function is the canonical gradient for the tangent space generated by these perturbations, so its variance equals the associated classical convolution bound. The argument uses orthogonal projection onto this tangent space. The connection is developed in Section~\ref{sec:bridge}, including correlated quantities and singular covariance matrices.

The main contributions are the following.
\begin{enumerate}[wide=0pt,label={\textup{(\roman*)}},itemsep=3pt]
\item We extend the RALU class to marginal-integral functionals with any finite number of specified nuisance components. Theorem~\ref{thm:ulb} characterizes the infimum of the variance over estimating functions satisfying the unbiasedness requirements by minimizing conditional variances.
\item We characterize the restrictions that robustness places on the form of the estimating function and on nonlinear transformations of conditional means (Lemma~\ref{lem:func-eq-maxarm} and Proposition~\ref{prop:affine-necessary}). These results explain when the exact unbiasedness requirements can be satisfied.
\item We derive explicit minimum-variance weights for maxima of affine functions, using conditional variances for separately observed treatment outcomes and conditional covariance matrices for jointly observed quantities (Theorems~\ref{thm:kink} and~\ref{thm:cov-kink}). For the treatment-specific models, Theorem~\ref{thm:achieve} and Corollary~\ref{cor:achieve-global} give conditions under which cross-fitted estimators attain the bound.
\item We identify conditions under which the optimal influence function coincides with the canonical gradient for a specified family of regular parametric submodels that preserve ties to first order, giving a classical convolution interpretation of the variance bound (Theorem~\ref{thm:bridge} and Corollary~\ref{cor:bridge-maxarm}).
\item The theory recovers the Xu--Guo optimal-treatment bound, gives a direct efficiency comparison with the nonparametric softmax estimator of \citet{Whitehouse2025Softmax}, and applies to $L^1$ calibration, covariate-assisted Balke--Pearl bounds, and mediation.
\end{enumerate}

The empirical application examines the value of covariance weighting for covariate-assisted Balke--Pearl bounds, following \citet{levis2025covariate}. These bounds combine several functions of conditional probabilities. Their estimating functions use the same observations and can be correlated, so the covariance between them affects precision. In the Card--NLSYM data, optimizing their weights using the full covariance matrix reduces the median estimated variance by $12.94\%$ for the lower endpoint and $16.52\%$ for the upper endpoint relative to equal weighting on the same estimated set of maximizing or minimizing functions. The estimated variance is lower in every repeated cross-fitting split at both endpoints.

Several complementary literatures address inference for nonregular targets. Generalized van Trees inequalities yield local minimax risk bounds without differentiability \citep{TakatsuKuchibhotla2024}. Directional delta methods and local asymptotic minimax theory address directionally differentiable functionals \citep{FangSantos2019,Fang2014LAM}, while smoothing methods replace a nondifferentiable target by a sequence of differentiable approximations \citep{Bibaut2017Smoothing,Chen2025Softmax,Whitehouse2025Softmax}. Online one-step estimation and subagging provide additional procedures for optimal-value inference \citep{LuedtkeVDL2016,ShiLuSong2020}. These works address local risk, distributional approximation, or the construction of valid nonregular inference. The present paper derives asymptotic-variance bounds for estimating functions satisfying specified robustness requirements and constructs estimators that attain those bounds under the stated conditions.

Recent work on aggregated conditional linear programs develops de-biased estimators and entropic smoothing \citep{BenMichael2025CLP}, while \citet{KhandamiryanSemenova2026} establish regular inference under almost-sure uniqueness of the conditional dual optimizer. Our Balke--Pearl analysis treats the case where several conditional optimizers are tied on a set of positive probability. The relevant variance comparison then depends on their conditional covariance matrix, which determines the gain from combining their estimating functions.

Section~\ref{sec:setup} introduces the target functionals and nuisance parameters. Section~\ref{sec:ralu} defines the estimator class, and Section~\ref{sec:ulb} reduces its variance problem to conditional optimization. Section~\ref{sec:kink} derives the explicit optimal weights and constructs estimators that attain the bound. Section~\ref{sec:bridge} connects the results to classical efficiency. Sections~\ref{sec:numerics} and~\ref{sec:card-data} present the numerical study and empirical application, and Section~\ref{sec:disc} concludes. Proofs, further applications, and additional numerical details are collected in the appendices.

Throughout, $\R$ denotes the real line and $\N$ the natural numbers. For a probability measure $P$ on a measurable space $(\Omega,\mathcal{G})$ and a $\mathcal{G}$-measurable real-valued function $h$ with $\E_P|h|<\infty$, we write $\E_P[h]$ for the integral. Conditional expectations are denoted $\E_P[h\mid\mathcal{F}]$ or $\E_{Q}[h]$ for the regular conditional law $Q$. For such an integrable $h$, we write $\Var_P[h]\eqdef \E_P[(h-\E_P h)^2]\in[0,\infty]$; this extended variance is finite exactly when $h\in L^2(P)$. By a Polish space we mean a separable completely metrizable topological space; all Polish spaces are equipped with their Borel $\sigma$-algebras. The notation $A\eqdef B$ means ``$A$ is defined as $B$''.

\Needspace{10\baselineskip}
\section{Setup}\label{sec:setup}

\subsection{Observation model and target functional}\label{sec:obs}

We observe i.i.d.\ data $\{(X_i,W_i)\}_{i=1}^n$ drawn from an unknown law $P\in\cP$, where $X_i$ takes values in a Polish space $\cX$, the \emph{covariate space}, and $W_i$ in a Polish space $\cW$, the \emph{residual space}; both spaces are fixed throughout.  Here $\cP$ is a class of probability measures on the product space $(\cX\times\cW,\cB_\cX\otimes\cB_\cW)$, and the target of inference is a measurable real-valued functional $\Psi:\cP\to\R$.

The two coordinates play different roles: $X$ is the variable over which $\Psi$ integrates, while $W$ carries the additional information that identifies the pointwise summary entering the integrand; Section~\ref{sec:mi} formalizes this structure.  In the binary optimal-treatment-regime problem, $\cW=\{0,1\}\times\R$ collects treatment and outcome; in mediation analysis, treatment, mediator, and outcome; in instrumental-variable problems, instrument, treatment, and outcome.

Throughout, $\cM_1(\cdot)$ denotes the Borel probability measures on the indicated Polish space.  For each $P\in\cP$, $P_X\in\cM_1(\cX)$ denotes the marginal of $P$ on $\cX$, and $Q_x^P\in\cM_1(\cW)$ the regular conditional law of $W$ given $X=x$ under $P$, which exists and is $P_X$-a.s.\ unique by the disintegration theorem for Polish probability spaces.

\subsection{Marginal-integral structure}\label{sec:mi}

The functionals we study share a common form: a covariate average of a fixed score evaluated at a $P$-dependent summary of the conditional law.  We take this form as the primitive of the framework.

\begin{definition}[Marginal-integral structure]\label{def:mi}
A \emph{marginal-integral structure} for the target functional $\Psi$ is a tuple $(\cU,\psi,u)$ consisting of:
\begin{enumerate}[wide=0pt,label={(MI\arabic*)},itemsep=0pt]
\item a Polish space $\cU$, the \emph{summary space};
\item a measurable function $\psi:\cX\times\cU\to\R$, the \emph{scoring function};
\item a collection of measurable maps $u_P:\cX\to\cU$, indexed by $P\in\cP$, with $u_P(x)$ the \emph{pointwise summary} of $P$ at $x$,
\end{enumerate}
satisfying, for every $P\in\cP$,
\begin{equation}\label{eq:mi-rep}
\Psi(P)=\E_{X\sim P_X}\bigl[\psi(X,u_P(X))\bigr],
\end{equation}
the right-hand side being well-defined, with $\psi(\cdot,u_P(\cdot))\in L^1(P_X)$. The dependence of $u_P(x)$ on $P$ is through $Q_x^P$ alone, using the same conditional-law rule across $P$. Assumption~\ref{ass:meas-main} states the joint measurability required of the local quantities.
\end{definition}

Definition~\ref{def:mi} separates what the analyst knows from what must be learned.  The scoring function $\psi$ is fixed and known (in all of our examples the kink enters through $\psi$), while the summary $u_P$ carries all dependence on the law, and carries it locally: $u_P(x)$ is a functional of the conditional law at $x$ alone, a requirement that Assumption~\ref{ass:rich} extends to the working models and the structural component.

Representation~\eqref{eq:mi-rep} is a restriction relative to the chosen observation decomposition $(X,W)$. With the conditional laws held fixed, the target must be linear in the covariate marginal. Functionals that depend nonlinearly on this marginal mixing, such as an unconditional quantile in a general regression model, require a different representation. The distinction concerns how the target combines the conditional laws, rather than whether its pointwise summary is itself linear.

\subsection{Nuisance and auxiliary inputs}\label{sec:nuisance}

In applications the integrand $\psi(x,u_P(x))$ in~\eqref{eq:mi-rep} is assembled from a small number of statistical components with sharply different asymptotic roles: working models that must be fit from data and whose estimation errors enter the bias; a discrete structural component (an optimal arm, a sign, an active index) that can be plugged in directly under a margin condition; and inputs that the analyst chooses freely and that affect variance only.  The next definition separates the three roles from the outset.

\begin{definition}[Admissible nuisance]\label{def:nuisance}
An \emph{admissible nuisance} for $\Psi$ with marginal-integral structure $(\cU,\psi,u)$ is a triple $(\eta,\kappa,\Wext)$ specified as follows.
\begin{enumerate}[wide=0pt,label={(N\arabic*)},itemsep=2pt]
\item \textbf{Primary nuisance.} A vector of length $d\ge1$, $\eta=(\eta_1,\dots,\eta_d)$, where each \emph{working-model} functional assigns to $P\in\cP$ a specified measurable map $\eta_j(P):\cX\to\cN_{0,j}$, valued in a Polish space $\cN_{0,j}$. We write $\cN_0\eqdef\cN_{0,1}\times\cdots\times\cN_{0,d}$ and $\eta(P)(x)\eqdef(\eta_1(P)(x),\dots,\eta_d(P)(x))\in\cN_0$. The integer $d$ is the \emph{component count} of $\eta$.

\item \textbf{Structural component.} A collection of measurable maps $\kappa(P):\cX\to\cK_0$, indexed by $P\in\cP$, for a Polish space $\cK_0$. We require $\kappa$ to be determined measurably by the primary nuisance and the covariate: there is a measurable map $\kappa_\flat:\cX\times\cN_0\to\cK_0$ with
\begin{equation}\label{eq:kappa-flat}
\kappa(P)(x)=\kappa_\flat\bigl(x,\eta(P)(x)\bigr)\qquad
\text{for all }P\in\cP,\ x\in\cX.
\end{equation}
Thus two conditional laws with the same primary-nuisance value have the same structural component.

\item \textbf{Auxiliary working space.} A Polish space $\Wext$. An element $w\in\{\text{measurable maps }\cX\to\Wext\}$ is the analyst's free choice and is \emph{not} a functional of $P$.
\end{enumerate}
The triple must identify the integrand pointwise: there is a measurable function defined on the full domain $\psib:\cX\times\cN_0\times\cK_0\to\R$ such that
\begin{equation}\label{eq:pt-id}
\psi(x,u_P(x))=\psib(x,\eta(P)(x),\kappa(P)(x))\qquad P_X\text{-a.s., for every }P\in\cP.
\end{equation}
Only the restriction of $\psib$ to $\{(x,\eta(P)(x),\kappa(P)(x)):P\in\cP\}$ enters the theory; an arbitrary measurable extension to the rest of its domain is fixed once and for all.
\end{definition}

We call the fixed triple $(\eta,\kappa,\Wext)$, including the grouping of $\eta$ into $d$ components, the \emph{nuisance declaration}. A \emph{slot} is an input position of a score map: each primary component, the structural component, and the auxiliary input occupy separate slots.

The primary nuisance, structural component, and auxiliary input are distinguished by their roles in the $\sqrt n$-analysis of plug-in estimators.  Each component $\eta_j$ is a \emph{working model}: a functional of $P$ that the analyst estimates from data, so that errors in $\eta$ are bias-relevant; they will be controlled by the exact multiple robustness condition of Definition~\ref{def:ralu}.  The integer $d$ counts functionals, not scalar parameters.  The structural component $\kappa$ is bias-relevant as well but is handled differently: it is plugged in directly, and a margin condition renders the plug-in error $o_P(n^{-1/2})$.  The auxiliary $w$ affects variance and is freely chosen; a procedure may use estimated conditional variances to select it.  In binary-action OTR, $\eta$ collects the arm-wise outcome regressions and the propensity ($d=2$), $\kappa$ is the set of optimal arms, and $\Wext$ encodes working conditional variances; the further applications instantiate the triple analogously.

Equation~\eqref{eq:kappa-flat} ties the structural component to the working models and covariate. The rigidity lemmas vary the nuisance components while preserving this structural value. The kernel-substitution condition below ensures that these variations are realized by laws in the comparison model; structural compatibility alone does not establish model membership. The conditional-variance reduction uses only (M1), while the simultaneous local identities used for rigidity require (M1)--(M3) and the measurability conditions below.  In every worked example, the dependence in~\eqref{eq:kappa-flat} takes the form $\kappa=\argmax$, $\mathrm{sign}$, $\argmin$, or a constant, each determined by $\eta_1$. In the active-set examples, $\cK_0$ is a finite power set, so $\kappa$ can record all branches tied at a kink.

\subsection{Standing conditions}\label{sec:standing-model}

The lower-bound and rigidity arguments of Sections~\ref{sec:ulb} and~\ref{sec:kink} manipulate the law directly: they reweight its covariate marginal, substitute one working-model component at a time, and replace conditional laws by compatible alternatives.  The standing assumption records that $\cP$ supports these operations, and that the nuisance components $\eta$ and $\kappa$ are local: their values at $x$ depend on $P$ only through the conditional law there.

Fix a Borel graph $\{(x,Q):Q\in\mathfrak K_x\}$ of locally admissible conditional laws. The Borel maps $\eta_{\rm loc}(x,Q)$ and $\kappa_{\rm loc}(x,Q)$ give the primary nuisance and structural values determined by $Q$ at $x$.  For $P\in\cP$, define the compatible fibre
\begin{equation}\label{eq:QxP-rich}
\mathfrak{Q}_x^P\eqdef\bigl\{Q\in\mathfrak K_x:
\eta_{\rm loc}(x,Q)=\eta(P)(x),\
\kappa_{\rm loc}(x,Q)=\kappa(P)(x)\bigr\}.
\end{equation}
The set $\mathfrak Q_x^P$, called a compatible fibre, consists of admissible local laws that agree with $P$ in these values. It is specified by the local graph, independently of arbitrary choices of conditional versions on marginal null sets. Score-specific completeness and perturbation conditions are introduced after the corresponding identifying scores have been defined.

\begin{assumption}[Standing conditions on the model]\label{ass:rich}
The class $\cP$ satisfies:
\begin{enumerate}[wide=0pt,itemsep=2pt]
\item[(M1)] \textbf{Marginal reweighting closure.} For every $P\in\cP$ and measurable $A\subseteq\cX$ with $P_X(A)>0$, the law $P|_A$ with marginal $P_X(\cdot\cap A)/P_X(A)$ and unchanged conditional kernel belongs to $\cP$, with the same nuisance and structural readouts as $P$ on $A$.
\item[(M2)] \textbf{Integrable kernel substitution.} There is a finite Borel cost $C(x,Q)\ge1$ on the admissible graph, controlling the absolute target integrand and the moments required by the model. Every $P\in\cP$ satisfies $Q_x^P\in\mathfrak K_x$ almost surely and $\int C(x,Q_x^P)\,dP_X(x)<\infty$.  Every universally measurable selection $R_x\in\mathfrak K_x$ with $\int C(x,R_x)\,dP_X(x)<\infty$ defines a law in $\cP$ with marginal $P_X$.
\item[(M3)] \textbf{Locality.} At $P_X$-almost every $x$, $\eta(P)(x)=\eta_{\rm loc}(x,Q_x^P)$ and $\kappa(P)(x)=\kappa_{\rm loc}(x,Q_x^P)$. The primary integrand has the same locality property.
\end{enumerate}
\end{assumption}

\begin{assumption}[Measurability for kernel substitution]\label{ass:meas-main}
On the admissible graph $\{(x,Q):Q\in\mathfrak K_x\}$, the local nuisance and structural maps, the local target integrand $\psi\{x,u_{\rm loc}(x,Q)\}$, and the finite model cost $C(x,Q)$ are Borel, where $u_{\rm loc}(x,Q)$ is the conditional summary determined by $Q$. They satisfy
\begin{equation}\label{eq:model-cost-control}
\bigl|\psi\{x,u_{\rm loc}(x,Q)\}\bigr|\le C(x,Q),
\qquad
\int C(x,Q_x^P)\,dP_X(x)<\infty\quad(P\in\cP).
\end{equation}
Universally measurable selections are interpreted in the $P_X$-completion and replaced by Borel versions when forming probability kernels.
\end{assumption}

The cost $C$ is a model moment bound, independent of the candidate score map. For a selected field $R_x\in\mathfrak K_x$, set $A_m=\{x:C(x,R_x)\le m\}$ and use $R_x$ on $A_m$ and the original kernel $Q_x^P$ on $A_m^c$. The resulting cost is integrable, so (M2) puts each resulting law in $\cP$; the sets $A_m$ exhaust the domain because the local cost is finite. For finitely many jointly selected kernels use the sum of their costs. We refer to this operation as localized kernel substitution. Bounded-outcome models permit a constant cost; unbounded outcomes are allowed under the displayed moment bound.

A canonical bounded-outcome model makes the conditions above concrete and serves as a running example.  Let $\cX=[0,1]$, $A\in\{1,\ldots,K\}$, $0\le Y\le1$, and let $\cP$ consist of all joint laws with $\pi_a(x)\ge\epsilon$ for a fixed $\epsilon\in(0,1/K)$ and positive conditional arm variances.  The target is $\E[\max_a\E(Y\mid X,A=a)]$, the nuisance components are the arm regressions and the propensity, and the structural component is the active set.  The admissible graph is Borel, the constant cost $C=2$ satisfies (M2), and marginal restrictions and every admissible kernel substitution stay in $\cP$, so (M1)--(M3) hold.  The simplex-weighted augmented inverse-probability-weighted rule $\sum_{a\in k}\lambda_a(x,k)\{u_a(x)+\1\{A=a\}(Y-u_a(x))/\pi_a(x)\}$ is bounded by $1+1/\epsilon$ at admissible slots and is exactly unbiased whenever the regression or the propensity is correct, so the algebraic class of Section~\ref{sec:ralu} is nonempty.  A subclass on which one estimator learns the nuisances at the required rates is described after Definition~\ref{def:ralu}.

This bounded-outcome model is the canonical comparison model of the paper, and the reader may keep it in mind throughout.  It satisfies (M1)--(M3); the max-arm perturbation conditions of Section~\ref{sec:kink}; the one-wrong envelope condition of Corollary~\ref{cor:achieve-global}, because every admissible slot evaluation of the weighted rule is bounded; and the local closure of Section~\ref{sec:bridge}.  This model supplies the structural conditions for the max-arm results below; estimator attainment additionally uses the stated learning conditions on the chosen subclass.  The model is used again to illustrate the class and learnability conditions of Sections~\ref{sec:ralu} and~\ref{sec:achieve}.  The results are nevertheless stated for a general model satisfying the listed conditions, because the shared-score structures of Section~\ref{sec:cov-extended} use comparison models of their own, such as the finite-cell Balke--Pearl model of Section~\ref{sec:card}.

The comparison model $\cP$ and nuisance declaration now specify the exact identities used below. Sections~\ref{sec:ralu} and~\ref{sec:ulb} define the admissible score maps and reduce their variance minimization to a conditional problem.

\section{The class of robust asymptotically linear unbiased estimators}\label{sec:ralu}

The algebraic comparison uses the comparison model $\cP$. Asymptotic linearity is required on a specified learnable subclass $\cP_{\mathrm{ach}}\subseteq\cP$, where the subscript denotes attainment, and the nuisance functions can be learned at the required rates.  This subset need not be closed under kernel substitutions.  Exact robustness determines the variance benchmark; learning conditions determine its attainability.

\begin{definition}[RALU class]\label{def:ralu}
Fix a marginal-integral structure $(\cU,\psi,u)$ for $\Psi$ and an admissible nuisance $(\eta,\kappa,\Wext)$ with primary-nuisance component count $d$, and fix $\cP_{\mathrm{ach}}\subseteq\cP$. An estimator sequence $(T_n)_{n\ge1}$ belongs to $\TRL(\Psi;\eta,\kappa,\Wext)$ if there exists a measurable function
\begin{equation}\label{eq:f-domain}
f:\cX\times\cW\times\cN_0\times\cK_0\times\Wext\to\R
\end{equation}
such that:
\begin{enumerate}[wide=0pt,label={(\Alph*)},itemsep=2pt]
\item \textbf{Asymptotic linearity.}~For every $P\in\cP_{\mathrm{ach}}$ there is a measurable $w_{f,P}:\cX\to\Wext$ with
\begin{equation}\label{eq:AL}
T_n=\frac{1}{n}\sum_{i=1}^n f\bigl(X_i,W_i,\eta(P)(X_i),\kappa(P)(X_i),
w_{f,P}(X_i)\bigr)+o_P(n^{-1/2}).
\end{equation}
The map $w_{f,P}$ may depend on $P$, while the score map $f$ and the robustness requirement below are common to the comparison model.

\item \textbf{Exact multiple robustness.} For every $P,P'\in\cP$ and every measurable $w:\cX\to\Wext$, if there exists $j\in\{1,\dots,d\}$ such that
\begin{equation}\label{eq:agree-others}
\eta_l(P')(x)=\eta_l(P)(x)\quad\text{for all }l\neq j,\quad P_X\text{-a.s.,}
\end{equation}
then
\[
f\bigl(X,W,\eta(P')(X),\kappa(P)(X),w(X)\bigr)\in L^1(P)
\]
and
\begin{equation}\label{eq:MR-unbias}
\E_P\bigl[f\bigl(X,W,\eta(P')(X),\kappa(P)(X),w(X)\bigr)\bigr]=\Psi(P).
\end{equation}

\end{enumerate}
The function $f$ is an uncentered \emph{influence-function map} for $(T_n)$ on $\cP_{\mathrm{ach}}$. A full-slot map means a function defined on the entire domain in~\eqref{eq:f-domain}, including working nuisance values that differ from the truth. The algebraic family $\IFRL(\Psi;\eta,\kappa,\Wext)$ consists of \emph{all} measurable full-slot maps in~\eqref{eq:f-domain} satisfying (B) on $\cP$; membership does not presuppose an estimator satisfying (A). When the models need to be explicit we write $\TRL(\Psi;\cP_{\mathrm{ach}}\mid\cP)$, suppressing the fixed slot declaration.  We optimize the lower bound over $\IFRL$ and prove estimator attainment separately on $\cP_{\mathrm{ach}}$.
\end{definition}

In the bounded-outcome model of Section~\ref{sec:standing-model}, the simplex-weighted rule displayed there belongs to $\IFRL$, and the subclass of laws whose propensities and Bernoulli arm means are constant in $x$, with arm means in $[\delta,1-\delta]$ for a fixed $\delta>0$, is a $\cP_{\mathrm{ach}}$ for the estimator of Section~\ref{sec:achieve}: clipped empirical arm means and proportions are root-$n$ consistent, the tolerance $\tau_n=n^{-1/3}$ dominates their errors, and every positive branch gap at a fixed law is bounded away from zero.  That subclass is not closed under kernel substitution, which is precisely why the two roles are kept separate.

\label{rem:factorization-dependence}%
The nuisance grouping and structural information determine which estimating functions are compared. The algebraic family $\IFRL(\Psi;\eta,\kappa,\Wext)$ and its variance bound depend on this specification as well as on $(\Psi,\cP)$. Condition~(B) constrains $f$ only at laws $P'$ that agree with $P$ in all but one component. Refining the grouping of the same primary vector, with the structural input fixed, shrinks the set of such $P'$, weakens~(B), enlarges the class, and can only lower the bound; coarsening can only raise it. With $d=1$, robustness requires unbiasedness at every working value arising from a law in the model. Whether the class is nonempty then depends on the target and the structural information.

For binary OTR, the declaration treats all arm means as one component and the propensity as another, and keeps the true set of optimal arms in the structural input. Condition~(B) therefore demands unbiasedness when either the mean component or the propensity is correct. The benchmark compares representations satisfying this same requirement. Section~\ref{sec:bridge} identifies its classical interpretation through local experiments that preserve the active set.

The structural input in~(B) is always $\kappa(P)$, even when a primary nuisance component is replaced. In plain language, of the $d$ working models in $\eta$, \emph{at most one may disagree between $P$ and $P'$}, and the influence function remains unbiased under the mismatched plug-in.  The pattern \emph{at most $1$ wrong $\Rightarrow$ unbiased} subsumes the classical instances of multiple robustness.  For $d=2$ (ATE or OTR, with $\eta$ the outcome regression and the propensity), it reads \emph{either one of two correct}, which is standard double robustness; for $d=3$ (mediation, with the mediator-law pair as third component), it reads \emph{any two of three correct}, the triple-robust pattern of \cite{TchetgenShpitser2012}.

Exact robustness is verified by evaluating the expectation under each allowed nuisance substitution. A second-order expansion helps identify the required cancellations: cross-products between different component errors vanish when all but one component are correct, while terms involving only the remaining component must cancel in the estimating function itself.

No smoothness condition on $f$ is imposed; the integrability and equalities in (B) are part of membership. Assumptions~\ref{ass:rich} and~\ref{ass:meas-main} make the nuisance-substitution and conditional-law identities simultaneous over each compatible fibre.  Conditions~(A)--(B), asymptotic linearity and exact unbiasedness under one-component misspecification, are therefore the complete class definition.

Definition~1 of \cite{XuGuo} imposes only measurability on $f$, together with asymptotic linearity and double-robust unbiasedness. Under the nonparametric binary-OTR specialization, with the outcome-regression and propensity-score blocks as the two primary nuisance components and the active set equivalently encoded by the sign of the treatment contrast, evaluation at the correct nuisance limits yields the same unit-sum binary tie-slope program as in \cite{XuGuo}. Consequently, the two analyses have the same variance-minimizing limiting representation and sharp bound at such laws.  The correspondence is between limiting programs; the finite-$n$ constructions differ, Xu and Guo attaining the common bound through their two-parameter adaptive-smoothing estimator and Section~\ref{sec:achieve} through a tolerance active set.

Definition~\ref{def:ralu} imposes no square-integrability on $f$.  The lower bound of Theorem~\ref{thm:ulb} is stated directly on the variance of $f$ evaluated at the true nuisance values under $P$; if $\Var_P[f^P_w]=\infty$ (with notation introduced in~\eqref{eq:fPw} below), the bound holds trivially.  Estimator-level statements requiring $\sqrt n$-consistency are corollaries (Corollary~\ref{cor:est-level}) and carry their own square-integrability and uniform-integrability conditions.

The auxiliary $w$ ranges freely over measurable maps $\cX\to\Wext$; it enters only through the last argument of $f$, and (B) places no correctness requirement on it.  Different choices of $w$ specify different evaluations of an algebraic map and may yield different estimator constructions.  The pointwise lower bound turns out not to depend on $w$ (Lemma~\ref{rem:vmin-w-invariant}); the auxiliary matters at the \emph{achievement} stage, where the analyst's working variance determines which member of the class is constructed (Theorem~\ref{thm:achieve}, Section~\ref{sec:weighted-family}).

\section{Conditional-variance reduction}\label{sec:ulb}

Exact robustness fixes the conditional mean of every admissible evaluation at the true nuisance values. Its variance therefore splits into a fixed marginal term and a conditional residual variance. Theorem~\ref{thm:ulb} proves that minimizing these conditional variances gives the exact global infimum over $\IFRL$. It also establishes simultaneous local identities that constrain every admissible map. In Section~\ref{sec:kink}, the rigidity results convert those identities into affine representations, and Theorems~\ref{thm:kink} and~\ref{thm:cov-kink} solve the resulting finite-dimensional quadratic problems.

\subsection{Pointwise admissibility and the pointwise variance floor}
\label{sec:pt-admissibility}

For a measurable $f$ as in~\eqref{eq:f-domain}, $P\in\cP$, and a measurable $w:\cX\to\Wext$, define its \emph{evaluation at the true nuisance values}
\begin{equation}\label{eq:fPw}
f^P_w(x,W)\eqdef f\bigl(x,W,\eta(P)(x),\kappa(P)(x),w(x)\bigr),
\quad x\in\cX,\ W\in\cW.
\end{equation}
The notation $f_w^P$ will be used without introducing a separate name for this evaluation.

Recall the $(\eta,\kappa)$-compatible fibre in~\eqref{eq:QxP-rich}. Its kernels have the same local nuisance and structural values as $P$ at $x$. Their measurable selections can be realized by the localized kernel substitution justified above.  The score-specific assumptions below supply the additional completeness used for closed forms.

\begin{definition}[Pointwise admissibility]\label{def:pt-adm}
Let $P\in\cP$ and $w:\cX\to\Wext$ be fixed. We say a measurable $f:\cX\times\cW\times\cN_0\times\cK_0\times\Wext\to\R$ is \emph{pointwise admissible at $(P,w)$} if, outside one $P_X$-null set, both of the following hold with absolutely convergent integrals:
\begin{enumerate}[wide=0pt,label={(\Roman*)},itemsep=2pt]
\item \textbf{Nuisance-substitution invariance.} For every $j\in\{1,\dots,d\}$ and every local slot value $v$ realized by an admissible kernel $R\in\mathfrak K_x$, called a witness for $v$, with $\eta_{{\rm loc},l}(x,R)=\eta_l(P)(x)$ for $l\ne j$ and $v=\eta_{{\rm loc},j}(x,R)$,
\begin{equation}\label{eq:PUnb}
\begin{aligned}
&\E_{Q_x^P}\!\Bigl[f\bigl(x,W,\eta_1(P)(x),\dots,
  v,\dots,\\[-2pt]
&\hspace{7em}\eta_d(P)(x),\kappa(P)(x),w(x)\bigr)\Bigr]\\
&\hspace{3em}=\psib\bigl(x,\eta(P)(x),\kappa(P)(x)\bigr).
\end{aligned}
\end{equation}
\item \textbf{Law-variation invariance.} For every $Q'\in\mathfrak{Q}_x^P$,
\begin{equation}\label{eq:LVI}
\E_{Q'}\!\bigl[f^P_w(x,\cdot)\bigr]
=\psib\bigl(x,\eta(P)(x),\kappa(P)(x)\bigr),
\end{equation}
where $f^P_w$ is defined in~\eqref{eq:fPw}.
\end{enumerate}
\end{definition}

The two invariances are complementary: \eqref{eq:PUnb} fixes the truth-law expectation as the substituted nuisance value varies, while \eqref{eq:LVI} fixes the expectation at the true nuisance values as the integrating law $Q'$ varies in the $(\eta,\kappa)$-equivalence class. Theorem~\ref{thm:ulb} shows that every $f\in\IFRL$ satisfies \emph{both}.

\begin{definition}[Pointwise variance lower bound]\label{def:vmin}
For $P\in\cP$ and $w:\cX\to\Wext$ write $v_f(x)\eqdef\Var_{Q_x^P}[f^P_w(x,\cdot)]$ for a map $f\in\IFRL$ with $\Var_P[f^P_w]<\infty$, and let $\mathcal V_{P,w}$ be the family of these conditional-variance functions. Define
\begin{equation}\label{eq:vmin}
\Vmin(\cdot\,;P,w)\eqdef\operatorname*{ess\,inf}\nolimits_{v\in\mathcal V_{P,w}}v
\qquad\text{in }L^0(P_X),
\end{equation}
the $P_X$-a.e.\ largest measurable lower bound for $\mathcal V_{P,w}$. For a nonempty family, it is the pointwise infimum of a countable subfamily, up to $P_X$-null sets. All statements about $\Vmin$ are statements in $L^0(P_X)$, the space of measurable functions identified up to almost-sure equality. The \emph{variance floor} at $(P,w)$ is
\begin{equation}\label{eq:Vlow}
\Vlow(P,w)\eqdef\Var_{P_X}\!\bigl[\psi(X,u_P(X))\bigr]
+\E_{P_X}\!\bigl[\Vmin(X;P,w)\bigr].
\end{equation}
When the competitor family is empty we set $\Vmin\eqdef+\infty$, and $\Vlow$ is read as an extended-real benchmark.
\end{definition}

\label{rem:ess-inf-necessary}%
Conditional variances are defined only up to $P_X$-null sets, and the exceptional set in a rigidity inequality may depend on the candidate map $f$. The essential infimum in Definition~\ref{def:vmin} accommodates both facts. An inequality $v_f\ge H$ holding $P_X$-a.e.\ for each $f$ separately gives $\Vmin\ge H$ $P_X$-a.e., because $H$ is a measurable lower bound for the whole family. A map $f^\ast\in\IFRL$ with $\Var_P[(f^\ast)^P_w]<\infty$ and $v_{f^\ast}=H$ $P_X$-a.e.\ gives the reverse inequality. The closed-form theorems below establish both steps, identifying $\Vmin=H$ in $L^0(P_X)$ independently of the chosen conditional-variance versions.

The competitors in Definition~\ref{def:vmin} are full-slot maps satisfying (B) throughout $\cP$. The local identities constrain their evaluations by varying both the conditional law and the working nuisance values. A map that realizes the minimizing conditional variance establishes algebraic sharpness; an estimator with that evaluation in its asymptotic-linear expansion establishes statistical attainment.

\begin{lemma}[$\Vmin$ does not depend on $w$]\label{rem:vmin-w-invariant}
For every $P\in\cP$ and all measurable $w,w':\cX\to\Wext$,
\[
\Vmin(\cdot\,;P,w)=\Vmin(\cdot\,;P,w')\quad\text{in }L^0(P_X),
\qquad \Vlow(P,w)=\Vlow(P,w').
\]
\end{lemma}

Lemma~\ref{rem:vmin-w-invariant} shows that changing the auxiliary map does not change the conditional or total variance floor when optimization ranges over all admissible score maps. We henceforth write the common quantities as $\Vmin(\cdot;P)$ and $\Vlow(P)$. The evaluation $f_w^P$ still depends on $w$, which determines the constructed estimator and whether it attains the common floor.

\subsection{The influence-function-level lower bound}\label{sec:ulb-main}

Theorem~\ref{thm:ulb} identifies the total variance infimum by optimizing conditional residual variances. Its first part proves this reduction using marginal reweighting. Its second part makes the local unbiasedness identities simultaneous over the admissible conditional laws and working nuisance values, providing the constraints used in the closed-form analysis.

\begin{theorem}[Conditional-variance reduction and exact algebraic floor]
\label{thm:ulb}
Fix $P\in\cP$ and a measurable $w:\cX\to\Wext$.
\begin{enumerate}[wide=0pt,label={\textup{(\roman*)}},itemsep=2pt]
\item Under (M1), every $f\in\IFRL$ satisfies
\[
\E_{Q_x^P}\!\left[f_w^P(x,\cdot)\right]
=\psi\{x,u_P(x)\}\quad P_X\text{-a.e.},
\]
and, in the extended-real sense,
\begin{equation}\label{eq:ulb-bound}
\begin{aligned}
\Var_P\!\bigl[f^P_w\bigr]
&=\Var_{P_X}\!\bigl[\psi\{X,u_P(X)\}\bigr]
+\E_{P_X}\!\bigl[\Var_{Q_X^P}\{f_w^P(X,\cdot)\}\bigr]\\
&\ge\Vlow(P).
\end{aligned}
\end{equation}
Moreover,
\begin{equation}\label{eq:ulb-inf}
\inf_{f\in\IFRL}\Var_P\!\bigl[f^P_w\bigr]\ =\ \Vlow(P).
\end{equation}
\item Under (M1)--(M3) and Assumption~\ref{ass:meas-main}, every $f\in\IFRL$ satisfies the following identity outside one $P_X$-null set: for every $j\in\{1,\dots,d\}$, every $Q'\in\mathfrak Q_x^P$, and every $v$ with a witness $R\in\mathfrak K_x$ such that $\eta_{{\rm loc},l}(x,R)=\eta_l(P)(x)$ for $l\ne j$ and $\eta_{{\rm loc},j}(x,R)=v$, the integral is absolutely convergent and
\begin{equation}\label{eq:combined-pt}
\begin{aligned}
&\E_{Q'}\!\bigl[f\bigl(x,W,\eta_1(P)(x),\dots,v,\dots,\eta_d(P)(x),
\kappa(P)(x),w(x)\bigr)\bigr]\\
&\hspace{3em}=\psib\bigl(x,\eta(P)(x),\kappa(P)(x)\bigr).
\end{aligned}
\end{equation}
The exceptional set may depend on $(f,P,w)$ but is common to all these $j,Q',v,R$.
\end{enumerate}
\end{theorem}

Theorem~\ref{thm:ulb}\textup{(i)} shows that robustness fixes the conditional mean of every admissible representation. The marginal contribution to variance is therefore common to all competitors, and the efficiency comparison reduces to their conditional residual variances. Equation~\eqref{eq:ulb-inf} identifies the integrated conditional floor with the global variance infimum using only (M1). Existence of a minimizing map is established under the realization conditions below.

Theorem~\ref{thm:ulb}\textup{(ii)} makes the constraints simultaneous over the local nuisance values and integrating kernels. It includes both invariances in Definition~\ref{def:pt-adm} on one common null set, which is essential for the rigidity results that follow. Realization by a map and attainment by an estimator are established under the additional conditions in Sections~\ref{sec:weighted-family}--\ref{sec:achieve}.

\subsection{Estimator-level corollary}\label{sec:est-corollary}

The influence-function bound transfers to estimators through asymptotic linearity and the central limit theorem.  The corollary records the transfer, together with the uniform-integrability condition that a variance, rather than distributional, statement costs.

\begin{corollary}[Estimator-level lower bound]\label{cor:est-level}
Assume (M1). Let $T_n\in\TRL(\Psi;\eta,\kappa,\Wext)$ have influence function $f$ with auxiliary $w_f$, and suppose $\Var_P[f^P_{w_f}]<\infty$ for some $P\in\cP_{\mathrm{ach}}$. Then
\begin{equation}\label{eq:CLT-RALU}
\sqrt n\bigl(T_n-\Psi(P)\bigr)\ \Rightarrow\
\mathcal{N}\!\bigl(0,\Var_P[f^P_{w_f}]\bigr).
\end{equation}
In particular, the asymptotic variance of $\sqrt n(T_n-\Psi(P))$ is bounded below by $\Vlow(P)$. If additionally the family $\{n(T_n-\Psi(P))^2\}_{n\ge1}$ is uniformly integrable under $P$, then
\begin{equation}\label{eq:liminf-var}
\lim_{n\to\infty}\Var_P\!\bigl(\sqrt n\,T_n\bigr)=\Var_P[f^P_{w_f}]
\ge\Vlow(P).
\end{equation}
\end{corollary}

Corollary~\ref{cor:est-level} transfers the influence-function bound to the normal limit of any estimator in the class with a square-integrable representation. Uniform integrability is needed only for the convergence of the estimator's variance in~\eqref{eq:liminf-var}; asymptotic linearity and square integrability suffice for the distributional conclusion in~\eqref{eq:CLT-RALU}. Thus the same lower bound governs the variance of the limiting normal distribution, and, under the additional moment condition, the limit of the finite-sample variances.

\section{Closed forms and attainment for kink-type functionals}\label{sec:kink}

The conditional-variance reduction of Theorem~\ref{thm:ulb} becomes explicit when the exact robustness identities determine the admissible slopes. We first treat a maximum of affine arm values, with one arm observed per subject; optimal treatment values have this form. Rigidity gives an affine score family (Section~\ref{sec:func-eq}), variance minimization gives its optimal weights (Section~\ref{sec:weighted-family}), and Theorem~\ref{thm:kink} identifies the harmonic variance floor under an explicit oracle-realization condition. Section~\ref{sec:achieve} gives a cross-fitted estimator, while Section~\ref{sec:delimitation} explains where affineness is necessary. Balke--Pearl bounds and $L^1$ calibration use the separate shared-score structure of Section~\ref{sec:cov-extended}, in which several candidate scores are observed together and the optimization uses their full covariance.

\subsection{Max-arm structure}\label{sec:max-arm}

The closed form rests on one structural hypothesis: the residual data decompose into arms, exactly one of which is observed per subject, and the target maximizes over arm summaries.  The definition makes this precise.

\begin{definition}[Max-arm structure]\label{def:max-arm}
A marginal-integral structure $(\cU,\psi,u)$ is of \emph{max-arm type} if there exist:
\begin{enumerate}[wide=0pt,label={(K\arabic*)},itemsep=2pt]
\item a finite set $\cA$ with $|\cA|=K$, the \emph{arm space}, and a Polish space $\widetilde\cW$ such that the residual space factorizes as $\cW=\cA\times\widetilde\cW$, with a generic $W=(A,\widetilde W)$;
\item for each $P\in\cP$, the conditional law $Q_x^P$ disintegrates as $Q_x^P(\mathrm{d}a,\mathrm{d}\widetilde w)=\pi_a^P(x)\,Q_{x,a}^P(\mathrm{d}\widetilde w)$ where $\pi_a^P(x)\eqdef P(A=a\mid X=x)$ and $Q_{x,a}^P$ is the regular conditional law of $\widetilde W$ given $(X,A)=(x,a)$;
\item a measurable arm-functional $u_{a}:\cP\to\{\text{measurable maps }\cX\to\cU_0\}$ for each $a\in\cA$, valued in a fixed Polish $\cU_0\subseteq\R$, such that $u_P(x)=(u_{P,a}(x))_{a\in\cA}\in\cU_0^\cA$;
\item a measurable scalar function $g:\cX\times\cA\times\cU_0\to\R$, smooth in $u_a\in\cU_0$ for each $(x,a)$, such that the scoring function takes the \emph{max-arm form}
\begin{equation}\label{eq:max-form}
\psi(x,u)=\max_{a\in\cA}\,g(x,a,u_a),\qquad u=(u_a)_{a\in\cA}\in\cU.
\end{equation}
\end{enumerate}
\end{definition}

We call $g(x,a,\cdot)$ the arm-$a$ link and $g(x,a,u_{P,a}(x))$ its branch value at $P$. The branch is active when its value equals the maximum in~\eqref{eq:max-form}.

The primary nuisance then has $d=2$ components: $\eta_1(P)$, the vector of arm-conditional summaries $(u_{P,a})_{a\in\cA}$, and $\eta_2(P)$, the arm-propensity $\pi^P(\cdot)\eqdef(\pi_a^P(\cdot))_{a\in\cA}$.  The structural component is the active set
\begin{equation}\label{eq:Aact}
\cA^{\mathrm{act}}(x;P)\eqdef\argmax_{a\in\cA}g\bigl(x,a,u_{P,a}(x)\bigr),
\end{equation}
encoded as $\kappa(P)(x)=\cA^{\mathrm{act}}(x;P)$. The structural space $\cK_0$ is the finite collection of nonempty subsets of $\cA$; thus a structural input $k$ is itself an active set, and we write $\cA^\ast(k)=k$. We also write $\cA^\ast(x;P)\eqdef\cA^{\mathrm{act}}(x;P)$.  The auxiliary space $\Wext$ may encode arm-specific working variances for an identifying score $h_a$ (Definitions~\ref{def:ident-score} and~\ref{def:work-var} below). For the max-arm RALU results, the propensity-slot space is the open probability simplex and every law in the declared model satisfies $\pi_a^P(x)>0$ for all $a\in\cA$, $P_X$-a.e.  Thus the full-slot rules below are real-valued on their entire declared domain; the uniform overlap constants needed for feasible estimation are imposed separately in Assumption~\ref{ass:achieve-reg}.

Because $\cA$ is finite, the active set is finite, the quadratic minimizations below attain their minima, and the positive gap $\Delta_+(x)$ in~\eqref{eq:pos-gap} separates the active and inactive arms whenever $\cA^\ast(x;P)\ne\cA$.  Infinite arm spaces require an additional isolated-gap or sieve condition and are outside the present scope. Binary OTR is the max-arm structure with $\cA=\{0,1\}$, $h_a=Y$, and $g$ the identity; multi-arm OTR takes $\cA=\{1,\dots,K\}$.

\subsection{Identifying score and working variances}\label{sec:ident-score}

To formulate the closed-form expression for $\Vmin$, we introduce two ingredients tied to the arm-functional $u_a$.

\begin{definition}[Identifying score]\label{def:ident-score}
An \emph{identifying score} for the arm-functional family $(u_a)_{a\in\cA}$ is a measurable family of functions $h_a:\widetilde\cW\to\R$, one per arm, such that
\begin{equation}\label{eq:ident-score}
u_{P,a}(x)=\E_{Q_{x,a}^P}[h_a(\widetilde W)]\qquad
\text{for all }P\in\cP,\ x\in\cX,\ a\in\cA.
\end{equation}
\end{definition}

We use $\bm u_P$ for the vector of arm means and $\bm e_a\in\R^{|\cA|}$ for the coordinate vector of arm $a$. In the perturbation conditions below, $D_x$ records changes in arm means that keep the active set fixed. The set $\mathcal R_P(x)$ contains the inactive arms that can be varied individually in this way; the conditions need not make every inactive arm individually reachable.

\begin{assumption}[Max-arm perturbation conditions]\label{ass:armwise}
For every $P\in\cP$ and $P_X$-a.e.\ $x$, the following conditions hold.
\begin{enumerate}[wide=0pt,label={(X\arabic*)},itemsep=2pt]
\item \textbf{Armwise fixed-mean completeness.}  For every $a\in\cA$, hold $\pi^P(x)$ and the off-arm conditional laws fixed.  For every $\epsilon\in(0,1)$ and every finitely supported probability law $\nu$ on $\supp Q_{x,a}^P$ satisfying $\E_\nu[h_a]=u_{P,a}(x)$, the admissible arm laws contain $(1-\epsilon)Q_{x,a}^P+\epsilon\nu$.  The resulting correspondence has Borel graph and finite model-admissibility cost $C$ as in Assumption~\ref{ass:rich}.  Selections with integrable cost glue to a law in $\cP$ with the same $P_X$, $\eta$, and $\kappa$; arbitrary selections are localized to $\{C\le m\}$ and equal the baseline elsewhere.

\item \textbf{Structure-preserving directions.}  There is a correspondence $x\mapsto D_x\subseteq\R^{|\cA|}$ with Borel graph and $\bm 0\in D_x$, together with a measurable correspondence of perturbed inactive arms $x\mapsto\mathcal R_P(x)\subseteq \cA\setminus\cA^\ast(x;P)$.  For every $a\in\mathcal R_P(x)$, $t\bm e_a\in D_x$ for all sufficiently small $t$ that leave $a$ inactive.  Moreover, $D_x$ contains a jointly measurable tie-preserving path $\bm r_x(t)$ for $|t|<\epsilon_P(x)$, where $\epsilon_P:\cX\to(0,\infty)$ is measurable.  Its inactive coordinates are zero, and it satisfies
\[
g\bigl(x,a,u_{P,a}(x)+r_{x,a}(t)\bigr)
=g\bigl(x,a,u_{P,a}(x)\bigr)+t,
\qquad a\in\cA^\ast(x;P).
\]

\item \textbf{Mean-shifting conditional laws.}  For every $\bm\delta\in D_x$ there is a conditional law $Q_{x,\bm\delta}$ on $\cA\times\widetilde\cW$ having arm probabilities $\pi^P(x)$ and arm score means $\bm u(Q_{x,\bm\delta})=\bm u_P(x)+\bm\delta$, where $u_a(Q)\eqdef\E_Q[h_a(\widetilde W)\mid A=a]$, such that
\begin{equation}\label{eq:bridge}
Q_{x,\bm\delta}\ll Q_x^P,\qquad Q_{x,\bm 0}=Q_x^P.
\end{equation}
The corresponding perturbation changes $\eta_1=\bm u_P$ only and preserves $\eta_2=\bm\pi^P$ and $\kappa$.  The correspondence of mean-shifting laws has Borel graph, and its kernels have finite model-admissibility cost $C$.  Every universally measurable selection $x\mapsto Q_{x,\bm\delta(x)}$ with integrable cost glues to a law in $\cP$ with marginal $P_X$. The baseline direction $\bm0$ is used outside each cost-localized set.
\end{enumerate}
\end{assumption}

In OTR, $h_a(y)=y$ identifies the conditional outcome regression $u_{P,a}$. The existence of an identifying score is the \emph{linear-functional} character of $u_a$ with respect to its conditional law, a property that is automatic when $u_a$ is a conditional expectation.

\begin{definition}[Working variance]\label{def:work-var}
The \emph{conditional variance} of the identifying score is
\begin{equation}\label{eq:sigma-def}
\sigma_a^2(x;P)\eqdef\Var_{Q_{x,a}^P}\!\bigl[h_a(\widetilde W)\bigr].
\end{equation}
A \emph{working variance} is a measurable map $\widehat\sigma^2:\cX\times\cA\to(0,\infty)$ chosen by the analyst, serving as $\Wext$-valued auxiliary input. We write $\widehat\sigma_a^2(x)\eqdef\widehat\sigma^2(x,a)$.
\end{definition}

\subsection{Rigidity of admissible influence functions}\label{sec:func-eq}

Exact multiple robustness pins the influence function pointwise: using pointwise admissibility (Definition~\ref{def:pt-adm}) and the armwise score and mean conditions in Assumption~\ref{ass:armwise}, the following lemma reduces every admissible evaluation $f_w^P$ to an affine family in the identifying scores, with the slope freedom that the rest of the section optimizes.

\begin{lemma}[Max-arm functional form]\label{lem:func-eq-maxarm}
Assume the max-arm structure of Definition~\ref{def:max-arm} with a scalar identifying score $h_a:\widetilde\cW\to\R$ (Definition~\ref{def:ident-score}).  Suppose Assumptions~\ref{ass:rich} and~\ref{ass:armwise} and Assumption~\ref{ass:meas-main} hold.  Fix $P\in\cP$, a measurable $w$, and a score map $f\in\IFRL$.  Assume $g(x,a,\cdot)$ is continuously differentiable with $g_a^\prime(x)\eqdef\partial_{u_a}g(x,a,u_{P,a}(x))\neq 0$; assume further that every \emph{active} link $g(x,a,\cdot)$, $a\in\cA^{\mathrm{act}}(x;P)$, is affine on $\{u_{P,a}(x)+r_{x,a}(t):|t|<\epsilon_P(x)\}$, and that $\pi_a^P(x)>0$ and $\sigma_a^2(x;P)<\infty$ for all $a\in\cA$, with $\sigma_a^2(x;P)>0$ for $a\in\cA^{\mathrm{act}}(x;P)$, for $P_X$-a.e.\ $x$.

Then there exist measurable functions $\beta_a,c_a:\cX\to\R$, $a\in\cA$, and a set $N_{f,P,w}\in\mathcal B_{\cX}$ with $P_X(N_{f,P,w})=0$ such that, for every $x\notin N_{f,P,w}$,
\begin{equation}\label{eq:lem-form}
f^P_w(x,(a,\widetilde w))\;=\;\beta_a(x)\,h_a(\widetilde w)\,+\,c_a(x),
\qquad Q_{x,a}^P\text{-a.s.},\quad a\in\cA,
\end{equation}
with the slopes constrained as follows.
\begin{enumerate}[wide=0pt,label={\textup{(\arabic*)}},itemsep=2pt]
\item \textbf{Strict region.} If $\cA^{\mathrm{act}}(x;P)=\{a^\ast\}$ is a singleton, then
\begin{equation}\label{eq:lem-strict}
\beta_{a^\ast}(x)=\frac{g_{a^\ast}^\prime(x)}{\pi_{a^\ast}^P(x)},
\qquad
\beta_a(x)=0\quad\text{for }a\in\mathcal R_P(x).
\end{equation}

\item \textbf{Kink region.} If $|\cA^{\mathrm{act}}(x;P)|\geq 2$, then $\beta_a(x)=0$ for every $a\in\mathcal R_P(x)$, and the active-arm slopes $(\beta_a)_{a\in\cA^{\mathrm{act}}(x;P)}$ satisfy the single linear identity
\begin{equation}\label{eq:lem-kink}
\sum_{a\in\cA^{\mathrm{act}}(x;P)}\pi_a^P(x)\cdot
\frac{\beta_a(x)}{g_a^\prime(x)}\;=\;1,
\end{equation}
\end{enumerate}
The intercepts satisfy
\[
\sum_{a\in\cA}\pi_a^P(x)
\{\beta_a(x)u_{P,a}(x)+c_a(x)\}
=\psib(x,\eta(P)(x),\kappa(P)(x)).
\]
\end{lemma}

Lemma~\ref{lem:func-eq-maxarm} reduces the optimization over measurable estimating functions to their arm-specific slopes and intercepts. Fixed-mean completeness gives the affine form in~\eqref{eq:lem-form}. The mean-shifting laws in Assumption~\ref{ass:armwise}~(X3) then turn exact robustness into a finite-difference identity: individual changes in the inactive arms indexed by $\mathcal R_P(x)$ force their slopes to vanish, and the tie-preserving path gives~\eqref{eq:lem-kink}. An inactive arm outside $\mathcal R_P(x)$ can retain a conditionally centered term, but its nonnegative variance contribution cannot reduce the lower bound.

\subsection{The weighted family and optimal weights}\label{sec:weighted-family}

For the explicit construction in this subsection, take affine links $g(x,a,u)=g_a^\prime(x)u+\rho(x,a)$ with $g_a^\prime(x)\ne0$, positive arm probabilities, finite arm-score variances, and positive variances on the active set. All these requirements are understood $P_X$-a.e. at the evaluating law $P$. Here $\rho(x,a)$ is the known affine intercept. We abbreviate $\cA^\ast(x)=\cA^\ast(x;P)$ and $\psib(x)=\psib(x,\eta(P)(x),\kappa(P)(x))=\psi(x,u_P(x))$ within calculations at that law.

For $P\in\cP$ and a measurable weight map $\bm\lambda:\cX\to[0,1]^{\cA}$ supported on $\cA^\ast(x;P)$ and satisfying $\sum_{a\in\cA^\ast(x;P)}\lambda_a(x)=1$, define
\begin{equation}\label{eq:weighted-eval}
\begin{aligned}
\phi_P^{\bm\lambda}(x,(a,\widetilde w))
&\eqdef\psib(x,\eta(P)(x),\kappa(P)(x))\\
&\quad+\1\{a\in\cA^\ast(x;P)\}
\frac{g_a^\prime(x)\lambda_a(x)}{\pi_a^P(x)}
\{h_a(\widetilde w)-u_{P,a}(x)\}.
\end{aligned}
\end{equation}

By Lemma~\ref{lem:func-eq-maxarm}, every admissible evaluation $f_w^P$ is affine in the identifying scores, with active-arm slopes $\beta_a=g_a^\prime\lambda_a/\pi_a^P$ parametrized by weight maps $\lambda$ on the active set summing to one (cf.~\eqref{eq:lem-kink}), and intercepts $c_a$ satisfying $\sum_a\pi_a^P\{\beta_au_{P,a}+c_a\}=\psib$.  The intercepts are then fixed by a second minimization.  Writing $r_a\eqdef\beta_au_{P,a}+c_a$ for the arm-level conditional means, the law of total variance within $x$ gives, for every finite-variance evaluation,
\begin{equation}\label{eq:within-between}
\Var_{Q_x^P}\bigl[f_w^P(x,\cdot)\bigr]
=\sum_{a}\pi_a^P\beta_a^2\sigma_a^2+\sum_a\pi_a^P\,(r_a-\psib)^2,
\end{equation}
so that the intercepts minimizing the second sum equalize every arm-level mean to $\psib$, and a nonzero slope on an inactive arm outside $\mathcal R_P(x)$, which the constraints leave free, can only enlarge the first sum.  Equalizing the arm-level conditional means and setting the inactive slopes to zero reduces the optimization to active-arm coefficients summing to one. These coefficients may initially be signed. The variance minimizer is strictly positive, so the nonnegative-weight family in~\eqref{eq:weighted-eval} contains an optimizer even though it does not contain every RALU evaluation.

For class membership we need a map defined at working nuisance inputs, not only the evaluation at the truth. Write $e=(e_1,e_2)\in\cN_0$ for the mean and propensity inputs and $k\in\cK_0$ for a nonempty active set. Let $\omega(x,k)$ be supported on $k$ with nonnegative coordinates summing to one. The corresponding full nuisance rule is
\begin{multline}\label{eq:f-form-weighted-slot}
f^{\omega}\bigl(x,(a,\widetilde w),e,k,v\bigr)
\eqdef\sum_{a'\in\cA^\ast(k)}\omega_{a'}(x,k)\,g\bigl(x,a',e_{1,a'}\bigr)\\
+\1\{a\in\cA^\ast(k)\}\,
\frac{g_a^\prime(x)\,\omega_a(x,k)}{e_{2,a}}\,
\bigl(h_a(\widetilde w)-e_{1,a}\bigr)
\end{multline}
When the weights use working variances, $\omega$ also depends on $(e_2,v)$; its rule is given in~\eqref{eq:lambda-opt-working}. We suppress these additional arguments in~\eqref{eq:f-form-weighted-slot}. The rule does not use the working mean $e_1$ in its weights.

The intercept of $f^\omega$ is the $\omega$-weighted plug-in of the arm values (the general form of the working-weighted tie intercept of \cite{XuGuo}), rather than an arbitrary off-truth extension of $\psib$: the identification~\eqref{eq:pt-id} specifies $\psib$ only at nuisance and structural values generated by laws in $\cP$.

At the true nuisance and structural inputs the intercept equals $\sum_{a'}\omega_{a'}(x)\,g(x,a',u_{P,a'}(x)) =\psib(x,\eta(P)(x),\kappa(P)(x))$ (the active-arm $g$-values being tied), recovering the evaluation at the true nuisance values in~\eqref{eq:weighted-eval}.  The algebra for nuisance-substitution invariance~\eqref{eq:PUnb} is exact: for the mean slot, the affine identity $g(x,a',e_{1,a'})+g_{a'}^\prime(x)(u_{P,a'}(x)-e_{1,a'}) =g(x,a',u_{P,a'}(x))$ cancels the intercept's deviation for every $e_1$; for the propensity input, the corrections are conditionally centered for every $e_2$.

These are conditional expectation identities whenever the displayed terms are integrable. Their use as condition~(B) on the whole comparison model also requires absolute integrability for each compatible pair of laws. This global requirement is separate from the algebra; it is guaranteed, for example, by the envelope in~\eqref{eq:maxarm-onewrong-envelope}.

Minimizing the true conditional variance over the unit-sum constraint yields the \emph{oracle weights}
\begin{equation}\label{eq:lambda-opt-oracle}
\lambda_a^\ast(x)\eqdef
\frac{\pi_a^P(x)/[(g_a^\prime(x))^2\,\sigma_a^2(x;P)]}
{\sum_{b\in\cA^\ast(x)}\pi_b^P(x)/[(g_b^\prime(x))^2\,\sigma_b^2(x;P)]},
\end{equation}
for $a\in\cA^\ast(x)$, and $\lambda_a^\ast(x)=0$ otherwise. Their optimal value is
\begin{equation}\label{eq:var-oracle-value}
\sum_{a\in\cA^\ast(x)}\frac{(g_a^\prime)^2(\lambda_a^\ast)^2\sigma_a^2}
{\pi_a^P}(x)
=\biggl(\sum_{a\in\cA^\ast(x)}
\frac{\pi_a^P(x)}{(g_a^\prime(x))^2\,\sigma_a^2(x;P)}\biggr)^{\!-1},
\end{equation}
the harmonic variance value.

\begin{assumption}[Oracle realization for max-arm scores]\label{ass:oracle-maxarm}
At the evaluating law $P$, there exist a full nuisance map $f_P^\ast\in\IFRL$ and a measurable auxiliary input $w_P^\ast$ such that $(f_P^\ast)^P_{w_P^\ast}=\phi_P^{\bm\lambda^\ast}$ $P$-a.s., where $\phi_P^{\bm\lambda^\ast}$ is defined by~\eqref{eq:weighted-eval} and~\eqref{eq:lambda-opt-oracle}, and $\Var_P[\phi_P^{\bm\lambda^\ast}]<\infty$.
\end{assumption}

Assumption~\ref{ass:oracle-maxarm} requires an estimating function that satisfies the robustness identities throughout the comparison model and has the proposed optimal evaluation at $P$. The formula~\eqref{eq:f-form-weighted-slot} supplies one when every evaluation with one misspecified nuisance component is integrable. Corollary~\ref{cor:achieve-global} gives a sufficient envelope for this integrability and, separately, learning conditions under which a feasible estimator attains the same variance.

\subsection{The closed-form variance floor at kink points}\label{sec:vmin-closed}

The variance minimization above bounds every admissible score map. Oracle realization makes the bound exact.

\begin{theorem}[Harmonic lower bound and exact floor]\label{thm:kink}
Assume the max-arm structure of Definition~\ref{def:max-arm}, an identifying score as in Definition~\ref{def:ident-score}, and Assumptions~\ref{ass:rich}, \ref{ass:meas-main}, and~\ref{ass:armwise}. Fix $P\in\cP$. Suppose $g(x,a,u)=g_a^\prime(x)u+\rho(x,a)$ for all $u\in\R$, with $g_a^\prime(x)\ne0$. Suppose also that $\pi_a^P(x)>0$ and $\sigma_a^2(x;P)<\infty$ for every arm, and $\sigma_a^2(x;P)>0$ on the active set, for $P_X$-a.e.\ $x$. Then
\[
\Vmin(x;P)\ge
\biggl(\sum_{a\in\cA^\ast(x)}
\frac{\pi_a^P(x)}{(g_a^\prime(x))^2\,\sigma_a^2(x;P)}\biggr)^{-1}
\qquad P_X\text{-a.e.}
\]
If Assumption~\ref{ass:oracle-maxarm} holds at $P$, then
\begin{equation}\label{eq:vmin-harmonic}
\Vmin(x;P)
=\biggl(\sum_{a\in\cA^\ast(x)}
\frac{\pi_a^P(x)}{(g_a^\prime(x))^2\,\sigma_a^2(x;P)}\biggr)^{\!-1}
\qquad P_X\text{-a.e.},
\end{equation}
and its integral, plus $\Var_{P_X}[\psi(X,u_P(X))]$, is attained by the map in that assumption.
\end{theorem}

On a strict branch with a nondegenerate reachable interval of active-arm means, affineness is necessary, by Proposition~\ref{prop:affine-necessary}, so that without affineness the infimum in~\eqref{eq:vmin} would there run over an empty set.  It is not forced along a persistent tie: the governing restriction is the exact perturbation equation along a tie-preserving path.  A nonlinear link can satisfy that equation; the harmonic form above is the case in which it reduces to a single linear constraint.  A separate first-order relaxation is discussed in Section~\ref{sec:delimitation}.

Theorem~\ref{thm:kink} bounds the variance of every admissible estimating function using the exact robustness identities. Assumption~\ref{ass:oracle-maxarm} makes this bound sharp by supplying a function with the minimizing variance. The bound depends on the true conditional variances and is independent of the analyst's working variances, as Lemma~\ref{rem:vmin-w-invariant} requires. Section~\ref{sec:achieve} gives the additional learning conditions for an estimator to attain it.

The expression~\eqref{eq:vmin-harmonic} is the harmonic-mean-type formula that recovers the Xu--Guo OTR bound when $\cA=\{0,1\}$ and $g(x,a,u_a)=u_a$ (so $g_a^\prime\equiv 1$).

Under the equality conditions of Theorem~\ref{thm:kink}, an analyst who commits to a working variance $w(x)=(\widehat\sigma^2_a(x))_a$ and minimizes the working-variance proxy obtains the working weights
\begin{equation}\label{eq:lambda-opt-working}
\widehat\lambda_a(x) \eqdef
\frac{\pi_a^P(x)/[(g_a^\prime(x))^2\widehat\sigma_a^2(x)]}
{\sum_{b\in\cA^\ast(x)}\pi_b^P(x)/[(g_b^\prime(x))^2\widehat\sigma_b^2(x)]},
\end{equation}
whose \emph{true} conditional variance is
\begin{equation}\label{eq:vmin-closed}
V_{\widehat\lambda}(x;P,w)
\eqdef
\frac{\displaystyle\sum_{a\in\cA^\ast(x)}\pi_a^P(x)\,\sigma_a^2(x;P)
\big/\bigl[(g_a^\prime(x))^2\widehat\sigma_a^4(x)\bigr]}
{\Bigl(\displaystyle\sum_{b\in\cA^\ast(x)}\pi_b^P(x)\big/
\bigl[(g_b^\prime(x))^2\widehat\sigma_b^2(x)\bigr]\Bigr)^{\!2}}
\;\ge\;\Vmin(x;P),
\end{equation}
with equality if and only if $\widehat\sigma_a^2(x)=c_P(x)\sigma_a^2(x;P)$ for all active arms and some $c_P(x)>0$: the price of a misspecified working variance is an achieved variance, not a lower bound.

The working-weighted value~\eqref{eq:vmin-closed} converges to the ($w$-free) bound~\eqref{eq:vmin-harmonic} as $\widehat\sigma^2\to\sigma^2(\cdot;P)$ (proportionality on the active set suffices). Under the learning conditions of Theorem~\ref{thm:achieve}, a consistent variance learner therefore yields an estimator attaining the bound.  This is the principle behind the adaptive smoothing of \cite{XuGuo}, whose asymmetry parameter is replaced by any consistent estimate of the optimal tie weight, in its general form; it is the precise sense in which the auxiliary slot matters: not for the bound, but for which member of $\TRL$ the analyst can construct (Theorem~\ref{thm:achieve}).

On the strict region $\{|\cA^\ast(x)|=1\}$, only one arm $a^\ast(x)$ is active, so $\lambda^\ast_{a^\ast}(x)=1$ and~\eqref{eq:vmin-harmonic} reduces to
\begin{equation}\label{eq:vmin-strict}
\Vmin(x;P)\bigm|_{|\cA^\ast(x)|=1}
=\frac{(g_{a^\ast}^\prime(x))^2\sigma_{a^\ast}^2(x;P)}{\pi_{a^\ast}^P(x)}.
\end{equation}
For OTR ($g_a^\prime\equiv 1$), this is the standard AIPW conditional residual variance for the optimal arm.

\begin{corollary}[Recovery of the optimal-treatment-regime bound]
\label{rem:otr-recovery-main}
Under the equality conditions of Theorem~\ref{thm:kink}, for binary OTR ($\cA=\{0,1\}$, $g(x,a,u_a)=u_a$, so $g'\equiv 1$), \eqref{eq:vmin-harmonic} reduces on the strict region to $\sigma_{a^\ast}^2/\pi_{a^\ast}^P$ and on the kink region to the harmonic mean $(\pi_1^P/\sigma_1^2+\pi_0^P/\sigma_0^2)^{-1}$, recovering the efficiency bound of \cite{XuGuo} as a special case: Theorem~3 there under conditional homoscedasticity, where the kink term reduces to $\sigma^2$, and Theorem~A.1 of their supplement in general.
\end{corollary}

The max-arm structure allows any finite number of treatment arms. Section~\ref{sec:cov-extended} considers jointly observed scores, and Section~\ref{sec:card-data} illustrates the corresponding covariance weighting for Balke--Pearl bounds.

For optimal-treatment values, where $g(x,a,u_a)=u_a$, the nonparametric softmax estimator of \citet[Theorem~3.3]{Whitehouse2025Softmax} has an asymptotic representation assigning weight $1/|\cA^\ast(x)|$ to each optimal arm.  For an active set $S$ of size $m$, its conditional residual variance is $m^{-2}\sum_{a\in S}\sigma_a^2(x;P)/\pi_a^P(x)$.  This value is no smaller than the harmonic RALU floor in~\eqref{eq:vmin-harmonic}.  Equality holds exactly when $\sigma_a^2(x;P)/\pi_a^P(x)$ is constant over $a\in S$.  Whenever the corresponding total variances are finite, the integrated improvement is strict if this condition fails on a set of positive $P_X$-probability.  Corollary~\ref{cor:whitehouse-otr-dominance} proves this variance comparison. At a binary tie, the Xu--Guo asymmetry parameter $t_0(x)$ is the active-arm weight $\lambda_1(x)$.  Their smoothing scale controls approximation and validity, while adapting $t_0(x)$ implements the optimizer in~\eqref{eq:lambda-opt-oracle}.

\subsection{Attainment by a cross-fitted estimator}\label{sec:achieve}

We now show that the lower bound is attained, under regularity conditions, by an explicit RALU estimator built on the weighted family~\eqref{eq:weighted-eval}.

The estimator retains every arm whose fitted branch value is within a shrinking tolerance of the fitted maximum. This allows it to keep treatments that are truly tied even when their fitted values differ. A direct plug-in $\operatorname{argmax}$ can split such ties: under nondegenerate continuously distributed estimation error it does so almost surely, however small the errors. Write $\widehat\eta_n=(\widehat\eta_{1,n},\widehat\eta_{2,n})$ for the fitted arm means and propensities and $\widehat\sigma_n^2$ for the fitted working variances. We omit $n$ and the training-fold superscript when stating the pointwise rule; all fitted functions are evaluated at $x$.

\begin{definition}[Tolerance active set]\label{def:tol-face}
Let $\widehat g_a(x)\eqdef g(x,a,\widehat\eta_{1,a}(x))$ and let $\tau_n\downarrow0$ be a tolerance sequence.  Set
\begin{equation}\label{eq:tol-face}
\widehat\cA_{\tau_n}(x)\eqdef\Bigl\{a\in\cA:\
\max_{b}\widehat g_b(x)-\widehat g_a(x)\le\tau_n\Bigr\},
\end{equation}
which the estimator supplies to the structural slot, with weights normalized on it.
\end{definition}

Here and below $g_a(x)\eqdef g(x,a,u_{P,a}(x))$ denotes the true branch value, so that $\|\widehat g-g\|_\infty$ is the $P_X$-essential supremum over $x$ of $\max_{a\in\cA}|\widehat g_a(x)-g_a(x)|$; the three-argument link $g(x,a,u)$ itself is known.

The estimator is computed in four steps. First, use the training folds to fit the arm regressions, propensities, and working variances. Second, form the tolerance active set~\eqref{eq:tol-face} and normalize the weight map $\omega$ on it; these weights determine how much each candidate arm contributes to precision. Third, evaluate the augmented rule $f^\omega$ of~\eqref{eq:f-form-weighted-slot} on the held-out observations, keeping each observation separate from the data used to fit its nuisance functions. Finally, average these contributions across folds. In symbols, the cross-fitted estimator is
\begin{equation}\label{eq:Tnstar-def}
T_n^{\omega}\eqdef\frac{1}{n}\sum_{i=1}^n
f^{\omega}\!\bigl(X_i,W_i;\widehat\eta_n^{(-k(i))}(X_i),
\widehat\cA_{\tau_n}^{(-k(i))}(X_i),\widehat\sigma_n^{2,(-k(i))}(X_i)\bigr),
\end{equation}
where $\{1,\dots,n\}$ is partitioned into $\mathcal K\ge 2$ folds $I_1,\dots,I_{\mathcal K}$ whose sizes differ by at most one (calligraphic $\mathcal K$ is the fold count, distinct from the arm count $K=|\cA|$), $k(i)$ is the fold containing $i$, and $\widehat\eta_n^{(-k)}, \widehat\cA_{\tau_n}^{(-k)},\widehat\sigma_n^{2,(-k)}$ are fitted using $\{1,\dots,n\}\setminus I_k$. For a measurable function $r$ of $(X,W)$, write $\mathbb P_n r=n^{-1}\sum_{i=1}^n r(X_i,W_i)$.

\begin{assumption}[Regularity for attainment]\label{ass:achieve-reg}
We assume the following. Every rate condition below holds for the maximum over the training folds.
\begin{enumerate}[wide=0pt,label={(R\arabic*)},itemsep=2pt]
\item \textbf{Measurable active set and inputs.} The maps $x\mapsto\cA^\ast(x;P)$, $\pi^P$, $g'$, and $\sigma^2(\cdot;P)$ are measurable.

\item \textbf{Positive-gap margin condition.}~Let
\begin{equation}\label{eq:pos-gap}
\Delta_+(x)\eqdef\min_{a\notin\cA^\ast(x;P)}
\bigl\{\psib(x)-g\bigl(x,a,u_{P,a}(x)\bigr)\bigr\}\in(0,\infty],
\end{equation}
the gap between the active value and the nearest \emph{strictly inactive} arm, with $\Delta_+(x)=\infty$ when every arm is active. For every $t>0$ small,
\begin{equation}\label{eq:margin-cond}
P_X\bigl(0<\Delta_+(X)\le t\bigr)\le C\,t^{\gamma}
\end{equation}
for some $\gamma>0$ and $C<\infty$.

\Needspace{7\baselineskip}
\item \textbf{Plug-in convergence.} There exist measurable, $\cN_0$-valued nuisance estimators $\widehat\eta_n$ and a tolerance active set $\widehat\cA_{\tau_n}$ such that the componentwise product rate
\begin{equation}\label{eq:product-rate}
\max_{a\in\cA}\bigl\|\widehat\eta_{1,a}-u_{P,a}\bigr\|_{L^2(P_X)}\,
\bigl\|\widehat\eta_{2,a}-\pi_a^P\bigr\|_{L^2(P_X)}=o_P(n^{-1/2})
\end{equation}
holds, and $\|\widehat\eta_n-\eta(P)\|_2=o_P(1)$.

The tolerance sequence $\tau_n\downarrow0$ of Definition~\ref{def:tol-face} satisfies $\|\widehat g-g\|_\infty=o_P(\tau_n)$ and $\sqrt n\,\tau_n^{1+\gamma}\to0$.  The working-variance estimator $\widehat\sigma^2_n$ converges in $L^2(P_X)$ at rate $o_P(1)$ to a nonrandom, measurable, $P$-specific limit $\widehat\sigma^2_{\infty,P}$; foldwise,
\[
\max_{k\le\mathcal K}
\bigl\|\widehat\sigma_n^{2,(-k)}-\widehat\sigma^2_{\infty,P}
\bigr\|_{L^2(P_X)}=o_P(1).
\]

\item \textbf{Affineness of $g$ in $u_a$.} For every arm $a\in\cA$ and $P_X$-a.e.\ $x$,
\begin{equation}\label{eq:affine-g}
g(x,a,u)=g_a^\prime(x)\cdot u+\rho(x,a)\qquad\text{for all }u\in\R,
\end{equation}
for some scalars $g_a^\prime(x)\in\R$ and $\rho(x,a)\in\R$.

\item \textbf{Moments and cross-fitting.} $\E_{P_X}[\psi(X,u_P(X))^2]<\infty$, the number of folds $\mathcal K$ is fixed with balanced fold sizes.

\item \textbf{Uniform bounds.} There exist constants $0<\epsilon_0\le 1$ and $0<c\le C<\infty$ such that, $P_X$-a.e.: overlap $\pi_a^P(x)\ge\epsilon_0$ for all $a\in\cA$, and the plug-in propensity estimators take values in $[\epsilon_0/2,1]$; the slopes satisfy $c\le|g_a^\prime(x)|\le C$; the identifying-score conditional variances satisfy $\sigma_a^2(x;P)\le C$; and the working variances satisfy $c\le\widehat\sigma^2_{n,a}(x),\widehat\sigma^2_{\infty,P}(x,a) \le C$.
\end{enumerate}
\end{assumption}

Condition~(R2) excludes exact ties from the margin event and therefore allows positive tie mass.  Conditions~(R2)--(R3) require the explicit tolerance window
\[
\|\widehat g-g\|_\infty\ll\tau_n
\ll n^{-1/[2(1+\gamma)]}.
\]
Under~(R6), the working-variance weight map is uniformly Lipschitz in its propensity and variance slots on each fixed active set. Oracle weights are treated as a fixed measurable map at the evaluating law and require no slotwise Lipschitz property.  The theorem below covers the max-arm structure; attainment for the Balke--Pearl and mediation structures requires structure-specific conditions.  Uniform integrability is not used for asymptotic linearity and is imposed only for the variance-convergence conclusion.

\begin{theorem}[Asymptotic linearity of the weighted estimator]\label{thm:achieve}
Assume the max-arm structure, identifying scores as in Definition~\ref{def:ident-score}, and Assumption~\ref{ass:meas-main}. Let $\omega$ be either a fixed measurable map $(x,k)\mapsto\omega(x,k)\in[0,1]^{\cA}$ or the working-variance weight rule in~\eqref{eq:lambda-opt-working}, with
\[
\sum_{a\in\cA^\ast(k)}\omega_a(x,k)=1
\]
for every nonempty structural value $k$.  Fix $P_0\in\cP$ and suppose the foldwise conditions in Assumption~\ref{ass:achieve-reg} hold at $P_0$.  Then the cross-fitted estimator in~\eqref{eq:Tnstar-def} satisfies
\[
T_n^\omega
=\mathbb P_n\bigl[(f^\omega)^{P_0}_{w_{f,P_0}}\bigr]
+o_{P_0}(n^{-1/2}),
\qquad
w_{f,P_0}(x)=\bigl(\widehat\sigma^2_{\infty,P_0}(x,a)\bigr)_a.
\]
The limiting score has variance
\begin{equation}\label{eq:achieve-eq}
\Var_{P_0}\!\bigl[(f^{\omega})^{P_0}_{w_{f,P_0}}\bigr]
=\Var_{P_{0,X}}[\psi(X,u_{P_0}(X))]
+\E_{P_{0,X}}\!\Bigl[\sum_{a\in\cA^\ast(X;P_0)}
\frac{(g_a^\prime)^2\,\bar\omega_a^2\,\sigma_a^2}
{\pi_a^{P_0}}(X)\Bigr],
\end{equation}
where $\bar\omega$ denotes the realized weights ($\bar\omega=\omega$ for a fixed map; $\bar\omega=\widehat\lambda(\cdot;\widehat\sigma^2_{\infty,P_0})$ when the weights are obtained from working variances). If $\{n(T_n^\omega-\Psi(P_0))^2:n\ge1\}$ is uniformly integrable, then $n\Var_{P_0}(T_n^\omega)$ converges to the right-hand side of \eqref{eq:achieve-eq}.
\end{theorem}

Theorem~\ref{thm:achieve} establishes asymptotic linearity at a law where the stated learning conditions hold. The next corollary combines this result with the global integrability requirement for exact robustness and identifies the equality cases.

\begin{corollary}[Class membership and attainment]\label{cor:achieve-global}
Let the comparison model $\cP$ satisfy the lower-bound hypotheses of Theorem~\ref{thm:kink} at each $P\in\cP$, without imposing Assumption~\ref{ass:oracle-maxarm}. Fix a weight rule $\omega$ of Theorem~\ref{thm:achieve}. Suppose one collection of learning rules, one fold scheme, and one tolerance sequence satisfy the hypotheses of Theorem~\ref{thm:achieve} at every $P\in\cP_{\mathrm{ach}}\subseteq\cP$. Suppose also that, for every $P,P'\in\cP$ satisfying~\eqref{eq:agree-others} for some $j$, the one-wrong envelope
\begin{equation}\label{eq:maxarm-onewrong-envelope}
\begin{aligned}
\mathcal E_{P,P'}(X,A,\widetilde W)
\eqdef\sum_{a\in\cA^\ast(X;P)}\Biggl[&
\left|g\bigl(X,a,u_{P',a}(X)\bigr)\right|\\
&+\frac{\1\{A=a\}|g_a^\prime(X)|}{\pi_a^{P'}(X)}
\left\{|h_a(\widetilde W)|+|u_{P',a}(X)|\right\}\Biggr]
\end{aligned}
\end{equation}
satisfies $\E_P[\mathcal E_{P,P'}]<\infty$. Then the full-slot rule $f^\omega$ belongs to $\IFRL$ on $\cP$, and $T_n^\omega\in\TRL(\Psi;\cP_{\mathrm{ach}}\mid\cP)$. For every $P\in\cP_{\mathrm{ach}}$:
\begin{enumerate}[wide=0pt,label={(\roman*)},itemsep=2pt]
\item A fixed oracle map $\lambda^\ast(P)$, extended and normalized on every structural value, gives a member attaining $\Vlow(P)$ at the law $P$.
\item A single rule that reads $w_{f,P}=\sigma^2(\cdot;P)$ from the auxiliary input gives an adaptive member attaining $\Vlow(P)$ at every law where the variance learner is consistent.
\item The working-variance rule satisfies
\[
\Var_P[(f^{\widehat\lambda})^P_{w_{f,P}}]\ge\Vlow(P),
\]
with equality if and only if $\widehat\sigma^2_{\infty,P}\propto\sigma^2(\cdot;P)$ on the active sets.
\end{enumerate}
Under the fixed-oracle construction in~\textup{(i)} or the adaptive construction in~\textup{(ii)},
\begin{equation}\label{eq:achieve-sandwich}
\Var_P[\phi_P^{\bm\lambda^\ast}]
=\inf_{f\in\IFRL}\Var_P\!\bigl[f^P_w\bigr]=\Vlow(P)
\qquad\text{for every measurable }w,
\end{equation}
so the algebraic variance floor is exact and attained.
\end{corollary}

Corollary~\ref{cor:achieve-global} separates exact robustness on $\cP$ from learnability on $\cP_{\mathrm{ach}}$, which need not be closed under conditional-law substitutions. Part~(i) allows a different optimal map for each law, with quantifier order $\forall P\,\exists f_P$. Part~(ii) supplies one adaptive rule, with order $\exists f\,\forall P$ on the subclass where the variance learner is consistent. An oracle map fixed for one law need not be optimal at another.

The oracle-weighted member in~(i) is an element of $\TRL$, the analogue of the estimator of \cite{XuGuo} that uses the optimal tie weight directly; it is infeasible without knowledge of $\sigma^2(\cdot;P)$, which is what part~(ii) remedies.  By Lemma~\ref{lem:func-eq-maxarm}, the infimum in~\eqref{eq:achieve-sandwich} runs over all full-slot maps satisfying condition~(B) on $\cP$. The affine form in Lemma~\ref{lem:func-eq-maxarm} identifies the active-arm structure, so the infimum is not restricted to the weighted family.

Theorem~\ref{thm:achieve} links the estimator to the variance bound through three sources of error: empirical fluctuation, nuisance bias, and active-set error. In particular, the exact conditional bias identity removes every contribution from an error in only one nuisance component, leaving a product of errors from different components; the margin condition controls active-set recovery; and the conditional variance identity gives the stated limit.

\subsection{Delimitation of the exact theory}\label{sec:delimitation}

The explicit variance formulas and the attaining estimator use affine transformations of the arm means. Proposition~\ref{prop:affine-necessary} explains when exact robustness requires this restriction. On a positive-probability region with one optimal arm and a nondegenerate reachable interval of its mean, a nonaffine link makes condition~(B) impossible to satisfy. When several branches remain tied under the allowed perturbations, the constraint instead involves their joint changes, and nonlinear links can satisfy it. The proposition therefore distinguishes the strict-branch restriction from the more general behavior at persistent ties.

\begin{proposition}[Necessity of affineness for exact unbiasedness]
\label{prop:affine-necessary}
Adopt the max-arm structure of Definition~\ref{def:max-arm} with a scalar identifying score $h_a$ (Definition~\ref{def:ident-score}), and suppose Assumptions~\ref{ass:rich} and~\ref{ass:armwise} and Assumption~\ref{ass:meas-main} hold.  Suppose there is a set $\cX_0\subseteq\cX$ with $P_X(\cX_0)>0$ such that for each $x\in\cX_0$:
\begin{enumerate}[wide=0pt,label={\textup{(\roman*)}},itemsep=1pt]
\item $x$ lies in the strict region, with active arm $a^\ast(x)$ ($g(x,a^\ast,u_{P,a^\ast}(x))>g(x,a,u_{P,a}(x))$ for $a\neq a^\ast$); and
\item there is a nondegenerate convex set $J(x)\subseteq\cU_0$ such that, for every $m\in J(x)$,
\[
\bm\delta_m\eqdef\{m-u_{P,a^\ast}(x)\}\bm e_{a^\ast}\in D_x,
\]
and the mean-shifting law $Q_{x,\bm\delta_m}$ of Assumption~\ref{ass:armwise}~(X3) changes only the arm-$a^\ast$ mean, preserving every other arm mean, the remaining primary components, and $\kappa(P)(x)$; the model is closed under conditional mixtures of these mean-shifted laws; and
\item $g(x,a^\ast(x),\cdot)$ is continuously differentiable and fails to be affine on $J(x)$.
\end{enumerate}
Then $\IFRL=\varnothing$ and consequently $\TRL(\Psi;\eta,\kappa,\Wext)=\varnothing$.
\end{proposition}

Proposition~\ref{prop:affine-necessary} shows that the affine-link requirement is necessary for a nonempty exact class on a positive-probability reachable strict branch. For $|\cA|=1$, it applies to the scalar target $\E_{P_X}[g\{X,a_0,u_{P,a_0}(X)\}]$. Only the model-reachable range matters: the link can be nonlinear elsewhere. Along a persistent tie, the exact perturbation identity remains the governing condition; for affine links it reduces to the single linear constraint in~\eqref{eq:lem-kink}.

For example, with a single Bernoulli outcome and target $\Psi(P)=\{\E_PY\}^2$, declaring only the mean and no structural information would require $(1-p)f(0,e)+pf(1,e)=p^2$ for every $p$. The left side is affine in $p$, so no exactly robust map exists on an interval of means. Replacing exact robustness by first-order orthogonality would define a different class and a separate variance-minimization problem.


\subsection{Beyond arm partition: correlated candidate scores}\label{sec:cov-extended}

The harmonic form of Theorem~\ref{thm:kink} uses residuals $\1\{A=a\}\{h_a-u_{P,a}\}$ whose conditional cross-covariances vanish for distinct arms. When candidate scores are all observed on the same residual data, their full covariance matrix determines the corresponding variance minimization. We develop this shared-data counterpart using a fixed nuisance declaration; its rigidity argument differs from the arm-partition argument because the observation supports and admissible nuisance substitutions differ.

\begin{definition}[Min-of-linear structure]\label{def:min-of-linear}
The functional $\Psi$ has \emph{min-of-linear structure} if $\psi(x,u)=\min_{k=1,\dots,K}u_k$, and there exist measurable identifying scores $h_k:\cX\times\cW\to\R$ such that
\begin{equation}\label{eq:min-lin-id}
u_{P,k}(x)=\E_{Q_x^P}[h_k(x;W)]\qquad\text{for }k=1,\dots,K,
\end{equation}
with finite second moments under $Q_x^P$.
\end{definition}

Throughout this subsection, fix a declaration $(\eta,\kappa,\Wext)$ as in Definition~\ref{def:nuisance}, with arbitrary component count $d$. One component, indexed by $j_\circ$, determines the mean vector through a measurable map $H_{j_\circ}:\cN_{0,j_\circ}\to\R^K$, so that $H_{j_\circ}\{\eta_{j_\circ}(P)(x)\}=\bm u_P(x)$. The class $\IFRL$ and its floor $\Vmin$ refer to this declaration, including substitutions in every one of its components. The perturbation conditions below hold the other components and the structural value fixed while varying this mean component.

A canonical declaration is $d=1$, $\eta_1(P)=\bm u_P$, $\kappa(P)(x)=\argmin_l u_{P,l}(x)$, and $\Wext$ the Polish space of positive-semidefinite $K\times K$ matrices used as working covariances. Here $H_1$ is the identity and the structural value is the active set itself. The explicit weighted construction below also permits a more detailed structural value $k$ from which a measurable map $\cA^\ast(k)$ recovers the active set. The lower-bound argument applies to any fixed declaration satisfying the stated perturbation conditions. A max-of-linear target is handled by negating the scores and the target.

\begin{definition}[Conditional covariance and active set]\label{def:cov-active}
For a min-of-linear structure, define the conditional covariance matrix
\begin{equation}\label{eq:Sigma-def}
\bm\Sigma(x;P)\eqdef\Cov_{Q_x^P}\!\bigl(\bm h(x;W)\bigr)
\in\R^{K\times K},\qquad
\bm h(x;W)\eqdef(h_1(x;W),\dots,h_K(x;W))^\top,
\end{equation}
and the active set
\begin{equation}\label{eq:Aast-min}
\cA^\ast(x;P)\eqdef\bigl\{k\in\{1,\dots,K\}:u_{P,k}(x)=\min_l u_{P,l}(x)\bigr\}.
\end{equation}
\end{definition}

\begin{assumption}[Shared-score perturbation conditions]
\label{ass:shared-perturb}
For the fixed declaration and index $j_\circ$ above, every $P\in\cP$ satisfies the following for $P_X$-a.e.\ $x$.
\begin{enumerate}[wide=0pt,label={(COV\arabic*)},itemsep=2pt]
\item Fixed-mean completeness. For every $\epsilon\in(0,1)$ and every finitely supported probability law $\nu$ on $\supp Q_x^P$ with $\E_\nu[\bm h(x;W)]=\bm u_P(x)$, the law $(1-\epsilon)Q_x^P+\epsilon\nu$ belongs to $\mathfrak Q_x^P$.

\item Structure-preserving directions. There is a correspondence $x\mapsto D_x\subseteq\operatorname{range}\bm\Sigma(x;P)$ with Borel graph and $\bm0\in D_x$. Writing
\[
T_0(x)\eqdef\{\bm\delta\in\operatorname{range}\bm\Sigma(x;P):
\bm\delta_{\cA^\ast(x;P)}=\bm0\},
\]
and
\[
T_{\cA^\ast}(x)\eqdef\{\bm\delta\in\operatorname{range}
\bm\Sigma(x;P):\delta_i=\delta_j\text{ for all }
i,j\in\cA^\ast(x;P)\},
\]
$D_x$ contains a relative neighbourhood of $\bm0$ in $T_0(x)$ and a relative neighbourhood of $\bm0$ in $T_{\cA^\ast}(x)$ on which the inactive coordinates remain strictly inactive.

\item Mean-shifting conditional laws. For every $\bm\delta\in D_x$ there are a conditional law $Q_{x,\bm\delta}$ and a compatible value $e_{j_\circ}(x,\bm\delta)\in\cN_{0,j_\circ}$ such that
\[
Q_{x,\bm\delta}\ll Q_x^P,\qquad
\E_{Q_{x,\bm\delta}}[\bm h(x;W)]=\bm u_P(x)+\bm\delta,
\qquad Q_{x,\bm0}=Q_x^P,
\]
and $H_{j_\circ}\{e_{j_\circ}(x,\bm\delta)\}=\bm u_P(x)+\bm\delta$, with $e_{j_\circ}(x,\bm0)=\eta_{j_\circ}(P)(x)$. The correspondence $(x,\bm\delta)\mapsto(Q_{x,\bm\delta},e_{j_\circ}(x,\bm\delta))$ has Borel graph and finite model cost in the sense of Assumption~\ref{ass:meas-main}. Each universally measurable selection with integrable cost defines a law $P_{\bm\delta}\in\cP$ with marginal $P_X$, primary component $\eta_{j_\circ}(P_{\bm\delta})(x)=e_{j_\circ}(x,\bm\delta)$, and all other primary components and $\kappa$ equal to those under $P$. Selections of nonintegrable cost are first restricted to sets of bounded cost, with the baseline kernel and nuisance value elsewhere. The fixed-mean kernels in \textup{(COV1)} satisfy the same requirement.
\end{enumerate}
\end{assumption}

Lemma~\ref{lem:func-eq-cov} establishes the following shared-score rigidity property under Assumptions~\ref{ass:rich}, \ref{ass:meas-main}, and~\ref{ass:shared-perturb}: every admissible evaluation has the form $\alpha+\bm\beta^\top\bm h$ outside one map-dependent null set. The coefficient vector is defined modulo $\ker\bm\Sigma(x;P)$, with a corresponding adjustment of the intercept. When the active covariance block is positive definite, there is an active-supported representative with coefficients summing to one. Its conditional variance is $\bm\beta^\top\bm\Sigma\bm\beta$.

For a declaration with $\cA^\ast\{\kappa(P)(x)\}=\cA^\ast(x;P)$, the full-domain weighted rule is especially simple. For a measurable, possibly signed weight map with support on $\cA^\ast(k)$ and $\sum_l\omega_l(x,k,v)=1$, define
\begin{equation}\label{eq:cov-full-main}
f_{\rm cov}^{\omega}(x,W,e,k,v)
\eqdef\sum_{l\in\cA^\ast(k)}\omega_l(x,k,v)\,h_l(x;W).
\end{equation}
This is the weighted mean plug-in plus its residual correction: the two occurrences of $H_{j_\circ}(e_{j_\circ})$ cancel. At true structural values its conditional expectation is $\min_l u_{P,l}(x)$, independently of all primary inputs. To obtain a member of $\IFRL$, these identities must also meet the integrability requirement in Definition~\ref{def:ralu}. A sufficient envelope is $\sup_{x,k,v}\|\omega(x,k,v)\|_1\le L<\infty$ and $\E_P[\max_l|h_l(X;W)|]<\infty$ for every $P\in\cP$; another full-domain envelope may be used instead. A cap establishes unrestricted optimality only if it contains the oracle weights.

When $\bm\Sigma_{\cA^\ast(x)}(x;P)$ is positive definite, the unit-sum quadratic program has oracle weights
\begin{equation}\label{eq:cov-lambda}
\bm\lambda^\ast(x)=
\frac{\bm\Sigma_{\cA^\ast(x)}(x;P)^{-1}\,\bm1_{\cA^\ast(x)}}
{\bm1_{\cA^\ast(x)}^\top\,\bm\Sigma_{\cA^\ast(x)}(x;P)^{-1}\,\bm1_{\cA^\ast(x)}},
\end{equation}
where $\bm1_{\cA^\ast(x)}$ is the all-ones vector of length $|\cA^\ast(x)|$ and the weights are zero off the active set. Denote the proposed oracle evaluation by
\begin{equation}\label{eq:cov-IF}
\phi_{P,{\rm cov}}^\ast(x,W)=\min_l u_{P,l}(x)
+\sum_{k\in\cA^\ast(x)}\lambda_k^\ast(x)\bigl[h_k(x;W)-u_{P,k}(x)\bigr].
\end{equation}

\begin{theorem}[Covariance-extended closed form]\label{thm:cov-kink}
For the fixed declaration above, assume Definition~\ref{def:min-of-linear} and Assumptions~\ref{ass:rich}, \ref{ass:meas-main}, and~\ref{ass:shared-perturb}. Fix $P\in\cP$ and suppose $\bm\Sigma_{\cA^\ast(x)}(x;P)$ is positive definite for $P_X$-a.e.\ $x$.
\begin{enumerate}[wide=0pt,label={\textup{(\roman*)}},itemsep=2pt]
\item For $P_X$-a.e.\ $x$,
\begin{equation}\label{eq:cov-lower}
\Vmin(x;P)\ge
\frac{1}{\bm1_{\cA^\ast(x)}^\top\,\bm\Sigma_{\cA^\ast(x)}(x;P)^{-1}\,\bm1_{\cA^\ast(x)}}.
\end{equation}
\item If one full-domain map $f^\ast\in\IFRL$ and one measurable auxiliary map $w_\ast$ satisfy $(f^\ast)^P_{w_\ast}=\phi_{P,{\rm cov}}^\ast$ $P$-a.s.\ with finite variance under $P$, then, for $P_X$-a.e.\ $x$,
\begin{equation}\label{eq:cov-Vmin}
\Vmin(x;P)=
\frac{1}{\bm1_{\cA^\ast(x)}^\top\,\bm\Sigma_{\cA^\ast(x)}(x;P)^{-1}\,\bm1_{\cA^\ast(x)}}.
\end{equation}
The oracle evaluation attains the conditional floor and the total variance floor $\Vlow(P)$.
\end{enumerate}
\end{theorem}

Theorem~\ref{thm:cov-kink} identifies the conditional covariance matrix as the quantity governing the efficiency comparison. Equality additionally requires the global integrability of an oracle rule such as~\eqref{eq:cov-full-main}; writing down its evaluation at $P$ alone does not establish this. Equal weighting gives conditional variance $|\cA^\ast|^{-2}\bm1^\top\bm\Sigma_{\cA^\ast}\bm1$, which is no smaller than the quadratic-program value in~\eqref{eq:cov-Vmin}. Equality holds exactly when $\bm\Sigma_{\cA^\ast}\bm1$ is proportional to $\bm1$. Equal weighting therefore minimizes variance in the unit-sum family precisely under this covariance condition. The Balke--Pearl application can yield a strict improvement, while the two active $L^1$ calibration scores cancel and make equal weighting optimal at its kink.

For a rank-deficient active block, the quadratic-program value is $(\bm1^\top\bm\Sigma_{\cA^\ast}^{+}\bm1)^{-1}$ when $\bm1$ belongs to its range and is zero otherwise; the latter branch uses a unit-sum vector in the kernel. Under the declaration and perturbation conditions of Theorem~\ref{thm:cov-kink}, with positive semidefiniteness replacing positive definiteness, this program value equals the RALU floor if its measurable minimizing evaluation is the truth evaluation of one full-domain map in $\IFRL$ and has finite variance under $P$. On the strict region $\{|\cA^\ast(x)|=1\}$, the positive-definite formula reduces to $\Vmin(x;P)=\Sigma_{k^\ast k^\ast}(x;P)$ under the equality conditions of Theorem~\ref{thm:cov-kink}. Proposition~\ref{prop:cov-singular} proves the singular extension. Covariance affects the optimum only when several scores are active.

\section{A bridge to classical efficiency}\label{sec:bridge}

We now connect the variational bound to classical semiparametric efficiency. At a tie, a perturbation that moves the optimal branches apart can prevent pathwise differentiability. Along directions in which these branches have the same first-order change, however, the target has a linear derivative. Conditional branch gradients represent that derivative, and projection onto the tangent space of feasible perturbations gives the associated classical efficiency bound. In the canonical models below, this bound equals the variational optimum. We work at a law $P$: the projection result needs square-integrable branch gradients, and its statistical interpretation uses the expansions and feasible local experiments specified next.

Write $O=(X,W)$ for the observation and use $J$ for the number of branches in this section. In the max-arm setting, $J=|\cA|$ and $b_a(P)(x)=g\{x,a,u_{P,a}(x)\}$, with the arms relabeled as $1,\ldots,J$ when needed. For the min-of-linear setting, apply this section to the transformed target $\widetilde\Psi=-\Psi$ and branches $b_j(P)=-u_{P,j}$; their conditional gradients are $-(h_j-u_{P,j})$, and the resulting target gradient is negated again to recover the gradient for $\Psi$. Hence the variance bound is unchanged by this transformation. Within the section, $\Psi$ denotes the target in max form.

Set $m_P(x)=\max_{j\le J}b_j(P)(x)$, $q_P(X)=m_P(X)-\Psi(P)$, and $\cA^*(x)=\argmax_j b_j(P)(x)$. For every law $R$ in the model the target is
\begin{equation}\label{eq:bridge-target}
\Psi(R)=\E_R\{m_R(X)\},\qquad m_R(x)=\max_{j\le J}b_j(R)(x),
\end{equation}
so that along a path both the covariate marginal and the branches change.  Consider paths $dP_t=(1+ts)\,dP$ with bounded mean-zero $s$.
\begin{enumerate}[wide=0pt,label={(G\arabic*)},itemsep=2pt]
\item $q_P\in L^2(P)$, and there are $\zeta_j\in L^2(P)$ with $\E_P[\zeta_j\mid X]=0$, $j\le J$; write $\zeta=(\zeta_1,\ldots,\zeta_J)^\top$. The linear space of first-order active-branch constraints is
\begin{equation}\label{eq:bridge-H}
 \mathcal H=\left\{s\in L_0^2(P):
 \E_P[(\zeta_i-\zeta_j)s\mid X]=0
 \text{ on }\{i,j\in\cA^*(X)\},\ i,j\le J\right\}.
\end{equation}
Here $L_0^2(P)=\{s\in L^2(P):\E_Ps=0\}$; $\Pi_{\mathcal H}$ denotes orthogonal projection in $L^2(P)$.
\item On the tilts considered in \textup{(G3)},
\begin{equation}\label{eq:bridge-branch-expansion}
 b_j(P_t)=b_j(P)+t\E_P[\zeta_j s\mid X]+o_{L^1(P_X)}(|t|),
 \qquad dP_t=(1+ts)\,dP.
\end{equation}
\item A linear space $\mathcal S$ of bounded scores in the constraint space $\mathcal H$ of \eqref{eq:bridge-H} is specified such that, for every finite subset $s_1,\ldots,s_\ell\in\mathcal S$, the law $dP_t=(1+\sum_l t_ls_l)dP$ belongs to the model for all sufficiently small $t\in\R^\ell$ and the branch expansion holds with remainder $o_{L^1(P_X)}(\|t\|)$.
\end{enumerate}

The identity that links this space to efficiency is the following. For a bounded score $s\in\mathcal H$ satisfying~(G2), the two-sided derivative of $t\mapsto\Psi(P_t)$ at $t=0$ exists and equals
\begin{equation}\label{eq:bridge-derivative}
\frac{d}{dt}\Psi(P_t)\Big|_{t=0}
=\E_P[q_P(X)\,s]+\E_P\bigl[\E_P[\zeta_{j_0}s\mid X]\bigr]
=\E_P\bigl[(q_P+\zeta_{j_0})\,s\bigr],
\end{equation}
where $j_0(X)$ is any measurable selection from $\cA^*(X)$.  The first term is the change of the covariate marginal, whose density ratio along the path is exactly $1+t\E_P[s\mid X]$; the second is the common first-order change of the active branches, which does not depend on the selected branch precisely because $s\in\mathcal H$.  For a feasible bounded tilt satisfying~\eqref{eq:bridge-branch-expansion} whose score lies outside $\mathcal H$, the one-sided derivatives are $\E_P[q_Ps]+\E_P[\max_{j\in\cA^*(X)}\E_P[\zeta_js\mid X]]$ from the right and the corresponding minimum from the left, which differ.

Among the feasible bounded tilts satisfying~\eqref{eq:bridge-branch-expansion}, the two-sided derivative therefore exists exactly for those with score in $\mathcal H$.  By~\eqref{eq:bridge-derivative}, every $D_\lambda=q_P+\lambda^\top\zeta$ with unit-sum active weights and $D_\lambda\in L^2(P)$ is a gradient of $\Psi$ relative to $\mathcal H$, and the gradient that lies in $\mathcal H$ itself is their common projection.

Let $\bm\Gamma(X)=\E_P[\zeta\zeta^\top\mid X]$ and $\bm\Gamma_{\rm act}(X)=\bm\Gamma_{\cA^*(X)}(X)$. This is the covariance of the conditional branch gradients; for the fixed shared scores of Section~\ref{sec:cov-extended}, it equals $\bm\Sigma(X;P)$, while for max-arm gradients it includes the inverse propensity factors.  Choose measurable unit-sum active weights $\lambda^*$ minimizing $\lambda^\top\bm\Gamma_{\rm act}\lambda$, and set $D^*=q_P+\lambda^{*\top}\zeta$.  Weights outside the active set are zero.  The notation allows signed weights and singular $\bm\Gamma_{\rm act}$.  The statistical experiment is the one specified in (G3), with tangent closure $\mathcal T=\overline{\mathcal S}^{\,L^2(P)}$.

\begin{theorem}[Branch-gradient bridge]\label{thm:bridge}
Assume \textup{(G1)}.
\begin{enumerate}[wide=0pt,label={\textup{(\roman*)}},itemsep=2pt]
\item \textup{Projection.}  For $j_0(X)=\min\cA^*(X)$, $\Pi_{\mathcal H}(q_P+\zeta_{j_0})=D^*$; the same projection holds for every square-integrable active unit-sum combination.
\item \textup{Value.} The squared norm is
\begin{equation}\label{eq:bridge-value}
\begin{aligned}
 \|D^*\|_2^2&=\Var_P\{m_P(X)\}+\E_P[v_{\rm br}(\bm\Gamma_{\rm act}(X))],\\
 v_{\rm br}(M)&=\begin{cases}
 (\1^\top M^+\1)^{-1},&\1\in\operatorname{range}(M),\\
 0,&\1\notin\operatorname{range}(M),
 \end{cases}
\end{aligned}
\end{equation}
where $M$ is positive semidefinite.
\item \textup{Classical interpretation.}  Under \textup{(G2)} and \textup{(G3)}, the canonical gradient in the specified local experiments is $\Pi_{\mathcal T}D^*$.  If $D^*\in\mathcal T$, it is $D^*$, and the convolution lower bound for regular estimators is $\|D^*\|_2^2$.
\end{enumerate}
\end{theorem}

Theorem~\ref{thm:bridge} identifies the covariance-minimizing combination as the projection of any active branch gradient. Square integrability suffices for this projection; neither the branch gradients nor the minimizing weights need be bounded. Bounded scores are dense in $\mathcal H$ under the conditional constraints, so $\mathcal T=\mathcal H$ when their tilts are feasible. In a constrained model, the feasible tangent space can be smaller. Equality of active-branch derivatives is a condition at $t=0$; the affine max-arm paths below also preserve the active set at nonzero $t$.

For the affine max-arm problem, the branch gradients are
\begin{equation}\label{eq:bridge-maxarm-gradient}
 \zeta_a(O)=\frac{\1\{A=a\}g'_a(X)}{\pi_a^P(X)}
 \{h_a(\widetilde W)-u_{P,a}(X)\}.
\end{equation}
For a binary tie with $g'_a\equiv1$, so that $\zeta_a=\1\{A=a\}\{Y-\mu_a(X)\}/\pi_a(X)$, the active block is $\bm\Gamma_{\rm act}=\operatorname{diag}(\sigma_1^2/\pi_1,\sigma_0^2/\pi_0)$, and the constraint defining $\mathcal H$ reads $\E_P[\zeta_1s\mid X]=\E_P[\zeta_0s\mid X]$, that is, $\dot\mu_1(x)=\dot\mu_0(x)$: on the tie set the two arms carry information about one common perturbation of the tied mean. Weighting them by $\pi_a/\sigma_a^2$ gives $\lambda^*$ and $v_{\rm br}(\bm\Gamma_{\rm act})=(\pi_1/\sigma_1^2+\pi_0/\sigma_0^2)^{-1}$, the expression in~\eqref{eq:intro-binary}; Theorem~\ref{thm:bridge} extends this to signed weights and singular blocks.

Here a direct link with exact robustness is available, through the following closure condition on the model.
\begin{enumerate}[wide=0pt,label={(L)},itemsep=2pt]
\item Call a score $s=\alpha(X)+r(O)$ \emph{admissible} if $\alpha$ is bounded and mean zero, $r$ is bounded, and \textup{(L1)}~$\E_P[r\mid X,A]=0$; \textup{(L2)}~$d_a^r(X)\eqdef\E_P[\zeta_ar\mid X]$ takes a common value over $a\in\cA^*(X)$; and \textup{(L3)}~$r$ vanishes outside a covariate set on which $\max_a\E_P[|\zeta_a|\mid X]$ is bounded and the smallest strictly positive inactive gap $\min_{a\notin\cA^*(X)}\{m_P(X)-b_a(P)(X)\}$ is bounded away from zero (an empty minimum being $+\infty$).  For every finite family $s_1,\dots,s_\ell$ of admissible scores there is a neighbourhood of $0\in\R^\ell$ on which the law $dP_t=(1+\sum_{l=1}^{\ell}t_ls_l)\,dP$ belongs to $\cP$.
\end{enumerate}
Write $\mathcal S_0$ for the linear space of admissible scores in \textup{(L)}, and define $\mathcal H_0=\{s\in\mathcal H:\E_P[s\mid X,A]=\E_P[s\mid X]\}$. The bounded-tilt experiment generated by $\mathcal S_0$ uses the finite-dimensional paths specified in \textup{(L)}. These paths keep the propensity and active set fixed and vary only the primary mean component. Condition \textup{(L)} holds in the bounded-outcome model of Section~\ref{sec:standing-model}. Corollary~\ref{cor:bridge-maxarm} identifies the tangent closure and efficiency bound of this experiment.

For a fixed measurable auxiliary map $w:\cX\to\Wext$, write $f_P=f_w^P$. The following corollary uses the local paths in \textup{(L)} rather than the completeness and bridge-reachability assumptions of Section~\ref{sec:func-eq}.

\begin{corollary}[Equality for affine max-arm models]\label{cor:bridge-maxarm}
Fix $P\in\cP$ and the max-arm declaration of Section~\ref{sec:max-arm}, with identifying scores~\eqref{eq:ident-score}. Suppose $g(x,a,u)=g'_a(x)u+\rho(x,a)$, $g'_a(x)\ne0$, $\pi_a^P(x)>0$ for all arms, and $0<\sigma_a^2(x;P)<\infty$ on the active set, for $P_X$-a.e.\ $x$. Assume $q_P\in L^2(P)$ and $\zeta_a$ in~\eqref{eq:bridge-maxarm-gradient} belongs to $L^2(P)$ for every arm. Assume condition~\textup{(L)} and the marginal-reweighting condition \textup{(M1)} of Assumption~\ref{ass:rich}.
\begin{enumerate}[wide=0pt,label={\textup{(\roman*)}},itemsep=2pt]
\item Every $f\in\IFRL$ and measurable $w:\cX\to\Wext$ satisfy, for $P_X$-a.e.\ $x$ and allowing infinite conditional variance,
\[
 \Var_P(f_P\mid X)\ge
 \left\{\sum_{a\in\cA^*(X)}
 \frac{\pi_a^P(X)}{(g'_a(X))^2\sigma_a^2(X;P)}\right\}^{-1}.
\]
If one member of $\IFRL$ realizes the evaluation $\Psi(P)+D^*$, then the conditional essential infimum of the members' variances and the integrated variational bound are attained, and the latter equals~\eqref{eq:bridge-value}.
\item The bounded-tilt experiment generated by $\mathcal S_0$ has tangent closure $\mathcal H_0$ and canonical gradient $D^*$. Its convolution lower bound is $\|D^*\|_2^2$.
\item An estimator that is asymptotically linear at $P$ with a square-integrable truth evaluation from $\IFRL$ is regular in this experiment.
\end{enumerate}
\end{corollary}

Corollary~\ref{cor:bridge-maxarm} equates the variational and classical bounds in an experiment that holds the propensity and active set fixed. Its conditional lower bound applies to every admissible evaluation, including those with inactive-arm residuals or unequal arm-level intercepts. Thus the local experiment supplies a lower bound for the full class, rather than only for the displayed weighted family.

Corollary~\ref{cor:bridge-maxarm} concerns the affine max-arm structure.  For other declarations, Theorem~\ref{thm:bridge} identifies the canonical gradient in the feasible experiment as $\Pi_{\mathcal T}D^*$ under its stated path conditions, while the exact-class floor depends on the finite-substitution constraints and on realization; in constrained persistent-tie models it can be strictly larger than $\|D^*\|_2^2$.

For a general declaration, the regularity conclusion holds when feasible, exactly $\kappa$-preserving, one-component substitution scores span the stated tangent space. For Balke--Pearl branches, the centered doubly robust signals give the $\zeta_j$ in Theorem~\ref{thm:bridge}, with the classical interpretation determined by the feasible paths in the declared cell model. Section~\ref{sec:card-data} uses their unit-sum combinations for the empirical variance comparison. For a minimum, both the branches and their conditional gradients are negated, as specified at the start of this section.

\section{Numerical studies}\label{sec:numerics}

The numerical study evaluates the efficiency gain from adapting the weights of tied treatment arms. Our primary comparison is between the proposed RALU estimator and the nonparametric softmax estimator of \citet{Whitehouse2025Softmax}. Their Theorem~3.3 identifies the equal-weight asymptotic variance, and the comparison after Corollary~\ref{rem:otr-recovery-main} predicts a strict gain for RALU when tied arms carry unequal information. An additional equal-weight estimator uses the same estimated active set as RALU to isolate the effect of variance-adaptive weighting. The adaptive procedure of \citet{XuGuo} provides an optimality benchmark in the homoscedastic binary control.

\subsection{Design and analytical comparison}\label{sec:num-design}

Let $\cA=\{0,1\}$.  Conditional on $X$ and $A=a$, let $Y\sim\cN\{\mu_a(X),\sigma_a^2\}$.  Set
\[
\mu_0(x)=1,\qquad \mu_1(x)=1+1.2(x-q)_+,
\qquad \sigma_0^2=1/4,\qquad \sigma_1^2=1.
\]
Given $X=x$, let $A$ be Bernoulli with $\pi_1(x)=0.40+0.20x$ and $\pi_0(x)=1-\pi_1(x)$.  The covariate law has density
\[
 f_q(x)=\1\{0\le x<q\}
 +\frac{2(x-q)}{1-q}\1\{q\le x\le1\}.
\]
Thus the arms are tied on $\{X\le q\}$, with probability $q$, and the positive treatment-effect gap has quadratic margin near zero.  We consider $q=0.30$ and the higher-tie-mass value $q=0.80$. These designs combine positive-probability treatment ties with unequal arm information, the structure for which the theory predicts a strict RALU efficiency gain.  An information-equality control retains the $q=0.80$ outcome functions and covariate law but sets $\pi_1\equiv1/2$ and $\sigma_0^2=\sigma_1^2=1$.

Write $v_a(x)=\sigma_a^2/\pi_a(x)$.  On a binary tie, the RALU and Whitehouse limiting conditional residual variances are
\[
 V_{\rm R}(x)=\frac{v_0(x)v_1(x)}{v_0(x)+v_1(x)},\qquad
 V_{\rm eq}(x)=\frac{v_0(x)+v_1(x)}4,
\]
so that
\begin{equation}\label{eq:num-binary-gap}
 V_{\rm eq}(x)-V_{\rm R}(x)
 =\frac{\{v_0(x)-v_1(x)\}^2}{4\{v_0(x)+v_1(x)\}}\geq0.
\end{equation}
The inequality is strict on any positive-probability tie subset where $v_0\ne v_1$, and equality holds exactly when the two inverse-information terms agree.  Let $\mathcal V_{\rm R}$ and $\mathcal V_{\rm eq}$ denote the corresponding total variances after adding the common variance of the conditional target and integrating over $X$.  The equal-weight limiting representation has total variance $\mathcal V_{\rm eq}$. The ratio $\mathcal V_{\rm eq}/\mathcal V_{\rm R}$ is the theoretical benchmark in Table~\ref{tab:otr-efficiency}.

\begin{table}[!htb]
\centering
\small
\caption{Whitehouse--RALU variance comparison at $n=64{,}000$ with $1{,}000$ paired replications. The limit ratio is $\mathcal V_{\rm eq}/\mathcal V_{\rm R}$. The empirical columns compare the implemented Whitehouse estimator (W) and the additional equal-weight estimator (eq) with RALU (R), with paired $95\%$ Monte Carlo intervals. The final column gives coverage of nominal $95\%$ intervals for the unsmoothed optimal value.}
\label{tab:otr-efficiency}
\begin{tabular}{@{}lcccc@{}}
\toprule
design & limit ratio & $\widehat V_{\rm W}/\widehat V_{\rm R}$ & $\widehat V_{\rm eq}/\widehat V_{\rm R}$ & coverage R/W\\
\midrule
$q=.30$, unequal information & 1.065739 & 1.077 $[1.041,1.116]$ & 1.061 $[1.030,1.094]$ & 0.951/0.948\\
$q=.80$, unequal information & 1.309629 & 1.390 $[1.299,1.483]$ & 1.315 $[1.245,1.390]$ & 0.939/0.936\\
$q=.80$, information equality & 1.000000 & 1.139 $[1.099,1.183]$ & 1.0007 $[0.9997,1.0018]$ & 0.931/0.921\\
Xu--Guo control, $q=.80$ & 1.248895 & 1.509 $[1.400,1.631]$ & 1.254 $[1.183,1.326]$ & 0.952/0.924\\
\bottomrule
\end{tabular}
\end{table}

\subsection{Paired implementation and interpretation}\label{sec:num-results}

The oracle experiment inserts the true active set and nuisances and isolates the variance difference in~\eqref{eq:num-binary-gap}. All three feasible estimators use the same observations, two-fold partition, outcome regressions, and propensity estimates. RALU and the additional equal-weight estimator also share the tolerance active set and conditional-variance estimates. Their common learner fits each arm regression separately over $\{1,(X-q)_+\}$ and the propensity over $\{1,X\}$, with a constant residual variance within each arm. The RALU active-set tolerance is $1.2n^{-1/4}$.

For the Whitehouse estimator, write
\begin{equation}\label{eq:num-whitehouse-smoothing}
\begin{aligned}
 w_a^\beta(\bm u)&=\frac{\exp(\beta u_a)}{\sum_b\exp(\beta u_b)},
 &s_\beta(\bm u)&=\sum_a w_a^\beta(\bm u)u_a,\\
 d_a^\beta(\bm u)&=w_a^\beta(\bm u)\bigl[1+\beta\{u_a-s_\beta(\bm u)\}\bigr].
\end{aligned}
\end{equation}
Following Proposition~3.2 and equation~(4) of \citet{Whitehouse2025Softmax}, each held-out observation contributes the score
\begin{equation}\label{eq:num-whitehouse-score}
 s_{\beta_n}\{\widehat{\bm\mu}(X)\}
 +\sum_{a=0}^1 d_a^{\beta_n}\{\widehat{\bm\mu}(X)\}
 \frac{\1\{A=a\}}{\widehat\pi_a(X)}\{Y-\widehat\mu_a(X)\}.
\end{equation}
We average these scores over the held-out folds and use their empirical variance for the standard error. The smoothing sequence is $\beta_n=1.5n^{1/6}\log\log n$, following the growth rule in their Section~5.2, equation~(11), with margin exponent two. The derivative coefficients in~\eqref{eq:num-whitehouse-smoothing} converge to equal weights on the truly optimal arms and zero on the others. With the root-$n$ nuisance fits and finite Gaussian moments in this panel, $\beta_n/\sqrt n\to0$ controls the nuisance remainder and $\sqrt n\beta_n^{-3}\to0$ controls the smoothing bias, giving the equal-weight asymptotic variance used above.

In the terms of Definition~\ref{def:ralu}, the learnable subclass $\cP_{\mathrm{ach}}$ in this study consists of laws with the displayed covariate density, a linear propensity $\pi_1(x)=\alpha_0+\alpha_1x$ taking values in the learner's clipping interval $[0.05,0.95]$ on $[0,1]$, and arm regressions $\mu_a(x)=\beta_0+\beta_{1,a}(x-q)_+$ with a common intercept and $\beta_{1,1}\ne\beta_{1,0}$. The conditional outcome laws are Gaussian with one of the variance pairs used in the four designs. This subclass includes the two unequal-information designs and the two controls.

The common-intercept specification makes the treatment effect $(\beta_{1,1}-\beta_{1,0})(x-q)_+$, so $P_X(0<\Delta_+\le t)$ is proportional to $t^2$ and the margin exponent is $\gamma=2$. Least-squares fits over $\{1,(X-q)_+\}$ and $\{1,X\}$ converge in supremum norm at rate $n^{-1/2}$. Consequently $\tau_n=1.2n^{-1/4}$ lies in the window $\|\widehat g-g\|_\infty\ll\tau_n\ll n^{-1/6}$ of Assumption~\ref{ass:achieve-reg}. Regressions with distinct intercepts can instead cross where the covariate density is positive and have margin exponent one; the common-intercept restriction specifies the subclass for which the stated tolerance is justified.

We use $n=4{,}000$, $16{,}000$, and $64{,}000$, with $2{,}000$, $2{,}000$, and $1{,}000$ Monte Carlo replications, respectively.  The RALU, Whitehouse, and additional equal-weight comparisons reuse the same observations and fold assignment.  We report $n$ times the empirical variance, root-$n$ scaled bias, estimated-standard-error calibration, and Wald coverage; paired resampling quantifies Monte Carlo uncertainty in empirical variance ratios.  These diagnostics measure finite-sample approximation to the exact benchmarks.

At $n=64{,}000$, the empirical Whitehouse/RALU variance ratios are $1.077$ for $q=.30$ and $1.390$ for $q=.80$, with paired Monte Carlo intervals $[1.041,1.116]$ and $[1.299,1.483]$. The corresponding RALU/Whitehouse coverage pairs are $.951/.948$ and $.939/.936$. Their root-$n$ scaled biases are $-.021/-.121$ and $-.183/-.294$, respectively. The direct comparison therefore shows lower variance for RALU in both unequal-information designs, with similar coverage. The larger gain at $q=.80$ is consistent with the greater probability of a treatment tie.

The additional equal-weight/RALU ratios are $1.061$ and $1.315$ in these two designs, close to the asymptotic ratios $1.065739$ and $1.309629$. This comparison isolates the gain from weighting the active arms by their conditional information. In the information-equality control the theoretical ratio is one and the additional equal-weight/RALU ratio is $1.0007$. The implemented Whitehouse/RALU ratio is $1.139$ at this sample size, with coverage $.921$ for Whitehouse and $.931$ for RALU; finite-sample smoothing and nuisance estimation affect this ratio even when the limiting representations coincide.

The last row is the Xu--Guo positive control. Their adaptive procedure uses the same simulated observations with its prescribed nested sample splits and role-specific nuisance fits.  Under homoscedasticity, their boundary weight $t_0(x)=\pi_1(x)$ equals the RALU optimizer.  With $q=0.80$ and $\pi_1=0.75$, the RALU and Xu--Guo limiting representations therefore share the exact variance $1.0714026667$, while equal-face weighting has variance $1.3380693333$.  The Xu--Guo procedure therefore serves as a positive control: both estimators target the same optimum, and the comparison checks that the feasible RALU estimator reaches it. At $n=64{,}000$, the end-to-end Xu--Guo estimator has scaled variance $1.126$, root-$n$ bias $-0.011$, coverage $.951$, and mean selected bandwidth $0.126$. Its finite-sample variance ratio relative to feasible RALU is $1.101\,[1.058,1.147]$.  The identity $t_0=\lambda_1^\ast$ establishes their common asymptotic benchmark; the reported ratio describes their finite-sample variances.

The numerical results connect the theoretical comparison with the implemented Whitehouse estimator and show how the gain changes with the probability and information structure of treatment ties. We next examine covariance weighting in an observational Balke--Pearl application, where both the active set and its score covariance are learned from data.

\section{Empirical application: covariate-assisted Balke--Pearl bounds}
\label{sec:card}
\label{sec:card-data}

We analyze $3{,}010$ observations from the National Longitudinal Survey of Young Men, introduced by \citet{Card1995}, using the processed replication data of \citet{Wang2017NLSData} and the observational instrumental-variable design of \citet{levis2025covariate}.  The instrument indicates proximity to a four-year college, the exposure indicates education beyond high school, and the binary outcome indicates whether 1976 hourly wage exceeds the fixed full-cohort cutoff of $537.5$ cents.  This cutoff, equal to the median in the processed analysis sample, was fixed before nuisance fitting and cross-fitting.  We retain the supplied survey weights and adjust for age, parental education, southern and metropolitan residence, race, IQ, and the corresponding missingness indicators.  The weighted first stage is $0.118$.  The targets are the survey-weighted covariate-assisted lower and upper Balke--Pearl endpoints, respectively the maximum and minimum of eight conditional branches before covariate aggregation.  Their causal interpretation relies on consistency, conditional instrument independence, exclusion, and positivity; monotonicity is not imposed.

The theory is instantiated as follows.  Write $p_{ay\mid z}(x)=P(A=a,Y=y\mid X=x,Z=z)$ for the conditional cell probabilities and $\rho_z(x)=P(Z=z\mid X=x)$ for the instrument propensity.  Each upper-endpoint branch is an affine function of the cells, $U_k(x)=\sum_{a,y,z}c^{(k)}_{ayz}\,p_{ay\mid z}(x)$ with integer coefficients; for instance $U_2=p_{00\mid1}+p_{11\mid1}$ and $U_1=-p_{01\mid0}+p_{00\mid1}+p_{01\mid1}+p_{11\mid1}$, and the target is $\E_{P_X}[\min_kU_k(X)]$, the lower endpoint being handled by negation.  The full-slot score for branch $k$ is the doubly robust
\begin{equation}\label{eq:card-dr-score}
\varphi_k(O;p',\rho')=U_k(p')+\sum_z\frac{\1\{Z=z\}}{\rho'_z}
\sum_{a,y}c^{(k)}_{ayz}\bigl\{\1\{A=a,Y=y\}-p'_{ay\mid z}\bigr\},
\end{equation}
whose conditional bias is $\sum_z(1-\rho_z/\rho'_z)\{M_{kz}(p')-M_{kz}(p)\}$ with $M_{kz}(p)=\sum_{a,y}c^{(k)}_{ayz}p_{ay\mid z}$; it vanishes when either slot is correct, which is condition~(B) with $d=2$.

At the true nuisance values, the centered residuals $\varphi_k^P-U_k=\sum_z\1\{Z=z\}\{c^{(k)}_{AYz}-M_{kz}\}/\rho_z$ play the role of the centered identifying scores $h_k-u_{P,k}$ in Definition~\ref{def:min-of-linear} and the branch gradients $\zeta_k$ in Section~\ref{sec:bridge}. Their conditional covariance matrix has entries
\[
\widetilde{\bm\Sigma}_{kl}(x)
=\sum_z\frac{1}{\rho_z(x)}
\left\{\sum_{a,y}p_{ay\mid z}(x)c^{(k)}_{ayz}c^{(l)}_{ayz}
-M_{kz}(x)M_{lz}(x)\right\}.
\]
This matrix determines the variance-minimizing weights on the estimated active set $\widehat S$. If $\widehat{\bm\Sigma}_{\widehat S}$ is nonsingular, the weights are
\[
\widehat{\bm\lambda}
=\frac{\widehat{\bm\Sigma}_{\widehat S}^{-1}\1}
{\1^\top\widehat{\bm\Sigma}_{\widehat S}^{-1}\1}.
\]
For a rank-deficient matrix, Section~\ref{sec:cov-extended} gives two cases: normalized pseudoinverse weights when $\1$ lies in the matrix range, and a unit-sum vector in its kernel otherwise. In either case, the weights can be negative.

For these nuisance-dependent branch scores, the exact robustness identity is given by~\eqref{eq:card-dr-score}. We minimize conditional variance over their unit-sum weighted combinations. The covariance program has the same form as in Theorem~\ref{thm:cov-kink}, while its robustness justification uses the displayed doubly robust identity in place of the fixed identifying-score assumption of that theorem. Proposition~\ref{prop:bp-rigidity} proves the corresponding lower bound, and Theorems~\ref{thm:bp-achieve} and~\ref{thm:bp-stable-rank} give the attainment conditions.

The analysis protocol fixes $20$ stratified five-fold cross-fitting partitions.  Survey-weighted ridge multinomial regressions estimate the four conditional $(A,Y)$ cell probabilities within each instrument level, and a weighted ridge logistic regression estimates the instrument propensity.  The estimated active set uses tolerance $\tau_n=0.25n^{-1/4}=0.0338$.  For each split, we use the full estimated conditional covariance matrix and compare three scores: RALU weighting, equal weighting on exactly the same estimated face, and hard selection of one estimated active branch.  All three methods use the same observations, folds, and off-fold nuisance predictions. By the complete score we mean the plug-in branch value together with its doubly robust correction; both enter the reported variance calculation.

We report estimated variances of the complete scores evaluated at the fitted nuisances. Their asymptotic interpretation additionally requires control of nuisance-product bias and active-set estimation error, together with convergence of the complete score in $L^2$. These are learning conditions of the procedure, separate from the fitted variance comparison. To stabilize the weights, the implementation projects onto $\{\lambda:\bm1^\top\lambda=1,\ \|\lambda\|_1\le5\}$ and uses equal weights if the covariance calculation fails its numerical guard. Neither adjustment was activated in the primary weighted analysis on any of the $20$ splits at either endpoint. The largest recorded lower-endpoint $\ell^1$ norm was $3.64$, so the primary analysis used the unconstrained covariance-minimizing weights.

\begin{table}[!htb]
\centering
\small
\caption{Card--NLSYM application.  Method columns report the componentwise median point estimate, with the median estimated standard error in parentheses, across $20$ cross-fitting splits.  The variance ratio is computed within each split and then summarized by its median; the percentage is $100\{1-\widehat V_{\rm R}/\widehat V_{\rm eq}\}$.}
\label{tab:card-results}
\begin{tabular}{@{}lcccc@{}}
\toprule
endpoint & RALU & equal face & hard branch
& $\widehat V_{\rm eq}/\widehat V_{\rm R}$ (reduction)\\
\midrule
lower & $-0.3960\ (0.0203)$ & $-0.4034\ (0.0217)$
      & $-0.4007\ (0.0235)$ & $1.1486\ (12.94\%)$\\
upper & $\phantom{-}0.5233\ (0.0213)$ & $\phantom{-}0.5283\ (0.0233)$
      & $\phantom{-}0.5257\ (0.0257)$ & $1.1979\ (16.52\%)$\\
\bottomrule
\end{tabular}
\end{table}

At the primary tolerance, the estimated active set contains more than one branch for a median $53.2\%$ of the weighted observations at the lower endpoint and $66.7\%$ at the upper endpoint.  The median variance ratios in Table~\ref{tab:card-results} correspond to estimated variance reductions of $12.94\%$ and $16.52\%$, respectively.  RALU has lower estimated complete-score variance in all $20$ splits for both endpoints. Across the lower-endpoint splits, the reduction ranges from $7.75\%$ to $22.16\%$; across the upper-endpoint splits, it ranges from $13.12\%$ to $23.11\%$.  For the lower endpoint, the direction also persists in all $60$ split-by-tolerance comparisons and throughout the unweighted ridge and weighted boosting sensitivity analyses.

\begin{figure}[t]
\centering
\includegraphics[width=.96\linewidth]{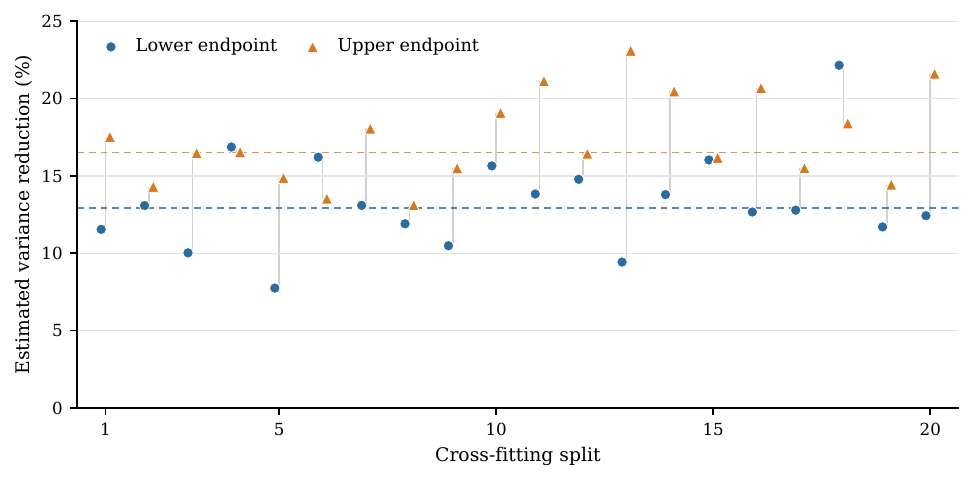}
\caption{Estimated RALU variance reduction across the $20$ cross-fitting splits in the Card--NLSYM application.  Each point is $100\{1-\widehat V_{\rm R}/\widehat V_{\rm eq}\}$ from one shared five-fold partition at the primary active-set tolerance; dashed lines mark the endpoint medians.  Every split favors RALU.}
\label{fig:card-variance-reduction}
\end{figure}

The application provides a real-data counterpart to the correlated-score theory.  When several Balke--Pearl branches enter the estimated active set and carry unequal information, adapting their weights to the full covariance structure produces a material estimated score-variance gain over the equal-weight rule.

\section{Discussion}\label{sec:disc}

This paper formulates efficient inference as the design of a minimum-variance asymptotic linear representation subject to robustness and unbiasedness constraints. For marginal-integral functionals, exact unbiasedness under nuisance substitutions determines the feasible score maps without requiring pathwise differentiability. The conditional-variance reduction identifies their exact variance infimum. Under the stated structural and perturbation conditions, oracle realization gives sharp formulas for the principal piecewise-affine structures, and the learning and margin conditions yield feasible cross-fitted attainment in the max-arm subclass. The distinction between the algebraic infimum, a minimizing representation, and an attaining estimator is part of the variational formulation.

The resulting efficiency geometry is explicit.  Conditionally separated identifying scores lead to a harmonic variance formula, whereas jointly observed scores lead to a covariance-weighted quadratic program, with ordinary inverse and Moore--Penrose formulas under the stated positive-definite and singular-covariance conditions. The binary optimal-treatment formula agrees exactly with the bound of \citet{XuGuo}.  Their two-parameter adaptive smoothing construction separates the smoothing scale needed for validity from the asymmetric boundary weight that determines efficiency; the latter is exactly the binary active-set weight optimized by the RALU quadratic program.  The present theory extends this variational principle to multiple treatment arms, general maximum-type scores, covariate-aggregated Balke--Pearl bounds, $L^1$ calibration error, and cross-world mediation.  For static optimal-treatment values with finite asymptotic variances, it also answers the efficiency question raised by \citet{Whitehouse2025Softmax}: RALU has strictly smaller asymptotic variance than their nonparametric softmax estimator when arms with unequal conditional information are tied on a set of positive probability.

The static and dynamic treatment-regime developments in \citet{XuGuo} belong to the same RALU program. The broader marginal-integral formulation also accommodates targets such as cross-world mediation, where the conditional summary contains both outcome and mediator components. Each application requires its own nuisance declaration: the robustness identities and learning conditions must be checked for that declaration before asserting efficiency or attainment.

The bridge theorem explains when the same geometry is also classical. On feasible local experiments that preserve active branches to first order, the optimal residual is the projection of a conditional branch gradient.  In affine max-arm models, exactly tie-preserving paths make this interpretation compatible with the finite substitution identities. The resulting interpretation complements the exact finite-substitution comparison.

The Card--NLSYM application shows that this covariance geometry has visible consequences in data.  Relative to equal weighting on the same estimated active set, RALU reduces the median estimated variance by $12.94\%$ for the lower Balke--Pearl endpoint and $16.52\%$ for the upper endpoint, and every repeated cross-fitting split favors RALU. The estimated score-variance improvement therefore also appears in the empirical application.

Exact robustness is substantive.  On sufficiently rich regions with a unique active branch, it forces the relevant link to be affine.  On persistent tie regions, it instead imposes a finite-dimensional linear system that can admit nonlinear links.  The bound is also indexed by the declared nuisance factorization and robustness pattern.  These features identify precisely which asymptotically linear estimators are compared and explain how the theory remains meaningful when the derivative-based construction is unavailable.

Important next directions include local-uniform versions of the RALU bound, robustness patterns beyond one misspecified component, attainment at unstable covariance-rank boundaries without spectral separation, and continuous actions.  Replacing exact robustness by first-order orthogonality would define a different estimator class and requires a separate lower-bound and attainment analysis.

\clearpage
\part*{Proofs and additional results}
\setcounter{section}{0}
\renewcommand{\thesection}{S.\arabic{section}}
\renewcommand{\theHsection}{supp.\arabic{section}}

Sections~\ref{ssec:ulb-proof}--\ref{ssec:achieve-proof} prove the variance reduction, functional-form, and attainment results. Sections~\ref{ssec:model-verification} and~\ref{ssec:further-apps} verify the model conditions and develop the applications, and Section~\ref{ssec:bridge} proves the connection to classical efficiency. Sections~\ref{ssec:num-checks} and~\ref{ssec:card-details} give additional numerical and empirical details; Section~\ref{ssec:limitations} discusses the scope of the results.

\section{Conditional-variance reduction and the exact infimum}\label{ssec:ulb-proof}

We prove Theorem~\ref{thm:ulb}, Lemma~\ref{rem:vmin-w-invariant}, and Corollary~\ref{cor:est-level}. We first recall Assumption~\ref{ass:meas-main} and specify the selection graphs used in the proofs.  Throughout this section $\cP$ is the rich model used in condition (B). The family $\IFRL$ contains all measurable full-slot maps satisfying that condition.  An estimator need only satisfy asymptotic linearity on $\cP_{\mathrm{ach}}\subseteq\cP$; no learning assumption is used here.

\begin{assumption}[Measurability and selection graphs]\label{ass:meas}
The spaces of observations and slot values are Polish.  The graph $\{(x,Q):Q\in\mathfrak K_x\}$ in (M2), the local nuisance readouts, the local target integrand, and its finite model cost $C(x,Q)$ are Borel.  We require
\[
 |\psi\{x,u_{\rm loc}(x,Q)\}|\le C(x,Q),\qquad
 \int C(x,Q_x^P)\,dP_X(x)<\infty\quad(P\in\cP).
\]
Here $u_{\rm loc}$ is the conditional-law readout of the primary integrand.  The following names record the three selection operations.
\begin{enumerate}[wide=0pt,itemsep=2pt]
\item[\textup{(Sel-slot)}] A slot value is selected together with an admissible witness kernel $R$.  For index $j$, the graph is
\[
 \{(x,v,R):R\in\mathfrak K_x,\
   \eta_{{\rm loc},l}(x,R)=\eta_l(P)(x)\ (l\ne j),\
   v=\eta_{{\rm loc},j}(x,R)\}.
\]
It is Borel, and its model cost is $C(x,R)$.  Using the witness coordinate avoids requiring that the projection onto $(x,v)$ be Borel.
\item[\textup{(Sel-law-graph)}] The compatible graph is
\[
 \mathfrak Q_x^{P,\mathrm{loc}}
 =\mathfrak Q_x^P
 =\{Q\in\mathfrak K_x:
       \eta_{\rm loc}(x,Q)=\eta(P)(x),\
       \kappa_{\rm loc}(x,Q)=\kappa(P)(x)\}.
\]
It is Borel by the preceding readout assumptions.
\item[\textup{(Sel-law-glue)}] Every selection with integrable model cost glues as in (M2).  A field $R$ of finite pointwise cost can always be used on $\{C(x,R_x)\le m\}$ and set equal to $Q_x^P$ elsewhere.  For finitely many jointly selected kernels use the sum of their costs.  Each resulting law belongs to $\cP$.
\end{enumerate}
Selections are understood in the $P_X$-completion; a Borel version may be used before forming a probability kernel.
\end{assumption}

The model cost determines which kernel substitutions produce laws in $\cP$. Integrability of a candidate score at each such law then follows from (B). The next lemma uses these two facts to obtain conditional identities simultaneously over an entire kernel graph.

\Needspace{15\baselineskip}
\begin{lemma}[Violation selection]\label{lem:violation-selection}
Let $\mu=P_X$, let $S$ be Polish, and let $\mathcal G\subseteq\cX\times S$ be a Borel graph with sections $\mathfrak S_x$.  Let $c:\mathcal G\to[0,\infty)$ be a finite Borel model cost.  Suppose a baseline selection $s_0$ satisfies $\int c(x,s_0(x))\,d\mu(x)<\infty$, and that every selection obtained by using an arbitrary selector on a set where its cost is bounded, and $s_0$ elsewhere, satisfies the required model-closure property. Assume:
\begin{enumerate}[wide=0pt,label={\textup{(VS\arabic*)}},itemsep=1pt]
\item The baseline is available on a $\mu$-full set.
\item A Borel envelope $I:\mathcal G\to[0,\infty]$ satisfies $I(x,s(x))<\infty$ almost surely for each gluable selection $s$.
\item A Borel gap $\gamma:\mathcal G_{\rm fin}\to\R$, where $\mathcal G_{\rm fin}=\{I<\infty\}$, satisfies $\gamma(x,s(x))=0$ almost surely for each gluable selection.
\end{enumerate}
Then outside one $\mu$-null set, independent of $s\in\mathfrak S_x$, every fibre element satisfies $I(x,s)<\infty$ and $\gamma(x,s)=0$.
\end{lemma}

\begin{proof}
Discard the fixed null set where the baseline is unavailable. For each integer $m\ge1$ the graph $\{(x,s):I(x,s)=\infty,\ c(x,s)\le m\}$ is Borel. If its analytic projection had positive $\mu$-measure, Jankov--von Neumann selection \cite{Kechris1995} would give a selector there; extend it by $s_0$ elsewhere.  Its cost is integrable and the localized gluing assumption applies, contradicting (VS2). All these projections are null.  Since $c$ is finite at every fibre element, their countable union covers every possible infinite-envelope violation.

For $k,m\ge1$ and $\varepsilon\in\{-1,+1\}$ apply the same argument to the Borel graph
\[
 \{(x,s):I(x,s)<\infty,\ 
        \varepsilon\gamma(x,s)\ge k^{-1},\ c(x,s)\le m\}.
\]
A positive-measure projection would give a gluable selection contradicting (VS3).  The countable union of all these null projections and the preceding domain projections is the required exceptional set.
\end{proof}

The first stage handles possible failures of conditional integrability. Truncating the model cost produces an admissible law even when the selected score has an infinite absolute conditional integral. Condition (B) supplies the contradiction, so the same argument covers both integrability failures and nonzero expectation gaps.

For reference set
\[
 F_{j,v}^{P,w}(x,z)=
 f\bigl(x,z,\eta_1(P)(x),\ldots,v,\ldots,\eta_d(P)(x),
          \kappa(P)(x),w(x)\bigr).
\]
All envelope integrals below are Borel by measurability of integrals of nonnegative Borel functions against a variable probability law. The four applications use the following graphs and costs.
\begin{enumerate}[wide=0pt,label={(VS-\alph*)},itemsep=3pt]
\item VS-law. Select $Q\in\mathfrak Q_x^P$, with cost $C(x,Q)$, baseline $Q_x^P$, and $I_{\rm law}(x,Q)=\int|f_w^P(x,z)|\,dQ(z)$. For a cost-localized selection, (M2) gives a compatible law $P''$. Condition (B) at $(P'',P'')$ implies $\int I_{\rm law}(x,Q_x)\,dP_X(x)<\infty$.

\item VS-slot. Select $(v,R)$ from (Sel-slot), with cost $C(x,R)$ and baseline $(\eta_j(P)(x),Q_x^P)$. Use $I_{\rm slot}(x,v,R)=\int|F_{j,v}^{P,w}(x,z)|\,dQ_x^P(z)$. After cost localization, $R$ glues to a same-marginal witness $P^\dagger$, and (B) at $(P,P^\dagger)$ implies integrability.

\item VS-pair. Select $(Q,v,R)$ jointly from the compatible-law and slot-witness graphs, with cost $C(x,Q)+C(x,R)$ and envelope
\[
 I_{\rm pair}(x,Q,v,R)=\int|F_{j,v}^{P,w}(x,z)|\,dQ(z).
\]
Localize both coordinates on the same bounded-cost set. They give actual laws $P''$ and $P^\dagger$, respectively, and (B) at $(P'',P^\dagger)$ implies integrability.

\item VS-bridge. Select a direction and its matched kernel bundle from the relevant (X3) or (COV3) graph. Its cost is the sum of $C$ over the generated kernels, and its envelope is $\int|f_w^P(x,z)|\,dQ_{x,\bm\delta}(z)$. The matched ambient law is formed only after cost localization; (B) at that law and the frozen working law supplies integrability. The kernels, nuisance values, and structural values must be selected as one bundle.
\end{enumerate}
In every case, including additional kernels used in a rigidity argument, the stipulated graph has finite model cost. Localization makes this cost integrable before the corresponding law is formed.

The compatible fibre is defined by the local graph in Assumption~\ref{ass:meas}. Locality identifies the nuisance and structural values of each substituted law on the bounded-cost set where it differs from the baseline. Lemma~\ref{lem:violation-selection} then supplies one exceptional set for all fibre elements, on both atomic and atomless covariate spaces.

For concreteness, let $\cX=[0,1]$, $A\in\{1,\ldots,K\}$, and $0\le Y\le1$. Let $\cP$ consist of all joint laws with $\pi_a(x)\ge\epsilon>0$ and positive conditional arm variances, for a fixed $\epsilon<1/K$. The local graph is Borel and the constant cost $C=2$ suffices for the target $\E\max_a\E[Y\mid X,A=a]$. Marginal restrictions and every admissible kernel substitution stay in this model. With regression and propensity as the two primary components and the active set as structural component, define
\[
 f_{\lambda}(x,a,y;u,\pi,k)
 =\sum_{b\in k}\lambda_b(x,k)
       \left\{u_b+
       \frac{\mathbf1\{a=b\}}{\pi_b}(y-u_b)\right\}.
\]
Here $\lambda_b(x,k)\ge0$ and $\sum_{b\in k}\lambda_b(x,k)=1$. Choose conditional versions satisfying the model's range constraints at every $x$, with a fixed admissible value on conditional null sets. At admissible inputs, $|f_\lambda|\le1+1/\epsilon$. Its conditional expectation at working values $(u,\pi)$ and true active set $k=\cA^\ast(x;P)$ is
\[
 \E_{Q_x^P}[f_\lambda(x,A,Y;u,\pi,k)]
 =\sum_{b\in k}\lambda_b(x,k)
       \left\{u_b+\frac{\pi_b^P(x)}{\pi_b}
                    \bigl(u_{P,b}(x)-u_b\bigr)\right\}.
\]
If $u=\bm u_P(x)$ or $\pi=\bm\pi^P(x)$, this equals $\sum_{b\in k}\lambda_b(x,k)u_{P,b}(x)=\max_bu_{P,b}(x)$. The bound and this identity prove (B). Define the rule to be zero outside the admissible input range, a Borel set. This gives a measurable full-slot member of $\IFRL$.

A simple $\cP_{\mathrm{ach}}$ consists of the laws in this rich model whose propensities and Bernoulli arm means are constant in $x$, with arm means in $[\delta,1-\delta]$ for fixed $\delta>0$. It includes continuous covariate distributions. Clipped empirical arm means and proportions are root-$n$ consistent on each training fold.  The tolerance $\tau_n=n^{-1/3}$ dominates their supnorm errors, while every positive branch gap at a fixed law is separated from zero, so the positive-gap margin condition holds with any exponent.  The construction in the attainment theorem therefore supplies a common estimator on this subclass. The example supplies both a comparison model supporting the required substitutions and a subclass on which one estimator attains the bound.

Unbounded outcomes also admit a nonempty comparison class.  For the max-of-means target, take conditional arm means in a fixed interval $(-M,M)$, propensities at least $\epsilon$, finite positive conditional second moments, and cost $C(x,Q)=1+\sum_a\E_Q[|Y|\mid A=a]$ with finite marginal integral. The simplex rule above obeys $|f|\le M+(|Y|+M)/\epsilon$ at every admissible one-wrong evaluation, so (B) is globally integrable.  Fixed-mean finite contaminations and bounded mean tilts remain locally available.  Finite-cost substitution therefore permits unbounded scores while excluding noise fields whose absolute first moment is not integrable across covariates.

\begin{lemma}[Auxiliary-input invariance]\label{lem:auxiliary-invariance}
For every $P\in\cP$ and measurable $w,w':\cX\to\Wext$, the conditional-variance competitor families at $w$ and $w'$ coincide in $L^0(P_X)$, and therefore so do their essential infima.
\end{lemma}

\begin{proof}
Let $m_x$ interchange $w(x)$ and $w'(x)$ while fixing other values. The map $(x,v)\mapsto m_x(v)$ is jointly measurable because the diagonal of the Polish space $\Wext$ is Borel. For $f\in\IFRL$ set
\[
 \widetilde f(x,z,e,k,v)=f(x,z,e,k,m_x(v)).
\]
For every measurable auxiliary field $a(x)$, the composition $m_x(a(x))$ is measurable. Thus (B) for $f$, including its integrability requirement, gives (B) for $\widetilde f$. Moreover,
\[
 \widetilde f^P_{w'}=f^P_w,\qquad m_x(m_x(v))=v.
\]
The transformation is therefore a bijection between the finite-variance competitor families, which have the same essential infimum in $L^0(P_X)$.

If $f$ represents an estimator on $\cP_{\mathrm{ach}}$, define $w_{\widetilde f,Q}(x)=m_x(w_{f,Q}(x))$ at each $Q$ in that subclass. Then
\[
 \begin{aligned}
 &\widetilde f(x,z,\eta(Q)(x),\kappa(Q)(x),w_{\widetilde f,Q}(x))\\
 &\qquad=f(x,z,\eta(Q)(x),\kappa(Q)(x),w_{f,Q}(x)),
 \end{aligned}
\]
so the same estimator retains its asymptotic-linear representation.
\end{proof}

\begin{proof}[Conditional identities and variance inequality in Theorem~\ref{thm:ulb}]
Fix $P\in\cP$, a measurable $w$, and $f\in\IFRL$. The conditional-mean and variance arguments in Steps~1 and~3 use (M1). Steps~2 and~2$'$ additionally use (M2)--(M3) and Assumption~\ref{ass:meas-main}.

Step 1: Conditional unbiasedness at the truth. Condition (B) at $(P,P)$ gives $f_w^P\in L^1(P)$ and
\begin{equation}\label{eq:f-mean-truth}
 \E_P[f_w^P]=\Psi(P).
\end{equation}
Let
\begin{equation}\label{eq:Delta-def}
 \Delta(x)=\E_{Q_x^P}[f_w^P(x,\cdot)]
              -\psi\{x,u_P(x)\}.
\end{equation}
Both terms are integrable.  For every $A$ with $P_X(A)>0$, (M1) preserves the nuisance and structural values on $A$, and the unchanged conditional kernel preserves the target integrand. Applying (B) at $(P|_A,P|_A)$ therefore gives
\begin{equation}\label{eq:f-mean-restr}
 \int_A\Delta(x)\,dP_X(x)=0.
\end{equation}
Taking $A=\{\Delta>\varepsilon\}$ or $\{\Delta<-\varepsilon\}$ shows
\begin{equation}\label{eq:pt-unbias-truth}
 \E_{Q_x^P}[f_w^P(x,\cdot)]
 =\psi\{x,u_P(x)\}
 =\psib(x,\eta(P)(x),\kappa(P)(x))
 \quad P_X\text{-a.s.}
\end{equation}

Step 2: One-slot substitution. First fix a cost-localized slot-witness selection $(v(x),R_x)$. Its glued witness $P^\dagger$ has the same marginal as $P$ and agrees in the other nuisance blocks.  Condition (B) at $(P,P^\dagger)$ gives
\begin{equation}\label{eq:f-mean-Pdagger}
 \E_P[F_{j,v(X)}^{P,w}(X,W)]=\Psi(P),
 \qquad F_{j,v(X)}^{P,w}\in L^1(P).
\end{equation}
Apply the same condition at $(P|_A,P^\dagger|_A)$ for arbitrary $A$ with positive marginal mass.  Both restricted laws exist by (M1), because their marginals agree.  The argument of Step 1 gives
\begin{equation}\label{eq:pt-adm-result}
 \E_{Q_x^P}[F_{j,v(x)}^{P,w}(x,\cdot)]
 =\psib(x,\eta(P)(x),\kappa(P)(x))
 \quad P_X\text{-a.s.}
\end{equation}
Now apply Lemma~\ref{lem:violation-selection} with the VS-slot graph, cost, and envelope.  The gap is the difference in \eqref{eq:pt-adm-result}.  Condition (B) gives (VS2), and the localized identity gives (VS3).  The conclusion holds outside one null set simultaneously for every local value having a witness kernel.  This proves condition (I) of pointwise admissibility. To relate this local statement to substitutions realized by whole laws, define
\begin{equation}\label{eq:NjP}
\cN_j^P\eqdef\bigl\{\eta_j(P'):P'\in\cP,\ P'_X=P_X,\
\eta_l(P')=\eta_l(P)\ P_X\text{-a.s. for }l\ne j\bigr\}.
\end{equation}
Every field in $\cN_j^P$ has a realizing law with the same marginal. The local constraint applies wherever that law's conditional kernel has the stipulated nuisance values; it does not constrain arbitrary null-set versions of the field.

Step 2$'$: Compatible-law variation. For a cost-localized selection $Q_x\in\mathfrak Q_x^P$, let $P''$ be its glued law.  Its nuisance and structural readouts agree with those of $P$.  Condition (B) at $(P'',P'')$ and Step 1 at $P''$ give
\begin{equation}\label{eq:locality-selection}
 \E_{Q_x}[f_w^P(x,\cdot)]
 =\psib(x,\eta(P)(x),\kappa(P)(x))
 \quad P_X\text{-a.s.}
\end{equation}
Apply the violation-selection lemma with VS-law.  We conclude that
\begin{equation}\label{eq:LVI-target}
 \E_Q[f_w^P(x,\cdot)]
 =\psib(x,\eta(P)(x),\kappa(P)(x))
 \quad\text{for every }Q\in\mathfrak Q_x^P
\end{equation}
outside one $P_X$-null set.  This proves condition (II). The argument includes conditional-integrability failures through the domain stage of the lemma and does not use truth-law variance.

Step 3: Variance decomposition. Write $m(x)=\psi\{x,u_P(x)\}$ and $\mu=\Psi(P)$. Step~1 and (B) give $m(X)=\E_P[f_w^P\mid X]$ with $f_w^P,m\in L^1(P)$. For almost every $x$, centering by this finite conditional mean gives
\[
 \E_{Q_x^P}\!\left[(f_w^P(x,W)-\mu)^2\right]
 =
 \E_{Q_x^P}\!\left[(f_w^P(x,W)-m(x))^2\right]
 +(m(x)-\mu)^2.
\]
The identity holds in $[0,\infty]$: the cross term has conditional expectation zero and is absolutely integrable, since $f_w^P(x,\cdot)$ is conditionally integrable. Integrating the nonnegative terms proves
\begin{equation}\label{eq:LTV}
 \Var_P(f_w^P)
 =\Var_{P_X}[\psi\{X,u_P(X)\}]
     +\E_{P_X}\!\left[\Var_{Q_X^P}\{f_w^P(X,\cdot)\}\right].
\end{equation}
If $\Var_P(f_w^P)<\infty$, its conditional variance belongs to the defining competitor family and is at least $\Vmin(\cdot;P,w)$ almost surely. If $\Var_P(f_w^P)=\infty$, the lower-bound inequality holds in the extended-real order. Thus $\Var_P(f_w^P)\ge\Vlow(P,w)$ in both cases. Lemma~\ref{lem:auxiliary-invariance} gives independence of $w$.
\end{proof}

The conditional identities yield the following combined pointwise constraint.

\begin{corollary}[Combined pointwise constraint]\label{cor:combined-pt}
Assume (M1)--(M3) and Assumption~\ref{ass:meas-main}. Fix $f\in\IFRL$, $P\in\cP$, and a measurable $w$. Outside one $P_X$-null set, for every $j$, every $Q\in\mathfrak Q_x^P$, and every local slot value $v$ with a witness $R\in\mathfrak K_x$ agreeing off $j$, the following integral is absolutely finite and
\begin{equation}\label{eq:full-pt-constraint}
 \E_Q[F_{j,v}^{P,w}(x,\cdot)]
 =\psib(x,\eta(P)(x),\kappa(P)(x)).
\end{equation}
\end{corollary}

\begin{proof}
Use the VS-pair graph, retaining the witness coordinate $R$. For each cost-localized joint selection, the two glued laws $P''$ and $P^\dagger$ have the same marginal and agree off $j$. Condition (B) at $(P'',P^\dagger)$ gives integrability and the mean identity.  Repeating it after simultaneous marginal restriction gives \eqref{eq:full-pt-constraint} almost surely for that selection. The violation-selection lemma makes integrability and the identity simultaneous over the entire joint graph.  Selecting law and slot together avoids an uncountable union of exceptional sets.  Finally take the finite union over $j$.
\end{proof}

\begin{proof}[Proof of Corollary~\ref{cor:est-level}]
Fix $P\in\cP_{\mathrm{ach}}$ and the auxiliary map appearing in (A). Condition (B) gives $\E_P[f^P_{w_f}]=\Psi(P)$. With $V=\Var_P[f^P_{w_f}]<\infty$, asymptotic linearity and the i.i.d.\ central limit theorem yield
\begin{align*}
 Z_n&\eqdef\sqrt n\{T_n-\Psi(P)\}\\
 &=n^{-1/2}\sum_{i=1}^n\{f^P_{w_f}(X_i,W_i)-\Psi(P)\}+o_P(1)
 \ \Rightarrow\ \mathcal N(0,V).
\end{align*}
Theorem~\ref{thm:ulb} gives $V\ge\Vlow(P)$. If $\{Z_n^2\}$ is uniformly integrable, convergence in distribution to $Z\sim\mathcal N(0,V)$ implies
\[
 \E_P[Z_n^2]\longrightarrow \E[Z^2]=V,\qquad
 \E_P[Z_n]\longrightarrow \E[Z]=0.
\]
Hence $\Var_P(\sqrt n\,T_n)=\E_P[Z_n^2]-(\E_P[Z_n])^2\to V$ \cite{Billingsley1995}.
\end{proof}

Finite patching also gives equality at the infimum.  Theorem~\ref{thm:ulb} states equality between the infimum of the variance over $\IFRL$ and $\Vlow(P,w)$.  Step~3 above proves the inequality $\ge$.  The reverse inequality rests on the following closure property of the algebraic class.

\begin{lemma}[Finite patching preserves (B)]\label{lem:patching}
Let $f,g\in\IFRL$ and let $A\subseteq\cX$ be measurable.  Under \textup{(M1)}, the full-slot map
\[
h(x,z,e,k,v)\eqdef\1_A(x)\,f(x,z,e,k,v)+\1_{A^c}(x)\,g(x,z,e,k,v)
\]
belongs to $\IFRL$.
\end{lemma}

\begin{proof}
Fix $j$, a pair $R,R'\in\cP$ satisfying $\eta_l(R')=\eta_l(R)$ $R_X$-a.s.\ for $l\ne j$, and a measurable auxiliary map $w$. Denote evaluations at $(\eta(R')(x),\kappa(R)(x),w(x))$ by $f_{R'},g_{R'},h_{R'}$. Condition (B) gives $f_{R'},g_{R'}\in L^1(R)$, so $h_{R'}=\1_Af_{R'}+\1_{A^c}g_{R'}$ is integrable.

If $R_X(A)>0$, (M1) puts $R|_A$ in $\cP$ and preserves the nuisance and structural values on $A$. Consequently $(R|_A,R')$ still satisfies the agreement condition in (B). Keeping the working law $R'$ unchanged, we obtain
\begin{align*}
 \E_R[\1_Af_{R'}]
 &=R_X(A)\E_{R|_A}[f_{R'}]\\
 &=R_X(A)\Psi(R|_A)
 =\int_A\psi\{x,u_R(x)\}\,dR_X(x).
\end{align*}
Both sides are zero if $R_X(A)=0$. Applying the same argument to $g$ on $A^c$ gives
\[
 \E_R[h_{R'}]
 =\int_A\psi\{x,u_R(x)\}\,dR_X(x)
  +\int_{A^c}\psi\{x,u_R(x)\}\,dR_X(x)
 =\Psi(R).
\]
Thus the measurable full-slot map $h$ satisfies (B).
\end{proof}

\begin{proposition}[The floor is the algebraic infimum]\label{prop:inf-equality}
Under \textup{(M1)}, for every $P\in\cP$ and measurable $w$,
\[
\inf_{f\in\IFRL}\Var_P\bigl[f^P_w\bigr]=\Vlow(P,w)
\]
in the extended-real sense.
\end{proposition}

\begin{proof}
If no member of $\IFRL$ has finite variance at $(P,w)$, both sides are $+\infty$ by Definition~\ref{def:vmin}. Otherwise choose $f_1,f_2,\ldots\in\IFRL$ with finite variances whose conditional-variance functions realize the essential infimum:
\[
 \inf_{j\ge1}v_{f_j}(x)=\Vmin(x;P,w)\qquad P_X\text{-a.s.}
\]
Fix Borel versions of these functions. Let $h_1=f_1$ and, recursively, set
\[
 A_n=\{v_{h_{n-1}}\le v_{f_n}\},\qquad
 h_n=\1_{A_n}h_{n-1}+\1_{A_n^c}f_n.
\]
Lemma~\ref{lem:patching} gives $h_n\in\IFRL$. Its truth evaluation is square-integrable because it selects between two square-integrable evaluations. On each of the two sets it agrees with the selected map, so
\[
 v_{h_n}=\min(v_{h_{n-1}},v_{f_n})
        =m_n\eqdef\min_{j\le n}v_{f_j}\qquad P_X\text{-a.s.}
\]
The conditional mean is the common function $\psi\{x,u_P(x)\}$ by Step~1. Thus
\[
 \Var_P[h^P_{n,w}]
 =\Var_{P_X}[\psi\{X,u_P(X)\}]+\E_{P_X}[m_n].
\]
Finally, $m_n\downarrow\Vmin(\cdot;P,w)$ almost surely and $0\le m_n\le v_{f_1}\in L^1(P_X)$. Dominated convergence gives
\[
 \lim_{n\to\infty}\Var_P[h^P_{n,w}]
 =\Var_{P_X}[\psi\{X,u_P(X)\}]
   +\E_{P_X}[\Vmin(X;P,w)]
 =\Vlow(P,w).
\]
This proves the upper inequality for the infimum; Step~3 proves the lower inequality.
\end{proof}

Proposition~\ref{prop:inf-equality} constructs a sequence of admissible maps whose variances approach $\Vlow(P,w)$. The closed-form theorems below give conditions for a minimizing map, and the attainment results construct an estimator with the corresponding asymptotic variance on $\cP_{\mathrm{ach}}$.

\section{Proofs of the rigidity lemmas}\label{ssec:rigidity}
Every rigidity result in this section is a statement about an algebraic score map $f\in\IFRL$ satisfying the exact identity~(B) on the comparison model $\cP$.  No estimator or asymptotic-linear expansion is used in these proofs.  Model-admissible perturbations are localized by the finite-cost convention of~(M2) and Assumption~\ref{ass:meas}; condition~(B) then supplies the integrability of the selected score evaluations. The rigidity lemmas rest on the following moment-problem reduction.

\begin{lemma}[Moment-problem affine reduction]\label{lem:moment-affine}
Let $\cZ$ be a Polish space, $\bm h=(h_1,\dots,h_K):\cZ\to\R^K$ measurable, and $Q$ a probability measure on $\cZ$ with $\E_Q[\lVert\bm h\rVert^2]<\infty$, mean $\bm m^\ast\eqdef\E_Q[\bm h]$, and covariance $\bm\Sigma\eqdef\Cov_Q[\bm h]$.  Let $\phi:\cZ\to\R$ be a fixed measurable representative, and let $\mathfrak F$ be a family of probability measures containing $Q$ such that $\bm h$ and $\phi$ are integrable under every $Q'\in\mathfrak F$.  Assume one of the following two alternatives.
\begin{enumerate}[wide=0pt,label={\textup{(\alph*)}},itemsep=2pt]
\item \emph{Mean-varying form.}~The family $\mathfrak F$ is closed under finite mixtures and contains, for every $z\in\supp Q$, the spikes $(1-\epsilon)Q+\epsilon\delta_z$ for all sufficiently small $\epsilon>0$.  Moreover, $\E_{Q'}[\phi]$ depends on $Q'\in\mathfrak F$ only through $\E_{Q'}[\bm h]$.
\item \emph{Fixed-mean form.}~For every finitely supported probability measure $\nu$ on $\supp Q$ with $\E_\nu[\bm h]=\bm m^\ast$ and every $\epsilon\in(0,1)$, the contamination $(1-\epsilon)Q+\epsilon\nu$ belongs to $\mathfrak F$. Moreover, $\E_{Q'}[\phi]$ is constant over $\{Q'\in\mathfrak F:\E_{Q'}[\bm h]=\bm m^\ast\}$.
\end{enumerate}
Then there exist $\alpha\in\R$ and $\bm\beta\in\R^K$ such that
\begin{equation}\label{eq:moment-affine}
\phi(z)=\alpha+\bm\beta^{\!\top}\bm h(z)
\qquad Q\text{-a.s.}
\end{equation}
and $\bm\beta$ is unique modulo $\ker\bm\Sigma$.  In particular, if $\bm\Sigma$ is nonsingular (e.g.\ when $K=1$ and $\Var_Q[h]>0$), then $\bm\beta$ is unique. Under alternative~\textup{(b)}, existence of the affine representation already holds with only $\E_Q[\|\bm h\|]<\infty$; the statements involving $\bm\Sigma$ retain the finite-second-moment hypothesis.
\end{lemma}

\begin{proof}[Proof of Lemma~\ref{lem:moment-affine}]
Under alternative~\textup{(a)}, define $L(\bm m)=\E_{Q'}\phi$ whenever $\E_{Q'}\bm h=\bm m$. The hypothesis makes $L$ well-defined on the convex set $M$ of attainable means, and mixture closure gives
\[
L\{\lambda\bm m_1+(1-\lambda)\bm m_2\}
=\lambda L(\bm m_1)+(1-\lambda)L(\bm m_2),\qquad 0\le\lambda\le1.
\]
Thus $L$ extends to an affine function $\alpha+\bm\beta^\top\bm m$ on $\R^K$. For $z\in\supp Q$, apply this identity to $Q'_z=(1-\epsilon)Q+\epsilon\delta_z$ and subtract the identity at $Q$:
\[
\epsilon\{\phi(z)-\E_Q\phi\}
=\epsilon\bm\beta^\top\{\bm h(z)-\bm m^\ast\}.
\]
Division by $\epsilon$ proves~\eqref{eq:moment-affine}.

Under alternative~\textup{(b)}, let $A$ be the affine hull of the essential range of $\bm h$, and choose a $Q$-full measurable $S\subseteq\supp Q$ with $\bm h(S)\subseteq A$. Put $C=\operatorname{conv}\bm h(S)$. Then $\bm m^\ast\in\operatorname{ri}C$: otherwise a supporting affine functional, nonconstant on $A$, would satisfy $\ell(\bm h)-\ell(\bm m^\ast)\ge0$ almost surely with expectation zero. It would therefore vanish almost surely, contradicting the definition of $A$. Here we use $\operatorname{ri}\overline C=\operatorname{ri}C$ for finite-dimensional convex sets.

Let $\mu$ be a finitely supported signed measure on $S$ with
\[
\mu(S)=0,\qquad \int\bm h\,d\mu=\bm0.
\]
If $\mu\ne0$, its Jordan parts have a common mass $c>0$ and a common normalized score mean $\bm r$. Because $\bm m^\ast\in\operatorname{ri}C$, choose $\delta>0$ and a finitely supported probability measure $\rho$ on $S$ with $\E_\rho\bm h=\bm m^\ast+\delta(\bm m^\ast-\bm r)$. With $d=c/\delta$, both
\[
\nu=\frac{\mu^++d\rho}{c+d},\qquad
\nu'=\frac{\mu^-+d\rho}{c+d}
\]
have score mean $\bm m^\ast$. Constancy of the expectations under their contaminations with $Q$ implies $\E_\nu\phi=\E_{\nu'}\phi$, hence $\int\phi\,d\mu=0$.

Choose anchor points $z_0,\ldots,z_r\in S$ whose score values affinely span $A$. For every $z\in S$, write $\bm h(z)=\sum_{i=0}^r a_i\bm h(z_i)$ with $\sum_i a_i=1$. The signed measure $\delta_z-\sum_i a_i\delta_{z_i}$ has zero mass and score mean, so the preceding identity gives $\phi(z)=\sum_i a_i\phi(z_i)$. Interpolating the anchor values proves the affine representation on $S$. This argument requires only first moments.

When second moments exist, two affine representations differ by $\gamma_0+\bm\gamma^\top\bm h=0$ almost surely. Therefore $\bm\gamma^\top\bm\Sigma\bm\gamma=0$, equivalently $\bm\gamma\in\ker\bm\Sigma$, and $\gamma_0=-\bm\gamma^\top\bm m^\ast$. This proves precisely the stated uniqueness modulo the kernel.

For a measurable family of conditional laws, a measurable coefficient choice follows from bounded normalization. Set $z=(1,\bm h^\top)^\top$, $v=z/(1+\|\bm h\|)$, and
\[
G=\E_Q[vv^\top],\qquad
b=\E_Q\!\left[v\,\frac{\phi}{1+\|\bm h\|}\right],
\qquad \theta=G^+b.
\]
These expectations exist because $v$ is bounded and $\phi\in L^1(Q)$. If $\theta_0$ is any affine representation, then $b=G\theta_0$, so $\theta-\theta_0\in\ker G$ and $z^\top\theta=\phi$ almost surely. Conditional integration and Borel measurability of the Moore--Penrose inverse give the asserted measurable coefficients.
\end{proof}
\begin{proof}[Proof of Lemma~\ref{lem:func-eq-maxarm}]
Throughout, abbreviate $u_a^\ast\eqdef u_{P,a}(x)$, $\pi_a\eqdef \pi_a^P(x)$, $f_a(\widetilde w;s)\eqdef f(x,(a,\widetilde w); s,\eta_2(P)(x),\kappa(P)(x),w(x))$ for a slot vector $s=(s_a)_{a\in\cA}$, and write $\cA^\ast=\cA^{\mathrm{act}}(x;P)$, $f_a\eqdef f_a(\cdot;u^\ast)$.

Corollary~\ref{cor:combined-pt} and Lemma~\ref{lem:matched-bridge} give the fixed-mean and mean-shifting identities outside one $P_X$-null set for this map. Their proofs use the model-cost localization in (M2). Fix $x$ outside that set. The absolute continuity in (X3) allows the affine representation under $Q_x^P$ to be integrated against every bridge kernel.

Step 0 (Fixed-mean affineness on each arm). Assumption~\ref{ass:armwise}\textup{(X1)} permits fixed-mean contamination of each arm-conditional law while the arm probabilities and all other arm laws remain fixed. The combined pointwise identity therefore makes the expectation of the arm section constant over these contaminations. We apply the scalar fixed-mean reduction separately to each arm.

Lemma~\ref{lem:moment-affine} in its fixed-mean form, applied with $Q=Q_{x,a}^P$ and $K=1$, gives measurable $c_a(x)$, $\beta_a(x)$ with
\begin{equation}\label{eq:lem-form-step0}
f^P_w(x;a,\widetilde w)=c_a(x)+\beta_a(x)\,h_a(\widetilde w)
\qquad Q_{x,a}^P\text{-a.s.}
\end{equation}
On an inactive arm with $\sigma_a^2(x;P)=0$, this coefficient is not unique; we choose the representative $\beta_a(x)=0$.  Positivity on the active arms makes their coefficients unique, which is all the active-slope program below requires.

Step 1 (Bridge identity). Let $\bm\delta\in D_x$ be an admissible direction of Assumption~\ref{ass:armwise}~(X2). Corollary~\ref{cor:matched-arm} (the arm-partition coefficient form of Lemma~\ref{lem:matched-bridge}), instantiated with the arm-mean component as the single disagreeing one and with the affine representation supplied by Step~0, gives
\begin{equation}\label{eq:maxarm-matched}
\sum_{a\in\cA}\pi_a^P(x)\,\beta_a(x)\,\delta_a
=\max_a g\bigl(x,a,u_{P,a}(x)+\delta_a\bigr)-\max_a g\bigl(x,a,u_{P,a}(x)\bigr),
\end{equation}
at $P_X$-a.e.\ $x$ and simultaneously for every $\bm\delta\in D_x$. Every constraint below is an instance of~\eqref{eq:maxarm-matched}; none requires leaving the $\kappa$-stratum, and no ambient configuration with a different active set is invoked.

Step 2 (Reachable inactive arms). For $a\notin\cA^\ast(x;P)$ whose coordinate direction is reachable, take $\bm\delta=t\bm e_a$ with $|t|$ below the gap to the active value, admissible by~(X2).  The right-hand side of~\eqref{eq:maxarm-matched} is zero, so $\beta_a(x)=0$.  For an inactive arm whose mean direction is \emph{not} reachable no constraint on $\beta_a(x)$ follows from this identity.  The variance decomposition below discards the corresponding contribution because it is nonnegative.

Step 3 (Active arms). Let $\bm r_x(t)$ be the tie-preserving path of~(X2), which keeps the inactive coordinates fixed and moves every active arm value by the common amount $t$. Then~\eqref{eq:maxarm-matched} reads $\sum_{a\in\cA^\ast(x)}\pi_a^P(x)\beta_a(x)\,r_{x,a}(t)=t$ for all small $t$, an exact perturbation identity valid for $C^1$ links (Remark~\ref{rem:c1-retied}); under the affine hypothesis of the statement, $r_{x,a}(t)=t/g_a^\prime(x)$ on the reachable neighbourhood, whence
\begin{equation}\label{eq:lem-kink-step3}
\sum_{a\in\cA^\ast(x;P)}\pi_a^P(x)\,
\frac{\beta_a(x)}{g_a^\prime(x)}=1 .
\end{equation}
The strict region is the special case $|\cA^\ast(x;P)|=1$, giving $\beta_{a^\ast}(x)=g_{a^\ast}^\prime(x)/\pi_{a^\ast}^P(x)$; no separate argument, and in particular no perturbation into a singleton configuration, is needed.

\medskip
The pointwise unbiasedness identity gives $\sum_a\pi_a^P(x)\,\E_{Q_{x,a}^P}[\beta_ah_a+c_a]=\psib(x)$, as required.
\end{proof}

\begin{remark}[$C^1$ links and the exact perturbation identity]\label{rem:c1-retied}
For merely $C^1$ active links the tie-preserving constraint of Step~3 is the exact perturbation identity $\sum_{a\in\cA^\ast(x)}\pi_a^P(x)\beta_a(x)\,r_{x,a}(t)=t$ for all small $t$, with $r_{x,a}(t)=g_a^{-1}\{g_a(x,a,u_{P,a}(x))+t\}-u_{P,a}(x)$; differentiating at $t=0$ yields the active-slope constraint in the main paper as a \emph{necessary} condition.  For nonlinear links the derivative condition is not a characterization and the residual degree count fails: with $\pi=(\tfrac12,\tfrac12)$, $u=(0,0)$, $g_1(m)=m$, $g_2(m)=m+m^2$, the vector $\bm\beta=(1,1)$ satisfies \eqref{eq:lem-kink} but violates the exact perturbation identity at $t=0.1$ (left side $0.0958039892\ne0.1$), whereas $\bm\beta=(2,0)$ satisfies it exactly for every small $t$; the exact solution set there is a single point, not a line.  In the strict region a nondegenerate bridge-reachable mean window forces the active link to be affine on that window, by the single-arm instance of the identity (Proposition~\ref{prop:affine-necessary}, localized), so the affine hypothesis of the lemma concerns reachable neighbourhoods rather than adding a modelling restriction; in the kink region, movement along the tie-preserving locus alone forces no per-arm affineness, as the example shows.
\end{remark}

Write $\cA^\ast(x)=\cA^\ast(x;P)$.  At $P_X$-a.e.\ $x$, let the tie-preserving bridge have a nondegenerate displacement window $T_x=(-\varepsilon_x,\varepsilon_x)$, with $\pi_a^P(x)>0$ and $0<\sigma_a^2(x;P)<\infty$ for every $a\in\cA^\ast(x)$, and write
\begin{equation}\label{eq:retied-displacement}
r_{x,a}(t)\eqdef g_a^{-1}\bigl(g_a(x,a,u_{P,a}(x))+t\bigr)-u_{P,a}(x),
\qquad a\in\cA^\ast,\ t\in T_x,
\end{equation}
for the arm-mean displacement that raises every active link value by $t$.  The tie-preserving path is jointly measurable on the variable-domain graph
\begin{equation}\label{eq:bridge-graph}
\mathcal G\eqdef\{(x,t):|t|<\varepsilon_x\},
\qquad
x\mapsto\varepsilon_x\ \text{measurable},
\end{equation}
which is Borel. We use a measurable choice of local inverse branch for which $t\mapsto r_{x,a}(t)$ is continuous. Assumption~\ref{ass:armwise}\textup{(X3)} supplies a Borel graph of compatible direction--kernel pairs. Kernel choices are retained as coordinates in this graph when applying Lemma~\ref{lem:matched-bridge}; a jointly Borel choice of a kernel for every $(x,t)$ is not required.

\begin{theorem}[Exact perturbation restriction and the pointwise floor at a tie]
\label{thm:finite-path}
Adopt the max-arm structure of Definition~\ref{def:max-arm} with scalar identifying scores and fix $P\in\cP$ and a measurable auxiliary map $w:\cX\to\Wext$.  Assume Assumptions~\ref{ass:rich}, \ref{ass:armwise}, and~\ref{ass:meas}, and the tie-preserving bridge setup above. Let $f\in\IFRL$ be one global algebraic score map.  Then there is a $P_X$-null set $N_f$, which may depend on $f$, such that at every $x\notin N_f$ the function $f^P_w(x,\cdot)$ is armwise affine, $f^P_w(x,(a,\widetilde w))=c_a(x)+\beta_a(x)h_a$, its slopes vanish on every bridge-reachable inactive arm, and the active slopes satisfy the \emph{exact perturbation equation}
\begin{equation}\label{eq:finite-path}
\sum_{a\in\cA^\ast}\pi_a^P(x)\,\beta_a(x)\,r_{x,a}(t)\;=\;t
\qquad\text{for every }t\in T_x .
\end{equation}
Write $\cB(x)\subseteq\R^{|\cA^\ast(x)|}$ for the set of active slope vectors satisfying~\eqref{eq:finite-path}.  Inactive slopes may be set to zero when minimizing variance.  Then:
\begin{enumerate}[wide=0pt,label={\textup{(\roman*)}},itemsep=2pt]
\item if $P_X\{x:\cB(x)=\varnothing\}>0$ then $\IFRL=\varnothing$;
\item every $f\in\IFRL$ with finite truth-evaluation variance satisfies $\Var_{Q_x^P}[f^P_w(x,\cdot)]\ge H(x;P)$ at $P_X$-a.e.\ $x$, and therefore, taking the essential infimum over the class,
\begin{equation}\label{eq:finite-path-vmin-lower}
\Vmin(\cdot\,;P)\;\ge\;H(\cdot\,;P)\qquad\text{in }L^0(P_X),
\end{equation}
where
\begin{equation}\label{eq:finite-path-floor}
H(x;P)\eqdef\inf_{\bm\beta\in\cB(x)}
\sum_{a\in\cA^\ast(x)}\pi_a^P(x)\,\sigma_a^2(x;P)\,\beta_a^2 ,
\end{equation}
with $\inf\varnothing=+\infty$;
\item \emph{(closed form)} let $\bm c(t)\eqdef(\pi_a^P(x)r_{x,a}(t))_{a\in\cA^\ast}$ and let $q$ be the dimension of $\operatorname{span}\{\bm c(t):t\in T_x\}$. Choose $t_1,\dots,t_q\in T_x$ with $\bm C\eqdef[\bm c(t_1),\dots,\bm c(t_q)]^{\!\top}$ of full row rank and put $\bm b\eqdef(t_1,\dots,t_q)^{\!\top}$.  If some $t\in T_x$ has $(\bm c(t),t)$ outside the row space of $[\bm C\ \bm b]$, then $\cB(x)=\varnothing$.  Otherwise $\cB(x)=\{\bm\beta:\bm C\bm\beta=\bm b\}$ and, with $\bm D\eqdef\operatorname{diag}(\pi_a^P\sigma_a^2)_{a\in\cA^\ast}$,
\begin{equation}\label{eq:finite-path-closed}
\bm\beta^\ast(x)=\bm D^{-1}\bm C^{\!\top}
\bigl(\bm C\bm D^{-1}\bm C^{\!\top}\bigr)^{-1}\bm b,
\qquad
H(x;P)=\bm b^{\!\top}\bigl(\bm C\bm D^{-1}\bm C^{\!\top}\bigr)^{-1}\bm b ,
\end{equation}
with the convention $\inf\varnothing=+\infty$ in~\eqref{eq:finite-path-floor}.
\end{enumerate}
\end{theorem}

\begin{remark}[Role of the global identity]
The theorem concerns a full-slot map in $\IFRL$, rather than only a section admissible at one law.  Indeed, $f^P_w(x,\cdot)\equiv\psib(x,\eta(P)(x),\kappa(P)(x))$ is pointwise admissible with zero conditional variance and zero slopes, but it violates~\eqref{eq:finite-path} at every nonzero displacement.  The cross-law identities in~(B) supply the missing restrictions.
\end{remark}

\begin{remark}[Measurability of the exact perturbation program]
Continuity of $t\mapsto r_{x,a}(t)$ makes \eqref{eq:finite-path} on $T_x$ equivalent to its restriction to the rational points of $T_x$.  Hence $\{(x,\bm\beta):\bm\beta\in\cB(x)\}$ is a countable intersection of Borel sets, and $H(\cdot\,;P)$ is measurable as the value of a Borel-graph program, using its $P_X$-completed version.
\end{remark}

\begin{proof}
The bridge identity first turns each tie-preserving perturbation into a linear equation for the active slopes.  We then use that equation to characterize feasibility, obtain the conditional quadratic program, and solve its finite-rank form.

Both ingredients use the global identity~(B), hence apply to every $f\in\IFRL$.  Apply the fixed-mean form of Lemma~\ref{lem:moment-affine} separately to every arm, as in Step~0 of the max-arm rigidity proof.  Identifying scores have finite first moments, so the lemma's $L^1$ existence statement covers inactive arms even when their second moments are infinite.  The tie setup supplies the finite positive second moments needed for the active variance program.  The lemma also supplies measurable coefficients. Corollary~\ref{cor:matched-arm}, applied first to reachable inactive directions, makes those inactive slopes vanish.  Applied along the tie-preserving bridge, it then gives, for each $t\in T_x$, the matched identity whose left side is $\sum_{a\in\cA^\ast}\pi_a^P\beta_a r_{x,a}(t)$ and whose right side is the target increment $t$; this is~\eqref{eq:finite-path}.  Each invocation excludes a $P_X$-null set depending on $f$ and on the displacement; the simultaneous-null-set argument of Lemma~\ref{lem:matched-bridge}, together with the countable-rational reduction above, collects them into the single $N_f$ of the statement. The affine specialization substitutes $r_{x,a}(t)=t/g_a^\prime$; here~\eqref{eq:finite-path} is retained in full.

For~(i), $\cB(x)=\varnothing$ says no slope vector is compatible with the identity; if this happens on a set of positive $P_X$-measure then no map in $\IFRL$ can exist, since by the display its evaluation $f^P_w$ would have to supply such a vector at $P_X$-a.e.\ $x$. The conclusion is about the class and not about pointwise admissible sections, of which there may be many; the constant section is one. For~(ii), the displayed variance decomposition is the armwise conditional decomposition, the second term being the $\pi^P$-variance of the arm-level means of $f^P_w$, which is nonnegative and vanishes when the intercepts are chosen to equalize them, as in Theorem~\ref{thm:kink}; slopes on non-reachable inactive arms only add further nonnegative terms.  Inactive contributions are kept as variances of the corresponding sections, interpreted in $[0,\infty]$, and need no second-moment assumption on their identifying scores.

For~(iii), \eqref{eq:finite-path} is a linear system in $\bm\beta$ indexed by $t$, so its solution set is either empty or the affine subspace $\{\bm C\bm\beta=\bm b\}$ determined by any $q$ independent rows, the consistency of the remaining rows being exactly the row-space condition stated.  On that subspace the program~\eqref{eq:finite-path-floor} is the equality-constrained least-squares problem $\min\bm\beta^{\!\top}\bm D\bm\beta$ subject to $\bm C\bm\beta=\bm b$ with $\bm D\succ0$ by positivity of $\pi_a^P$ and $\sigma_a^2$, whose solution and value are~\eqref{eq:finite-path-closed}.
\end{proof}

\begin{corollary}[Equality under oracle realization]
\label{cor:finite-path-attain}
Under the conditions of Theorem~\ref{thm:finite-path}, suppose there is a full-slot map $f^\star\in\IFRL$ with finite truth-evaluation variance, $\Var_P[(f^\star)^P_w]<\infty$.  Suppose its active slopes form a measurable selection $\bm\beta^\star(x)\in\argmin_{\bm\beta\in\cB(x)} \sum_{a\in\cA^\ast(x)}\pi_a^P\sigma_a^2\beta_a^2$ at $P_X$-a.e.\ $x$, whose slopes vanish on the inactive arms, and whose intercepts equalize the arm-level means so that the second term of the variance decomposition vanishes.  Then
\[
\Vmin(\cdot\,;P)=H(\cdot\,;P)\quad\text{in }L^0(P_X),
\]
attained by $f^\star$.
\end{corollary}

\begin{proof}
Theorem~\ref{thm:finite-path} gives $\Vmin\ge H$. The proposed map has conditional variance $H$: its active slopes minimize the quadratic form, and its inactive and between-arm variance terms vanish. Its finite total variance places it in the family defining $\Vmin$, giving the reverse inequality.
\end{proof}

A measurable solution of the conditional quadratic program need not extend to a full-slot map satisfying~(B) on $\cP$.  Conversely, an exactly robust map need not minimize that program.  These are the two algebraic requirements in the corollary.  Statistical attainment is a separate conclusion: on a learnable subclass $\cP_{\mathrm{ach}}\subseteq\cP$, one fixed estimator must have the attaining truth evaluation in its asymptotic-linear expansion at every law of that subclass.  The subclass need not be closed under the perturbations used to prove the lower bound.  Theorem~\ref{thm:achieve} supplies such a construction for affine links under its learning conditions.

\begin{corollary}[Affine links: the harmonic formula is the case $q=1$]
\label{cor:finite-path-affine}
If the active links are affine on the reachable window, so that $r_{x,a}(t)=t/g_a^\prime(x)$, then $q=1$, $\cB(x)$ is the single linear constraint $\sum_{a\in\cA^\ast}\pi_a^P\beta_a/g_a^\prime=1$ with $|\cA^\ast|-1$ residual degrees of freedom, and~\eqref{eq:finite-path-closed} reduces to
\[
\beta_a^\ast=\frac{H}{g_a^\prime\sigma_a^2},
\qquad
H(x;P)=\Bigl(\sum_{a\in\cA^\ast}
\frac{\pi_a^P(x)}{g_a^\prime(x)^2\,\sigma_a^2(x;P)}\Bigr)^{\!-1},
\]
the harmonic value of Theorem~\ref{thm:kink}.  The closed form of the affine theory is therefore the affine case of this program, not a separate result.
\end{corollary}

\begin{corollary}[Strict branch: affineness is forced]
\label{cor:finite-path-strict}
If $\cA^\ast(x)=\{a^\ast\}$ is a singleton, \eqref{eq:finite-path} reads $\pi_{a^\ast}^P\beta_{a^\ast}r_{x,a^\ast}(t)=t$ for every $t\in T_x$, which forces $r_{x,a^\ast}$ to be linear on $T_x$, that is, $g(x,a^\ast,\cdot)$ affine on the bridge-reachable mean window, with $\beta_{a^\ast}=g_{a^\ast}^\prime/\pi_{a^\ast}^P$ uniquely determined. This is the pointwise content of Proposition~\ref{prop:affine-necessary}, recovered from the same equation rather than from a separate argument.
\end{corollary}

\begin{remark}[At a persistent tie the class need not be empty, and the
exact-class floor can exceed the first-order Cram\'er--Rao value]\label{rem:persistent-tie} Affineness is forced on strict branches; it is not forced along a tie manifold.  The following model realizes the tied nonlinear branch.  Take $\cX$ a point, $\pi=(\tfrac12,\tfrac12)$, arm means $u_1(m)=m+m^2$ and $u_2(m)=m$ for $m$ in a neighbourhood of the origin, and links $g_1(u)=u$, $g_2(u)=u+u^2$.  Then $g_1(u_1)=g_2(u_2)=m+m^2$ for \emph{every} $m$: the two arms are tied along the entire model path, the second link is nonaffine on the reachable set, and the target is $m+m^2$. The sequence $T_n=\mathbb P_n[2\cdot\1\{A=1\}Y]$ is exactly unbiased at every law of this model; its slot rule reads no slot, so Definition~\ref{def:ralu}~(B) holds at every one-wrong configuration, and (A) holds because the sequence is its own empirical mean.  The class is therefore nonempty at a law whose active link is nonlinear on the reachable set, so the strict-branch necessity does not extend to tie manifolds.  Here $q=2$ and $\cB(x)$ is the single point $\bm\beta=(2,0)$.  Both benchmarks must be evaluated at this law's own arm variances, not at a normalization: with $Y\in\{-1,1\}$ and conditional means $u_a$, $\sigma_a^2=1-u_a^2$, so at $m=0.15$ one has $\sigma_1^2=0.97024375$ and $\sigma_2^2=0.9775$, and
\[
H=1.9404875,
\qquad
V_{\mathrm{CR}}=1.222490243 .
\]
Here $V_{\mathrm{CR}}$ is the first-order Cram\'er--Rao variance in this regular one-parameter model.  Indeed, its information for $m$ is $I(m)=\tfrac12(1+2m)^2/\sigma_1^2+\tfrac12/\sigma_2^2$ and the target derivative is $1+2m$, so $V_{\mathrm{CR}}=(1+2m)^2/I(m)$.  It equals the first-order harmonic program, whereas~(B) imposes the entire finite-displacement equation. The difference $H-V_{\mathrm{CR}}>0$ is therefore a cost of exact finite-substitution robustness; it is not a contradiction with classical efficiency.  The first-order minimizing section has no exact-robust realization, so the tangent projection does not identify the exact-robust floor in this example.  The exact floor is attained in the sense of Corollary~\ref{cor:finite-path-attain}, but not by the empirical rule exhibited above: $H$ bounds the \emph{total} conditional variance $\sum_a\pi_a\beta_a^2\sigma_a^2+\Var_{a\sim\pi}\bigl[\E[F\mid A=a]\bigr]$, not its slope part alone.  The sequence $\mathbb P_n[2\cdot\1\{A=1\}Y]$ has the unique feasible slope $(2,0)$ and hence within-arm variance $1.94048750$, but its two arm-level means differ, contributing a between-arm term $0.02975625$ and a total of $1.97024375$.  Replace it by the arm-one augmented inverse-probability-weighted (AIPW) rule
\begin{equation}\label{eq:tie-aipw-member}
F\eqdef u_1+\frac{\1\{A=1\}}{\pi_1}\bigl(Y-u_1\bigr),
\end{equation}
which satisfies exact unbiasedness on every one-wrong face, carries the same unique feasible slope $(1/\pi_1,0)=(2,0)$, and has $\E[F\mid A=1]=\E[F\mid A=2]=u_1$, so its between-arm term vanishes and its total variance is exactly $H=1.94048750$.  The hypotheses of Corollary~\ref{cor:finite-path-attain} therefore hold at this law. For an actual estimator, use the arm-one sample mean (with an arbitrary bounded value if that arm is unobserved).  Since $\pi_1=1/2$, the zero-count event has exponentially small probability and the ratio expansion gives influence function $F-u_1$ at every law of this bounded parametric model.
\end{remark}

\begin{remark}\label{rem:lem-vs-XuGuo}
For binary OTR with $g(x,a,u_a)=u_a$ (so $g_a^\prime\equiv 1$), \eqref{eq:lem-strict} gives $\beta_{a^\ast}=1/\pi_{a^\ast}$ on the optimal arm and $\beta_{1-a^\ast}=0$ on the suboptimal arm, recovering the AIPW form found in the supplement of~\cite{XuGuo}. In the kink region, \eqref{eq:lem-kink} reduces to $\sum_a\pi_a^P(x)\,\beta_a=1$, matching the slope constraint derived in the supplement of \cite{XuGuo}.  Their smoothing scale $h_n$ and boundary-asymmetry parameter $t_0$ play mathematically different roles.  The scale $h_n$ controls approximation error and validity, while $t_0$ selects the active-face coordinate and hence the first-order variance.  At a binary tie, identifying $t_0=\lambda_1$ turns their adaptive choice into the oracle weight in the weighted family in Section~\ref{sec:weighted-family}; under conditional homoscedasticity this gives $t_0^\ast=\pi_1$, and under heteroscedasticity it gives~\eqref{eq:xuguo-t0-opt}.  The extra asymmetry parameter, absent from symmetric one-parameter softmax smoothing, is what permits the boundary weight to be optimized without altering the validity role of $h_n$.

The max-arm proof derives the slope constraint by shifting the arm means while fixing the working inputs. True conditional variances then enter the quadratic minimization, which yields~\eqref{eq:xuguo-t0-opt} in the binary heteroscedastic case.
\end{remark}

\begin{remark}[Comparison with Lemma~\ref{lem:func-eq-maxarm} and use
in Theorem~\ref{thm:cov-kink}]\label{rem:lem-cov-vs-maxarm} Lemma~\ref{lem:func-eq-cov} parallels Lemma~\ref{lem:func-eq-maxarm} but is adapted to vector-valued shared-data scores. Two structural differences determine the covariance formula. First, in the arm-partition setting of Lemma~\ref{lem:func-eq-maxarm}, the centered residuals $\1\{A=a\}(h_a-u_{P,a})$ have zero conditional cross-covariance: their products vanish for distinct arms and each has conditional mean zero. Their covariance matrix is therefore diagonal, and the active slopes satisfy $\sum_a\pi_a^P\beta_a/g_a^\prime=1$. In the shared-data setting, $\bm\Sigma$ may have rank strictly less than $K$. For example, the eight Balke--Pearl branch means are affine functions of six free conditional cell probabilities (Section~\ref{sec:bp-bound}). Coefficients of a fixed affine representation are then determined only modulo $\ker(\bm\Sigma)$. This ambiguity leaves $\bm\beta^\top\bm\Sigma\bm\beta$ unchanged, and the singular variance formula uses the Moore--Penrose inverse (Section~\ref{sec:cov-singular}).

Second, Lemma~\ref{lem:func-eq-cov}'s proof perturbs the mean of the integrating measure (via $P^\dagger$), so that the constraint on $\bm\beta$ at the truth slot is obtained from a single algebraic subtraction at the truth slot, with no slot-continuity hypothesis on $\bm\beta(\cdot)$; the proof of Lemma~\ref{lem:func-eq-maxarm} follows the same mean-perturbation architecture.
\end{remark}

The next lemma supplies the shared-score functional form used in Theorem~\ref{thm:cov-kink}.

\begin{lemma}[Bridge identity for jointly selected paths]\label{lem:matched-bridge}
Fix $P\in\cP$, a measurable $w$, and a map $f\in\IFRL$. Assume Assumptions~\ref{ass:rich} and~\ref{ass:meas}. Let $D_x$ and the compatible mean-shifting kernels be supplied by Assumption~\ref{ass:armwise}\textup{(X2)--(X3)} or Assumption~\ref{ass:shared-perturb}\textup{(COV2)--(COV3)}. Retain the direction, kernel, and any required nuisance value together as coordinates of the Borel graph in that assumption. Write $Q_{x,\bm\delta}$ for any kernel in the corresponding section. Its finite model cost is $C(x,Q_{x,\bm\delta})$, independent of $f$.

Outside one $P_X$-null set, the following integrals are absolutely finite and the identity holds simultaneously for every $\bm\delta\in D_x$ and every such compatible kernel:
\begin{equation}\label{eq:matched-bridge}
\E_{Q_{x,\bm\delta}}\bigl[f^P_w(x,\cdot)\bigr]
-\E_{Q_x^P}\bigl[f^P_w(x,\cdot)\bigr]
=\psib\bigl(x;\bm u_P(x)+\bm\delta,\kappa(P)(x)\bigr)
-\psib\bigl(x;\bm u_P(x),\kappa(P)(x)\bigr).
\end{equation}
The shorthand on the right denotes the target integrand with the mean-determining primary component varied and all other primary components held at their values under $P$. The left-hand side is an expectation difference; its coefficient forms appear in Corollaries~\ref{cor:matched-shared} and~\ref{cor:matched-arm}.
\end{lemma}

\begin{proof}
A jointly selected direction, kernel, and nuisance value form one compatible perturbation. Select this entire tuple from the Borel graph in the relevant bridge assumption. Localize it to $A_m=\{x:C(x,Q_{x,\bm\delta(x)})\le m\}$ and use the baseline tuple on $A_m^c$. The selected field is universally measurable and has a Borel version in the $P_X$-completion. Its model cost is integrable, so (M2) gives an ambient law $P_{\bm\delta}\in\cP$ with marginal $P_X$. By the bridge assumption, this law changes only the designated primary component and preserves $\kappa(P)$.

Apply condition~(B) with ambient law $P_{\bm\delta}$ and working law $P$. It gives integrability of $f_w^P$ under the selected kernel and the target mean under $P_{\bm\delta}$. Apply the same identity after restricting both laws to any covariate set of positive marginal probability, using (M1). The argument in Step~1 of Section~\ref{ssec:ulb-proof} gives the corresponding conditional mean identity almost surely. Subtracting the truth identity yields the gap
\begin{equation}\label{eq:mb-gap}
\begin{aligned}
\Gamma(x,\bm\delta,Q)
&\eqdef\E_Q\bigl[f^P_w(x,\cdot)\bigr]-\E_{Q_x^P}\bigl[f^P_w(x,\cdot)\bigr]\\
&\quad-\bigl\{\psib(x;\bm u_P(x)+\bm\delta,\kappa(P)(x))\\
&\hspace{5em}-\psib(x;\bm u_P(x),\kappa(P)(x))\bigr\}.
\end{aligned}
\end{equation}
It is zero almost surely for each localized joint selection.

Apply Lemma~\ref{lem:violation-selection} to the joint graph, with cost $C(x,Q)$, baseline the truth tuple, and envelope $I(x,\bm\delta,Q)=\int|f_w^P(x,z)|\,dQ(z)$. Condition~(B) gives (VS2), and the conditional identity above gives (VS3). The lemma first makes every integral finite outside one null set and then gives $\Gamma=0$ there for every compatible tuple. This proves~\eqref{eq:matched-bridge} with the required simultaneous quantifiers. In particular, the argument uses the Borel correspondence supplied by the main assumptions and does not assume a jointly Borel selector on the full direction graph.
\end{proof}

\begin{corollary}[Shared-score coefficient form]\label{cor:matched-shared}
In the min-of-linear (shared-score) structure, if in addition $f^P_w(x,\cdot)=\alpha(x)+\bm\beta(x)^{\!\top}\bm h(x;\cdot)$ holds $Q_x^P$-a.s., then, since $Q_{x,\bm\delta}\ll Q_x^P$ carries the representation and $\E_{Q_{x,\bm\delta}}[\bm h]=\bm u_P+\bm\delta$, the left side of~\eqref{eq:matched-bridge} equals $\bm\beta(x)^{\!\top}\bm\delta$:
\begin{equation}\label{eq:matched-shared}
\bm\beta(x)^{\!\top}\bm\delta
=\psib\bigl(x;\bm u_P(x)+\bm\delta,\kappa(P)(x)\bigr)
-\psib\bigl(x;\bm u_P(x),\kappa(P)(x)\bigr).
\end{equation}
The value does not depend on the choice of coefficient vector: two choices differ by an element of $\ker\bm\Sigma(x;P)$, and $\bm\delta$ is reachable, hence in $\operatorname{range}\bm\Sigma(x;P)$.
\end{corollary}

\begin{corollary}[Arm-partition coefficient form]\label{cor:matched-arm}
In the max-arm (arm-partition) structure with $W=(A,\widetilde W)$, arm probabilities $\pi_a^P(x)$ preserved by the bridge, and $f^P_w(x,(a,\widetilde w))=c_a(x)+\beta_a(x)\,h_a(x;\widetilde w)$ holding $Q_x^P$-a.s., the left side of~\eqref{eq:matched-bridge} equals $\sum_a\pi_a^P(x)\beta_a(x)\delta_a$, the arm weights entering through the arm-partition expectation $\E_{Q_x}[\,\cdot\,]=\sum_a\pi_a^P(x)\,\E_{Q_{x,a}}[\,\cdot\,]$:
\begin{equation}\label{eq:matched-arm}
\sum_{a\in\cA}\pi_a^P(x)\,\beta_a(x)\,\delta_a
=\psib\bigl(x;\bm u_P(x)+\bm\delta,\kappa(P)(x)\bigr)
-\psib\bigl(x;\bm u_P(x),\kappa(P)(x)\bigr).
\end{equation}
The base lemma carries no arm weights and no coefficients; both belong to the representation, not to the bridge.
\end{corollary}

\begin{lemma}[Min-of-linear functional form]\label{lem:func-eq-cov}
Assume the min-of-linear structure of Definition~\ref{def:min-of-linear} with vector identifying score $\bm h(x;\cdot)=(h_1(x;\cdot),\dots,h_K(x;\cdot))^\top$.  Suppose Assumptions~\ref{ass:rich}, \ref{ass:meas}, and~\ref{ass:shared-perturb} hold.  Fix $P\in\cP$, a measurable $w$, and a map $f\in\IFRL$. Then there are measurable maps $\alpha:\cX\to\R$ and $\bm\beta:\cX\to\R^K$ and a single $P_X$-null set $N_{f,P,w}$ such that, for every $x\notin N_{f,P,w}$, the following representation and coefficient restrictions hold. The matched mean identities used to obtain the restrictions hold simultaneously for every admissible bridge kernel.
\begin{equation}\label{eq:lemcov-form}
f^P_w(x;W)=\alpha(x)+\bm\beta(x)^{\!\top}\bm h(x;W)
\qquad Q_x^P\text{-a.s.}
\end{equation}
The affine random variable is unique $Q_x^P$-a.s., and its parametrizations are related by
\begin{equation}\label{eq:lemcov-reparam}
(\alpha(x),\bm\beta(x))\;\mapsto\;
\bigl(\alpha(x)-\E_{Q_x^P}\!\bigl[\bm\gamma^{\!\top}\bm h(x;\cdot)\bigr],\;
\bm\beta(x)+\bm\gamma\bigr),\qquad
\bm\gamma\in\ker\bm\Sigma(x;P),
\end{equation}
so the class in $\R^K/\ker\bm\Sigma(x;P)$, its induced linear functional on $\operatorname{range}\bm\Sigma(x;P)$, and $\bm\beta^{\!\top}\bm\Sigma(x;P)\bm\beta$ are invariant.
\begin{enumerate}[wide=0pt,label={\textup{(\arabic*)}},itemsep=2pt]
\item \textbf{Active-supported representative.}~There is a representative $\bm b(x)=(\bm b_{\cA^\ast(x)}(x),\bm 0)$ supported on the active set and an adjusted intercept
\begin{equation}\label{eq:lemcov-alpha-b}
\alpha_b(x)\eqdef\alpha(x)+
  \{\bm\beta(x)-\bm b(x)\}^{\!\top}\bm u_P(x)
\end{equation}
such that $f_w^P(x;W)=\alpha_b(x)+\bm b(x)^{\!\top}\bm h(x;W)$ $Q_x^P$-a.s.  This pair induces the same linear functional on $\operatorname{range}\bm\Sigma(x;P)$ as every admissible coefficient vector.  It is unique precisely when $\ker\bm\Sigma(x;P)\cap\{\bm\gamma:\bm\gamma_{\cA^{\ast c}}=\bm 0\} =\{\bm 0\}$, which holds when $\bm\Sigma_{\cA^\ast(x)}$ is positive-definite.
\item \textbf{Normalization.}~If there is a reachable direction $\bm\delta\in\operatorname{range}\bm\Sigma(x;P)$ with $\bm\delta_{\cA^\ast(x)}=t\bm 1$ for some $t\ne0$, then
\begin{equation}\label{eq:lemcov-unit-sum}
\bm 1_{\cA^\ast(x)}^{\!\top}\bm b_{\cA^\ast(x)}(x)=1 .
\end{equation}
Such a direction exists whenever $\bm\Sigma_{\cA^\ast(x)}$ is positive-definite, and more generally whenever $\bm 1\in\operatorname{range}(\bm\Sigma_{\cA^\ast(x)})$; take $\bm v_{\cA^\ast}=\bm\Sigma_{\cA^\ast}^{+}\bm 1$, $\bm v=\bm 0$ off $\cA^\ast$, and $\bm\delta=t\,\bm\Sigma(x;P)\bm v$.
\item \textbf{Strict region.}~If $\cA^\ast(x;P)=\{k^\ast\}$ and the normalizing direction of~(2) exists, then $b_{k^\ast}(x)=1$.
\end{enumerate}
The adjusted intercept satisfies $\alpha_b(x)+\bm b_{\cA^\ast(x)}(x)^{\!\top} \bm u_{P,\cA^\ast(x)}(x)=\min_l u_{P,l}(x)$ for the active-supported representative.
\end{lemma}

\begin{remark}[Null sets and coefficient representatives]
The null-set qualification is essential.  If $P_X$ is atomless, $f_w^P$ can be redefined at one point by a nonaffine function without changing its global expectations.  The lemma therefore gives one exceptional set that is uniform over the fibres used in the downstream essential infimum and quadratic program.

The normalization in~\eqref{eq:lemcov-unit-sum} concerns the selected active-supported representative, not every representative of the quotient class.  For example, if $\bm h=(Z_1,Z_2,Z_1+c)$ for independent standardized $Z_1,Z_2$ and $\cA^\ast=\{1,2\}$, then $\ker\bm\Sigma=\operatorname{span}\{(1,0,-1)\}$ and $\bm\Sigma_{\cA^\ast}=\bm I_2$, while $(\tfrac12,\tfrac12,0)$ and $(\tfrac32,\tfrac12,-1)$ represent the same section and have active sums $1$ and $2$.  If no direction with $\bm\delta_{\cA^\ast}=t\bm1$, $t\ne0$, is reachable, no unit-sum conclusion follows.  At a tie, for instance, $\bm h(W)=(W,-W)$ with Rademacher $W$ has $\operatorname{range}\bm\Sigma=\operatorname{span}\{(1,-1)\}$ and no nonzero tie-preserving direction.
\end{remark}

\begin{proof}[Proof of Lemma~\ref{lem:func-eq-cov}]
Fix $P,w,f$. Corollary~\ref{cor:combined-pt} and Lemma~\ref{lem:matched-bridge} make the fixed-mean and bridge identities simultaneous outside one $P_X$-null set. Fix $x$ outside that set and abbreviate $Q=Q_x^P$, $S=\cA^\ast(x;P)$, $\bm u=\bm u_P(x)$, and $\bm\Sigma=\bm\Sigma(x;P)$.

Step 1 (Affine representation). By (COV1), every fixed-mean finite contamination belongs to the compatible fibre. Its expectation of $f_w^P(x,\cdot)$ is therefore $\min_lu_l$. Lemma~\ref{lem:moment-affine} gives measurable coefficients with $f_w^P=\alpha+\bm\beta^\top\bm h$ $Q$-almost surely. The same lemma gives~\eqref{eq:lemcov-reparam}. In particular, this representation is conditionally square-integrable because $\bm h\in L^2(Q)$, even if its variance integrated over $x$ is infinite.

Step 2 (Mean-shift identity). For every $\bm\delta\in D_x$, (COV3) supplies $Q_{x,\bm\delta}\ll Q$ with score mean $\bm u+\bm\delta$. Absolute continuity transfers the affine representation to this kernel. Subtracting the truth identity from the matched bridge identity yields
\begin{equation}\label{eq:lemcov-master-subtraction}
\bm\beta(x)^{\!\top}\bm\delta
=\min_l\,(u_l+\delta_l)-\min_lu_l.
\end{equation}
For $\bm\delta$ in a sufficiently small neighbourhood of zero in $T_0(x)$, the right side is zero. Scaling any vector of $T_0(x)$ into that neighbourhood shows that $\bm\beta^\top\bm\delta=0$ on all of $T_0(x)$.

Step 3 (Active-supported coefficients). Let $E=\operatorname{range}\bm\Sigma$ and let $R:E\to\R^{|S|}$ restrict a vector to its active coordinates. Then $\ker R=T_0(x)$. Step 2 implies that the functional $\bm\delta\mapsto\bm\beta^\top\bm\delta$ factors through $R$: if $R\bm\delta=R\bm\delta'$, their difference lies in $T_0(x)$ and the functional takes the same value. Extend the induced linear functional on $R(E)$ to $\R^{|S|}$, and represent it by $\bm b_S$. Extending $\bm b_S$ by zero gives
\[
(\bm\beta-\bm b)^\top\bm\delta=0\quad(\bm\delta\in E),
\qquad \bm\beta-\bm b\in E^\perp=\ker\bm\Sigma.
\]
Consequently $\alpha_b=\alpha+(\bm\beta-\bm b)^\top\bm u$ gives the same affine random variable, proving~\eqref{eq:lemcov-alpha-b}.

A measurable representative is obtained explicitly. Put $\bm c(x)=\Cov_{Q_x^P}\{\bm h,f_w^P\}=\bm\Sigma\bm\beta$. The preceding factorization gives $\bm c\in\operatorname{range}(\bm\Sigma_{:,S})$, so define
\begin{equation}\label{eq:canon-selector}
\bm b_S(x)=\bm\Sigma_{:,S}(x;P)^+\bm c(x),
\qquad \bm b_k(x)=0\quad(k\notin S).
\end{equation}
This solves $\bm\Sigma_{:,S}\bm b_S=\bm c$ and hence has the required equivalence class. Conditional moments and pseudoinversion are Borel, and there are finitely many active-set strata, so the choice is measurable. Two active-supported representatives differ by a kernel vector supported on $S$. This proves the uniqueness criterion in (1); positive definiteness of $\bm\Sigma_S$ makes that intersection trivial.

Step 4 (Normalization and intercept). Suppose $\bm1\in\operatorname{range}\bm\Sigma_S$. Set $\bm v_S=\bm\Sigma_S^+\bm1$ and $\bm v=0$ off $S$. Then $(\bm\Sigma\bm v)_S=\bm1$, and (COV2) makes $\bm\delta=t\bm\Sigma\bm v$ admissible for sufficiently small nonzero $t$. Equation~\eqref{eq:lemcov-master-subtraction} gives
\[
t=\bm\beta^\top\bm\delta=\bm b^\top\bm\delta
=t\bm1^\top\bm b_S.
\]
Thus $\bm1^\top\bm b_S=1$, with $b_{k^\ast}=1$ in the singleton case. The same argument applies to any reachable direction whose active block is $t\bm1$. If there is no such direction, the lemma imposes no normalization. Finally, $\E_Qf_w^P=\min_lu_l$ gives $\alpha_b+\bm b_S^\top\bm u_S=\min_lu_l$.
\end{proof}

\begin{remark}[Uniqueness of the active-supported coefficients]\label{rem:lemcov-moment-problem}
For two active-supported representatives of the same evaluation, their difference $\bm\gamma$ satisfies $\bm\gamma\in\ker\bm\Sigma$ and $\bm\gamma_{S^c}=0$. Hence
\[
0=\bm\gamma^\top\bm\Sigma\bm\gamma
=\bm\gamma_S^\top\bm\Sigma_S\bm\gamma_S.
\]
A positive-definite active block therefore makes this representative unique for the given evaluation. The example in the preceding null-set remark shows why coefficients with nonzero inactive coordinates can still represent that evaluation. Proposition~\ref{prop:cov-singular} handles a singular active block.
\end{remark}

\begin{corollary}[Single-arm functional form]\label{cor:func-eq-single}
Assume the marginal-integral setup with $d=1$, $\eta_1(P)=u_P$, and a scalar identifying score $h:\cX\times\cW\to\R$ satisfying $u_P(x)=\E_{Q_x^P}[h(x;W)]$. Suppose Assumptions~\ref{ass:rich} and~\ref{ass:meas} hold. Assume the scalar fixed-mean completeness condition: for every $\epsilon\in(0,1)$, each contamination $(1-\epsilon)Q_x^P+\epsilon\nu$, where $\nu$ is finitely supported on $\supp Q_x^P$ and $\E_\nu[h]=u_P(x)$, belongs to $\mathfrak Q_x^P$. Let $D_x$ have Borel graph and contain zero. For every $\delta\in D_x$, assume there is an admissible kernel $Q\ll Q_x^P$ with $\E_Q[h]=u_P(x)+\delta$ and unchanged structural value. These direction--kernel pairs must have Borel graph and finite model cost, with the truth kernel at zero. Thus (M2) applies to each cost-localized joint selection. These are the scalar fixed-mean and mean-shifting conditions, without restricting the scalar target integrand to be linear.  Fix $P$, a measurable $w$, and one map $f\in\IFRL$.  Then for $P_X$-a.e.\ $x$,
\begin{equation}\label{eq:cor-single-form}
f^P_w(x;W)=\alpha(x)+\beta(x)\,h(x;W)\qquad Q_x^P\text{-a.s.,}
\end{equation}
for measurable $\alpha(x),\beta(x)\in\R$, the pair being unique when $\Var_{Q_x^P}(h)>0$.  The remaining conclusions are stated in increasing order of the reachability conditions they require.

\begin{enumerate}[wide=0pt,label={\textup{(\roman*)}},itemsep=2pt]
\item \emph{Secant representation, unconditionally.}~For every reachable $\delta\in D_x$,
\begin{equation}\label{eq:cor-single-secant}
\beta(x)\,\delta=\psib(x,u_P(x)+\delta,\kappa(P)(x))
-\psib(x,u_P(x),\kappa(P)(x)),
\end{equation}
an exact finite difference with no limit taken and no differentiability used.
\item \emph{Derivative, given a nonzero reachable sequence.}~If there exist $\delta_n\in D_x\setminus\{0\}$ with $\delta_n\to0$, and $\psib(x,\cdot,\kappa(P)(x))$ is differentiable at $u_P(x)$, then dividing~\eqref{eq:cor-single-secant} by $\delta_n$ gives
\begin{equation}\label{eq:cor-single-slope}
\beta(x)=\partial_u\psib(x,u,\kappa(P)(x))\Bigm|_{u=u_P(x)} .
\end{equation}
\item \emph{Local affineness, given a reachable interval.}~If $D_x$ contains an interval around $0$, then~\eqref{eq:cor-single-secant} holds on it and $\psib(x,\cdot,\kappa(P)(x))$ is affine there, with slope $\beta(x)$.
\end{enumerate}

The reachable-sequence condition in (ii) concerns the comparison model. Positive conditional variance makes the slope of a given affine evaluation unique, but does not supply mean-shifting laws. If $D_x=\{0\}$, the secant identity places no further restriction on that slope.
\end{corollary}

\begin{proof}
The scalar fixed-mean completeness condition and Lemma~\ref{lem:moment-affine} give the measurable affine representation~\eqref{eq:cor-single-form}. If $h$ is not almost surely constant, its slope and intercept are unique.

For a cost-localized joint selection $(\delta(x),Q_x)$, (M2) supplies a law $P_\delta$ with the perturbed mean and unchanged structural value. Since $d=1$, condition~(B) applies to the ambient law $P_\delta$ with working law $P$. Marginal restriction under (M1) and the joint-graph argument of Lemma~\ref{lem:matched-bridge} give, outside one null set, $\E_Q[f_w^P(x,\cdot)]=\psib(x,u_P(x)+\delta,\kappa(P)(x))$ for every admissible pair. Absolute continuity carries the affine representation from $Q_x^P$ to $Q$. Subtracting its expectation at the truth gives~\eqref{eq:cor-single-secant}.

Dividing this finite-difference identity by $\delta_n$ and taking the stated limit proves~\eqref{eq:cor-single-slope}. If the identity holds on an interval of displacements, it directly expresses the target integrand as an affine function on the corresponding mean interval. Only the target integrand is differentiated; no smoothness of $f$ is required.
\end{proof}
\section{Optimal weights and the singular covariance extension}\label{ssec:cov-proof}

\begin{proof}[Proof of Theorem~\ref{thm:kink}]
Fix $P$, a measurable $w$, and a map $f\in\IFRL$ with finite truth-evaluation variance. We first bound its conditional variance outside the exceptional set in Lemma~\ref{lem:func-eq-maxarm}, then use oracle realization for equality.

Define the measurable function
\begin{equation}\label{eq:vmin-pt-prob}
H(x)\eqdef\Bigl(\sum_{a\in\cA^\ast(x)}
\frac{\pi_a^P(x)}{(g_a^\prime(x))^2\sigma_a^2(x;P)}\Bigr)^{-1}.
\end{equation}
The inequality $v_f\ge H$ for each finite-variance member makes $H$ a lower bound for the essential infimum. A finite-variance oracle member with conditional variance $H$ then gives equality.

Fix $x$ outside the map-dependent exceptional set for the following calculations.

Step (i): functional form. Lemma~\ref{lem:func-eq-maxarm} (proved in Section~\ref{ssec:rigidity}) establishes that the evaluation at the true nuisance values $f_x^a(\widetilde w)\eqdef f^P_w(x,(a,\widetilde w))$ has the affine form
\begin{equation}\label{eq:f-form-active}
f_x^a(\widetilde w)\;=\;c_a(x)+\beta_a(x)\,h_a(\widetilde w)
\qquad\text{$Q_{x,a}^P$-a.s.,}
\end{equation}
with $\beta_a(x)=0$ for every locally reachable $a\notin\cA^\ast(x)$ (unreachable inactive arms carry unrestricted slopes, discarded below as a nonnegative variance contribution), and the slopes $(\beta_a)_{a\in\cA^\ast(x)}$ on the active set constrained as follows:

\begin{itemize}[leftmargin=2em,itemsep=0pt]
\item Strict region $(|\cA^\ast(x)|=1$, optimal arm $a^\ast$): $\beta_{a^\ast}(x)=g_{a^\ast}^\prime(x)/\pi_{a^\ast}^P(x)$ is uniquely determined by \eqref{eq:lem-strict}.

\item Kink region $(|\cA^\ast(x)|\geq 2)$: the active-arm slopes satisfy \eqref{eq:lem-kink},
\begin{equation}\label{eq:slope-constraint-recall}
\sum_{a\in\cA^\ast(x)}\pi_a^P(x)\cdot\frac{\beta_a(x)}{g_a^\prime(x)}=1,
\end{equation}
yielding $|\cA^\ast(x)|-1$ residual degrees of freedom to be optimized in Step~(iii).
\end{itemize}

The intercepts $c_a(x)$ satisfy the conditional unbiasedness identity
\begin{equation}\label{eq:alpha-constraint-step-i}
\sum_{a\in\cA}\pi_a^P(x)\bigl[c_a(x)+\beta_a(x)\,u_{P,a}(x)\bigr]
=\psib(x,\eta(P)(x),\kappa(P)(x)),
\end{equation}
by~\eqref{eq:pt-unbias-truth} and $\E_{Q_{x,a}^P}h_a=u_{P,a}(x)$.

Step (ii): Variance decomposition. Condition on the observed arm in the affine representation~\eqref{eq:f-form-active}.

Using the disintegration $Q_x^P=\sum_a\pi_a^P(x)\,Q_{x,a}^P$, the law of total variance applied conditionally to $A$ gives
\begin{equation}\label{eq:var-decomp-arm}
\Var_{Q_x^P}\!\bigl[f^P_w(x,\cdot)\bigr]
=\sum_{a\in\cA}\pi_a^P(x)\,\Var_{Q_{x,a}^P}\!\bigl[f_x^a(\widetilde W)\bigr]
+\Var_{a\sim\pi^P(x)}\!\bigl[\E_{Q_{x,a}^P}[f_x^a(\widetilde W)]\bigr].
\end{equation}
For every arm, the affine representation gives
\begin{equation}\label{eq:var-active}
\Var_{Q_{x,a}^P}\!\bigl[f_x^a(\widetilde W)\bigr]
=\beta_a(x)^2\,\sigma_a^2(x;P).
\end{equation}
Consequently, the first term of~\eqref{eq:var-decomp-arm} is
\[
\sum_{a\in\cA}\pi_a^P(x)\,\beta_a(x)^2\,\sigma_a^2(x;P),
\]
and is bounded below by the active-arm contribution
\begin{equation}\label{eq:var-active-sum}
\sum_{a\in\cA^\ast(x)}\pi_a^P(x)\,\beta_a(x)^2\,\sigma_a^2(x;P).
\end{equation}
The omitted terms are precisely the nonnegative within-arm variances from unreachable inactive arms.  The conditional arm means are $\E_{Q_{x,a}^P}[f_x^a]=c_a(x)+\beta_a(x)u_{P,a}(x)$, so the second term of~\eqref{eq:var-decomp-arm} is also nonnegative.  Thus every admissible evaluation satisfies
\begin{equation}\label{eq:var-step-ii}
\Var_{Q_x^P}\!\bigl[f^P_w(x,\cdot)\bigr]
\geq\sum_{a\in\cA^\ast(x)}\pi_a^P(x)\,\beta_a(x)^2\,\sigma_a^2(x;P).
\end{equation}

The weighted family attains equality at the minimizing active slopes: it sets every inactive slope to zero and chooses
\begin{equation}\label{eq:alpha-choice}
c_a(x)=\psib(x,\eta(P)(x),\kappa(P)(x))-\beta_a(x)u_{P,a}(x)
\end{equation}
for every arm.  With this choice all conditional arm means equal $\psib(x,\eta(P)(x),\kappa(P)(x))$, and hence
\begin{equation}\label{eq:cross-arm-zero}
\Var_{a\sim\pi^P(x)}\!\bigl[\E_{Q_{x,a}^P}[f_x^a]\bigr]=0,
\end{equation}
while the inactive within-arm term vanishes.  The pointwise minimization therefore reduces to the active-arm slope parameters $(\beta_a)_{a\in\cA^\ast(x)}$, subject to the slope constraint~\eqref{eq:slope-constraint-recall} from Lemma~\ref{lem:func-eq-maxarm}.

Step (iii): Variance optimization. In the strict region $(|\cA^\ast(x)|=1)$, Lemma~\ref{lem:func-eq-maxarm} fixes the slope uniquely as $\beta_{a^\ast}=g_{a^\ast}^\prime/\pi_{a^\ast}^P$. Substituting into~\eqref{eq:var-step-ii} gives
\begin{equation}\label{eq:vmin-strict-deriv}
\Var_{Q_x^P}\!\bigl[f^P_w(x,\cdot)\bigr]
\geq\pi_{a^\ast}^P(x)\cdot\biggl(\frac{g_{a^\ast}^\prime(x)}{\pi_{a^\ast}^P(x)}\biggr)^{\!2}
\sigma_{a^\ast}^2(x;P)
=\frac{(g_{a^\ast}^\prime(x))^2\,\sigma_{a^\ast}^2(x;P)}{\pi_{a^\ast}^P(x)}.
\end{equation}
This is the singleton case of~\eqref{eq:vmin-harmonic}; the oracle map has weight one on the active arm and intercepts~\eqref{eq:alpha-choice}.

In the kink region $(|\cA^\ast(x)|\geq 2)$, Lemma~\ref{lem:func-eq-maxarm}'s constraint~\eqref{eq:slope-constraint-recall} suggests the natural reparametrization
\begin{equation}\label{eq:reparam-lambda}
\lambda_a(x)\eqdef\frac{\pi_a^P(x)\,\beta_a(x)}{g_a^\prime(x)}
\qquad\text{for }a\in\cA^\ast(x),
\end{equation}
under which~\eqref{eq:slope-constraint-recall} becomes $\sum_{a\in\cA^\ast(x)}\lambda_a(x)=1$. Inverting: $\beta_a=g_a^\prime\,\lambda_a/\pi_a^P$. The active lower-bound term in \eqref{eq:var-step-ii} becomes
\begin{equation}\label{eq:var-weighted}
V_{\mathrm{act}}(x;P,\bm\lambda)
\eqdef\sum_{a\in\cA^\ast(x)}\frac{(g_a^\prime(x))^2\,\lambda_a(x)^2\,\sigma_a^2(x;P)}
{\pi_a^P(x)}.
\end{equation}
The decomposition~\eqref{eq:var-decomp-arm} shows that every admissible evaluation has variance at least $V_{\mathrm{act}}$. The inactive-arm and between-arm contributions vanish for the oracle evaluation.

Write $v_a=(g_a^\prime)^2\sigma_a^2/\pi_a^P>0$ on the active set. Cauchy--Schwarz gives
\[
1=\left(\sum_a\lambda_a\right)^2
\le\left(\sum_av_a\lambda_a^2\right)\left(\sum_av_a^{-1}\right),
\]
where these sums run over active arms. Equality holds precisely at $\lambda_a^\ast=H/v_a$, which gives~\eqref{eq:lambda-opt-oracle} and~\eqref{eq:var-oracle-value}. Therefore $v_f\ge H$ almost surely for each finite-variance member, and $\Vmin\ge H$.

Under Assumption~\ref{ass:oracle-maxarm}, the oracle evaluation comes from a map in $\IFRL$ and has finite total variance. It therefore belongs to the defining competitor family, and its conditional variance $H$ gives $\Vmin\le H$. Total variance then gives the asserted total floor. All oracle coefficients are measurable by their formulas.

Replacing $\sigma_a^2$ by $\widehat\sigma_a^2$ in the same minimization gives~\eqref{eq:lambda-opt-working}. Substitution into~\eqref{eq:var-weighted} gives~\eqref{eq:vmin-closed}. Since $v_a\lambda_a^\ast=H$ and both weight vectors sum to one,
\[
V_{\widehat\lambda}-H
=\sum_{a\in\cA^\ast}v_a(\widehat\lambda_a-\lambda_a^\ast)^2.
\]
Equality is equivalent to $\widehat\lambda=\lambda^\ast$, or to a common proportionality factor between $\widehat\sigma_a^2$ and $\sigma_a^2$ on the active set.
\end{proof}

\begin{remark}[Membership for the covariance family]\label{rem:cov-membership}
For the fixed declaration of Section~\ref{sec:cov-extended}, write $\bm m_e=H_{j_\circ}(e_{j_\circ})$ for the mean vector determined by the primary input $e\in\cN_0$. For structural input $k\in\cK_0$, the covariance-weighted rule satisfies
\begin{equation}\label{eq:cov-IF-slot}
\begin{aligned}
f^{\omega}(x,W,e,k,v)
&\eqdef\sum_{l\in\cA^\ast(k)}\omega_l(x,k)\,m_{e,l}\\
&\quad+\sum_{l\in\cA^\ast(k)}\omega_l(x,k)
  \bigl(h_l(x;W)-m_{e,l}\bigr)\\
&=\sum_{l\in\cA^\ast(k)}\omega_l(x,k)\,h_l(x;W),
\end{aligned}
\end{equation}
the plug-in and its correction cancelling exactly because the data are shared (no propensity denominators); $f^{\omega}$ therefore reads only the structural slot (for working-covariance weights, $\omega$ may also read the auxiliary slot as in Remark~\ref{rem:cov-wv}; the verifications require $\sum_{l\in\cA^\ast(k)}\omega_l(x,k,v)=1$ at every nonempty structural value $k$).  Integrability is a separate part of condition~(B).  A sufficient full-slot domain is
\[
\sup_{x,k,v}\|\bm\omega(x,k,v)\|_1\le L<\infty,
\qquad \E_P\!\left[\max_{l\le K}|h_l(X;W)|\right]<\infty
\quad\text{for every }P\in\cP.
\]
Then $|f^\omega|\le L\max_l|h_l|$ for every auxiliary section, including at all one-wrong substitutions.  Once this L1 requirement holds, the expectation identity in~(B) follows: for any admissible pair with the mean component misspecified, $\E_P[f^{\omega}(\cdot,\eta(P'),\kappa(P),\cdot)] =\E_{P_X}\!\bigl[\sum_{l\in\cA^\ast(X)}\omega_l(X)\,u_{P,l}(X)\bigr] =\E_{P_X}[\min_l u_{P,l}(X)]=\Psi(P)$, all active-arm means being tied on the active set.  Thus this is one map in $\IFRL$ on the rich comparison model.  The cap is sufficient, not necessary: any other full-slot envelope establishing all these L1 requirements also suffices.  To identify the unrestricted optimum using a capped family, the minimizing oracle weights must lie within that cap.

For asymptotic linearity, use fixed balanced folds and the estimated set
\[
\widehat\cA_{\tau_n}(x)=\{l:\widehat u_l(x)-\min_b\widehat u_b(x)\le\tau_n\}.
\]
At the evaluating law, assume $\max_l\E_P[h_l(X;W)^2\mid X]\le C$, the min-of-linear version of the positive-gap margin condition (R2), and, foldwise,
\[
\|\widehat{\bm u}-\bm u_P\|_\infty=o_P(\tau_n),\qquad
\sqrt n\tau_n^{1+\gamma}\longrightarrow0.
\]
Write $\widehat\omega_k(X)$ for the fitted unit-sum weights, supported on $\widehat\cA_{\tau_n}^{(-k)}(X)$, and $\omega_0(X)$ for nonrandom unit-sum limiting weights supported on $\cA^*(X;P)$. Assume $\sup_x\|\widehat\omega_k(x)\|_1=O_P(1)$, $\sup_x\|\omega_0(x)\|_1<\infty$, and $\|\1_{H_k^c}(\widehat\omega_k-\omega_0)\|_{L^2(P_X)}=o_P(1)$, where $H_k=\{\widehat\cA_{\tau_n}^{(-k)}\ne\cA^*\}$ and the vector norm inside $L^2$ is Euclidean. Fixed weight maps satisfy the last condition with zero error.

Set $f_0=\omega_0^\top\bm h$ and $D_k=(\widehat\omega_k-\omega_0)^\top\bm h$. For $T_n^\omega=\sum_k(n_k/n)\mathbb P_{n,k}(\widehat\omega_k^\top\bm h)$, the exact decomposition is
\[
T_n^\omega-\mathbb P_nf_0
=\sum_k\frac{n_k}{n}(\mathbb P_{n,k}-P)D_k
+\sum_k\frac{n_k}{n}PD_k.
\]
On events where the fitted weight norms are bounded, the conditional second-moment assumption gives
\[
PD_k^2\le C'\{\|\1_{H_k^c}(\widehat\omega_k-\omega_0)\|_{L^2(P_X)}^2+P_X(H_k)\}=o_P(1).
\]
Conditional Chebyshev bounds the empirical sum by $o_P(n^{-1/2})$. The min-of-linear version of Lemma~\ref{lem:face-recovery} gives $P_X(H_k)=O_P(\tau_n^\gamma)$ and, from the signed-weight bound~\eqref{eq:R2-face-bound}, $|PD_k|=O_P(\tau_n^{1+\gamma})=o_P(n^{-1/2})$. Therefore $T_n^\omega=\mathbb P_nf_0+o_P(n^{-1/2})$.

For estimated covariance weights, a sufficient condition for the coefficient convergence just used is
\[
\bigl\|\|\widehat{\bm\Sigma}-\bm\Sigma\|_{\mathrm{op}}\bigr\|_{L^2(P_X)}=o_P(1),
\]
together with eigenvalues in a fixed interval $[c,C]\subset(0,\infty)$ for the candidate covariance blocks. Matrix inversion and normalization are uniformly Lipschitz on this domain. A globally integrable map with oracle limiting weights then attains~\eqref{eq:cov-Vmin}. Using the same learners and tolerance at every law gives one estimator on the corresponding learnable subclass. Proposition~\ref{prop:cov-singular} supplies the algebraic optimum for singular blocks. The same feasible argument applies when the chosen coefficients satisfy the displayed norm and convergence conditions.
\end{remark}

\begin{remark}[Working-variance extension]\label{rem:cov-wv}
Theorem~\ref{thm:cov-kink} expresses the bound in the true covariance $\bm\Sigma$; the oracle weights $\bm\lambda^\ast= \bm\Sigma_{\cA^\ast}^{-1}\bm 1/\bm 1^\top\bm\Sigma_{\cA^\ast}^{-1}\bm 1$ attain it. With a positive-definite working covariance $\widehat{\bm\Sigma}$, whether estimated or prespecified, the analyst uses weights $\widehat{\bm\lambda}=\widehat{\bm\Sigma}_{\cA^\ast}^{-1}\bm 1/\bm 1^\top \widehat{\bm\Sigma}_{\cA^\ast}^{-1}\bm 1$, and the resulting variance under truth is
\begin{equation}\label{eq:cov-V-working}
\Var_{Q_x^P}\bigl[f^P_w\bigr]=\widehat{\bm\lambda}^\top\,\bm\Sigma_{\cA^\ast(x)}\,\widehat{\bm\lambda}
=\frac{\bm 1^\top\widehat{\bm\Sigma}_{\cA^\ast}^{-1}\bm\Sigma_{\cA^\ast}\widehat{\bm\Sigma}_{\cA^\ast}^{-1}\bm 1}
{\bigl(\bm 1^\top\widehat{\bm\Sigma}_{\cA^\ast}^{-1}\bm 1\bigr)^2},
\end{equation}
which equals the variance floor in~\eqref{eq:cov-Vmin} if and only if $\widehat{\bm\Sigma}_{\cA^\ast}^{-1}\bm 1\propto\bm\Sigma_{\cA^\ast}^{-1}\bm 1$ (in particular when $\widehat{\bm\Sigma}\propto\bm\Sigma$), and strictly exceeds it otherwise, by the uniqueness of the minimizer of the underlying quadratic program. This is the covariance-version of the ``adaptive variance weighting'' principle.
\end{remark}

\begin{proof}[Proof of Theorem~\ref{thm:cov-kink}]
Lemma~\ref{lem:func-eq-cov} first reduces $f^P_w$ to an affine function of the shared score vector.  We minimize the resulting covariance quadratic form under the unit-sum constraint to obtain the lower bound.  Under the additional oracle-realizability condition, the minimizing evaluation belongs to the comparison class and gives equality.

Step (i): functional form via Lemma~\ref{lem:func-eq-cov}. Choose the canonical active-supported representative supplied by Lemma~\ref{lem:func-eq-cov}.  Then $f^P_w$ satisfies
\begin{equation}\label{eq:cov-affine}
f^P_w(x;W)=\alpha(x)+\bm b(x)^{\!\top}\bm h(x;W)
\qquad Q_x^P\text{-a.s.,}
\end{equation}
with $b_k(x)=0$ for $k\notin\cA^\ast(x)$ and the active-supported sub-vector $\bm b_{\cA^\ast(x)}(x)$ satisfying the unit-sum identity $\bm 1_{\cA^\ast(x)}^{\!\top}\bm b_{\cA^\ast(x)}(x)=1$ (modulo $\ker\bm\Sigma_{\cA^\ast(x)}$, which is trivial under the positive-definiteness hypothesis).  The intercept is fixed by
\begin{equation}\label{eq:cov-pt-unbias}
\alpha(x)+\bm b_{\cA^\ast(x)}(x)^{\!\top}\bm u_{P,\cA^\ast(x)}(x)
=\min_l u_{P,l}(x).
\end{equation}

Step (ii): variance. The conditional variance is
\begin{equation}\label{eq:cov-var}
\Var_{Q_x^P}\!\bigl[f^P_w(x;W)\bigr]
=\bm b_{\cA^\ast(x)}(x)^{\!\top}\,\bm\Sigma_{\cA^\ast(x)}(x;P)\,
\bm b_{\cA^\ast(x)}(x),
\end{equation}
since $b_k(x)=0$ off active and only active-set components contribute to the quadratic form.

Step (iii): Lagrangian minimization. We minimize $\bm b^\top\bm\Sigma_{\cA^\ast}\bm b$ over $\bm b\in\R^{|\cA^\ast(x)|}$ subject to $\bm 1^\top\bm b=1$. By the method of Lagrange multipliers, $\nabla_{\bm b}(\bm b^\top\bm\Sigma_{\cA^\ast}\bm b-2\mu(\bm 1^\top\bm b-1))=0$ gives $\bm\Sigma_{\cA^\ast}\bm b=\mu\bm 1$, so
\begin{equation}\label{eq:cov-beta-opt}
\bm b^\ast=\mu\,\bm\Sigma_{\cA^\ast(x)}^{-1}\bm 1,
\qquad\mu=\frac{1}{\bm 1^\top\bm\Sigma_{\cA^\ast(x)}^{-1}\bm 1}
\quad(\text{from }\bm 1^\top\bm b=1).
\end{equation}
This is \eqref{eq:cov-lambda} (with $\bm\lambda^\ast=\bm b^\ast$, the identification of slopes with weights here because of the unit-sum constraint). The constrained minimum is
\begin{equation*}
(\bm b^\ast)^\top\bm\Sigma_{\cA^\ast}\bm b^\ast
=\mu^2\bm 1^\top\bm\Sigma_{\cA^\ast}^{-1}\bm\Sigma_{\cA^\ast}\bm\Sigma_{\cA^\ast}^{-1}\bm 1
=\mu^2\bm 1^\top\bm\Sigma_{\cA^\ast}^{-1}\bm 1
=\mu=\frac{1}{\bm 1^\top\bm\Sigma_{\cA^\ast}^{-1}\bm 1}.
\end{equation*}
Every admissible finite-variance evaluation satisfies the same constraint, so this value is a lower bound for $\Vmin(x;P)$.

Under the additional oracle-realizability condition for equality, the intercept $\alpha^\ast(x)$ is fixed by~\eqref{eq:cov-pt-unbias}: $\alpha^\ast(x)=\min_l u_{P,l}(x)-\bm b^{\ast\top}\bm u_{P,\cA^\ast(x)}(x)$. Substituting into~\eqref{eq:cov-affine} and rewriting in centered form yields \eqref{eq:cov-IF}.  This computes the infimum over the affine class pinned by Lemma~\ref{lem:func-eq-cov}.  The oracle-realizability condition supplies a map in $\IFRL$ with this truth evaluation and finite $P$-variance; hence it belongs to the competitor family defining $\Vmin$ and gives equality with the constrained minimum.  Remark~\ref{rem:cov-membership} gives sufficient envelope conditions for this map membership and states separately the learning conditions for estimator attainment.
\end{proof}

\begin{corollary}[Barycentric and covariance-optimal active-face weights]
\label{cor:barycentric-covariance}
Under the hypotheses of Theorem~\ref{thm:cov-kink}, including its oracle-realizability condition for equality, for $P_X$-a.e.\ $x$ write $S=\cA^\ast(x)$ and $m=|S|$.  The barycentric coefficient $\bm\lambda^{\rm bar}=m^{-1}\bm 1_S$ has conditional variance
\begin{equation}\label{eq:barycentric-cov-var}
V_{\rm bar}(x)=m^{-2}\bm 1_S^\top\bm\Sigma_S(x;P)\bm 1_S.
\end{equation}
The covariance-optimal RALU member satisfies
\begin{equation}\label{eq:barycentric-cov-gap}
V_{\rm bar}(x)-\Vmin(x;P)
=\bigl(\bm\lambda^{\rm bar}-\bm\lambda^\ast\bigr)^\top
\bm\Sigma_S(x;P)
\bigl(\bm\lambda^{\rm bar}-\bm\lambda^\ast\bigr)\ge0.
\end{equation}
Equality holds if and only if $\bm\Sigma_S(x;P)\bm 1_S=c(x)\bm 1_S$ for some scalar $c(x)$. Thus equal active-face weighting is efficient exactly when all active scores have the same covariance row sum.
\end{corollary}

\begin{proof}
Both $\bm\lambda^{\rm bar}$ and $\bm\lambda^\ast$ have unit sum. The first-order condition for the constrained quadratic program gives $\bm\Sigma_S\bm\lambda^\ast=\mu\bm 1_S$.  Expanding the quadratic at $\bm\lambda^\ast$ therefore eliminates the cross term and yields \eqref{eq:barycentric-cov-gap}.  Positive definiteness makes equality equivalent to $\bm\lambda^{\rm bar}=\bm\lambda^\ast$.  The formula for $\bm\lambda^\ast$ then gives the stated row-sum condition, and the converse follows from the same first-order condition.
\end{proof}

\begin{remark}[Consistency with the arm-partition variance formula]\label{rem:thm-kink-corollary}
Under arm-partition (Definition~\ref{def:max-arm}) with identifying score $h_k(x;W)=g_k^\prime(x)\cdot\1\{A=k\}\cdot\bigl(h_k(x;\widetilde W)-u_{P,k}(x)\bigr)/\pi_k^P(x) +\text{const}$, the covariance matrix $\bm\Sigma$ is diagonal with entries
\begin{equation*}
\Sigma_{kk}(x;P)=\frac{(g_k^\prime(x))^2\sigma_k^2(x;P)}{\pi_k^P(x)}
\quad(k\in\cA^\ast(x)).
\end{equation*}
The matrix inverse $\bm\Sigma_{\cA^\ast}^{-1}$ is then the diagonal of reciprocals, and $\bm 1^\top\bm\Sigma_{\cA^\ast}^{-1}\bm 1= \sum_{k\in\cA^\ast}\pi_k^P/[(g_k^\prime)^2\sigma_k^2]$, recovering \eqref{eq:vmin-harmonic}. The \emph{values} therefore coincide, the simplification arising entirely from the diagonality of $\bm\Sigma$ under the centered arm-partition embedding.  This is a consistency relation rather than a logical corollary: the two theorems quantify over classes with different declared nuisances (the arm-partition class also carries propensity robustness) and different section structures (per-arm intercepts versus a single affine form in $\bm h$), so the agreement of the two formulas is a consistency check across structures rather than a reduction; cf.\ the remark following Lemma~\ref{lem:func-eq-cov}.
\end{remark}

\begin{remark}[Singular-locus parametrization]\label{rem:lemcov-singular}
At the singular locus $\bm 1\in\ker\bm\Sigma_{\cA^\ast(x)}(x;P)$ (e.g., when the active scores satisfy $h_{k_1}+\dots+h_{k_{|\cA^\ast|}}\equiv c$ $Q_x^P$-a.s.), the normalizing direction of Lemma~\ref{lem:func-eq-cov}~(2) does not exist, since $\bm 1\notin\operatorname{range}(\bm\Sigma_{\cA^\ast})$, and no unit-sum identity is asserted; adding $\gamma\bm1$ to $\bm b_{\cA^\ast}$ leaves the variance unchanged, and the corresponding constant shift is absorbed into the intercept to preserve the same evaluation.  In this regime, the identified object is the equivalence class $\bm\beta+\ker\bm\Sigma_{\cA^\ast}$, and the variance functional $\bm\beta^{\!\top}\bm\Sigma\bm\beta$ collapses to a degenerate quadratic form.  Proposition~\ref{prop:cov-singular} identifies $\Vmin$ in the positive-semidefinite regime under its oracle-realization condition.
\end{remark}

\subsection{Singular extension via Moore--Penrose pseudo-inverse}\label{sec:cov-singular}

The positive-definiteness assumption on $\bm\Sigma_{\cA^\ast(x)}$ in Theorem~\ref{thm:cov-kink} fails when the active-arm scores satisfy an algebraic constraint forcing a linear combination $\bm\lambda^\top\bm h_{\cA^\ast}(x;W)$ to be deterministic under $Q_x^P$ for some non-zero $\bm\lambda$. The Balke--Pearl problem provides a concrete example (Section~\ref{sec:bp-degenerate}) where vanishing observed cells make distinct bound expressions coincide as random variables, producing a rank-deficient active covariance.

The extension is via the Moore--Penrose pseudo-inverse $\bm\Sigma_{\cA^\ast}^+$. Recall that $\bm\Sigma_{\cA^\ast}^+\bm\Sigma_{\cA^\ast}\bm\Sigma_{\cA^\ast}^+=\bm\Sigma_{\cA^\ast}^+$, and $\bm\Sigma_{\cA^\ast}^+\bm 1$ is the minimum-norm solution to $\bm\Sigma_{\cA^\ast}\bm\beta=\bm 1$ when one exists, that is, when $\bm 1\in\operatorname{range}(\bm\Sigma_{\cA^\ast})$, and equals $\bm 0$ when $\bm 1\in\ker(\bm\Sigma_{\cA^\ast})$.

\begin{proposition}[Singular extension]\label{prop:cov-singular}
Adopt the setup of Theorem~\ref{thm:cov-kink} but allow $\bm\Sigma_{\cA^\ast(x)}$ to be merely positive-semidefinite, and assume, as in that theorem, only \emph{oracle realization}: the measurable minimizing evaluation specified below is the truth evaluation of one full-slot map in $\IFRL$ and has finite variance under $P$.  No realization of every signed unit-sum weight field is required.  Then at $P_X$-a.e.\ $x$ the essential infimum of Definition~\ref{def:vmin} equals
\begin{equation}\label{eq:cov-Vmin-sing}
\Vmin(x;P)=\begin{cases}
\dfrac{1}{\bm 1^\top\bm\Sigma_{\cA^\ast(x)}^+\bm 1}
&\text{if }\bm 1\in\operatorname{range}\bigl(\bm\Sigma_{\cA^\ast(x)}\bigr),\\[3pt]
0&\text{if }\bm 1\notin\operatorname{range}\bigl(\bm\Sigma_{\cA^\ast(x)}\bigr),
\end{cases}
\end{equation}
and the infimum is attained by the covariance-weighted family with a measurable minimizing weight map. In the first regime the denominator is strictly positive and the minimizer is unique modulo $\ker\bm\Sigma_{\cA^\ast}$. These kernel shifts leave the centered influence representation unchanged.  Moreover $\Vmin(x;P)=0$ if and only if $\ker\bm\Sigma_{\cA^\ast(x)}\not\subseteq\{\bm 1\}^\perp$, equivalently if and only if some weight vector $\bm\lambda^\ast$ with $\bm 1^\top\bm\lambda^\ast=1$ makes $\bm\lambda^{\ast\top}\bm h_{\cA^\ast}(x;W)$ deterministic under $Q_x^P$; when $\bm 1\in\ker\bm\Sigma_{\cA^\ast(x)}$ this holds with $\bm\lambda^\ast=\bm 1/|\cA^\ast(x)|$.
\end{proposition}

\begin{proof}
Value of the constrained program. Decompose $\bm 1=\bm 1_r+\bm 1_0$ with $\bm 1_r\in\operatorname{range}(\bm\Sigma_{\cA^\ast})$ and $\bm 1_0\in\ker(\bm\Sigma_{\cA^\ast})$, an orthogonal decomposition by symmetry, and write $\bm\lambda=\bm\lambda_r+\bm\lambda_0$ accordingly; the objective of $\min\{\bm\lambda^\top\bm\Sigma_{\cA^\ast}\bm\lambda: \bm 1^\top\bm\lambda=1\}$ is $\bm\lambda_r^\top\bm\Sigma_{\cA^\ast}\bm\lambda_r$ and the constraint reads $\bm 1_r^\top\bm\lambda_r+\bm 1_0^\top\bm\lambda_0=1$.  If $\bm 1\notin\operatorname{range}(\bm\Sigma_{\cA^\ast})$, then $\bm 1_0\ne\bm 0$, and $\bm\lambda_r=\bm 0$, $\bm\lambda_0=\bm 1_0/\|\bm 1_0\|^2$ attain value $0$.  If $\bm 1\in\operatorname{range}(\bm\Sigma_{\cA^\ast})$, kernel components affect neither objective nor constraint, the problem restricts to the range where $\bm\Sigma_{\cA^\ast}$ is positive-definite, and the stationarity system yields $\bm\lambda_r^\ast=\bm\Sigma_{\cA^\ast}^{+}\bm 1\big/ (\bm 1^\top\bm\Sigma_{\cA^\ast}^{+}\bm 1)$ with value $1/(\bm 1^\top\bm\Sigma_{\cA^\ast}^{+}\bm 1)$, strictly positive since $\bm 1^\top\bm\Sigma_{\cA^\ast}^{+}\bm 1=0$ would force $\bm 1\perp\operatorname{range}(\bm\Sigma_{\cA^\ast})\ni\bm 1$.

Lower bound by the $\bm\Sigma$-semi-inner-product. Fix one map $f\in\IFRL$ with finite variance at the evaluating law. By Lemma~\ref{lem:func-eq-cov}, there is an $f$-dependent $P_X$-null set; fix $x$ outside it with $f^P_w(x,\cdot)\in L^2(Q_x^P)$.  Then $f^P_w(x,\cdot)=\alpha+\bm\beta^{\!\top}\bm h$ $Q_x^P$-a.s.\ for some $\bm\beta$, and $\Var_{Q_x^P}[f^P_w(x,\cdot)]=\bm\beta^{\!\top}\bm\Sigma\bm\beta$ with $\bm\Sigma\eqdef\bm\Sigma(x;P)$.  We argue directly on $\bm\beta$, with no choice of representative; all quantities below are invariant under $\bm\beta\mapsto\bm\beta+\bm\gamma$, $\bm\gamma\in\ker\bm\Sigma$.

Assume $\bm 1\in\operatorname{range}(\bm\Sigma_{\cA^\ast})$ and set $\bm v_{\cA^\ast}\eqdef\bm\Sigma_{\cA^\ast}^{+}\bm 1$, $\bm v=\bm 0$ off $\cA^\ast$.  Then $(\bm\Sigma\bm v)_{\cA^\ast} =\bm\Sigma_{\cA^\ast}\bm\Sigma_{\cA^\ast}^{+}\bm 1=\bm 1$, while the inactive block of $t\,\bm\Sigma\bm v$ is $O(t)$, so for $t>0$ small the direction $\bm\delta\eqdef t\,\bm\Sigma\bm v$ lies in $\operatorname{range}\bm\Sigma$, preserves the active face, and is admissible by Assumption~\ref{ass:shared-perturb}~(COV2). Its active block being $t\bm 1$, the master subtraction~\eqref{eq:lemcov-master-subtraction} gives
\begin{equation}\label{eq:cs-normalisation}
\bm\beta^{\!\top}\bm\Sigma\bm v=1 .
\end{equation}
Cauchy--Schwarz for the positive semidefinite form $\langle\bm a,\bm b\rangle_{\bm\Sigma}\eqdef\bm a^{\!\top}\bm\Sigma\bm b$ now yields
\[
1=\bigl(\bm\beta^{\!\top}\bm\Sigma\bm v\bigr)^2
\le\bigl(\bm\beta^{\!\top}\bm\Sigma\bm\beta\bigr)
\bigl(\bm v^{\!\top}\bm\Sigma\bm v\bigr),
\qquad
\bm v^{\!\top}\bm\Sigma\bm v
=\bm 1^{\!\top}\bm\Sigma_{\cA^\ast}^{+}\bm 1 ,
\]
whence $\Var_{Q_x^P}[f^P_w(x,\cdot)]\ge(\bm 1^{\!\top} \bm\Sigma_{\cA^\ast}^{+}\bm 1)^{-1}$, with equality for $\bm\beta=\bm v/(\bm v^{\!\top}\bm\Sigma\bm v)$.  If instead $\bm 1\notin\operatorname{range}(\bm\Sigma_{\cA^\ast})$ the program value is $0$ and the lower bound is the nonnegativity of a variance; no claim is made in that case that every admissible active block is feasible for a unit-sum program.

Attainment. The covariance-weighted family of Remark~\ref{rem:cov-membership}, with the measurable weight map $\bm\lambda^\ast(x)=\bm\Sigma_{\cA^\ast}^{+}\bm 1\big/ (\bm 1^\top\bm\Sigma_{\cA^\ast}^{+}\bm 1)$ in the first regime and $\bm\lambda^\ast(x)=(\bm I-\bm\Sigma_{\cA^\ast} \bm\Sigma_{\cA^\ast}^{+})\bm 1\big/ \|(\bm I-\bm\Sigma_{\cA^\ast}\bm\Sigma_{\cA^\ast}^{+})\bm 1\|^2$ in the second (the pseudo-inverse is a Borel function of the matrix entries), is the truth evaluation of a map in $\IFRL$ by the proposition's oracle-realization hypothesis.  Full-slot integrability is part of that hypothesis, not a consequence of unit sum; sufficient envelope conditions appear in Remark~\ref{rem:cov-membership}. The finite-variance hypothesis puts this evaluation in the competitor family defining $\Vmin$.  Its evaluation at the true nuisance values has conditional variance equal to the displayed value.  Exhibiting this member therefore bounds the infimum \emph{above} by the displayed value; together with the lower bound for every finite-variance member of $\IFRL$, this proves~\eqref{eq:cov-Vmin-sing}.

Zero-variance characterization. $\Vmin(x;P)=0$ holds if and only if some feasible $\bm\lambda$ lies in $\ker\bm\Sigma_{\cA^\ast}$, i.e.\ some $\bm k\in\ker\bm\Sigma_{\cA^\ast}$ has $\bm 1^\top\bm k\ne 0$, which is the condition $\ker\bm\Sigma_{\cA^\ast}\not\subseteq\{\bm 1\}^\perp$, equivalently $\bm 1_0\ne\bm 0$, equivalently $\bm 1\notin\operatorname{range}(\bm\Sigma_{\cA^\ast})$; determinism of $\bm\lambda^{\ast\top}\bm h_{\cA^\ast}$ under $Q_x^P$ is $\bm\lambda^\ast\in\ker\bm\Sigma_{\cA^\ast}$ with unit sum.
\end{proof}

\begin{remark}[$\Vmin=0$ as ``deterministic recovery'']\label{rem:Vmin-zero}
The interpretation of $\Vmin(x;P)=0$ is that, at the specific distribution $P$ and point $x$, some unit-sum combination of the arm-scores is an exactly deterministic function of $W$ conditionally on $X=x$, so the conditional target is recovered without conditional variance \emph{at that single $x$}; the variance of an estimator of $\Psi$ retains the marginal term $\Var_{P_X}[\psib]$, and estimator attainment is a separate question. Its $P_X$-mass is determined by the structure of the model, and may be zero, positive or one; there is no generic-position argument to be had here, because $P$-space carries no relevant measure, and continuity of the distribution of $X$ does not by itself make the singular locus $P_X$-null.  The $L^1$ calibration of Section~\ref{sec:lp1} is the extreme case: if $\E_P[Y\mid X]=g(X)$ identically, the kink set is all of $\cX$, the active block is
$\sigma_Y^2\bigl(\begin{smallmatrix}1&-1\\-1&1\end{smallmatrix}\bigr)$
everywhere, and the singular locus has $P_X$-mass one, so $\Vlow(P)$ is $0$ rather than being unaffected.  At the opposite extreme, when $\bm\Sigma_{\cA^\ast(x)}$ is nonsingular at $P_X$-a.e.\ $x$ the locus is empty.  The singular extension is therefore not a technicality that integrates away; it is needed exactly when the model's identifying scores are linearly dependent on the truth support.
\end{remark}

\Needspace{19\baselineskip}
\section{Estimator attainment and strict-branch affineness}\label{ssec:achieve-proof}
The tolerance rule has the following deterministic property. All statements in the lemma are understood $P_X$-almost surely, consistently with the essential supremum in Assumption~\ref{ass:achieve-reg}.

\begin{lemma}[Active-set recovery]\label{lem:face-recovery}
Let $E_n=\{\|\widehat g-g\|_\infty\le\tau_n/2\}$. On $E_n$:
\begin{enumerate}[wide=0pt,label={\textup{(\roman*)}},itemsep=2pt]
\item $\cA^*(x;P)\subseteq\widehat\cA_{\tau_n}(x)$;
\item every retained inactive arm has gap at most $2\tau_n$, and hence
\[
P_X\{\widehat\cA_{\tau_n}(X)\ne\cA^*(X;P)\}\le C(2\tau_n)^\gamma;
\]
\item for any unit-sum weights on $S=\widehat\cA_{\tau_n}(x)$, including signed weights,
\begin{equation}\label{eq:R2-face-bound}
\left|\sum_{a\in S}\widehat\omega_a\{g_a(x)-\psib(x)\}\right|
\le 2\tau_n\|\widehat\omega\|_1
\1\{0<\Delta_+(x)\le2\tau_n\}.
\end{equation}
\end{enumerate}
\end{lemma}
\begin{proof}
Write $e(x)=\max_a|\widehat g_a(x)-g_a(x)|$. For an active arm $a$, $\max_b\widehat g_b-\widehat g_a\le2e\le\tau_n$. For any retained arm, the defining inequality gives
\[
0\le\psib-g_a\le\tau_n+2e\le2\tau_n.
\]
Thus an inactive retained arm implies $0<\Delta_+\le2\tau_n$. The margin condition gives (ii); summing the last gap inequality against $|\widehat\omega_a|$ gives (iii), with zero contribution if all retained arms are active.
\end{proof}

\begin{proof}[Proof of Theorem~\ref{thm:achieve} and Corollary~\ref{cor:achieve-global}]
Fix the evaluating law $P$, and set $f_0=(f^\omega)^P_w$ with $w=\widehat\sigma^2_{\infty,P}$. For fold $k$, write $\widehat f_k$ for the fitted evaluation in~\eqref{eq:Tnstar-def}, $D_k=\widehat f_k-f_0$, $n_k=|I_k|$, and $\mathbb P_{n,k}$ for the empirical measure on $I_k$. Then
\begin{equation}\label{eq:R_n-decomp}
T_n^\omega-\mathbb P_nf_0
=\underbrace{\sum_{k=1}^{\mathcal K}\frac{n_k}{n}(\mathbb P_{n,k}-P)D_k}_{R_{1,n}}
+\underbrace{\sum_{k=1}^{\mathcal K}\frac{n_k}{n}PD_k}_{R_{2,n}}.
\end{equation}
Every bound below holds conditionally on the training sample of the indicated fold; the fixed number of folds permits taking their maximum.

We first bound the conditional second moment of $D_k$. Suppress $k$, put $\varepsilon_a=h_a(\widetilde W)-u_{P,a}(X)$, $\Delta_{1,a}=\widehat\eta_{1,a}-u_{P,a}$ and $\Delta_{2,a}=\widehat\eta_{2,a}-\pi_a^P$, and extend fitted and limiting weights by zero outside their respective active sets. Define
\[
B=\sum_a\widehat\omega_a(g_a-\psib),\qquad
\widehat c_a=\frac{g_a'\widehat\omega_a}{\widehat\eta_{2,a}},\qquad
c_a=\frac{g_a'\bar\omega_a}{\pi_a^P}.
\]
Affineness gives the exact decomposition
\[
D_k=B+\sum_a\1\{A=a\}(\widehat c_a-c_a)\varepsilon_a
+\sum_a g_a'\widehat\omega_a\Delta_{1,a}
\left(1-\frac{\1\{A=a\}}{\widehat\eta_{2,a}}\right).
\]
Writing $H=\{\widehat\cA_{\tau_n}\ne\cA^*\}$, the tolerance rule gives $B^2\le2\tau_n^2\1_H+8e^2$, where $e\le C\max_a|\Delta_{1,a}|$. On $H^c$, the working-variance weights are uniformly Lipschitz in their propensity and variance inputs by (R6); fixed weights agree exactly there. On $H$, all weights and inverse propensity factors remain bounded. Consequently, for $\tau_n\le1$,
\begin{align}
PD_k^2
&\le C\left\{P_XB^2+\sum_a\|\widehat c_a-c_a\|_2^2
+\sum_a\|\Delta_{1,a}\|_2^2\right\}\notag\\
&\le C\left\{\|\widehat\eta-\eta(P)\|_2^2
+\|\widehat\sigma^2-\widehat\sigma^2_{\infty,P}\|_2^2
+P_X(H)\right\}=o_P(1).
\label{eq:R1-var-bound}
\end{align}
The variance-input term is omitted for fixed weights. The last equality follows from (R3), $P(E_n)\to1$, and Lemma~\ref{lem:face-recovery}. If $\mathcal G_k$ is the training-data $\sigma$-field for fold $k$, conditional independence gives, for each $\epsilon>0$,
\[
P\left(\sqrt{n_k}|(\mathbb P_{n,k}-P)D_k|>\epsilon\mid\mathcal G_k\right)
\le\min\{1,PD_k^2/\epsilon^2\}.
\]
The right side tends to zero in probability and is bounded by one, so its expectation tends to zero. Balanced fold sizes and a union bound over $k$ yield $R_{1,n}=o_P(n^{-1/2})$.

For the conditional bias, direct integration of the same decomposition gives
\begin{align}
\E_P[\widehat f_k\mid X,\mathcal G_k]-\psib
&=\sum_{a\in S}\widehat\omega_a(g_a-\psib)
+\sum_{a\in S}\widehat\omega_ag_a'
\frac{\Delta_{1,a}\Delta_{2,a}}{\widehat\eta_{2,a}},
\qquad S=\widehat\cA_{\tau_n}.
\label{eq:R2-product-pointwise}
\end{align}
The fitted weights are simplex weights in this theorem. Lemma~\ref{lem:face-recovery} and (R2) therefore bound the absolute integral of the first sum, on $E_n$, by $C\tau_n^{1+\gamma}$. For the second sum, overlap and Cauchy--Schwarz give
\begin{equation}\label{eq:R2-eta-bound}
\left|P_X\sum_{a\in S}\widehat\omega_ag_a'
\frac{\Delta_{1,a}\Delta_{2,a}}{\widehat\eta_{2,a}}\right|
\le C\sum_a\|\Delta_{1,a}\|_2\|\Delta_{2,a}\|_2
=o_P(n^{-1/2}).
\end{equation}
Assumption (R3) imposes the tolerance window
\begin{equation}\label{eq:tau-window}
\|\widehat g-g\|_\infty=o_P(\tau_n),\qquad
\tau_n=o\{n^{-1/[2(1+\gamma)]}\}.
\end{equation}
It follows that $R_{2,n}=o_P(n^{-1/2})$, and~\eqref{eq:R_n-decomp} proves
\begin{equation}\label{eq:AL-target}
T_n^\omega=\mathbb P_n(f^\omega)^P_w+o_P(n^{-1/2}).
\end{equation}

For global robustness, fix $P,P'\in\cP$ agreeing in one of the two nuisance components and fix any measurable auxiliary input. The envelope~\eqref{eq:maxarm-onewrong-envelope} ensures absolute integrability. At the true active set, the first sum in~\eqref{eq:R2-product-pointwise} is zero. Taking $\Delta_{1,a}=u_{P',a}-u_{P,a}$ and $\Delta_{2,a}=\pi_a^{P'}-\pi_a^P$, agreement in either component also makes every product zero. Thus
\begin{equation}\label{eq:fstar-MR}
\E_P[f^\omega\{X,W;\eta(P')(X),\kappa(P)(X),w(X)\}]=\Psi(P),
\end{equation}
which proves $f^\omega\in\IFRL$. The common learning rules give~\eqref{eq:AL-target} at every $P\in\cP_{\mathrm{ach}}$, proving the class-membership assertion.

At the true inputs, conditional centering and the fact that exactly one arm is observed imply
\[
\Var_P(f_0)=\Var_{P_X}(\psib)+
\E_{P_X}\left[\sum_{a\in\cA^*}
\frac{(g_a')^2\bar\omega_a^2\sigma_a^2}{\pi_a^P}\right].
\]
This is~\eqref{eq:achieve-eq}, and it is finite under (R5)--(R6). For $d_a=(g_a')^2\sigma_a^2/\pi_a^P$ on the active set, the minimizing weights satisfy $d_a\lambda_a^*=c$ with $c=(\sum_{a\in\cA^*}d_a^{-1})^{-1}$. Every unit-sum vector obeys
\[
\sum_{a\in\cA^*}d_a\bar\omega_a^2-c
=\sum_{a\in\cA^*}d_a(\bar\omega_a-\lambda_a^*)^2\ge0.
\]
Equality holds precisely at $\bar\omega=\lambda^*$, since the active variances are positive. For the working-variance rule, this is equivalent to $\widehat\sigma^2_{\infty,P,a}=c_P(X)\sigma_a^2(X;P)$ on the active set, for a measurable $c_P(X)>0$. Fixing the oracle weights at one law, or using a variance learner with the true variance as its limit, therefore realizes the harmonic value. The lower bound of Theorem~\ref{thm:kink} gives~\eqref{eq:achieve-sandwich}; auxiliary invariance of the infimum follows from Lemma~\ref{rem:vmin-w-invariant}. This proves all three attainment assertions.

Finally, the i.i.d. central limit theorem for $f_0-\Psi(P)$ and~\eqref{eq:AL-target} give asymptotic normality with variance~\eqref{eq:achieve-eq}. If the stated squared normalized errors are uniformly integrable, their first two moments converge to those of this normal limit. Hence $n\Var_P(T_n^\omega)$ converges to the same value.
\end{proof}
\begin{proof}[Proof of Proposition~\ref{prop:affine-necessary}]
Suppose toward a contradiction that $f\in\IFRL$, and fix a measurable $w$.  By Corollary~\ref{cor:combined-pt}, the combined pointwise constraint holds for $f$ at $P_X$-a.e.\ $x$; fix $x\in\cX_0$ in this full set, and write $a^\ast=a^\ast(x)$, $J=J(x)$, $u^\ast=\bm u_P(x)$.  Step~0 of the proof of Lemma~\ref{lem:func-eq-maxarm} (fixed-mean affineness, Lemma~\ref{lem:moment-affine}) supplies the representation $f^P_w(x;a^\ast,\cdot)=c(x)+\beta_{a^\ast}(x)h_{a^\ast}$ on $\supp Q^P_{x,a^\ast}$.  Corollary~\ref{cor:matched-arm}, applied with $\bm\delta_m=\{m-u_{P,a^\ast}(x)\}\bm e_{a^\ast}$ and with the arm-mean component as the single disagreeing one, then gives, for every $m\in J(x)$,
\[
\pi_{a^\ast}^P(x)\,\beta_{a^\ast}(x)\,\bigl(m-u_{P,a^\ast}(x)\bigr)
=g\bigl(x,a^\ast,m\bigr)-g\bigl(x,a^\ast,u_{P,a^\ast}(x)\bigr).
\]
Thus, for every $m\in J(x)$,
\[
g(x,a^\ast,m)=\pi_{a^\ast}^P(x)\beta_{a^\ast}(x)m+
\{g(x,a^\ast,u_{P,a^\ast})-\pi_{a^\ast}^P(x)\beta_{a^\ast}(x)u_{P,a^\ast}\}.
\]
This contradicts (iii). The pointwise identities used here follow from Assumptions~\ref{ass:rich}, \ref{ass:meas-main}, and~\ref{ass:armwise} by Corollaries~\ref{cor:combined-pt} and~\ref{cor:matched-arm}; their integrability is supplied by (B). Hence the contradiction applies to every member of $\IFRL$, and $\IFRL=\TRL(\Psi;\eta,\kappa,\Wext)=\varnothing$.
\end{proof}

\section{Verification of the model conditions in the applications}\label{ssec:model-verification}
\label{ssec:verification}

This section verifies the max-arm perturbation conditions~(X1)--(X3) and the shared-score perturbation conditions~(COV1)--(COV3) in the two structures that carry the closed forms. Bounded linear tilts give exact mean shifts for arm-specific scores and for finite shared-score models. All kernel replacements follow the model-cost localization convention of Assumption~\ref{ass:meas}.

\begin{proposition}[Max-arm structure with nonparametric outcome
model]\label{prop:verify-maxarm} Let the links be $g(x,a,u)=g_a^\prime(x)u+\rho(x,a)$ with measurable finite coefficients and $g_a^\prime(x)\ne0$. Let $\cP$ contain all max-arm laws with overlap $\pi_a(x)\ge\epsilon_0>0$, finite positive arm-conditional score variances, and finite integrated model cost
\[
C(x,Q)\eqdef 1+\sum_{a\in\cA}\E_{Q_a}
\bigl[|h_a|+|g_a^\prime(x)h_a+\rho(x,a)|\bigr],
\qquad \E_{P_X}[C(X,Q_X^P)]<\infty.
\]
Here $Q_a$ is the arm-$a$ conditional law of $Q$. Assume the kernel graphs satisfy Assumption~\ref{ass:meas-main}. Then conditions~(X1)--(X3) in Assumption~\ref{ass:armwise}, with cost-localized gluing, hold at every such evaluating law, including on positive-mass ties. No exponential-moment condition is needed.
\end{proposition}

\begin{proof}
For~(X1), every mixture $(1-\epsilon)Q_{x,a}^P+\epsilon\nu$ with $\epsilon\in(0,1)$ and finitely supported $\nu$ of the same mean has finite positive conditional variance and finite pointwise cost. A measurable field of these mixtures need not have integrable cost. On the sets where its cost is at most $m$, paste it into the baseline kernel and leave the baseline elsewhere. The resulting law has integrable cost, preserves all means and propensities, and belongs to $\cP$. These sets exhaust the covariate space as $m\to\infty$, giving exactly the localized gluing required for~(X1).

For~(X2)--(X3), fix an arm and write $u=\E_Q h$ for its baseline mean. For a positive integer $M$, define
\[
b_M=\operatorname{clip}(h,-M,M)-\E_Q\operatorname{clip}(h,-M,M),
\qquad v_M=\E_Q[(h-u)b_M].
\]
Finite conditional variance implies $v_M\to\Var_Q(h)>0$. Choose the least integer $M=M(x,a)$ for which $v_M>0$. This is measurable, and $|b_M|\le2M$. For $|\delta|<\varepsilon_a(x)\eqdef v_M/(4M)$, set
\begin{equation}\label{eq:bounded-mean-tilt}
\frac{\mathrm dQ_{x,a}^{\delta}}{\mathrm dQ_{x,a}^P}
=1+\frac{\delta b_M}{v_M}.
\end{equation}
Its density lies between $1/2$ and $3/2$, integrates to one, and gives the exact identity $\E_{Q_{x,a}^{\delta}}h_a=u_{P,a}(x)+\delta$. It preserves support and finite positive conditional variance.

Let $D_x$ consist of the vectors with $|\delta_a|<\varepsilon_a(x)$ whose perturbed arm values have the same active set as the truth. Its graph is Borel. Every inactive-arm coordinate direction is in $D_x$ for sufficiently small displacement. The active path $r_{x,a}(t)=t/g_a^\prime(x)$ for $a\in\cA^\ast(x;P)$, and zero otherwise, lies in $D_x$ whenever
\[
|t|<\tfrac12\min\bigl\{1,\Delta_+(x),
\min_{a\in\cA^\ast(x;P)}\varepsilon_a(x)|g_a^\prime(x)|\bigr\}.
\]
The minimum uses the convention $\Delta_+=\infty$ when all arms are active. This positive measurable radius proves~(X2), including exact preservation of ties.

Finally mix these arm kernels with the unchanged probabilities $\pi^P(x)$ to obtain $Q_{x,\bm\delta}$. This kernel is jointly measurable and its density relative to $Q_x^P$ is between $1/2$ and $3/2$. Thus it preserves the propensity component, changes only the arm means, and has the required active set. Moreover $C(x,Q_{x,\bm\delta})\le(3/2)C(x,Q_x^P)$, so every measurable selection from this bounded-density bridge has integrable model cost and glues directly. This proves~(X3) without asserting gluing for uncontrolled kernel fields.
\end{proof}

\begin{proposition}[Shared-score structure on a finite residual
space]\label{prop:verify-shared} Assume Assumptions~\ref{ass:rich} and~\ref{ass:meas-main}. Let $\cW$ be finite and $\bm h$ a fixed $\R^K$-valued map on $\cW$. Fix $P\in\cP$ with $\supp Q_x^P=\mathcal S$ for $P_X$-almost every $x$. Assume the admissible local graph contains $(1-\epsilon)Q_x^P+\epsilon\nu$ for every $\epsilon\in(0,1)$ and every probability measure $\nu$ on $\mathcal S$ with score mean $\bm u_P(x)$. Suppose the declaration satisfies the following conditions.
\begin{enumerate}[wide=0pt,label={\textup{(V\arabic*)}},itemsep=1pt]
\item \emph{Fibre preservation.}~One of the following, which are the two ways of making the fixed-mean contaminations stay inside the fibre of \eqref{eq:QxP-rich}.
\begin{enumerate}[leftmargin=1.6em,label={\textup{(\alph*)}},itemsep=1pt]
\item \emph{Saturated declaration.}~The bridged component is mean-saturated, $\eta_{j_\circ}(P')=\bm u(P')$, and there is a \emph{single} support-preserving neighbourhood $\cU_x$ of $Q_x^P$ on which every other declared component $\eta_l$, $l\ne j_\circ$, together with $\kappa$, is constant, and which contains the whole contamination segment $(1-\epsilon)Q_x^P+\epsilon\nu$ for \emph{all} $\epsilon\in(0,1)$ and every admissible $\nu$ used below.
\item \emph{Invariant fibre.}~Every support-preserving interior contamination with $\E_{q}\bm h=\bm u_P(x)$ leaves \emph{all} declared components and $\kappa$ unchanged.
\end{enumerate}
\item \emph{Index preservation.}~There is a fixed primary index $j_\circ$ and a measurable readout $H$ with $\bm u(P')=H\{\eta_{j_\circ}(P')\}$, and every score-tilt kernel below changes $\eta_{j_\circ}$ only, leaving $\eta_l$, $l\ne j_\circ$, and $P_X$ unchanged.
\item \emph{Model and structural preservation.} There is a positive measurable radius $\rho(x)$ such that, for every $\bm\delta\in T_{\cA^\ast}(x)$ with $\|\bm\delta\|<\rho(x)$ for which the linear-tilt probabilities below are positive, that kernel belongs to $\mathfrak K_x$ and has structural value $\kappa(P)(x)$.
\end{enumerate}
Then (COV1)--(COV3) hold at $P$, with $D_x$ a relative neighbourhood of $\bm0$ in $T_{\cA^\ast}(x)$. For the canonical declaration $\eta_1=\bm u$ and $\kappa=\argmin_lu_l$, structural preservation in (V3) follows from the active-set restriction and the positive gap to inactive coordinates.
\end{proposition}

\begin{proof}
For (COV1), each prescribed contamination is admissible and has support $\mathcal S$ because $\epsilon<1$. Under (V1)(a), its bridged component is the unchanged mean vector and the other components and $\kappa$ are constant on the specified neighbourhood. Under (V1)(b), all these equalities hold directly. Thus each contamination belongs to $\mathfrak Q_x^P$.

For mean shifts, write $p_i(x)=Q_x^P(\{w_i\})>0$, $\bm h_i=\bm h(w_i)$, $\bm u=\sum_i p_i\bm h_i$, and $\bm\Sigma=\sum_i p_i(\bm h_i-\bm u)(\bm h_i-\bm u)^\top$. For $\bm\delta\in\operatorname{range}\bm\Sigma$, define
\[
q_i(x,\bm\delta)=p_i(x)\left\{1+(\bm h_i-\bm u)^\top\bm\Sigma^+\bm\delta\right\}.
\]
Direct summation gives
\[
\sum_iq_i=1,\qquad
\sum_iq_i\bm h_i=\bm u+\bm\Sigma\bm\Sigma^+\bm\delta
=\bm u+\bm\delta.
\]
Set $B(x)=\max_i\|\bm\Sigma^+(\bm h_i-\bm u)\|$. If $\|\bm\delta\|<\{2(1+B(x))\}^{-1}$, every density ratio $q_i/p_i$ lies between $1/2$ and $3/2$.

Let $\Delta(x)=\min_{l\notin\cA^\ast(x)}\{u_l-\min_ku_k\}$, with $\Delta=\infty$ if all indices are active. Choose
\[
\varepsilon(x)=\min\left\{1,\rho(x),\frac{1}{2(1+B(x))},\frac{\Delta(x)}4\right\},
\qquad
D_x=\{\bm\delta\in T_{\cA^\ast}(x):\|\bm\delta\|<\varepsilon(x)\}.
\]
This is a Borel graph and contains relative neighbourhoods of zero in both $T_{\cA^\ast}(x)$ and its subspace $T_0(x)$. The common active shift and the gap bound preserve the active set, proving (COV2).

The kernel with masses $q_i$ is jointly Borel, dominated by $Q_x^P$, and has exactly the desired mean. Condition (V2) preserves every primary component other than $j_\circ$, while (V3) preserves $\kappa$ and places the kernel in the admissible graph. Its local nuisance value is Borel by Assumption~\ref{ass:meas-main}. Finite model cost and (M2), after localization if needed, give every required substituted law. This proves (COV3).
\end{proof}

\begin{remark}[Checking the declaration]
For the canonical declaration $d=1$, $\eta_1=\bm u$, and $\kappa=\argmin_lu_l$, fixed-mean contaminations preserve the declaration, and the mean shifts above preserve $\kappa$ on $D_x$. Thus the proposition applies whenever these kernels belong to the comparison model. A more detailed declaration requires its additional components to satisfy (V1)--(V3). For example, with $\mathcal S=\{-1,0,1\}$, $\bm h(W)=(W,W)$, and an additional component $\eta_2(q)=\log\{q(-1)q(1)/q(0)^2\}$, moving from the uniform law to $(0.4,0.2,0.4)$ preserves the score mean but changes $\eta_2$ by $\log4$. Mean preservation alone therefore does not verify (V1).
\end{remark}

The second proposition explains the ``modulo $\ker\bm\Sigma$'' qualifications in Lemma~\ref{lem:func-eq-cov}: the kernel directions are unreachable by every law equivalent to the truth, so no argument based on such laws can identify the corresponding coefficients.

\section{Applications}\label{ssec:further-apps}
\subsection{Recovery of the OTR efficiency bound}\label{sec:otr-recovery}

This subsection proves Corollary~\ref{rem:otr-recovery-main} by specializing the harmonic formula to binary treatments. The calculation also identifies the optimal weights under unequal conditional outcome variances.

Take $\cX=\R^p$, $\cW=\{0,1\}\times\R$, with $W=(A,Y)$. The target is $\Psi(P)=\E[\max(\E_P[Y\mid X,A=0],\E_P[Y\mid X,A=1])]$, with assumptions: SUTVA, no unmeasured confounding, and overlap ($\pi_a(x)\eqdef P(A=a\mid X=x)\in(0,1)$, $P_X$-a.s.).

The marginal-integral representation is $\cU=\R^2$, $\psi(x,u_0,u_1)=\max(u_0,u_1)$, $u_{P,a}(x)=\E_P[Y\mid X=x,A=a]$. The structure is of max-arm type (Definition~\ref{def:max-arm}) with $\cA=\{0,1\}$, $\widetilde\cW=\R$, $\widetilde W=Y$, $g(x,a,u_a)=u_a$ (so $g_a^\prime\equiv 1$), and $h_a(y)=y$.

The primary nuisance has $d=2$: $\eta_1(P)=(\E_P[Y\mid X,A=a])_{a=0,1}$, $\eta_2(P)=\pi_1^P(\cdot)$.  The structural component is the set-valued active set
\[
\kappa(P)(x)=\argmax_{a\in\{0,1\}}u_{P,a}(x).
\]
In the binary problem this active set is in one-to-one correspondence with $\operatorname{sgn}\{u_{P,1}(x)-u_{P,0}(x)\}\in\{-1,0,1\}$, the structural encoding used by \cite{XuGuo}.  The auxiliary working space encodes $\widehat\sigma_a^2(x)$, working variances of $Y$ given $(X,A=a)$.

Under the perturbation conditions of Proposition~\ref{prop:verify-maxarm} and oracle realization in Assumption~\ref{ass:oracle-maxarm}, Theorems~\ref{thm:ulb} and~\ref{thm:kink} give the conditional variance floor
\begin{equation}\label{eq:otr-vmin}
\Vmin(x;P)
=\biggl(\sum_{a\in\cA^\ast(x)}\frac{\pi_a^P(x)}{\sigma_a^2(x;P)}\biggr)^{-1},
\end{equation}
with $\cA^\ast(x)=\argmax_{a\in\{0,1\}}u_{P,a}(x)$. On the strict region ($u_{P,0}(x)\neq u_{P,1}(x)$), $|\cA^\ast(x)|=1$, and the strict-branch formula reduces this to $\sigma_{a^\ast(x)}^2(x;P)/\pi_{a^\ast(x)}^P(x)$, the standard AIPW conditional variance. On the kink set ($u_{P,0}(x)=u_{P,1}(x)$), both arms are active and~\eqref{eq:otr-vmin} gives the harmonic form, recovering precisely the Xu--Guo formula in their Theorem~A.1.

The total floor is
\begin{equation}\label{eq:otr-Vlow}
\Vlow(P)=\Var_{P_X}[\max(u_{P,0},u_{P,1})]
+\E_{P_X}\!\left[\biggl(\sum_{a\in\cA^\ast(X)}
\frac{\pi_a^P(X)}{\sigma_a^2(X;P)}\biggr)^{-1}\right],
\end{equation}
matching the OTR formula in \cite[Supplement, Theorem~A.1]{XuGuo}.

At the correct outcome-regression and propensity-score limits, the binary problem has the same unit-sum constraint and variance-minimization program as in \citet{XuGuo}. Their two-parameter adaptive smoothing and the cross-fitted construction of Theorem~\ref{thm:achieve} attain this common limiting optimum under their respective estimation conditions. Equation~\eqref{eq:otr-vmin} includes unequal conditional outcome variances.

The dynamic optimal-policy value also has a marginal-integral representation.  Let $H_1$ be the baseline history, let $W$ collect the remaining longitudinal trajectory, and let $u_{P,a_1}(H_1)$ be the optimal continuation value following initial action $a_1$, obtained by the usual backward recursion under sequential identification.  Then
\[
\Psi_{\mathrm{dyn}}(P)
=\E_P\!\left[\max_{a_1}u_{P,a_1}(H_1)\right]
\]
has \eqref{eq:mi-rep} with $X=H_1$ and scoring function $\psi(H_1,u)=\max_{a_1}u_{a_1}$.  This establishes representation-level inclusion; the recursive nuisance construction and its RALU attainment argument are those of \cite[Supplement~B]{XuGuo}, not a corollary of the one-stage max-arm theorem above.

Xu and Guo also treat the extension from a single treatment decision to multiple-time-point dynamic treatment regimes.  That development is part of the same RALU program and is not presented here as a new application.  The contribution of the present paper is the general marginal-integral variational theory and the additional max-arm, shared-score, Balke--Pearl, calibration, and mediation developments.

\subsection{Multi-arm optimal treatment values}\label{sec:multi-arm-otr}

The framework's max-arm closed form (Theorem~\ref{thm:kink}) was specialized in Section~\ref{sec:otr-recovery} to binary-action OTR. We verify it also handles the $K$-arm case, which is a useful stress-test of the harmonic-mean structure.

\subsubsection{Setup}\label{sec:multi-arm-setup}
Let $\cA=\{1,\dots,K\}$ for some $K\ge 2$, $\cW=\cA\times\cY$ with $\cY=\R$. Target:
\begin{equation}\label{eq:multi-arm-target}
\Psi(P)=\E_{P_X}\!\bigl[\max_{a\in\cA}u_{P,a}(X)\bigr],
\end{equation}
with $u_{P,a}(x)=\E_P[Y\mid X=x,A=a]$. The max-arm structure (Definition~\ref{def:max-arm}) applies with $g(x,a,u_a)=u_a$ ($g_a^\prime\equiv 1$), $h_a(y)=y$, $\widetilde W=Y$, $\widetilde\cW=\R$.

The active set
\begin{equation}\label{eq:multi-Aast}
\cA^\ast(x;P)=\bigl\{a\in\cA:u_{P,a}(x)=\max_b u_{P,b}(x)\bigr\}
\end{equation}
can have any cardinality from $1$ (a unique optimal arm) up to $K$ (all arms tied).

\subsubsection{Closed-form variance bound}\label{sec:multi-arm-Vlow}
Under the same perturbation and oracle-realization conditions, Theorem~\ref{thm:kink} gives, on each region $\{x:\cA^\ast(x)=S\}$ for nonempty $S\subseteq\cA$, the harmonic form
\begin{equation}\label{eq:multi-arm-Vmin}
\Vmin(x;P)\bigm|_{\cA^\ast(x)=S}
=\biggl(\sum_{a\in S}\frac{\pi_a^P(x)}{\sigma_a^2(x;P)}\biggr)^{\!-1}.
\end{equation}
The total floor is
\begin{equation}\label{eq:multi-arm-Vlow}
\Vlow(P)=\Var_{P_X}\!\bigl[\max_a u_{P,a}(X)\bigr]
+\E_{P_X}\!\left[\biggl(\sum_{a\in\cA^\ast(X)}\frac{\pi_a^P(X)}{\sigma_a^2(X;P)}\biggr)^{\!-1}\right].
\end{equation}

\begin{corollary}[Efficiency comparison with softmax inference]
\label{cor:whitehouse-otr-dominance}
Assume the equality conditions of Theorem~\ref{thm:kink} for $g(x,a,u_a)=u_a$. Let $S(x)=\cA^\ast(x;P)$ and $m(x)=|S(x)|$. Suppose the nonparametric softmax estimator of \citet{Whitehouse2025Softmax} has the asymptotic representation in their Theorem~3.3. Its conditional residual variance is
\begin{equation}\label{eq:whitehouse-otr-var}
V_{\rm W}(x)=\frac{1}{m(x)^2}
\sum_{a\in S(x)}\frac{\sigma_a^2(x;P)}{\pi_a^P(x)}.
\end{equation}
For $P_X$-a.e.\ $x$,
\begin{equation}\label{eq:whitehouse-otr-gap}
V_{\rm W}(x)\ge
\biggl(\sum_{a\in S(x)}
\frac{\pi_a^P(x)}{\sigma_a^2(x;P)}\biggr)^{-1}
=\Vmin(x;P).
\end{equation}
Equality holds if and only if $\sigma_a^2(x;P)/\pi_a^P(x)$ is constant over $a\in S(x)$. If $\max_a u_{P,a}(X)\in L^2(P_X)$ and $\E_{P_X}\{V_{\rm W}(X)\}<\infty$, the asymptotic variance of the softmax estimator is no smaller than~\eqref{eq:multi-arm-Vlow}. The inequality is strict if the equality condition fails on a set of positive $P_X$-probability.
\end{corollary}

\begin{proof}
The limiting representation in \citet[Theorem~3.3]{Whitehouse2025Softmax} is the centered evaluation, at the true nuisance functions, of
\begin{equation}\label{eq:whitehouse-full-slot}
F^{\rm W}(x,a,y;e,q,S)
=\frac{1}{|S|}\sum_{b\in S}
\left[e_b(x)+\frac{\1\{a=b\}}{q_b(x)}\{y-e_b(x)\}\right],
\end{equation}
where $e=(e_a)_a$ and $q=(q_a)_a$ are working outcome regressions and propensities. If $e=u_P$, its conditional expectation equals $|S|^{-1}\sum_{b\in S}u_{P,b}(x)$ for every $q$. If $q=\pi^P$, the same identity holds for every $e$, whenever the terms are integrable. At $S=S(x)$ this average equals $\max_a u_{P,a}(x)$.

Conditional on $X=x$, centered residuals from distinct treatment arms have disjoint support and zero covariance. Equal weighting therefore gives~\eqref{eq:whitehouse-otr-var}. Set $v_a(x)=\sigma_a^2(x;P)/\pi_a^P(x)$. The Cauchy--Schwarz inequality gives
\[
\left(\sum_{a\in S}v_a\right)
\left(\sum_{a\in S}v_a^{-1}\right)\ge |S|^2,
\]
which proves~\eqref{eq:whitehouse-otr-gap}; equality holds exactly when the positive numbers $(v_a)_{a\in S}$ are equal. Both representations have conditional mean $\max_a u_{P,a}(x)$, so total variance adds the same between-$X$ term. Integration proves the total comparison and its strictness criterion.
\end{proof}

Corollary~\ref{cor:whitehouse-otr-dominance} gives a direct asymptotic-variance comparison with the Whitehouse estimator. Section~\ref{ssec:otr-feasible-results} evaluates its softmax implementation alongside the feasible RALU estimator. For the binary optimal-treatment problem, the Xu--Guo boundary-asymmetry parameter is $t_0(x)=\lambda_1(x)$. Minimizing $t_0^2\sigma_1^2/\pi_1+(1-t_0)^2\sigma_0^2/\pi_0$ gives
\begin{equation}\label{eq:xuguo-t0-opt}
t_0^\ast(x)=
\frac{\pi_1^P(x)/\sigma_1^2(x;P)}
{\pi_1^P(x)/\sigma_1^2(x;P)+
 \pi_0^P(x)/\sigma_0^2(x;P)}.
\end{equation}
Under conditional homoscedasticity, this reduces to $t_0^\ast(x)=\pi_1^P(x)$.

If $P_X$ is absolutely continuous and each pairwise mean difference is smooth with a zero set of Lebesgue measure zero, treatment ties have probability zero. Then~\eqref{eq:multi-arm-Vlow} reduces to the usual strict-region integral $\E_{P_X}[\sigma_{a^\ast(X)}^2(X;P)/\pi_{a^\ast(X)}^P(X)]$, plus the variance of the conditional target. Positive-probability ties also occur with continuous covariates when two treatment means agree on a region, as in Section~\ref{sec:num-design}. The harmonic term then contributes directly to the total efficiency bound.

\subsubsection{A three-arm tie}\label{sec:multi-arm-3tie}
At $x$ where three arms tie, $\cA^\ast(x)=\{1,2,3\}$, the formula gives
\begin{equation}\label{eq:multi-arm-3tie-Vmin}
\Vmin(x;P)
=\biggl(\frac{\pi_1}{\sigma_1^2}+\frac{\pi_2}{\sigma_2^2}+\frac{\pi_3}{\sigma_3^2}\biggr)^{-1}.
\end{equation}
This is strictly smaller than any single-arm AIPW variance $\sigma_a^2/\pi_a$, $a=1,2,3$. The optimal weights $\lambda_a^\ast=(\pi_a/\sigma_a^2)/(\sum_b\pi_b/\sigma_b^2)$ weight each active arm in proportion to $\pi_a/\sigma_a^2$.

The minimizing uncentered evaluation on this tie region is
\begin{equation}\label{eq:multi-arm-IF-3tie}
\begin{aligned}
f^\ast(X,A,Y)
&=\max_a u_{P,a}(X)
  +\sum_{a\in\cA^\ast(X)}\lambda_a^\ast(X)\\[-2pt]
&\hspace{5em}\,\frac{\1\{A=a\}}{\pi_a^P(X)}
  \bigl(Y-u_{P,a}(X)\bigr),\\
&\hspace{2em}\text{where }\cA^\ast(X)=\{1,2,3\}.
\end{aligned}
\end{equation}
This is the kink-region restriction;  on the strict region $\{X:\cA^\ast(X)=\{a^\ast(X)\}\}$, the evaluation takes the single-arm AIPW form $\max_a u_{P,a}(X)+\1\{A=a^\ast(X)\}/\pi_{a^\ast}^P(X) (Y-u_{P,a^\ast}(X))$, which is the special case of \eqref{eq:multi-arm-IF-3tie} with $\lambda_{a^\ast}^\ast=1$ and the singleton sum reducing to one term.  The full evaluation is the $\cA^\ast(X)$-dependent piecewise expression with the harmonic-weight prescription~\eqref{eq:multi-arm-Vmin} applied at each $X$.  This has the familiar multi-arm AIPW structure \citep{Cattaneo2010}; the harmonic tie weights are determined by the RALU variance minimization above.

The active set in~\eqref{eq:multi-arm-IF-3tie} can change as $X$ varies. The margin condition in Theorem~\ref{thm:achieve} controls estimation error near these changes by bounding the probability of a small positive gap between optimal and suboptimal arms.

\subsubsection{Verification of the framework conditions}\label{sec:multi-arm-passes}
For multi-arm OTR, exactly one treatment is observed per subject, and each arm summary is a conditional mean. Proposition~\ref{prop:verify-maxarm} verifies the perturbation conditions for any finite arm set. Assumption~\ref{ass:oracle-maxarm} supplies the remaining realization requirement for the exact harmonic formula; the same full nuisance rule and envelope used in the binary case apply with $\cA=\{1,\ldots,K\}$.

\subsection{Calibration error and exact robustness}\label{sec:lp-cal}

Let $g:\cX\to\R$ be a fixed measurable predictor and consider $\Psi_p(P)=\E_{P_X}[|\E_P(Y\mid X)-g(X)|^p]$, $p\ge1$, with $W=Y$ and $u_P(x)=\E_P[Y\mid X=x]$. Under the rich conditional-mean models specified below, $p=1$ admits an exactly robust estimating function when the active set is included in $\kappa$. For $p>1$, declaring the mean as the sole primary component and taking $\kappa$ constant makes the exact class empty on any model with the required nondegenerate mean perturbations.

\subsubsection{Nonlinear powers: emptiness under rich mean perturbations}\label{sec:lp-empty}

For $p>1$, use the declaration $\eta_1(P)(x)=\E_P[Y\mid X=x]$ and constant $\kappa$, and suppose that on a set of positive $P_X$-measure the comparison model permits conditional mixtures of laws with distinct means, together with the measurable localization needed to apply condition~(B) conditionally. The resulting mean segment is nondegenerate. The function $m\mapsto|m-g(x)|^p$ is nonaffine on every such segment, giving the same affine-expectation obstruction as Proposition~\ref{prop:affine-necessary}. The conclusion concerns the model-reachable means for this declaration.

The obstruction can also be seen directly in a mixture-rich conditional model. Fix a working mean and its estimating function, and vary the true conditional law along a mixture of two laws having different means. Its expectation is affine in the mixture parameter, whereas $|u-g(x)|^p$ is strictly convex along the corresponding nonconstant mean segment. With $d=1$ and constant $\kappa$, condition~(B) requires equality throughout that segment, which is impossible. This argument requires no differentiability of the estimating function with respect to its nuisance input.

Exact unbiasedness of the usual first-order correction would require
\begin{equation}\label{eq:lp-taylor-gap}
|m-g|^p+p\,|m-g|^{p-1}\mathrm{sign}(m-g)\,(u_P-m)=|u_P-g|^p
\qquad\text{for all }m,
\end{equation}
which would require the first-order Taylor expansion to be exact. For $p>1$ it is not exact on an open mean interval. A first-order orthogonal estimating function can still be constructed; exact robustness imposes the stronger finite-change identity. For $p=1$, the link is affine on either side of $g$, and including the active set in the structural component yields the construction below.

\subsubsection{Absolute calibration error}\label{sec:lp1}

For $p=1$ write $|m-g(x)|=\max\{m-g(x),\,g(x)-m\}$, a max of two affine links.  The natural marginal-integral structure is therefore the shared-data one of Section~\ref{sec:cov-extended}, with $K=2$ identifying scores read off the \emph{same} observation,
\begin{equation}\label{eq:lp1-scores}
h_+(x;Y)\eqdef Y-g(x),\qquad h_-(x;Y)\eqdef g(x)-Y,
\end{equation}
$m_P(x)\eqdef\E_P[Y\mid X=x]$ and $u_P(x)=(u_{P,+},u_{P,-})(x)=(m_P(x)-g(x),\,g(x)-m_P(x))$, $\psi(x,u)=\max\{u_+,u_-\}$, and structural component $\kappa(P)(x)=\cA^\ast(x;P)\subseteq\{+,-\}$ the set of maximizers.  The kink set is
\begin{equation}\label{eq:lp-kink-set}
\cK_P\eqdef\{x\in\cX:m_P(x)=g(x)\}
=\{x:\cA^\ast(x;P)=\{+,-\}\},
\end{equation}
on which $\kappa$ takes the two-element value; off $\cK_P$ it is a singleton.  Writing $\sigma_Y^2(x;P)\eqdef\Var_{Q_x^P}[Y]$, the score covariance is
\begin{equation}\label{eq:lp1-Sigma}
\bm\Sigma(x;P)=\sigma_Y^2(x;P)
\begin{pmatrix}1&-1\\-1&1\end{pmatrix},
\end{equation}
since $h_-=-h_+$ identically. Use $d=1$ with $\eta_1(P)=m_P$ and the set-valued $\kappa$ above. Assume $\E_P[|Y|+|g(X)|]<\infty$ throughout the comparison model, finite conditional variances as in Definition~\ref{def:min-of-linear}, and the shared-score perturbation conditions of Assumption~\ref{ass:shared-perturb}. A full-domain rule assigns $h_+$ to $\{+\}$, $h_-$ to $\{-\}$, and $(h_++h_-)/2$ to $\{+,-\}$; it is independent of the primary input, integrable, and unbiased at every true structural value. For use of the min-based results in Section~\ref{sec:cov-extended}, negate both scores and the target, and then negate the resulting estimating function. This preserves all variances. At the evaluating law $P$, assume $\E_{P_X}[(m_P(X)-g(X))^2+\1\{X\notin\cK_P\}\sigma_Y^2(X;P)]<\infty$. This is the square-integrability condition for the displayed minimizing evaluation.

Strict region. At $x\notin\cK_P$ the active face is a singleton, $\bm\Sigma_{\cA^\ast}=[\sigma_Y^2(x;P)]$ is a positive scalar when $\sigma^2_Y(x;P)>0$, and Theorem~\ref{thm:cov-kink} gives
\begin{equation}\label{eq:lp1-vmin-strict}
\Vmin(x;P)=\sigma_Y^2(x;P),
\end{equation}
attained by $f^P_w(x,Y)=|m_P(x)-g(x)|+\varsigma(x)\,(Y-m_P(x))$ with $\varsigma(x)=+1$ on $\{m_P>g\}$ and $-1$ on $\{m_P<g\}$. When $\sigma_Y^2(x;P)=0$, this evaluation is constant and variance nonnegativity gives the same formula.

Kink set. At $x\in\cK_P$ both indices are active and $\bm\Sigma_{\cA^\ast}(x;P)$ is the singular matrix in~\eqref{eq:lp1-Sigma}, whose range is contained in $\operatorname{span}\{(1,-1)\}$; hence $\bm 1\notin\operatorname{range}(\bm\Sigma_{\cA^\ast})$ and Proposition~\ref{prop:cov-singular} gives
\begin{equation}\label{eq:lp1-vmin-kink}
\Vmin(x;P)=0\qquad (x\in\cK_P),
\end{equation}
attained by the unit-sum weights $\bm\lambda^\ast=(\tfrac12,\tfrac12)$, for which $\tfrac12h_++\tfrac12h_-\equiv0$ and $f^P_w$ is the constant $\psib(x)=0$.  This is admissible: on the $\kappa$-preserving region the two active values coincide, which forces $m=g(x)$ and hence $\psib=0$ there, so the constant section satisfies the pointwise constraint at every compatible ambient law. The equal-face softmax limit of \citet[Corollary~4.8]{Whitehouse2025Softmax} uses the same $(1/2,1/2)$ coefficients.  Thus the barycentric member is already RALU-optimal for $L^1$ calibration, because the two active residual scores are exact negatives rather than because equal weighting is optimal for every tied maximum.

Consequently
\begin{equation}\label{eq:lp1-Vlow}
\Vlow(P)=\Var_{P_X}\bigl[|m_P(X)-g(X)|\bigr]
+\E_{P_X}\bigl[\1\{X\notin\cK_P\}\,\sigma_Y^2(X;P)\bigr].
\end{equation}

\begin{remark}[Structural information at exact calibration]\label{rem:lp1-kappa}
Consider instead the binary structural value $\varsigma(P)(x)=1$ when $m_P(x)\ge g(x)$ and $-1$ otherwise. At a tie, fixing $\varsigma=1$ permits nearby means above $g(x)$. Under the same conditional-law richness, exact unbiasedness over that half-interval forces the score $Y-g(x)$ and gives conditional variance $\sigma_Y^2(x;P)$. Choosing the negative branch at equality gives the analogous score $g(x)-Y$.

Under the set-valued convention used here, fixing $\kappa(P)(x)=\{+,-\}$ forces $m_P(x)=g(x)$. The target is then zero throughout the compatible conditional model, and the constant-zero score is admissible. Thus the two conventions specify different robustness constraints and different variance bounds at the tie; they agree away from it. This is the dependence on structural information described after Definition~\ref{def:ralu}.
\end{remark}

\begin{remark}[Comparison with $L^2$ calibration]\label{rem:lp-vs-l2}
For $p=2$, the target $\Psi_2(P)=\E_{P_X}[(m_P(X)-g(X))^2]$ has the usual pathwise derivative and a classical efficient influence function whenever the corresponding gradient is square integrable. Under the richness and declaration in Section~\ref{sec:lp-empty}, it nevertheless has no exactly robust map satisfying~(B). Thus pathwise differentiability and exact nuisance-substitution robustness impose different requirements.
\end{remark}

\subsection{Balke--Pearl bounds in instrumental variables}\label{sec:bp-bound}

The branches in this application share their residual data and hence lead to the same covariance quadratic program as the shared-score model.  Their doubly robust signals also depend on nuisance slots. We therefore establish the exact-class lower bound by a separate finite-cell argument and prove estimator attainment below.

\subsubsection{Setup}\label{sec:bp-setup}

Let $\cW=\cZ\times\cA\times\cY$ with $\cZ=\{0,1\}$ (binary IV), $\cA=\{0,1\}$ (binary treatment), $\cY=\{0,1\}$ (binary outcome). Write $W=(Z,A,Y)$. Under standard IV assumptions (random assignment of $Z$, exclusion restriction $Y(a,z)=Y(a)$), the conditional law of $(A,Y)$ given $(Z,X)$ is summarized by the eight probabilities
\begin{equation}\label{eq:bp-pmoments}
p_{ay\mid z}(x)\eqdef P(A=a,Y=y\mid Z=z,X=x),
\quad (a,y,z)\in\{0,1\}^3.
\end{equation}
Together with the IV propensity
\begin{equation}\label{eq:iv-prop}
\rho_z(x)\eqdef P(Z=z\mid X=x),
\end{equation}
these summarize the observed-data distribution conditional on $X$. The comparison model $\cP$ uses one overlap constant $\epsilon>0$: every truth and law-achievable working propensity lies in $[\epsilon,1-\epsilon]$.  Cell vectors lie in the two probability simplices, including their boundaries.  Richness and localized gluing refer to this comparison model; estimation conditions are imposed separately on $\cP_{\mathrm{ach}}\subseteq\cP$.

Applying the sharp bounds of \citet{BalkePearl1997} within each covariate stratum gives bounds on the conditional ATE $\tau(x)\eqdef \E[Y(1)-Y(0)\mid X=x]$ of the form
\begin{equation}\label{eq:bp-bounds}
\mathrm{LB}(x)\le\tau(x)\le\mathrm{UB}(x),
\end{equation}
where $\mathrm{UB}(x)=\min_{k=1,\dots,8}U_k(x)$ and $\mathrm{LB}(x)=\max_{k=1,\dots,8}L_k(x)$, each $U_k, L_k$ being a specific signed sum of $\{p_{ay\mid z}(x)\}$ with integer coefficients. We take the sixteen expressions in the form generated by the response-type linear program rather than by transcription: writing $q(x)$ for the vector of latent response-type probabilities over the sixteen combinations of $(A(0),A(1))$ and $(Y(0),Y(1))$, randomization of $Z$ and the exclusion restriction make the observed cells a linear image $\bm B\,q(x)=p(x)$, and the sharp bounds are the values of the linear programs $\max/\min\{c_{\mathrm{ATE}}^\top q: \bm Bq=p,\ q\ge0\}$. The upper-bound expressions are exactly the vertices of the dual feasible set $\{v:\bm B^\top v\ge c_{\mathrm{ATE}}\}$, taken modulo its one-dimensional lineality space (along which $p^\top v$ is constant because each $Z$-slice of $p$ sums to one).  There are eight such vertices, with coefficient vectors
\begin{equation}\label{eq:bp-U-explicit}
\arraycolsep=3pt
\begin{array}{lll}
U_1=-p_{01\mid0}+p_{00\mid1}+p_{01\mid1}+p_{11\mid1},
&\quad U_2=p_{00\mid1}+p_{11\mid1},\\[2pt]
U_3=p_{01\mid0}+p_{11\mid0}+p_{00\mid1}-p_{01\mid1}+p_{11\mid1},
&\quad U_4=p_{00\mid0}+p_{11\mid0},\\[2pt]
U_5=p_{00\mid0}+p_{01\mid0}+p_{11\mid0}-p_{01\mid1},
&\quad U_6=p_{00\mid0}-p_{01\mid0}+p_{11\mid0}+p_{01\mid1}+p_{11\mid1},\\[2pt]
U_7=p_{00\mid0}+p_{10\mid0}+p_{00\mid1}-p_{10\mid1}+p_{11\mid1},
&\quad U_8=p_{00\mid0}-p_{10\mid0}+p_{11\mid0}+p_{00\mid1}+p_{10\mid1},
\end{array}
\end{equation}
all arguments being evaluated at $x$; the eight lower-bound vectors are obtained in the same way from $-c_{\mathrm{ATE}}$.  No consolidation by symmetry is performed.  The minimum over a strict subset of the eight remains a valid upper bound, being a minimum over fewer valid bounds, but it is not generally \emph{sharp}.  All eight expressions are therefore retained.  As a check, the law with $Z$ randomized, $A=Z$ and $Y=A$ (all units compliers with $Y(1)=1$, $Y(0)=0$, so $\tau\equiv1$) gives $\min_k U_k(x)=1$, as it must.

For the population upper bound, $\mathrm{UB}=\E_{P_X}[\min_k U_k(X)]$. This is \eqref{eq:mi-rep} with $\cU=\R^8$, $u_P(x)=(U_k(x))_{k=1}^8$, and $\psi(x,u)=\min_k u_k$. For the population lower endpoint,
\[
\mathrm{LB}=\E_{P_X}[\max_k L_k(X)]
=-\E_{P_X}[\min_k\{-L_k(X)\}].
\]
Thus the same min-based construction is applied to the eight negated lower-bound expressions, and the resulting score and estimator are multiplied by $-1$ to recover $\mathrm{LB}$.  This final sign change leaves the covariance matrix, optimal weights, and all variance comparisons unchanged.  Only the upper endpoint is developed below.

\subsubsection{Identifying scores with shared residual data}
\label{sec:bp-gap}

Each $U_k(x)$ is a linear functional of $\{p_{ay\mid z}(x)\}$, and the IV propensity $\rho_z(x)=P(Z=z\mid X=x)$ enters any estimating function for it.  Pure inverse-propensity weighting is \emph{not} admissible here: the signal $\sum_{a,y,z}c^{(k)}_{ayz}\1\{(Z,A,Y)=(z,a,y)\}/\rho'_z$ is unbiased for $U_k$ only when $\rho'=\rho$, so it fails condition~(B) in the propensity direction.  The correct full-slot signal is the doubly robust one,
\begin{equation}\label{eq:bp-dr-score}
\varphi_k\bigl(O;p',\rho'\bigr)\eqdef U_k(p')
+\sum_{z}\frac{\1\{Z=z\}}{\rho'_z}
\sum_{a,y}c^{(k)}_{ayz}\bigl\{\1\{A=a,Y=y\}-p'_{ay\mid z}\bigr\},
\end{equation}
whose two primary slots are the cell probabilities $p'$ and the IV propensity $\rho'$.  Writing $M_{kz}(p)\eqdef\sum_{a,y}c^{(k)}_{ayz}p_{ay\mid z}$, a direct computation gives the exact bias identity
\begin{equation}\label{eq:bp-dr-onewrong}
\E_P\bigl[\varphi_k(O;p',\rho')\mid X=x\bigr]-U_k(x)
=\sum_{z}\Bigl(1-\frac{\rho_z}{\rho'_z}\Bigr)
\bigl\{M_{kz}(p')-M_{kz}(p)\bigr\},
\end{equation}
which vanishes identically when $p'=p$ or when $\rho'=\rho$: this is exactly the one-wrong robustness required by Definition~\ref{def:ralu}~(B) with $d=2$, and the joint-error term is again a product, matching~\eqref{eq:R2-product-pointwise}.

At the truth the residual signal is
\begin{equation}\label{eq:bp-truth-residual}
\varphi_k^P-U_k=\sum_z\frac{\1\{Z=z\}}{\rho_z}
\bigl\{c^{(k)}_{AYz}-M_{kz}\bigr\},
\end{equation}
so its covariance matrix is
\begin{equation}\label{eq:bp-Sigma-dr}
\widetilde{\bm\Sigma}_{kl}(x;P)=\sum_z\frac{1}{\rho_z(x)}
\Bigl[\sum_{a,y}p_{ay\mid z}(x)\,c^{(k)}_{ayz}c^{(l)}_{ayz}
-M_{kz}(x)M_{lz}(x)\Bigr].
\end{equation}
Centering within each instrument stratum removes between-$Z$ variation from the residual covariance. Formula~\eqref{eq:bp-Sigma-dr} is used in every variance program below.

\begin{remark}[Variance characterization for nuisance-dependent scores]\label{rem:bp-scope}
The branch scores in~\eqref{eq:bp-dr-score} depend on nuisance values. Proposition~\ref{prop:bp-rigidity} obtains their variance lower bound by varying the conditional cells while freezing the estimating function, and Definition~\ref{def:bp-face-family} supplies an exactly robust map with bounded weights. Together they identify the exact floor whenever the cap contains a minimizing weight. Theorems~\ref{thm:bp-achieve} and~\ref{thm:bp-stable-rank} establish the cross-fitted expansions on positive-definite and separated singular covariance strata. Corollaries~\ref{cor:bp-one-estimator} and~\ref{cor:bp-achieve-global} give one estimator satisfying these expansions throughout a learnable subclass.
\end{remark}

\medskip

Shared-score covariance. In the max-arm specialization, the residual space factorizes as $\cW=\cA\times\widetilde\cW$ with $A$ a discrete arm indicator and $\widetilde W$ having an arm-conditional law $Q_{x,a}^P$. This structure is \emph{essential} to the harmonic formula: the indicators for distinct arms have disjoint supports, so the covariance matrix is block diagonal.

In the Balke--Pearl problem, the candidate ``arms'' are the eight bound indices $k=1,\dots,8$. \emph{They do not partition $\cW$.} Each signal $\varphi_k$ involves $Z,A,Y$ jointly; two different signals share underlying data and have non-zero covariance.  Concretely, $U_1=-p_{01\mid0}+p_{00\mid1}+p_{01\mid1}+p_{11\mid1}$ and $U_2=p_{00\mid1}+p_{11\mid1}$ share the cells $p_{00\mid1}$ and $p_{11\mid1}$, so their residuals are correlated and $\widetilde{\bm\Sigma}$ is not diagonal.

The variance program therefore uses the full conditional covariance matrix. A calculation based only on the diagonal entries yields the same optimum only under the equality criterion given in Section~\ref{sec:bp-sigma}.

\subsubsection{Generalized harmonic form for correlated arm scores}
\label{sec:bp-correct}

The conditional-variance reduction of Theorem~\ref{thm:ulb} applies without an arm partition. For the Balke--Pearl scores, its explicit value comes from a unit-sum quadratic program based on~\eqref{eq:bp-Sigma-dr}. Proposition~\ref{prop:bp-rigidity} establishes the corresponding lower bound for the full exact class.

For Balke--Pearl, minimize the conditional variance of a unit-sum combination of the doubly robust residuals.  The covariance matrix $\widetilde{\bm\Sigma}(x;P)\in\R^{8\times 8}$ is given by~\eqref{eq:bp-Sigma-dr}.  Writing $\cA^\ast(x)$ for the set of indices attaining $\min_kU_k(x)$, define the \emph{algebraic benchmark} $b_{\mathrm{BP}}(x;P)\eqdef \nu\bigl(\widetilde{\bm\Sigma}_{\cA^\ast(x)}(x;P)\bigr)$; its value and minimizer are
\begin{equation}\label{eq:bp-Vmin}
b_{\mathrm{BP}}(x;P)=\frac{1}{\bm 1^{\top}\widetilde{\bm\Sigma}_{\cA^\ast(x)}^{+}\bm 1},
\qquad
\bm\lambda^\ast(x)=\frac{\widetilde{\bm\Sigma}_{\cA^\ast(x)}^{+}\bm 1}
{\bm 1^{\top}\widetilde{\bm\Sigma}_{\cA^\ast(x)}^{+}\bm 1},
\end{equation}
when $\bm 1\in\operatorname{range}(\widetilde{\bm\Sigma}_{\cA^\ast(x)})$, and $b_{\mathrm{BP}}(x;P)=0$ otherwise, as in Proposition~\ref{prop:cov-singular}.  Proposition~\ref{prop:bp-rigidity} gives $\Vmin(x;P)\ge b_{\mathrm{BP}}(x;P)$ under its same-face conditions where the benchmark is positive; on the zero-value region the inequality follows from nonnegativity. A minimizing evaluation of the bounded map in Definition~\ref{def:bp-face-family} gives the reverse inequality. Estimator attainment under the common learning conditions is established in Corollaries~\ref{cor:bp-one-estimator} and~\ref{cor:bp-achieve-global}. The minimizing doubly robust evaluation is
\begin{equation}\label{eq:bp-IF}
f^\ast(x;W)=\min_k u_{P,k}(x)+\sum_{k\in\cA^\ast(x)}\lambda_k^\ast(x)
\bigl[\varphi_k^P(x;W)-u_{P,k}(x)\bigr].
\end{equation}

\subsubsection{Computing \texorpdfstring{$\widetilde{\bm\Sigma}$}{Sigma-tilde} for Balke--Pearl}\label{sec:bp-sigma}

Formula~\eqref{eq:bp-Sigma-dr} is a within-stratum computation: for each $z$, the vector $(c^{(k)}_{AYz})_k$ is centred at $(M_{kz})_k$ and its covariance under $p_{\cdot\mid z}$ is inflated by $1/\rho_z$.  Two properties distinguish~\eqref{eq:bp-Sigma-dr} from the pure-IPW form.

First, distinct bound expressions do \emph{not} have disjoint supports: $U_1$ and $U_2$ of~\eqref{eq:bp-U-explicit} share the cells $p_{00\mid1}$ and $p_{11\mid1}$, so no off-diagonal entry vanishes for that reason, and the off-diagonal structure must be computed rather than argued away.

Second, positive correlation between active scores does not imply that ignoring the correlation understates the floor.  With
$\widetilde{\bm\Sigma}=\bigl(\begin{smallmatrix}100&9\\9&1\end{smallmatrix}\bigr)$
the full program gives $19/83\approx0.229$, whereas treating the block as diagonal and using the harmonic form gives $100/101\approx0.990$, an over- rather than an underestimate.  There is no general ordering.  Write $\bm M\eqdef\widetilde{\bm\Sigma}_{\cA^\ast}$ and $D\eqdef\operatorname{diag}(\bm M)$, and suppose $D$ is positive definite.  Across the positive-semidefinite branches, the full program $\inf_{\bm 1^\top\bm\lambda=1}\bm\lambda^\top\bm M\bm\lambda$ and the diagonal harmonic value $(\bm 1^\top D^{-1}\bm 1)^{-1}$ agree exactly when
\[
\bm 1\in\operatorname{range}(\bm M)
\quad\text{and}\quad
\bm 1^\top\bm M^+\bm 1=\bm 1^\top D^{-1}\bm 1.
\]
If $\bm 1\notin\operatorname{range}(\bm M)$, the full program is zero while the diagonal harmonic value is positive. The arm-partition structure guarantees this equality because its covariance block is diagonal.  Outside that structure either ordering is possible, and off-diagonal cancellations can also produce equality.

\subsubsection{Universal floor for Balke--Pearl}\label{sec:bp-Vlow}

At every law satisfying the hypotheses of Proposition~\ref{prop:bp-rigidity} on the positive-value region, Theorem~\ref{thm:ulb} and the finite-cell rigidity argument below give the following positive-semidefinite-safe value.  For every $m\ge1$ and every symmetric positive-semidefinite matrix $\bm M\in\R^{m\times m}$, define
\begin{equation}\label{eq:bp-nu}
\nu(\bm M)\eqdef
\begin{cases}
\bigl(\bm 1^{\!\top}\bm M^{+}\bm 1\bigr)^{-1},
&\bm 1\in\operatorname{range}(\bm M),\\
0,&\text{otherwise},
\end{cases}
\end{equation}
one has
\begin{equation}\label{eq:bp-Vlow-final}
\Vlow(P)\;\ge\;V_{\mathrm{BP}}(P)\eqdef
\Var_{P_X}\!\bigl[\min_k U_k(X)\bigr]
+\E_{P_X}\!\bigl[b_{\mathrm{BP}}(X;P)\bigr].
\end{equation}
When the bounded map contains an algebraic minimizer, the reverse inequality gives $\Vlow(P)=V_{\mathrm{BP}}(P)$. Theorems~\ref{thm:bp-achieve} and~\ref{thm:bp-stable-rank} establish expansions for the positive-definite and separated singular cases; Corollaries~\ref{cor:bp-one-estimator} and~\ref{cor:bp-achieve-global} give common implementations on learnable subclasses. These learning conditions are separate from the comparison-model conditions used for the lower bound.

Singular active covariance matrices can occur with strictly positive cells. For example, with both $Z$-slices equal to $p_{\cdot\mid z}=(0.1,0.1,0.7,0.1)$ and $\rho_1=0.4$, the branch values are $U=(0.2,0.2,0.3,0.2,0.2,0.3,0.3,0.3)$ and the active set is $\{1,2,4,5\}$. Its covariance eigenvalues are $(0,0.2312,0.6494,1.0361)$, and the pseudoinverse formula gives $\nu=0.16$ with weights $(0.25,0.15,0.35,0.25)$. The range-based definition~\eqref{eq:bp-nu} covers this case as well as positive-definite blocks. Under the stated rigidity and attainment conditions, the result gives efficient inference when several expressions tie; efficiency under a unique active expression is studied by \citet{levis2025covariate}.

\begin{remark}[Consequences of shared residual scores]
\label{rem:bp-gap-importance}
The Balke--Pearl residuals share observations, so their covariance matrix determines the unit-sum variance program. The diagonal harmonic formula is guaranteed by disjoint arm observations. For a general covariance block, its equality with the full program is characterized in Section~\ref{sec:bp-sigma}; either ordering is possible when that criterion fails. Proposition~\ref{prop:bp-rigidity} establishes the lower bound for the nuisance-dependent Balke--Pearl scores, and the bounded full-domain construction gives equality when it contains a minimizing weight.
\end{remark}

\subsubsection{Rigidity for Balke--Pearl: the doubly robust value is the floor}\label{sec:bp-rigidity}

Proposition~\ref{prop:bp-rigidity} establishes the lower bound for the nuisance-dependent Balke--Pearl scores. The finite conditional sample space makes the argument explicit: every admissible evaluation is a unit-sum combination of centered doubly robust branch scores plus a function of $Z$. Within-instrument centering makes that last term orthogonal to the branch scores, so minimizing variance sets it to zero.

\begin{definition}[Balke--Pearl same-face neighbourhood]
\label{def:bp-same-face-neighborhood}
Fix $P_0\in\cP$.  For a nonempty face $S\subseteq\{1,\ldots,8\}$, let $\mathcal H_S$ be the affine manifold of cell vectors in the two $Z$-slice simplices satisfying $(\bm c^{(k)}-\bm c^{(l)})^{\!\top}\bm p=0$ for all $k,l\in S$. A correspondence $x\mapsto\Gamma_x$ is a \emph{gluable same-face neighbourhood} at $P_0$ on a measurable set $G\subseteq\cX$ if its graph is Borel and, for $P_X$-a.e.\ $x\in G$, all its cell probabilities are positive, all inactive inequalities are strict, and
\begin{equation}\label{eq:bp-face-hull}
\Gamma_x\supseteq B_{\epsilon(x)}\{\bm p_0(x)\}
\cap\mathcal H_{\cA^\ast(x;P_0)}
\quad\text{for some }\epsilon(x)>0.
\end{equation}
Moreover, every universally measurable selection $q(x)\in\Gamma_x$ on $G$, extended by $\bm p_0(x)$ outside $G$, with integrable model-admissibility cost must induce a law $P_q\in\cP$ satisfying
\[
(P_q)_X=(P_0)_X,\qquad
\bm p(P_q)=q,\qquad \rho(P_q)=\rho(P_0),\qquad
\kappa(P_q)=\kappa(P_0).
\]
Other selections are localized to bounded-cost sets and joined to the baseline elsewhere, as in Assumption~\ref{ass:rich}~(M2) and Assumption~\ref{ass:meas}.  When no set is specified, $G=\cX$.
\end{definition}

The relative ball in~\eqref{eq:bp-face-hull} is taken in the full same-face affine manifold.  Its tangent space is
\[
T_S=\{\bm\delta:\text{both slice sums vanish and }
(\bm c^{(k)}-\bm c^{(l)})^{\!\top}\bm\delta=0, k,l\in S\}.
\]
A neighbourhood confined to a smaller curve or to an additional instrumental-variable constraint does not supply all directions in $T_S$.

\begin{proposition}[Finite-cell rigidity: lower bound]\label{prop:bp-rigidity}
Fix $P_0\in\cP$ with $0<\rho_z^0(x)<1$ for $z=0,1$ at $P_X$-a.e.\ $x$.  Put $G=\{x:\nu(\widetilde{\bm\Sigma}_{\cA^\ast(x)}(x;P_0))>0\}$ and suppose there exists a gluable same-face neighbourhood on $G$ in the sense of Definition~\ref{def:bp-same-face-neighborhood}. Let $f\in\IFRL$ with $f_w^{P_0}(x,\cdot)\in L^2(Q_x^{P_0})$, and write $S=\cA^\ast(x;P_0)$.  For $P_X$-a.e.\ $x\in G$, there exist $\bm\lambda(x)\in\R^S$ and $h_x:\{0,1\}\to\R$ such that $\bm 1^{\!\top}\bm\lambda=1$, $\E_{P_0}\{h_x(Z)\mid X=x\}=0$, and
\begin{equation}\label{eq:bp-rigid-form}
f^{P_0}_w(x,\cdot)=\psib(x;P_0)+\sum_{k\in\cA^\ast}\lambda_k
\bigl(\varphi_k^{P_0}-U_k(P_0)\bigr)+h_x(Z),
\qquad \sum_{k\in\cA^\ast}\lambda_k=1,
\end{equation}
and consequently
\[
\Var_{Q_x^{P_0}}\bigl[f^{P_0}_w(x,\cdot)\bigr]
=\bm\lambda^{\!\top}\widetilde{\bm\Sigma}_{\cA^\ast}\bm\lambda
+\Var_{Q_x^{P_0}}[h_x(Z)]
\;\ge\;\nu\bigl(\widetilde{\bm\Sigma}_{\cA^\ast}\bigr),
\]
with equality if and only if $h_x=0$ $Q_x^{P_0}$-a.s.\ and $\bm\lambda$ minimizes the unit-sum quadratic program.  On $G^c$ the program value is zero and variance nonnegativity gives the lower bound without positive-cell rigidity.  Hence, on all of $\cX$,
\[
\operatorname*{ess\,inf}_{f\in\IFRL}
\Var_{Q_x^{P_0}}\bigl[f^{P_0}_w(x,\cdot)\bigr]\;\ge\;
\nu\bigl(\widetilde{\bm\Sigma}_{\cA^\ast}\bigr)
\qquad\text{at }P_X\text{-a.e.\ }x,
\]
so $\nu(\widetilde{\bm\Sigma}_S)$ is a lower bound over the full class $\IFRL$ on the comparison model.
\end{proposition}

\begin{proof}
Fix $f$, $w$ and $P_0$. We first obtain a conditional identity while holding the estimating function fixed at the nuisance values of $P_0$. For each nonempty $S$, choose a basis $\delta_{S,1},\ldots,\delta_{S,d_S}$ of its tangent space $T_S$. There are finitely many such bases because there are finitely many active sets.

For rational $t>0$ and an integer $m$, consider the measurable set on which $x\in G$, $\cA^\ast(x;P_0)=S$, $p_0(x)+t\delta_{S,j}\in\Gamma_x$, and the cost of this replacement kernel is at most $m$. Replace $p_0$ by $p_0+t\delta_{S,j}$ on this set and keep the baseline elsewhere. Definition~\ref{def:bp-same-face-neighborhood} gives a law $P_q\in\cP$ with unchanged propensity and structural value. Applying (B) at $P_q$ with $P_0$ as the substitution law yields
\[
\E_{P_q}f_w^{P_0}=\Psi(P_q).
\]
Marginal restriction as in Assumption~\ref{ass:rich}\textup{(M1)} localizes this equality to $X=x$. The indices $(S,j,t,m)$ are countable, so these identities and the truth-law conditional identity hold outside one $P_X$-null set.

Fix $x\in G$ outside this null set and abbreviate $S=\cA^\ast(x;P_0)$, $p=p_0(x)$ and $r_c=\rho_{z(c)}^0(x)$. The neighbourhood condition supplies a nonzero rational $t$ for each basis direction, and the finite local cost supplies a corresponding $m$. Let $\bm f$ list the eight values of $f_w^{P_0}(x,\cdot)$. Subtracting the truth identity from the perturbed identity and dividing by $t$ gives
\[
\delta_{S,j}^{\!\top}(\bm r\circ\bm f-\bm c^{(k_0)})=0,
\qquad j=1,\ldots,d_S,
\]
where $k_0\in S$. Both the expectation and the common active branch are affine on $\mathcal H_S$. The same identity therefore holds in every tangent direction and implies
\begin{equation}\label{eq:bp-frozen-cell}
\E_{Q_{x,\bm q}}[f_w^{P_0}(x,\cdot)]
=\min_kU_k(\bm q),\qquad \bm q\in\Gamma_x.
\end{equation}

The tangent space and its orthogonal complement are
\begin{align*}
T_S&=\{\delta:\bm1_{z=0}^{\!\top}\delta=
\bm1_{z=1}^{\!\top}\delta=0,\quad
(\bm c^{(k)}-\bm c^{(k_0)})^{\!\top}\delta=0\ (k\in S)\},\\
T_S^\perp&=\operatorname{span}\{\bm1_{z=0},\bm1_{z=1},
\bm c^{(k)}-\bm c^{(k_0)}:k\in S\}.
\end{align*}
Thus $\bm r\circ\bm f\in\bm c^{(k_0)}+T_S^\perp$. The coefficients of the branch vectors in this affine space sum to one. Dividing cellwise by $\bm r$ and using
\[
\varphi_k^{P_0}-U_k(P_0)
=\frac{c^{(k)}_{AYZ}-M_{kZ}(p_0)}{\rho_Z^0}
\]
absorbs the remaining slice constants into a function $h_x(Z)$ and gives~\eqref{eq:bp-rigid-form}. The truth-law conditional mean implies $\E_{P_0}[h_x(Z)\mid X=x]=0$.

Each residual satisfies
\[
\E_{P_0}[\varphi_k^{P_0}-U_k(P_0)\mid X=x,Z]=0.
\]
It is therefore orthogonal to $h_x(Z)$, so
\[
\Var_{Q_x^{P_0}}(f_w^{P_0})
=\bm\lambda^{\!\top}\widetilde{\bm\Sigma}_S\bm\lambda
+\E_{P_0}[h_x(Z)^2\mid X=x].
\]
Proposition~\ref{prop:cov-singular} gives the lower bound and its equality criterion on $G$. On $G^c$ the benchmark is zero and variance nonnegativity completes the proof.
\end{proof}

\begin{remark}[Quadratic-program branches and attainment]
\label{rem:bp-rigidity-branches}
The value in Proposition~\ref{prop:bp-rigidity} is $\nu$ from \eqref{eq:bp-nu}, not the unqualified reciprocal $(\bm 1^{\!\top}\widetilde{\bm\Sigma}_S^+\bm 1)^{-1}$.  If $\bm 1\in\operatorname{range}(\widetilde{\bm\Sigma}_S)$, the minimizing set is
\[
\frac{\widetilde{\bm\Sigma}_S^+\bm 1}
{\bm 1^{\!\top}\widetilde{\bm\Sigma}_S^+\bm 1}
+\{\bm k\in\ker(\widetilde{\bm\Sigma}_S):\bm 1^{\!\top}\bm k=0\}.
\]
If $\bm 1\notin\operatorname{range}(\widetilde{\bm\Sigma}_S)$, it is $\{\bm k\in\ker(\widetilde{\bm\Sigma}_S): \bm 1^{\!\top}\bm k=1\}$ and the value is zero.  For example, at $\widetilde{\bm\Sigma}_S=\operatorname{diag}(1,0)$ the raw reciprocal equals one, whereas the unit-sum vector $(0,1)$ has value zero.

The bounded map of Definition~\ref{def:bp-face-family} achieves the conditional lower bound whenever its population weight belongs to the displayed minimizing set. On a zero-value kernel stratum this uses variance nonnegativity and requires no positive-cell rigidity. The integrated variance still contains $\Var_{P_X}\{\min_kU_k(X)\}$. Estimator attainment follows under the learning and common-rule conditions stated below.
\end{remark}

\begin{definition}[Balke--Pearl face-family slot and continuous cap]
\label{def:bp-face-family}
Use the declared nuisance $\eta=(\bm p,\rho)$, with cell-input space equal to the product of the two probability simplices and propensity-input space equal to the open binary probability simplex. The comparison laws retain the fixed overlap bound of Section~\ref{sec:bp-setup}. For $L\ge1$ and every nonempty $S\subseteq\{1,\ldots,8\}$, define
\begin{equation}\label{eq:bp-aux-family}
\Omega_L(S)=\{\bm\omega\in\R^S:\bm 1^{\!\top}\bm\omega=1,
\ \|\bm\omega\|_1\le L\},\qquad
\mathcal W_{\mathrm{aux},\mathrm{BP},L}
=\prod_{\varnothing\ne S\subseteq\{1,\ldots,8\}}\Omega_L(S).
\end{equation}
Thus $\mathcal W_{\mathrm{aux},\mathrm{BP},L}$ has $2^8-1=255$ face coordinates.  For $v\in\R^S$, let
\begin{equation}\label{eq:bp-continuous-cap}
\mathsf C_{L,S}(v)=\operatorname*{arg\,min}_{u\in\Omega_L(S)}
\|u-v\|_2.
\end{equation}
For $\boldsymbol\omega\in\mathcal W_{\mathrm{aux},\mathrm{BP},L}$, the slot map is
\[
f_{\mathrm{BP}}(x,O;\bm p,\rho,S,\boldsymbol\omega)
=\sum_{k\in S}\boldsymbol\omega[S]_k
\varphi_k(x,O;\bm p,\rho).
\]
Write $C_c=\max_{k,a,y,z}|c^{(k)}_{ayz}|$.  On the declared comparison model the full-slot bound
\[
|f_{\mathrm{BP}}|\le 2LC_c(1+\epsilon^{-1})
\]
holds for every cell input, law-achievable working propensity and auxiliary face family.  Thus the exact identity~\eqref{eq:bp-dr-onewrong} and facewise unit sum establish both L1 and the expectation part of~(B) on $\cP$, so this map belongs to $\IFRL$.  The $255$ coordinates index all nonempty subsets; no claim that all are attainable active faces is needed.
\end{definition}

\begin{assumption}[Positive-definite Balke--Pearl estimation conditions]
\label{ass:bp-pd-estimation}
Fix $P_0\in\cP$ and a $\mathcal K$-fold scheme whose nuisance learners are trained off-fold.  With $\widehat U_k=\bm c^{(k)}{}^{\!\top}\widehat{\bm p}$, assume:
\begin{enumerate}[wide=0pt,label={\textup{(B\arabic*)}},itemsep=1pt]
\item Each fitted cell vector $(\widehat p_{ay\mid z}(x))_{a,y}$ belongs to the probability simplex. For $z=0,1$ and some $c>0$, $c\le\rho_z^0(x)\le1-c$, and $(\widehat\rho_0,\widehat\rho_1)$ is a probability vector with $\widehat\rho_z\in[c/2,1-c/2]$.  Let $C_1<\infty$ be deterministic and satisfy
\[
\bigl\|\widetilde{\bm\Sigma}_{\cA^\ast(x)}(x;P_0)\bigr\|_{\mathrm{op}}
\le C_1
\qquad P_X\text{-a.e.}
\,.
\]
\item $\sigma_{\min}\{ \widetilde{\bm\Sigma}_{\cA^\ast(x)}(x;P_0)\}\ge c_0>0$.
\item With the convention $\min\varnothing=+\infty$,
\[
\gamma_0=\operatorname*{ess\,inf}_x
\min_{k\notin\cA^\ast(x)}\{U_k(x)-\min_jU_j(x)\}>0.
\]
\item For some $\tau_n\downarrow0$ with $\tau_n<\gamma_0/3$ eventually,
\[
\sup_x\|\widehat{\bm U}(x)-\bm U(x)\|_\infty=o_P(\tau_n),\qquad
\widehat\cA_{\tau_n}(x)=
\{k:\widehat U_k(x)\le\min_j\widehat U_j(x)+\tau_n\}.
\]
\item The foldwise product remainder satisfies
\[
\sum_{z=0}^1\left\|\,\|\widehat\rho-\rho^0\|
\,\|M_{\cdot z}(\widehat{\bm p})-M_{\cdot z}(\bm p^0)\|\right\|_{L^2(P_X)}
=o_P(n^{-1/2}).
\]
\Needspace{5\baselineskip}
\item The plug-in $\widehat{\bm\Sigma}$ from~\eqref{eq:bp-Sigma-dr} converges to $\widetilde{\bm\Sigma}$ entrywise in $L^1(P_X)$ in probability.
\item The plug-in scores are $L^2(P_0)$-consistent.
\end{enumerate}
All conditions hold foldwise, and all displays are understood at $P_X$-a.e.\ $x$ where applicable.
\end{assumption}

The product bound (B5) follows, for example, from foldwise $o_P(n^{-1/4})$ rates for its two factors in $L^4(P_X)$.  More generally, any H\"older-compatible pair of norms whose product is in $L^2(P_X)$ suffices.  Condition (B6) imposes no convergence rate and no uniform spectral control on the plug-in block.  The inverse weights do not enter the exact conditional-bias identity, and their empirical-process contribution is controlled by the pointwise cap.

\begin{definition}[Positive-definite Balke--Pearl estimator]
\label{def:bp-pd-estimator}
Under Assumption~\ref{ass:bp-pd-estimation}, for fixed $L\ge1$ set
\begin{equation}\label{eq:bp-stabilized-weights}
\widehat\lambda^{\rm raw}_{\rm pd}(x)=
\begin{cases}
\dfrac{\widehat{\bm\Sigma}_{\widehat\cA}(x)^{-1}\bm 1}
{\bm 1^{\!\top}\widehat{\bm\Sigma}_{\widehat\cA}(x)^{-1}\bm 1},
&\sigma_{\min}\{\widehat{\bm\Sigma}_{\widehat\cA}(x)\}\ge c_0/2,\\[2ex]
|\widehat\cA(x)|^{-1}\bm 1,&\text{otherwise},
\end{cases}
\qquad
\widehat\lambda(x)=\mathsf C_{L,\widehat\cA(x)}
\{\widehat\lambda^{\rm raw}_{\rm pd}(x)\},
\end{equation}
where $\widehat\cA=\widehat\cA_{\tau_n}$.  Complete this selected coordinate to the total face family
\begin{equation}\label{eq:bp-family-construction}
\widehat{\bm\omega}=(\widehat\omega_S)_{\varnothing\ne S\subseteq
\{1,\ldots,8\}},\qquad
\widehat\omega_S(x)=
\begin{cases}
\widehat\lambda(x),&S=\widehat\cA(x),\\
|S|^{-1}\bm 1_S,&S\ne\widehat\cA(x),
\end{cases}
\end{equation}
and define
\[
T_n=\mathbb P_n\!\left[f_{\mathrm{BP}}
\{X,O;\widehat{\bm p}^{(-k)},\widehat\rho^{(-k)},
\widehat\cA_{\tau_n},\widehat{\bm\omega}^{(-k)}\}\right].
\]
\end{definition}

Definition~\ref{def:bp-face-family} supplies one total map for every structural face.  A single eight-vector cannot do so: normalization on all singletons fixes every coordinate at one and contradicts normalization on any pair.  The unused coordinates in \eqref{eq:bp-family-construction} are uniform on their own faces, while the selected coordinate contains the adaptive weight.  Since $\Omega_L(S)$ is nonempty, compact, and convex, $\mathsf C_{L,S}$ is single-valued, Borel, $1$-Lipschitz, and the identity on $\Omega_L(S)$. The cap therefore preserves continuity at its boundary, whereas a hard accept-or-fallback test at $\|v\|_1=L$ is discontinuous there.  It also prevents an estimated nearly singular block on a small set from producing unbounded inverse weights; Remark~\ref{rem:bp-weight-adversary} gives an exact construction.  Facewise exact double robustness holds for every unit-sum coordinate, including both branches of \eqref{eq:bp-stabilized-weights}.

\begin{theorem}[Positive-definite branch: fixed-law expansion and benchmark equality]
\label{thm:bp-achieve}
Assume the Balke--Pearl model of Section~\ref{sec:bp-bound}, Assumption~\ref{ass:bp-pd-estimation}, and the slot and estimator in Definitions~\ref{def:bp-face-family}--\ref{def:bp-pd-estimator}.  Put
\[
\lambda^\ast(x)=
\frac{\widetilde{\bm\Sigma}_{\cA^\ast(x)}(x;P_0)^{-1}\bm 1}
{\bm 1^{\!\top}\widetilde{\bm\Sigma}_{\cA^\ast(x)}(x;P_0)^{-1}\bm 1},
\qquad
\lambda_L^\ast(x)=\mathsf C_{L,\cA^\ast(x)}\{\lambda^\ast(x)\}.
\]
Let $\bm\omega_L^\ast(x)$ be the element of $\mathcal W_{\mathrm{aux},\mathrm{BP},L}$ whose $\cA^\ast(x)$ coordinate is $\lambda_L^\ast(x)$ and whose other coordinates are facewise uniform, and define
\[
g_L^{P_0}=f_{\mathrm{BP}}
(X,O;\bm p^0,\rho^0,\cA^\ast,\bm\omega_L^\ast).
\]
For every fixed $L\ge1$,
\[
T_n=\mathbb P_n[g_L^{P_0}]+o_{P_0}(n^{-1/2}),
\]
so $T_n$ is asymptotically linear at $P_0$ with limiting evaluation $g_L^{P_0}$.  Its conditional variance is
\[
(\lambda_L^\ast)^{\!\top}
\widetilde{\bm\Sigma}_{\cA^\ast}\lambda_L^\ast
\ge b_{\mathrm{BP}}(x;P_0).
\]
Equality holds if and only if $\lambda_L^\ast=\lambda^\ast$, equivalently $\lambda^\ast\in\Omega_L(\cA^\ast)$.  In particular, every $L\ge\sqrt{C_1/c_0}$ gives
\[
\Var_{Q_x^{P_0}}\!\{g_L^{P_0}(x,\cdot)\}
=b_{\mathrm{BP}}(x;P_0)
\qquad P_X\text{-a.e.}
\]
\end{theorem}

Theorem~\ref{thm:bp-achieve} gives an asymptotic linear expansion and equality with the conditional variance program at the evaluating law. The full-domain map satisfies~(B) on $\cP$, and Proposition~\ref{prop:bp-rigidity} identifies the program value with $\Vmin$ under its perturbation conditions. Corollary~\ref{cor:bp-achieve-global} uses one common implementation throughout a learnable subclass.

\begin{proof}
Work on one validation fold of size $n_k\asymp n$, condition on its training observations, and suppress the fold index on fitted quantities. Write $\widehat f=f_{\mathrm{BP}}(\widehat\eta,\widehat\cA,\widehat{\bm\omega})$ and $g=g_L^{P_0}$. The decomposition is
\[
\mathbb P_{n_k}(\widehat f-g)
=(\mathbb P_{n_k}-P_0)(\widehat f-g)+P_0(\widehat f-g).
\]

By (B3)--(B4), the event
\[
E_n=\left\{\sup_x\|\widehat{\bm U}(x)-\bm U(x)\|_\infty
\le\tau_n/2,\quad \tau_n<\gamma_0/3\right\}
\]
has probability tending to one. On $E_n$, every active index lies within $\tau_n$ of the fitted minimum, whereas every inactive index has fitted gap at least $\gamma_0-\tau_n>\tau_n$. Hence $\widehat\cA(x)=\cA^\ast(x)$ throughout the fold's covariate domain.

For any nonempty $S$ and any unit-sum vector $\omega$, the conditional bias identity is
\begin{multline}\label{eq:bp-face-product}
\E_{P_0}[f_{\mathrm{BP}}(x,O;\widehat{\bm p},\widehat\rho,S,
\boldsymbol\omega)\mid X=x]-\min_jU_j(x)\\
=\sum_{k\in S}\omega_k\{U_k(x)-\min_jU_j(x)\}
+\sum_{k\in S}\omega_k\sum_z
\left(1-\frac{\rho_z^0(x)}{\widehat\rho_z(x)}\right)
\{M_{kz}(\widehat{\bm p})-M_{kz}(\bm p^0)\}(x).
\end{multline}
On $E_n$ the first term is zero. The cap and fitted overlap give
\[
|P_0(\widehat f-g)|
\le\frac{2L}{c}\sum_z
\left\|\|\widehat\rho-\rho^0\|
\,\|M_{\cdot z}(\widehat{\bm p})-M_{\cdot z}(\bm p^0)\|\right\|_{L^1(P_X)}
=o_P(n^{-1/2})
\]
by (B5) and $\|h\|_1\le\|h\|_2$. The weights appear only as a bounded multiplier of the nuisance-product error.

For weight convergence, (B6) implies convergence of the true active block in $P_X$-measure, in probability. The population eigenvalues are at least $c_0$, so the set on which the fitted block has an eigenvalue below $c_0/2$ has $P_X$-measure tending to zero. On its complement, the inverse-weight map is continuous. Projection onto $\Omega_L(S)$ is continuous and bounds both fitted and limiting weights by $L$. Bounded convergence in measure therefore gives, after padding by zeros outside the true active set,
\[
\|\widehat\lambda-\lambda_L^\ast\|_{L^2(P_X)}=o_P(1).
\]
There are finitely many branch scores, and their truth evaluations are uniformly bounded by overlap and finite cells. Together with (B7), this yields
\[
\|\widehat f-g\|_{L^2(P_0)}
\le L\max_k\|\widehat\varphi_k-\varphi_k^{P_0}\|_{L^2(P_0)}
+C\|\widehat\lambda-\lambda_L^\ast\|_{L^2(P_X)}
=o_P(1)
\]
on $E_n$, for a fixed finite $C$. Conditional Chebyshev now gives
$(\mathbb P_{n_k}-P_0)(\widehat f-g)=o_P(n^{-1/2})$. Since $P_0(E_n^c)\to0$, these convergence-in-probability statements extend to the whole sample space. Combining the fixed number of folds proves the expansion.

Finally, for an $m\times m$ positive-definite matrix $\bm\Sigma$, the optimal vector has
\begin{equation}\label{eq:bp-weight-norm}
\|\bm\lambda(\bm\Sigma)\|_1
\le\sqrt m\,\|\bm\lambda(\bm\Sigma)\|_2
\le\sqrt{\|\bm\Sigma\|_{\mathrm{op}}/\sigma_{\min}(\bm\Sigma)}.
\end{equation}
Indeed, comparison with the uniform feasible vector gives
\[
\sigma_{\min}(\bm\Sigma)\|\bm\lambda(\bm\Sigma)\|_2^2
\le\nu(\bm\Sigma)
\le\frac{\bm1^{\!\top}\bm\Sigma\bm1}{m^2}
\le\frac{\|\bm\Sigma\|_{\mathrm{op}}}{m}.
\]
Thus $L\ge\sqrt{C_1/c_0}$ contains the population optimizer. The variance formula and equality criterion follow from strict convexity of the positive-definite quadratic program.
\end{proof}

\begin{remark}[Weight stabilization]\label{rem:bp-weight-adversary}
Integrated covariance consistency need not control unregularized inverse weights. To see this at the matrix level, let the truth be $\bm I_2$ and change the fitted matrix only on a set of $P_X$-measure $n^{-2}$, to
\[
\begin{pmatrix}1-d+d^2&1\\1&1+d\end{pmatrix},\qquad d=n^{-2}.
\]
Its determinant is $d^3$ and its unit-sum minimizing vector is $(d^{-1},1-d^{-1})$. Thus its covariance error is $O(n^{-2})$ in $L^1(P_X)$ while the squared $L^2(P_X)$ weight error has order $n^2$. The estimator's eigenvalue threshold and continuous cap control this instability. The singular selectors below use spectral thresholds together with the same cap.
\end{remark}

\begin{corollary}[Integrated remainder criterion]\label{cor:bp-integrated-gap}
Use the cross-fitted estimator with the total face-family map and a fixed cap $L\ge1$. Suppose the fitted scores are square integrable and, on each fold, for a fixed $g_\infty^{P_0}\in L^2(P_0)$,
\begin{align*}
&\left\|\sum_{k\in\widehat S}\widehat\lambda_k
\{U_k-\min_jU_j\}\right\|_{L^1(P_X)}=o_P(n^{-1/2}),\\
&\left\|\sum_{k\in\widehat S}\widehat\lambda_k\sum_z
\left(1-\frac{\rho_z^0}{\widehat\rho_z}\right)
\{M_{kz}(\widehat{\bm p})-M_{kz}(\bm p^0)\}
\right\|_{L^1(P_X)}=o_P(n^{-1/2}),\\
&\left\|f_{\mathrm{BP}}(\widehat\eta,\widehat S,
\widehat{\bm\omega})-g_\infty^{P_0}\right\|_{L^2(P_0)}=o_P(1).
\end{align*}
Then $T_n=\mathbb P_n g_\infty^{P_0}+o_P(n^{-1/2})$ and $\E_{P_0}g_\infty^{P_0}=\Psi(P_0)$. If $g_\infty^{P_0}$ is a truth evaluation with a minimizing active-set weight, its conditional variance equals $b_{\mathrm{BP}}(x;P_0)$.
\end{corollary}

Corollary~\ref{cor:bp-integrated-gap} replaces uniform active-set and spectral separation by the integrated remainders needed for asymptotic linearity. For example, at the matrix field $\bm M(x)=\operatorname{diag}(x,0)$ with $X\sim\operatorname{Unif}(0,1]$, thresholding at $a_n$ selects $(0,1)$ for $x>a_n$ and the uniform vector for $x\le a_n$. The squared weight error integrates to $a_n/2\to0$ despite the absence of a uniform positive eigenvalue bound. This is a matrix illustration; applying the corollary to a Balke--Pearl model requires the three displayed remainders for its own cell probabilities. A common rule satisfying these remainders throughout a learnable subclass yields estimator membership as in Corollary~\ref{cor:bp-one-estimator}; minimizing limits and Proposition~\ref{prop:bp-rigidity} then give efficiency on that subclass.

\begin{proof}[Proof of Corollary~\ref{cor:bp-integrated-gap}]
Conditional on a training fold, subtract $g_\infty^{P_0}$ from the fitted score and decompose its validation average into a centered empirical term and its expectation. The third displayed condition and conditional Chebyshev make the centered term $o_P(n^{-1/2})$. Identity~\eqref{eq:bp-face-product} shows that the expectation differs from the target by at most the first two displayed norms, hence by $o_P(n^{-1/2})$. The same conditions and $L^2$ convergence imply $\E_{P_0}g_\infty^{P_0}=\Psi(P_0)$. Combining folds gives the stated asymptotic linear expansion. If the limiting truth weight minimizes the program, its conditional variance equals the benchmark.
\end{proof}

\begin{proposition}[Exact matrix-level rank-boundary obstruction]
\label{prop:bp-rank-boundary}
Let $u_0=(1,1)^{\!\top}/\sqrt2$, $k_0=(1,-1)^{\!\top}/\sqrt2$, $u_\theta=\cos\theta\,u_0+\sin\theta\,k_0$ and $k_\theta=-\sin\theta\,u_0+\cos\theta\,k_0$.  For $\epsilon>0$ set
\[
M_{\epsilon,\theta}=u_\theta u_\theta^{\!\top}
+\epsilon k_\theta k_\theta^{\!\top}.
\]
The unique unit-sum minimizer and its value are
\begin{align*}
\lambda_{\epsilon,\theta}
&=\frac{\cos\theta\,u_\theta-\epsilon^{-1}\sin\theta\,k_\theta}
{\sqrt2\{\cos^2\theta+\epsilon^{-1}\sin^2\theta\}},\\
\nu(M_{\epsilon,\theta})
&=\{2(\cos^2\theta+\epsilon^{-1}\sin^2\theta)\}^{-1}.
\end{align*}
Consequently, along $\theta=c\epsilon$, $\lambda_{\epsilon,\theta}\to(u_0-c k_0)/\sqrt2$, so different $c$ give different finite limits.  Along $\theta=c\sqrt\epsilon$ with $c\ne0$, $\|\lambda_{\epsilon,\theta}\|_1$ diverges at order $\epsilon^{-1/2}$ while $\nu(M_{\epsilon,\theta})\to\{2(1+c^2)\}^{-1}$.  At the rank-one limit $M_{0,0}=u_0u_0^{\!\top}$ the canonical minimizer is $u_0/\sqrt2$ and $\nu(M_{0,0})=1/2$.

Thus the exact full-rank oracle has no path-independent continuous extension at this rank boundary, its norm is not locally bounded, and even the benchmark is path-discontinuous.  This proves that no fixed cap can preserve equality throughout a whole matrix neighbourhood; it does \emph{not} assert that this abstract matrix path is realizable by Balke--Pearl cell laws.  Such an endogenous conclusion would require a separate covariance-realization lemma.
\end{proposition}

\begin{proof}[Proof of Proposition~\ref{prop:bp-rank-boundary}]
The vectors $u_\theta,k_\theta$ are orthonormal, so
\[
M_{\epsilon,\theta}^{-1}=u_\theta u_\theta^{\!\top}
+\epsilon^{-1}k_\theta k_\theta^{\!\top}.
\]
Their inner products with $\bm1$ are $\sqrt2\cos\theta$ and $-\sqrt2\sin\theta$, respectively. Substitution into the positive-definite unit-sum formulas gives the displayed optimizer and value. Using $\sin\theta=\theta+o(\theta)$ and $\cos\theta=1+o(1)$ yields both stated path limits and the divergent norm along $\theta=c\sqrt\epsilon$. At $\epsilon=\theta=0$, $\bm1=\sqrt2u_0$ belongs to the range and the pseudoinverse equals $u_0u_0^{\!\top}$, giving the stated minimizer and value $1/2$.
\end{proof}

\begin{definition}[Thresholded singular-branch selectors]
\label{def:bp-singular-selectors}
For the eigenpairs $(\widehat w_j,\widehat v_j)$ of the symmetrized active plug-in block, define
\begin{equation}\label{eq:bp-thresholded-weights}
\begin{aligned}
\widehat{\bm\Sigma}^{+,a}
&=\sum_{j:\widehat w_j>a}\widehat w_j^{-1}
\widehat v_j\widehat v_j^{\!\top},&
\widehat d_a&=\bm 1^{\!\top}\widehat{\bm\Sigma}^{+,a}\bm 1,\\
\widehat R_a
&=\|\{\bm I-\widehat{\bm\Sigma}\widehat{\bm\Sigma}^{+,a}\}\bm 1\|
 /\|\bm 1\|,&
\widehat\lambda^{\rm raw}_{\rm ran}
&=\begin{cases}
\widehat{\bm\Sigma}^{+,a_n}\bm 1/\widehat d_{a_n},
&\widehat R_{a_n}\le\kappa_n,\ \widehat d_{a_n}>0,\\
|\widehat\cA|^{-1}\bm 1,&\text{otherwise},
\end{cases}\\
\widehat\lambda_{\rm ran}
&=\mathsf C_{L,\widehat\cA}
 \{\widehat\lambda^{\rm raw}_{\rm ran}\}.&&
\end{aligned}
\end{equation}
For the kernel branch, define
\begin{equation}\label{eq:bp-kernel-weights}
\begin{aligned}
\widehat{\bm K}
&=\bm I-\widehat{\bm\Sigma}\widehat{\bm\Sigma}^{+,a_n},&
\widehat d_K&=\bm 1^{\!\top}\widehat{\bm K}\bm 1,\\
\widehat\lambda^{\rm raw}_{\rm ker}
&=\begin{cases}
\widehat{\bm K}\bm 1/\widehat d_K,
&|\widehat d_K|\ge\tfrac12\kappa_0^2\|\bm 1\|^2,\\
|\widehat\cA|^{-1}\bm 1,&\text{otherwise},
\end{cases}&
\widehat\lambda_{\rm ker}
&=\mathsf C_{L,\widehat\cA}
 \{\widehat\lambda^{\rm raw}_{\rm ker}\}.
\end{aligned}
\end{equation}
All quantities are evaluated at $x$; the argument is suppressed. For a truth block $\bm M=\widetilde{\bm\Sigma}_{\cA^\ast(x)}(x;P_0)$, the corresponding projected population weights are
\begin{align*}
\lambda_{L,\rm ran}^\ast(x)
&=\mathsf C_{L,\cA^\ast(x)}
\left\{\frac{\bm M^+\bm 1}{\bm 1^{\!\top}\bm M^+\bm 1}\right\},
&&\bm 1\in\operatorname{range}(\bm M),\\
\lambda_{L,\rm ker}^\ast(x)
&=\mathsf C_{L,\cA^\ast(x)}
\left\{\frac{\bm K\bm 1}{\bm 1^{\!\top}\bm K\bm 1}\right\},
\quad \bm K=\Pi_{\ker(\bm M)},
&&\bm 1\notin\operatorname{range}(\bm M).
\end{align*}
\end{definition}

\begin{assumption}[Separated singular Balke--Pearl strata]
\label{ass:bp-singular-strata}
Assume (B1) and (B3)--(B7) of Assumption~\ref{ass:bp-pd-estimation}.  At $P_X$-a.e.\ $x$, impose one of the following alternatives.
\begin{enumerate}[wide=0pt,label={\textup{(\alph*)}},itemsep=1pt]
\item \emph{Stable range.}~ $\bm 1\in\operatorname{range}(\widetilde{\bm\Sigma}_{\cA^\ast})$; $r(x)=\operatorname{rank}(\widetilde{\bm\Sigma}_{\cA^\ast})$ is constant on each cell of a measurable partition; every nonzero eigenvalue is at least $c_r>0$; and
\[
a_n\downarrow0,\quad a_n<c_r/2,\quad
\sup_x\|\widehat{\bm\Sigma}_{\widehat\cA}
-\widetilde{\bm\Sigma}_{\cA^\ast}\|=o_P(a_n),\quad
\kappa_n\downarrow0,\quad a_n=o(\kappa_n).
\]
Use $\widehat\lambda_{\rm ran}$ from~\eqref{eq:bp-thresholded-weights}.
\item \emph{Separated kernel.}~ $\bm 1\notin\operatorname{range}(\widetilde{\bm\Sigma}_{\cA^\ast})$, $\|\Pi_{\ker\widetilde{\bm\Sigma}_{\cA^\ast}}\bm 1\| \ge\kappa_0\|\bm 1\|$ for some $\kappa_0>0$, and every nonzero eigenvalue belongs to a common interval $[c_\ast,C_\ast]$ with $c_\ast>0$.  Moreover,
\[
a_n\downarrow0,\quad a_n<c_\ast/2,\quad
\sup_x\|\widehat{\bm\Sigma}_{\widehat\cA}
-\widetilde{\bm\Sigma}_{\cA^\ast}\|=o_P(a_n).
\]
Use $\widehat\lambda_{\rm ker}$ from~\eqref{eq:bp-kernel-weights}.
\end{enumerate}
The threshold inequalities are required only eventually.
\end{assumption}

\begin{theorem}[Stable-rank and kernel strata: fixed-law expansion and benchmark equality]
\label{thm:bp-stable-rank}
Assume the Balke--Pearl model and Assumption~\ref{ass:bp-singular-strata}.  Construct the face family and cross-fitted estimator of Definition~\ref{def:bp-pd-estimator}, replacing its selected positive-definite weight by the selector specified in Assumption~\ref{ass:bp-singular-strata}.  Let $\lambda_{L,\rm br}^\ast$ be the applicable population weight from Definition~\ref{def:bp-singular-selectors}, complete it by facewise uniform coordinates to $\bm\omega_{L,\rm br}^\ast$, and set
\[
g_{L,\rm br}^{P_0}=f_{\mathrm{BP}}
(X,O;\bm p^0,\rho^0,\cA^\ast,\bm\omega_{L,\rm br}^\ast).
\]
For every fixed $L\ge1$,
\begin{gather*}
T_n=\mathbb P_n[g_{L,\rm br}^{P_0}]+o_{P_0}(n^{-1/2}),
\\
\Var_{Q_x^{P_0}}(g_{L,\rm br}^{P_0})\ge b_{\mathrm{BP}}(x;P_0).
\end{gather*}
Equality holds if and only if $\lambda_{L,\rm br}^\ast$ belongs to the branch-specific minimizing set in Remark~\ref{rem:bp-rigidity-branches}.  Sufficient common-cap conditions are
\[
\begin{aligned}
L&\ge\sqrt{C_1/c_r}&&\text{on the stable-range stratum},\\
L&\ge1/\kappa_0&&\text{on the separated-kernel stratum}.
\end{aligned}
\]
Under these conditions, the respective limiting conditional variances equal the algebraic benchmark values
\[
(\bm 1^{\!\top}\widetilde{\bm\Sigma}_{\cA^\ast}^{+}\bm 1)^{-1}
\quad\text{and}\quad 0.
\]
For smaller $L$, the estimator remains bounded and asymptotically linear at $P_0$, but benchmark equality is not asserted.  Class-level membership and equality with $\Vmin$ require the common-rule conditions of Corollary~\ref{cor:bp-one-estimator}.
\end{theorem}

\begin{remark}[Denominator checks and weight bounds]\label{rem:bp-singular-selector-guards}
The denominator checks define finite weights for every fitted matrix. Under the respective convergence conditions they pass with probability tending to one: the range denominator has a positive population limit, and the kernel denominator is at least $\kappa_0^2\|\bm1\|^2$. The continuous projection separately enforces $\|\widehat\lambda\|_1\le L$, which controls the conditional bias and the integrated score error in the attainment proof.
\end{remark}

\begin{proof}
The proof of Theorem~\ref{thm:bp-achieve} applies once the selected weights converge in $L^2(P_X)$ to their stated limits. Its face identity and nuisance-product bound hold for every capped unit-sum vector.

On either stratum, let $e_n=\sup_x\|\widehat{\bm\Sigma}_{\widehat\cA}-\widetilde{\bm\Sigma}_{\cA^\ast}\|$. On the event of correct active-set recovery, $e_n=o_P(a_n)$ and the positive eigenvalues are bounded below. Weyl's inequality separates the empirical eigenvalues above $a_n$ from those corresponding to the population kernel. Spectral projection and inverse perturbation bounds therefore give
\[
\|\widehat\Pi_{\rm ran}-\Pi_{\rm ran}\|_{\infty}=O_P(e_n),
\qquad
\|\widehat{\bm\Sigma}^{+,a_n}-\widetilde{\bm\Sigma}_{\cA^\ast}^{+}\|_{\infty}=O_P(e_n),
\]
with constants depending on the stratum's fixed positive eigenvalue bound.

On the stable-range stratum, $\Pi_{\rm ran}\bm1=\bm1$, so $\widehat R_{a_n}=O_P(e_n)=o_P(\kappa_n)$. Moreover,
\[
\bm1^{\!\top}\widetilde{\bm\Sigma}_{\cA^\ast}^{+}\bm1
\ge |\cA^\ast|/C_1>0.
\]
The range test and denominator test therefore pass with probability tending to one, and the raw weights converge to the pseudoinverse weights. Applying~\eqref{eq:bp-weight-norm} on the range yields their $\ell^1$ bound $\sqrt{C_1/c_r}$.

On the separated-kernel stratum, $\widehat{\bm K}\to\bm K$ uniformly in probability and
\[
\bm1^{\!\top}\bm K\bm1=\|\bm K\bm1\|_2^2
\ge\kappa_0^2|\cA^\ast|.
\]
The denominator test passes with probability tending to one. The raw weights converge to $\bm K\bm1/(\bm1^{\!\top}\bm K\bm1)$, whose norm satisfies
\[
\left\|\frac{\bm K\bm1}{\bm1^{\!\top}\bm K\bm1}\right\|_1
\le\frac{\sqrt{|\cA^\ast|}}{\|\bm K\bm1\|_2}
\le\kappa_0^{-1}.
\]

Continuous projection preserves convergence and bounds every selected weight by $L$. Thus convergence holds in $L^2(P_X)$ on both strata, and the cross-fitted expansion follows. The stated cap bounds leave the respective population minimizers unchanged. For any cap, the equality criterion is precisely membership of the projected limit in the minimizing set of Proposition~\ref{prop:cov-singular}.
\end{proof}

\begin{remark}[Singular Balke--Pearl examples]\label{rem:bp-singular-strata}
The positive-cell example with $p_{\cdot\mid z}=(0.1,0.1,0.7,0.1)$ and $\rho_1=0.4$ has active-block eigenvalues $(0,0.2312,0.6494,1.0361)$ and $\bm1$ in its range. Its minimizing weights are $(0.25,0.15,0.35,0.25)$ and its conditional value is $0.16$, so its matrix belongs to the stable-range case.

The boundary example in Section~\ref{sec:bp-degenerate} has rank one and normalized kernel projection $\sqrt{11/12}$. Its kernel weight $(1/2,0,1/2,0)$ has variance zero. The unqualified pseudoinverse vector $(1,0,-1,1)$ instead has variance $57/28$. At this law nonnegativity and the bounded kernel construction establish $\Vmin=0$ directly. Theorem~\ref{thm:bp-stable-rank} supplies estimator attainment when its learning conditions hold. More general covariance fields with vanishing spectral gaps can be handled when the integrated conditions of Corollary~\ref{cor:bp-integrated-gap} hold.
\end{remark}

\begin{definition}[Unified stabilized spectral selector]
\label{def:bp-unified-selector}
Given the symmetrized plug-in $\widehat{\bm\Sigma}_{\widehat\cA}(x)$, threshold $a_n$, range tolerance $\kappa_n$, and cap $L$, set
\begin{equation*}
\begin{aligned}
\widehat r(x)&\eqdef
\begin{cases}
\lambda\bigl(\widehat{\bm\Sigma}^{+,a_n}\bigr),
&\text{if }\widehat R_{a_n}\le\kappa_n
\text{ and }\widehat d_{a_n}>0,\\[1ex]
\widehat{\bm K}\bm 1/\widehat d_K,
&\text{else if }|\widehat d_K|
\ge\tfrac12\kappa_0^2\|\bm 1\|^2,\\[1ex]
\bigl|\widehat\cA(x)\bigr|^{-1}\bm 1,&\text{otherwise,}
\end{cases}\\
\widehat\lambda(x)&\eqdef
\mathsf C_{L,\widehat\cA(x)}\{\widehat r(x)\},
\end{aligned}
\end{equation*}
where $\lambda(\bm M)\eqdef\bm M\bm 1/ (\bm 1^{\!\top}\bm M\bm 1)$ and $\widehat{\bm K}$ is defined in~\eqref{eq:bp-kernel-weights}.  The guards make $\widehat r$ total, and the common projection makes $\widehat\lambda$ unit-sum and bounded by $L$ in $\ell^1$.
\end{definition}

\begin{corollary}[One estimator across all three strata]
\label{cor:bp-one-estimator}
Fix the rich comparison model $\cP$ described above and let $\cP_{\mathrm{ach}}\subseteq\cP$ be a subclass on which one fold scheme, one family of learners, and one tuple $(\tau_n,a_n,\kappa_n,L,c_0,c_r,\kappa_0,C_1)$, fixed independently of $P$, satisfy the conditions of Theorem~\ref{thm:bp-achieve} on the positive-definite stratum and Theorem~\ref{thm:bp-stable-rank} on the stable-range and separated-kernel strata, and such that, independently of which strata the subclass contains, $\kappa_n\downarrow0$ and $a_n=o(\kappa_n)$.  Construct the estimator with Definition~\ref{def:bp-unified-selector}.  For each $P\in\cP_{\mathrm{ach}}$, let $\lambda_{L,\mathrm{br}}^\ast(x;P)$ denote the applicable projected population weight at $x$ and define the capped benchmark
\begin{equation}\label{eq:bp-capped-benchmark}
\begin{aligned}
V_{\mathrm{BP},L}(P)
&\eqdef\Var_{P_X}\!\left[\min_k U_k(X)\right]\\
&\quad+\E_{P_X}\!\left[
\lambda_{L,\mathrm{br}}^\ast(X;P)^{\!\top}
\widetilde{\bm\Sigma}_{\cA^\ast(X)}(X;P)
\lambda_{L,\mathrm{br}}^\ast(X;P)\right].
\end{aligned}
\end{equation}
Then
\[
T_n\in\TRL(\Psi;\cP_{\mathrm{ach}}\mid\cP)
\]
for every fixed $L\ge1$, and $\Vlow(P)\le V_{\mathrm{BP},L}(P)$ for every $P\in\cP_{\mathrm{ach}}$.

Here $\Vmin$ and $\Vlow$ use the maps satisfying~(B) on $\cP$; $\cP_{\mathrm{ach}}$ need not be closed under gluing.  If Proposition~\ref{prop:bp-rigidity} applies using perturbations in $\cP$ at every law in the subclass, with same-face neighbourhoods required only where $\nu>0$, then
\[
V_{\mathrm{BP}}(P)\le\Vlow(P)\le V_{\mathrm{BP},L}(P).
\]
If, in addition, the projected population limit is in the branch-specific minimizing set at $P_X$-a.e.\ $x$, then
\[
\Vmin(x;P)
=\nu\bigl(\widetilde{\bm\Sigma}_{\cA^\ast(x)}(x;P)\bigr)
\quad P_X\text{-a.e.},
\qquad
V_{\mathrm{BP},L}(P)=\Vlow(P)=V_{\mathrm{BP}}(P).
\]
The latter condition holds under the common branchwise cap bounds $L\ge\sqrt{C_1/c_0}$ on the positive-definite stratum, $L\ge\sqrt{C_1/c_r}$ on the stable-range stratum, and $L\ge1/\kappa_0$ on the separated-kernel stratum.  The hypotheses require the corresponding branchwise learning conditions throughout the subclass.
\end{corollary}

\begin{proof}
On the positive-definite stratum, $a_n<c_0/2$ eventually.  Covariance consistency and the fixed matrix dimension imply that, outside a set whose $P_X$-measure converges to zero in probability, every empirical active eigenvalue exceeds $a_n$, the thresholded range residual is zero, and the unified raw selector equals the ordinary-inverse selector of Theorem~\ref{thm:bp-achieve}.  The common continuous cap is uniformly bounded, so this agreement on the good set transfers the $L^2(P_X)$ weight convergence and hence the fixed-law expansion proved there.  The stable-range and separated-kernel expansions follow from Theorem~\ref{thm:bp-stable-rank}, provided the unified selector uses the branch that theorem prescribes.  On the stable-range stratum this is the content of $a_n=o(\kappa_n)$, as in that theorem.  On the separated-kernel stratum it uses $\kappa_n\downarrow0$: by covariance consistency with $a_n<c_\ast/2$ eventually, the thresholded range residual $\widehat R_{a_n}$ converges to $\|\Pi_{\ker\widetilde{\bm\Sigma}_{\cA^\ast}}\bm1\|/\|\bm1\| \ge\kappa_0>0$, so the range test $\widehat R_{a_n}\le\kappa_n$ fails with probability tending to one and the kernel branch is selected. The condition cannot be dropped when the subclass contains no stable-range law: at the rank-one law of Section~\ref{sec:bp-degenerate} the residual is $\sqrt{11/12}<1$, so $\kappa_n\equiv1$ would pass the range test permanently and return the unit-sum weights $(1,0,-1,1)$, with conditional variance $57/28$, whereas the kernel branch returns a unit-sum kernel vector with conditional variance $0$.  Because the fold scheme, learners, tuning tuple, all-face slot, and selector are common across the subclass, these fixed-law expansions belong to one estimator satisfying (A) on $\cP_{\mathrm{ach}}$.  The bounded full-slot construction in Definition~\ref{def:bp-face-family} independently gives~(B) on $\cP$.

The variance of that member is exactly~\eqref{eq:bp-capped-benchmark}, which proves $\Vlow\le V_{\mathrm{BP},L}$.  Proposition \ref{prop:bp-rigidity} gives the reverse lower comparison $V_{\mathrm{BP}}\le\Vlow$.  If the projected population weight is a branch-specific minimizer, its conditional quadratic value equals $\nu(\widetilde{\bm\Sigma}_{\cA^\ast})$ pointwise, proving all three equalities.  The displayed cap bounds make the projection inactive at the corresponding algebraic minimizer.
\end{proof}

\begin{remark}[Boundedness and equality for the unified selector]
The cap guarantees boundedness but may alter an algebraic minimizer. For a two-dimensional active block with kernel spanned by $(10,-9)$ and remaining eigenvalue $q=0.052559$, the kernel candidate is $(10,-9)$.  At $\kappa_0=0.05$ it passes the denominator guard, while its $\ell^1$ norm is $19$.  With $L=5$, projection returns $(3,-2)$, whose variance is $49q/181>0$ instead of the kernel benchmark $0$.

The selector must also read the estimated covariance.  For $\widetilde{\bm\Sigma}=\operatorname{diag}(1,4)$ and $\operatorname{diag}(4,1)$, the oracle weights are respectively $(0.8,0.2)$ and $(0.2,0.8)$, each with value $0.800$; either vector used at the other law has value $2.600$.  Definition \ref{def:bp-unified-selector} is one law-independent map whose argument adapts to both covariance matrices.
\end{remark}
\begin{corollary}[Global membership; Balke--Pearl equality on the
positive-definite branch]\label{cor:bp-achieve-global} Fix the rich comparison model $\cP$ described above.  Let $\cP_{\mathrm{ach}}\subseteq\cP$ be a subclass on which the foldwise conditions of Theorem~\ref{thm:bp-achieve} hold at every $P$ for \emph{one} fold scheme, one family of learning algorithms, one deterministic tolerance sequence $\tau_n$ and one stabilization rule, all fixed in advance and not depending on $P$.  Suppose the common-cap condition $L\ge\sqrt{C_1/c_0}$ holds and the comparison model supplies the localized perturbations and same-face hypotheses of Proposition~\ref{prop:bp-rigidity} at every $P\in\cP_{\mathrm{ach}}$.  Then $(T_n)$ satisfies Definition~\ref{def:ralu}~(A) at every such $P$ and $T_n\in\TRL(\Psi;\cP_{\mathrm{ach}}\mid\cP)$: its bounded full-slot map satisfies~(B) on $\cP$, whereas the common learning conditions supply~(A) on $\cP_{\mathrm{ach}}$.  No gluing closure is imposed on the learnable subclass.  Combined with the lower bound of Proposition~\ref{prop:bp-rigidity}, at every $P\in\cP_{\mathrm{ach}}$ whose active blocks are positive definite $P_X$-a.e.,
\[
\Vmin(x;P)=\nu\bigl(\widetilde{\bm\Sigma}_{\cA^\ast(x)}(x;P)\bigr)
=\bigl(\bm 1^{\!\top}\widetilde{\bm\Sigma}_{\cA^\ast(x)}^{-1}
\bm 1\bigr)^{-1}
\quad\text{at }P_X\text{-a.e.\ }x,
\]
and $\Vlow(P)=V_{\mathrm{BP}}(P)$, the right-hand side of~\eqref{eq:bp-Vlow-final}, so the inequality there is an equality on this subclass.  The one-estimator rank-deficient counterpart is Corollary~\ref{cor:bp-one-estimator}, with the branchwise common-cap conditions stated there; Theorem~\ref{thm:bp-stable-rank} alone is only the fixed-law ingredient.  Covariance fields beyond the separated cases can instead be assessed using Corollary~\ref{cor:bp-integrated-gap}.
\end{corollary}
\begin{proof}[Proof of Corollary~\ref{cor:bp-achieve-global}]
Apply Theorem~\ref{thm:bp-achieve} at each law in the learnable subclass using the common fold scheme, learners, tolerance, and stabilization rule. Its asymptotic linear representations satisfy~(A), while the integrable full-domain map in Definition~\ref{def:bp-face-family} satisfies~(B) on the comparison model. The common cap contains the oracle weights, so the represented conditional variance equals the quadratic-program value. Proposition~\ref{prop:bp-rigidity} gives the reverse lower comparison over all admissible maps. The conditional and integrated equalities follow.
\end{proof}

The mechanism is worth isolating, since it is what a general nuisance-dependent rigidity theorem would have to reproduce: the extra freedom left by the pointwise constraint is exactly a function of the instrument, and within-stratum centring of the doubly robust score makes that freedom orthogonal to the residuals, so it can only add variance.  Thus the frozen-cell identity pins down the span in \eqref{eq:bp-rigid-form}, and within-stratum centering makes the remaining function of the instrument orthogonal to the doubly robust residuals.

\begin{remark}[Frozen-evaluation identities]\label{rem:frozen-template}
The finite-cell proof illustrates a more general implication of exact robustness. Fix $P_0$ and a family of laws that varies one primary component while preserving the others, $P_X$ and $\kappa$. Write $f_0=f\{O;\eta(P_0),\kappa(P_0),w\}$. Suppose $f_0$ is integrable under the family, and that its conditional kernels have the representation
\[
Q_{x,\delta}-Q_{x,0}=L_x\delta,\qquad \delta\in D_x,
\]
where $L_x:T_x\to M_0(\cW)$ is linear and $D_x$ is a relatively open neighbourhood of zero in the finite-dimensional space $T_x$. Here $M_0(\cW)$ denotes finite signed measures of total mass zero. Assume these kernels are dominated by $Q_{x,0}$, their Borel graphs and finite model costs permit the localized substitutions of Assumption~\ref{ass:rich}, and the conditional target changes by $\langle c_x,\delta\rangle$.

Condition~(B) and marginal localization give
\begin{equation}\label{eq:frozen-template}
\langle L_x^\ast f_{0,x}-c_x,\delta\rangle=0
\quad(\delta\in D_x),\qquad
L_x^\ast f_{0,x}-c_x\in T_x^\circ,
\end{equation}
where $\langle L_x^\ast f,\delta\rangle=\int f\,d(L_x\delta)$ and $T_x^\circ$ is the annihilator in the dual pairing. The extension from $D_x$ to $T_x$ follows by linearity. In the Balke--Pearl application, $L_x\delta=\sum_c\rho_{z(c)}\delta_c\delta_{\{c\}}$, so $L_x^\ast f=\bm r\circ\bm f$ and the annihilator is $T_x^\perp$.

A variance characterization additionally requires an explicit description of this annihilator and a minimizing admissible map. For Balke--Pearl, the description yields the active residual span and an orthogonal function of $Z$; Definition~\ref{def:bp-face-family} supplies the admissible map. For smooth dominated paths the corresponding identity is $\E[f_{0,x}s_\delta]=\dot\psi_x[\delta]$ when differentiation of the frozen integrand is justified. The mediation proof below obtains it from separate bounded one-component tilts of the outcome and mediator laws. Its mediator density ratio is a functional of the single mediator-distribution component.
\end{remark}

\subsubsection{Numerical illustration at a kink point}\label{sec:bp-numerical}

The following calculations use the integer coefficient matrix of~\eqref{eq:bp-U-explicit}.  Take a single covariate stratum, so that $\Var_{P_X}[\min_k U_k]=0$ and the floor is the conditional one.  With $\rho=P(Z=1)=0.15399$ and conditional cells
\begin{equation}\label{eq:bp-num-law}
\setlength{\arraycolsep}{4pt}
\begin{array}{c|cc|cc}
& A=0,Y=0 & A=0,Y=1 & A=1,Y=0 & A=1,Y=1 \\\hline
Z=0 & 0.884931 & 0.026210 & 0.050839 & 0.038020\\
Z=1 & 0.095170 & 0.015781 & 0.050839 & 0.838210
\end{array}
\end{equation}
the eight bounds are
\[
\begin{aligned}
U={}&(0.922951,\ 0.933380,\ 0.981829,\ 0.922951,\\[-2pt]
&\phantom{(}0.933380,\ 1.750732,\ 1.818311,\ 1.018121),
\end{aligned}
\]
so $\min_kU_k=0.922951$ is attained at $\cA^\ast=\{1,4\}$: a genuine kink point with two tied active indices.  The active block of the doubly robust covariance~\eqref{eq:bp-Sigma-dr} has eigenvalues $(0.08094,\,0.34665)$ and is positive-definite, so the program of~\eqref{eq:bp-Vmin} gives
\begin{equation}\label{eq:bp-num-vmin}
\Vmin=\bigl(\bm 1^\top\widetilde{\bm\Sigma}_{\cA^\ast}^{-1}\bm 1\bigr)^{-1}
=0.075752,
\qquad
\bm\lambda^\ast=\frac{\widetilde{\bm\Sigma}_{\cA^\ast}^{-1}\bm 1}
{\bm 1^\top\widetilde{\bm\Sigma}_{\cA^\ast}^{-1}\bm 1}=(0.14974,\,0.85026).
\end{equation}
The optimal weights are interior, so the covariance optimization makes a nontrivial weighting choice. Monte Carlo simulation of the influence function under~\eqref{eq:bp-num-law} confirms the displayed variances. The equal-weight limit in \citet[Corollary~4.5]{Whitehouse2025Softmax} gives $0.121194$, compared with the covariance-optimal value above, a variance ratio of $1.60$. The residuals are centered within each $Z$-stratum, so their covariance excludes between-stratum variation. The calculations use the full-precision probabilities, with the tie-preserving six-decimal values displayed in~\eqref{eq:bp-num-law}.

\subsubsection{A singular active face}\label{sec:bp-degenerate}

One mechanism producing rank deficiency of $\bm\Sigma_{\cA^\ast}$ is support-induced algebraic dependence among the active scores.  The following boundary law makes this mechanism transparent.  Rank deficiency is not confined to boundary laws, as the strictly positive rank-three example above shows.  Take $p_{11\mid0}=p_{00\mid1}=p_{11\mid1}=0$ with $\rho=0.3$ and
\begin{equation}\label{eq:bp-num-law-sing}
\setlength{\arraycolsep}{4pt}
\begin{array}{c|cc|cc}
& A=0,Y=0 & A=0,Y=1 & A=1,Y=0 & A=1,Y=1 \\\hline
Z=0 & 0.05 & 0.05 & 0.90 & 0\\
Z=1 & 0 & 0.05 & 0.95 & 0
\end{array}
\end{equation}
Then $U=(0,\,0,\,0,\,0.05,\,0.05,\,0.05,\,0,\,0.1)$ and $\cA^\ast=\{1,2,3,7\}$, a fourfold tie.  The active block $\widetilde{\bm\Sigma}_{\cA^\ast}$ has eigenvalues $(0,\,0,\,0,\,0.678571)$: rank $1$, and $\bm 1\notin\operatorname{range}(\widetilde{\bm\Sigma}_{\cA^\ast})$.  The value of the program is therefore $0$, with unit-sum weights $\bm\lambda^\ast=(\tfrac12,0,\tfrac12,0)$ for which
\[
\tfrac12(\varphi_1^P-U_1)+\tfrac12(\varphi_3^P-U_3)=0
\qquad Q_x^P\text{-a.s.}:
\]
the coefficient vectors $c^{(1)}$ and $c^{(3)}$ sum to zero on every cell of positive probability, so the two centred residuals cancel.  The conditional variance of the corresponding residual combination is $0$, against $0.014137$ at weights spread equally over the active face. At this law the conditional benchmark is zero, so its lower bound uses only variance nonnegativity.  The displayed weights have $\ell^1$ norm one and hence define a section of the bounded full-slot map for every $L\ge1$.  Consequently $\Vmin=0$ here without a positive-cell neighbourhood.  This is a map-level statement; a common estimating procedure still needs the learning conditions of the separated-kernel attainment result.  With varying $X$, conditional zero would leave the marginal variance term intact.

\begin{remark}[Reading the singular case]\label{rem:singular-implication}
Singularity here is a property of the law, not of the framework: at laws where some observed cells vanish, distinct bound expressions become indistinguishable as random variables, and an exactly determined linear combination of active scores exists.  The framework locates this through the range criterion of Proposition~\ref{prop:cov-singular} rather than through a kernel membership test; the two differ, and the kernel test alone would misclassify laws with $\bm 1\notin\operatorname{range}$ but $\bm 1\notin\ker$.  In implementations the active covariance should be tested for rank and the pseudo-inverse used, rather than inverted naively.
\end{remark}

\subsection{Mediation: cross-world functional}\label{sec:mediation}

The mediation natural indirect effect (NIE) is $\Psi_{\mathrm{NIE}}(P)=\E[Y(1,M(1))]-\E[Y(1,M(0))]$. Both summands are identified under sequential ignorability \cite{Imai2010}, and a triple-robust IF for the NIE is the difference of triple-robust IFs for the two summands. Of these, the second summand, the cross-world counterfactual $\E[Y(1,M(0))]$, is the mediator-specific piece exhibiting the $d=3$ multiple robustness pattern; the first summand $\E[Y(1,M(1))]=\E[Y(1)]$ is a standard $d=2$ ATE-like target.

We focus on the cross-world piece $\Psi_{\mathrm{cross}}(P)\eqdef\E[Y(1,M(0))]$, which is the canonical setting where ``at most $1$ wrong of $3$'' robustness is a nontrivial extension of double robustness, and the framework should recover the \cite{TchetgenShpitser2012} triple-robust IF.

Under the overlap, dominance, square-integrability, and score-path conditions stated below, the cross-world functional has the classical efficient representation of \citet{TchetgenShpitser2012}. The same population map satisfies three exact unbiasedness identities, one for each possible misspecified nuisance component. These identities imply conditional orthogonality between its residual and the difference from any competing RALU map. The resulting variance decomposition gives the lower bound, and Corollary~\ref{cor:nie-achieve-global} constructs one estimator sequence that attains it throughout the specified learnable subclass.

\subsubsection{Setup}\label{sec:mediation-setup}

Let $\cW=\cA\times\cM\times\cY$ with $\cA=\{0,1\}$, $\cM$ a Polish mediator space, $\cY=\R$. Write $W=(A,M,Y)$. Under sequential ignorability assumptions \cite{Imai2010}, the cross-world expectation
\begin{equation}\label{eq:nie-target}
\Psi_{\mathrm{cross}}(P)=\E[Y(1,M(0))]
\end{equation}
is identified by
\begin{equation}\label{eq:nie-id}
\Psi_{\mathrm{cross}}(P)=\E_{P_X}\!\left[\int\eta_1(P)(X,m)\,
\mathrm{d}\eta_3(P)(m\mid X,A=0)\right],
\end{equation}
\Needspace{6\baselineskip}
with three nuisance functionals:
\begin{itemize}[leftmargin=2em,itemsep=0pt]
\item $\eta_1(P)(x,m)=\E_P[Y\mid X=x,A=1,M=m]$: outcome regression;
\item $\eta_2(P)(x)=P(A=1\mid X=x)$: treatment propensity;
\item $\eta_3(P)(\cdot\mid x,a)=P_M(\cdot\mid X=x,A=a)$: mediator distribution.
\end{itemize}

For later variance formulas, write
\[
\sigma_Y^2(x,m;P)\eqdef
\Var_P(Y\mid X=x,A=1,M=m).
\]

This fits \eqref{eq:mi-rep} with $\cU$ the space of (regression, mediator-distribution) pairs and
\begin{equation}\label{eq:nie-psi}
\psi(x,(\eta_1,\eta_3))=\int\eta_1(x,m)\,\mathrm{d}\eta_3(m\mid x,0).
\end{equation}
The structural component $\kappa$ is trivial (no kink). The number of working models is $d=3$.

\begin{lemma}[Jointly Borel Radon--Nikodym selector]
\label{lem:rn-selector}
Let $\cM$ be a Polish space, equip $\cP(\cM)$ with the weak topology, fix $C>1$, and set
\[
\mathfrak D_C\eqdef
\{(Q_0,Q_1)\in\cP(\cM)^2:Q_0\le C Q_1\}.
\]
The set $\mathfrak D_C$ is closed and hence Polish, and there is one Borel map
\begin{equation}\label{eq:rn-selector}
R:\cP(\cM)^2\times\cM\to[0,C]
\end{equation}
such that $R(Q_0,Q_1,\cdot)$ is a version of $\mathrm dQ_0/\mathrm dQ_1$ for every $(Q_0,Q_1)\in\mathfrak D_C$.
\end{lemma}

\begin{proof}
The order condition is equivalent to $\int\phi\,\mathrm dQ_0\le C\int\phi\,\mathrm dQ_1$ for every bounded nonnegative continuous $\phi$.  These inequalities are preserved under weak convergence, so $\mathfrak D_C$ is closed in the Polish product $\cP(\cM)^2$.

Choose refining finite partitions generated by a countable algebra $\{A_{n,i}\}$ for the Borel structure of $\cM$, and let $A_n(m)$ be the cell containing $m$.  Define
\[
R_n(Q_0,Q_1,m)\eqdef
\begin{cases}
Q_0\{A_n(m)\}/Q_1\{A_n(m)\},&Q_1\{A_n(m)\}>0,\\
0,&Q_1\{A_n(m)\}=0,
\end{cases}
\qquad
R\eqdef\bigl(\limsup_n R_n\bigr)\wedge C .
\]
Evaluation of a probability measure on a fixed Borel set is Borel, so each $R_n$ and hence $R$ are jointly Borel.  On $\mathfrak D_C$, the sequence is the martingale approximation to $\mathrm dQ_0/\mathrm dQ_1$ along the generated filtration and converges $Q_1$-a.e.\ to that derivative, which is bounded by $C$.
\end{proof}

\begin{definition}[Mediation slot space]
\label{def:med-slot-space}
Fix $C>1$, $\epsilon\in(0,\tfrac12)$, and a Polish space $\mathcal H$ of candidate regression functions with a jointly Borel evaluation map
\begin{equation}\label{eq:med-regression-slot}
\mathrm{ev}:\mathcal H\times\cM\to\R,
\qquad
\sup_{h\in\mathcal H,\,m\in\cM}
|\mathrm{ev}(h,m)|\le C_\mu<\infty .
\end{equation}
A pointwise mediation slot is an element of the Polish product
\[
\mathcal H\times[\epsilon,1-\epsilon]\times\mathfrak D_C.
\]
A law-level slot is a Borel map from $\cX$ into this product.  Writing $\bm\mu=\mathrm{ev}(h,\cdot)$ and choosing Borel kernel versions, define
\[
\begin{aligned}
\cN_3^{\mathrm{ov}}
&\eqdef\bigl\{(G_0,G_1):
 (G_0(\cdot\mid x),G_1(\cdot\mid x))\in\mathfrak D_C
 \text{ for all }x\bigr\},\\
r_G(x,m)&\eqdef
R\bigl(G_0(\cdot\mid x),G_1(\cdot\mid x),m\bigr).
\end{aligned}
\]
Values on exceptional sets are assigned one fixed default pair $(Q_\circ,Q_\circ)\in\mathfrak D_C$.
\end{definition}

\begin{assumption}[Mediation slot membership and uniform slack]
\label{ass:med-slot-domain}
Every $P\in\cP$ satisfies $\E_P[\1\{A=1\}|Y|]<\infty$. Every truth slot and every law-achievable working slot used in condition~(B) belongs to the law-level space of Definition~\ref{def:med-slot-space}.  Thus the regression component factors through a Borel field $h:\cX\to\mathcal H$, the treatment component lies in $[\epsilon,1-\epsilon]$, and the mediator component lies in $\cN_3^{\mathrm{ov}}$.  At every truth there is a constant $\delta>0$ for which the following intervals are nonempty and, for $P_X$-a.e.\ $x$,
\begin{equation}\label{eq:med-interior}
r_P(x,\cdot)\le C-\delta
\quad G_1^P(\cdot\mid x)\text{-a.e.},
\qquad
\epsilon+\delta\le\eta_2(P)(x)\le1-\epsilon-\delta .
\end{equation}
Every score path named below is required to remain in this declared slot domain for all sufficiently small positive and negative path parameters.
\end{assumption}

\subsubsection{Multiple robustness and the Tchetgen Tchetgen--Shpitser rule}\label{sec:nie-mr}

Condition~(B) of Definition~\ref{def:ralu}, specialized to $d=3$, requires: for $P,P'\in\cP$ with $\eta_l(P)=\eta_l(P')$ for $l\ne j$ (any single $j\in\{1,2,3\}$),
\begin{equation}\label{eq:nie-MR}
\E_P[f(X,W,\eta(P')(X),w(X))]=\Psi_{\mathrm{cross}}(P).
\end{equation}
Thus at least two of the three nuisance components must be correct, the triple-robustness pattern of \citet{TchetgenShpitser2012}. Section~\ref{sec:nie-IF} verifies these identities directly. Proposition~\ref{prop:nie-two-wrong-ceiling} shows that, under the bilinear mixture condition below, a guarantee requiring only the propensity to be correct is impossible.

The leave-one-out structure of the Tchetgen Tchetgen--Shpitser rule suggests the candidate below.  Its three exact one-wrong identities are proved by direct integration.  Proposition~\ref{prop:nie-two-wrong-ceiling} then proves, under the stated bilinear mixture condition, that the stronger guarantee allowing two simultaneous misspecifications is impossible on the declared mediation class.

\begin{assumption}[Bilinear mixture richness for mediation]
\label{ass:med-bilinear-rich}
On a set of covariate values of positive $P_X$-measure, the model contains a subfamily with the propensity $e$ and treated-mediator law $G_1$ fixed, two endpoint laws $Q^+,Q^-$, and every mixture $Q_\lambda=\lambda Q^++(1-\lambda)Q^-$, $0\le\lambda\le1$.  The endpoint outcome regressions and control-mediator laws obey
\[
 \mu_\lambda=\lambda\mu_++(1-\lambda)\mu_- ,\qquad
 G_{0,\lambda}=\lambda G_{0,+}+(1-\lambda)G_{0,-},
\]
and
\begin{equation}\label{eq:med-bilinear-witness}
 \int(\mu_+-\mu_-)\,\mathrm d(G_{0,+}-G_{0,-})\ne0.
\end{equation}
The subfamily is closed under the marginal gluing used to localize a violation.
\end{assumption}

\begin{proposition}[Sharp two-wrong impossibility]
\label{prop:nie-two-wrong-ceiling}
Under Assumption~\ref{ass:med-bilinear-rich}, no single measurable influence map can satisfy the strengthening of condition~(B) that permits the regression component and the mediator-pair component to be simultaneously wrong while only the propensity component is correct. Thus the one-wrong guarantee in~\eqref{eq:nie-MR} is maximal on this declared mediation class.
\end{proposition}

\begin{proof}
Fix the working nuisance value at an endpoint, say $\eta(Q^-)$, and write $h_x$ for the resulting measurable function of $W$. Put $\Delta\mu=\mu_+-\mu_-$ and $\Delta G_0=G_{0,+}-G_{0,-}$. For the conditional target along the mixture,
\[
q_x(\lambda)\eqdef\int\mu_\lambda\,\mathrm dG_{0,\lambda}
=q_x(0)+\lambda\left\{\int\Delta\mu\,\mathrm dG_{0,-}
 +\int\mu_-\,\mathrm d\Delta G_0\right\}
 +\lambda^2 c_x,
\qquad c_x\eqdef\int\Delta\mu\,\mathrm d\Delta G_0.
\]
The strengthened robustness condition would give $\E_{Q_\lambda}h_x=q_x(\lambda)$, because the propensity is unchanged. Integrability at both endpoints and the mixture identity instead imply
\[
\E_{Q_\lambda}h_x
=(1-\lambda)\E_{Q^-}h_x+\lambda\E_{Q^+}h_x,
\qquad
q_x(0)-2q_x(1/2)+q_x(1)=\tfrac12 c_x\ne0.
\]
These equations are incompatible. For a covariate-dependent mixture family, use the marginal localization in Assumption~\ref{ass:med-bilinear-rich} on a positive-measure set where $c_x$ has one sign, restricting further to finite model cost if necessary. The integrated second difference remains nonzero, while the expectation of the fixed working function has second difference zero. Only the three mixtures $\lambda=0,1/2,1$ are needed, so their exceptional sets can be removed together.
\end{proof}

A uniform-slack example satisfying the assumption is obtained as follows.  On a two-point mediator space take $G_1=(1/2,1/2)$, $G_{0,+}=(3/4,1/4)$, $G_{0,-}=(1/4,3/4)$, $\mu_+=(1,0)$ and $\mu_-=(0,1)$.  The target is $\lambda^2-\lambda+3/4$, whereas every frozen-slot expectation is affine; all Radon--Nikodym derivatives lie in $[1/2,3/2]\subset(0,2)$.  This explicit uniform-slack witness supplies the mediation-specific richness that the proposition uses, without invoking the shared-score fixed-mean condition~(COV1).  Mediation instead uses the bilinear-mixture richness stated in the proposition.

Section~\ref{sec:nie-IF} constructs and verifies an influence function achieving~\eqref{eq:nie-MR}; the leave-one-out algebraic structure above motivates its form, and Proposition~\ref{prop:nie-two-wrong-ceiling} establishes maximality on the declared mediation class.

\subsubsection{Verification of the Tchetgen Tchetgen--Shpitser influence function}\label{sec:nie-IF}

Write a working slot as $\eta=(\mu,e,G)$, where $G=(G_0,G_1)$ is the third primary component.  Definition \ref{def:med-slot-space} supplies the jointly Borel quantities
\[
r_\eta(x,m)\eqdef R\{G_0(\cdot\mid x),G_1(\cdot\mid x),m\},
\qquad
b_0^\eta(x)\eqdef\int\mu(x,m)\,\mathrm dG_0(m\mid x).
\]
\Needspace{7\baselineskip}
The population rule is
\begin{multline}\label{eq:nie-IF-candidate}
f_R^{\mathrm{TS}}(W;\eta)
=\frac{\1\{A=1\}}{e(X)}r_\eta(X,M)\{Y-\mu(X,M)\}\\
+\frac{\1\{A=0\}}{1-e(X)}\{\mu(X,M)-b_0^\eta(X)\}
+b_0^\eta(X).
\end{multline}
The structural slot is trivial and the auxiliary slot is unread.  The Radon--Nikodym factor is a deterministic Borel function of the mediator pair $G$, not an additional input. We write $f^{\mathrm{TS}}=f_R^{\mathrm{TS}}$ when the selector is clear. For every working slot, the common envelope
\[
|f_R^{\mathrm{TS}}|
\le \frac{C}{\epsilon}\1\{A=1\}(|Y|+C_\mu)
       +\frac{2C_\mu}{\epsilon}+C_\mu
\]
is integrable at every truth in $\cP$.  Thus the following exact identities establish the full condition~(B), including L1, on the comparison model independently of estimation rates.

Let $(\mu_P,e_P,G^P)$ denote the truth.  Direct conditioning gives the single identity
\begin{equation}\label{eq:nie-one-wrong-mean}
\begin{aligned}
\E_P\{f_R^{\mathrm{TS}}(W;\eta)\mid X\}
&=\frac{e_P}{e}\int r_\eta(\mu_P-\mu)\,\mathrm dG_1^P\\
&\quad+\frac{1-e_P}{1-e}
 \left\{\int\mu\,\mathrm dG_0^P-b_0^\eta\right\}+b_0^\eta .
\end{aligned}
\end{equation}
All functions and kernels in this display are evaluated at the same covariate value.

Case~(a): only $\eta_1$ wrong. Set $\eta_2=\eta_2(P)$, $\eta_3=\eta_3(P)$, $\eta_1$ arbitrary.  In~\eqref{eq:nie-one-wrong-mean}, $e=e_P$, $G=G^P$, and $r_\eta\,\mathrm dG_1^P=\mathrm dG_0^P$.  The first term is $\int(\mu_P-\mu)\,\mathrm dG_0^P$, the bracketed term is zero, and the result is $\int\mu_P\,\mathrm dG_0^P$.

Case~(b): only $\eta_2$ wrong. Set $\eta_1=\eta_1(P)$, $\eta_3=\eta_3(P)$, $\eta_2$ arbitrary.  Now $\mu=\mu_P$ and $G=G^P$.  Both residual integrals in \eqref{eq:nie-one-wrong-mean} vanish, independently of the working propensity, and $b_0^\eta=\int\mu_P\,\mathrm dG_0^P$.

Case~(c): only $\eta_3$ wrong. Set $\eta_1=\eta_1(P)$, $\eta_2=\eta_2(P)$, $\eta_3$ arbitrary.  Here $\mu=\mu_P$ and $e=e_P$.  The first integral in \eqref{eq:nie-one-wrong-mean} is zero.  The second term plus $b_0^\eta$ equals $\int\mu_P\,\mathrm dG_0^P$, regardless of the working mediator pair.

If $\mu=\mu_P$ while $e$ and $G$ are both substituted, the conditional bias is
\[
\left\{\frac{1-e_P}{1-e}-1\right\}
\left\{\int\mu_P\,\mathrm dG_0^P-b_0^\eta\right\}.
\]
The bias vanishes when either the propensity or mediator pair is correct. When both differ, the displayed product gives its exact value. Proposition~\ref{prop:nie-two-wrong-ceiling} rules out a uniform guarantee over all configurations with only the propensity correct.

\subsubsection{The mediator density ratio}
\label{sec:nie-ratio}

The identity $r_\eta\,\mathrm dG_1=\mathrm dG_0$ transports integration from the treated-mediator law to the control-mediator law appearing in the target. Lemma~\ref{lem:rn-selector} defines this change of measure as one jointly Borel function of the mediator pair throughout the declared domain. The pair $G=(G_0,G_1)$ therefore determines the ratio used by the population rule; no auxiliary input is needed.

\subsubsection{Pointwise variance minimum: an exact Pythagorean identity}
\label{sec:nie-rigid}

To compare an arbitrary admissible map with the Tchetgen Tchetgen--Shpitser rule, hold their nuisance inputs fixed and vary one component of the observation law. Exact robustness makes their expectations agree along each such path. The next assumption supplies enough bounded score directions to turn these equalities into a conditional variance decomposition.

Write, at the truth,
\[
 b_0(x)\eqdef\int\eta_1(P)(x,m)\,\mathrm dG_0^P(m\mid x),
 \qquad
 r_P(x,m)\eqdef R\{G_0^P(\cdot\mid x),G_1^P(\cdot\mid x),m\},
\]
and define the two centered components
\begin{align}
S_Y(W;P)&\eqdef
 \frac{\1\{A=1\}}{\eta_2(P)(X)}r_P(X,M)
 \{Y-\eta_1(P)(X,M)\},\nonumber\\
S_{M_0}(W;P)&\eqdef
 \frac{\1\{A=0\}}{1-\eta_2(P)(X)}
 \{\eta_1(P)(X,M)-b_0(X)\}.
\label{eq:nie-two-score-components}
\end{align}
Then $f^{\mathrm{TS}}-b_0=S_Y+S_{M_0}$ pointwise and
\begin{equation}\label{eq:nie-eff-bound}
\begin{aligned}
V^{\mathrm{TS}}(x;P)
&=\frac1{\eta_2(P)(x)}\int
  r_P(x,m)^2\sigma_Y^2(x,m;P)\,\mathrm dG_1^P(m\mid x)\\
&\quad+\frac1{1-\eta_2(P)(x)}\int
  \{\eta_1(P)(x,m)-b_0(x)\}^2\,\mathrm dG_0^P(m\mid x).
\end{aligned}
\end{equation}

\begin{assumption}[Mediation score paths]\label{ass:med-tangent}
At each law under consideration, $S_Y$ and $S_{M_0}$ in \eqref{eq:nie-two-score-components} belong to $L^2(P)$.  There are bounded, conditionally centered scores $s_{Y,k}$ and $s_{M_0,k}$ of two-sided bounded linear tilts $\mathrm dP_t=(1+ts)\mathrm dP$ in the comparison model, which move, respectively, only $Y\mid X,A{=}1,M$ and only $M\mid X,A{=}0$, such that
\[
s_{Y,k}\longrightarrow S_Y,
\qquad s_{M_0,k}\longrightarrow S_{M_0}
\qquad\text{in }L^2(P).
\]
For every bounded Borel $b(X)$, the products $b(X)s_{Y,k}$ and $b(X)s_{M_0,k}$ also generate bounded linear tilts in the declared model, preserving $P_X$ and moving only the indicated component.  All these paths stay in the regression slot space $\mathcal H$ and remain in the declared mediator-law domain for all sufficiently small positive and negative path parameters.
\end{assumption}

\begin{theorem}[Cross-world exact excess-variance identity]
\label{thm:nie-pythagoras}
Suppose Assumptions~\ref{ass:rich} and~\ref{ass:meas-main} hold, together with Assumptions~\ref{ass:med-slot-domain} and~\ref{ass:med-tangent}.  Let $f\in\IFRL$ for $\Psi_{\mathrm{cross}}$ on the comparison model $\cP$.  Fix a measurable auxiliary map $w$, suppose $f_w^P\in L^2(P)$, and put $g_w^P\eqdef f_w^P-f^{\mathrm{TS},P}$.  At $P_X$-a.e. $x$,
\begin{equation}\label{eq:nie-pythagoras}
\Var_{Q_x^P}\!\left[f^P_w(x,\cdot)\right]
=V^{\mathrm{TS}}(x;P)
+\E_{Q_x^P}\!\left[(g^P_w(x,\cdot))^2\right].
\end{equation}
Thus $V^{\mathrm{TS}}(x;P)$ is a pointwise lower bound for every such map in $\IFRL$, and equality for a given map holds if and only if $f^P_w(x,\cdot)=f^{\mathrm{TS},P}_w(x,\cdot)$ $Q_x^P$-a.s.  The verified integrable full-slot map $f_R^{\mathrm{TS}}\in\IFRL$ attains the pointwise infimum at the evaluating law.  The fixed-law expansion is stated in Theorem~\ref{thm:nie-achieve}; Corollary~\ref{cor:nie-achieve-global} supplies one estimator attaining it throughout a learnable subclass.
\end{theorem}

\begin{proof}
Suppress $w$ and write $g^P=f_w^P-f^{\mathrm{TS},P}$. The bounded regression and Assumption~\ref{ass:med-tangent} give $f^{\mathrm{TS},P}\in L^2(P)$, hence $g^P\in L^2(P)$. The conditional-unbiasedness conclusion of Theorem~\ref{thm:ulb}, applied to $f$ and to the directly verified map $f_R^{\mathrm{TS}}$, gives
\[
\E_P[g^P\mid X]=0.
\]
For a bounded score $s$ in Assumption~\ref{ass:med-tangent}, let $\mathrm dP_t=(1+ts)\mathrm dP$ and keep the working value equal to $\eta(P)$. The laws $P_t$ and $P$ differ in at most one nuisance component. Condition~(B) for the two maps therefore gives
\begin{equation}\label{eq:frozen-slot-zero}
0=\E_{P_t}[g\{W;\eta(P)\}]
 =\E_P[g^P]+t\E_P[g^Ps].
\end{equation}
The integrals are finite because $s$ is bounded and $g^P\in L^1(P)$. Taking any sufficiently small $t\ne0$ yields the exact identity
\begin{equation}\label{eq:score-identity}
\E_P[g^Ps]=0.
\end{equation}
For $j\in\{Y,M_0\}$ and bounded Borel $b(X)$, apply this identity to $b(X)s_{j,k}$. Cauchy--Schwarz gives
\[
\left|\E_P[g^P b(X)(S_j-s_{j,k})]\right|
\le\|b\|_\infty\|g^P\|_2\|S_j-s_{j,k}\|_2\longrightarrow0.
\]
Thus $\E_P[g^P b(X)S_j]=0$ for every bounded $b$, and consequently
\begin{equation}\label{eq:nie-score-orth}
\E_P[g^P S_Y\mid X]=0,
\qquad
\E_P[g^P S_{M_0}\mid X]=0.
\end{equation}
The components $S_Y$ and $S_{M_0}$ have conditional mean zero and disjoint treatment supports, so $S_YS_{M_0}=0$ pointwise. Since $f_w^P-b_0=S_Y+S_{M_0}+g^P$, expansion of the conditional square now gives
\[
\begin{aligned}
\Var_P(f_w^P\mid X)
&=\E_P[(S_Y+S_{M_0})^2\mid X]
 +2\E_P[g^P(S_Y+S_{M_0})\mid X]
 +\E_P[(g^P)^2\mid X]\\
&=V^{\mathrm{TS}}(X;P)+\E_P[(g^P)^2\mid X].
\end{aligned}
\]
The excess is zero precisely when $g^P=0$ conditionally almost surely. The map $f_R^{\mathrm{TS}}$ itself has $g^P=0$, belongs to $\IFRL$ by the verified integrable identities, and has finite variance under the stated assumptions; it therefore attains the infimum.
\end{proof}

Theorem~\ref{thm:nie-pythagoras} identifies the excess conditional variance as the squared distance from the Tchetgen Tchetgen--Shpitser estimating function. Exact robustness supplies the orthogonality needed for this identity.

\begin{remark}[Role of the slot domain]\label{rem:nie-slot-domain}
Definition~\ref{def:med-slot-space} supplies a real-valued measurable rule for every admissible nuisance input. Dominance is needed for the change of measure: if $G_0=\delta_{m_0}$ and $G_1=\delta_{m_1}$ with $m_0\ne m_1$, every measurable $r$ satisfies $\int_{\{m_0\}}r\,\mathrm dG_1=0$, whereas $G_0\{m_0\}=1$, so no Radon--Nikodym derivative exists. Treatment overlap makes the two inverse-probability factors finite. The uniform bounds on these factors and on $r$ yield the common integrable envelope and the product-rate bounds used below.
\end{remark}

\begin{remark}[Truncation and exact robustness]
\label{rem:nie-truncation-scope}
The population rule uses the exact ratio on its declared domain. To see the effect of changing it, take $G_0=(0.9,0.1)$, $G_1=(0.1,0.9)$, truth regression $\mu\equiv0$, and working regression $(1,0)$, with the propensity and mediator pair correct. Replacing $r=(9,1/9)$ by a cap at $2$ gives conditional bias
\[
\int(\widehat\mu-\mu)(r-\widetilde r)\,\mathrm dG_1
=(9-2)(0.1)=0.7,
\]
by~\eqref{eq:nie-product-bias}, while the exact ratio gives zero. A cap at or above the model-wide ratio bound leaves the population rule unchanged. For estimation, the bounded learner $\widetilde r$ is required to satisfy the consistency and product-rate conditions below.
\end{remark}

\begin{theorem}[Fixed-law expansion for the cross-fitted cross-world estimator]
\label{thm:nie-achieve}
Let $f_R^{\mathrm{TS}}$ denote the population slot map in~\eqref{eq:nie-IF-candidate}, with its ratio evaluated by the single Borel selector $R(G_0,G_1)$ of~\eqref{eq:rn-selector}.  This is the globally defined measurable map used in condition~(B).

Use a fixed finite number $\mathcal K$ of nonempty folds, with the partition chosen independently of the observations. Let $n_k$ be the size of fold $k$ and $\mathbb P_{n,k}$ its empirical measure. For fold $k$, fit the nuisance functions using observations outside that fold and put
\[
\widehat b_0^{(-k)}(x)\eqdef
 \int\widehat\eta_1^{(-k)}(x,m)\,
      \mathrm d\widehat G_0^{(-k)}(m\mid x)
\]
and define the finite-sample score
\begin{align}
\widetilde f_k^{\mathrm{TS}}
&\eqdef
 \frac{\1\{A=1\}}{\widehat e^{(-k)}(X)}
 \widetilde r^{(-k)}(X,M)
 \{Y-\widehat\eta_1^{(-k)}(X,M)\}\nonumber\\*
&\quad+
 \frac{\1\{A=0\}}{1-\widehat e^{(-k)}(X)}
 \{\widehat\eta_1^{(-k)}(X,M)-\widehat b_0^{(-k)}(X)\}
 +\widehat b_0^{(-k)}(X),\nonumber\\*
T_n&\eqdef\sum_{k=1}^{\mathcal K}\frac{n_k}{n}\mathbb P_{n,k}[\widetilde f_k^{\mathrm{TS}}].
\label{eq:nie-finite-score}
\end{align}
The finite-sample score uses the bounded learner $\widetilde r$, which may be a truncated ratio estimate. Its convergence and product rates determine whether the score has the population representation $f_R^{\mathrm{TS},P_0}$.

Fix $P_0$.  With all quantities below computed off fold, assume:
\begin{enumerate}[wide=0pt,label={\textup{(N\arabic*)}},itemsep=1pt]
\item \emph{Overlap and clipping.}  For some $\epsilon_0>0$, $\epsilon_0\le e\le1-\epsilon_0$ $P_X$-a.e., and $\widehat e$ is clipped to $[\epsilon_0/2,1-\epsilon_0/2]$.

\item \emph{Boundedness.}  The outcome and fitted regression are uniformly bounded, the truth ratio satisfies $0\le r\le C$, and the single finite-sample learner satisfies $0\le\widetilde r\le C'$. 

\item \emph{Foldwise score stability.}
\begin{equation}\label{eq:nie-score-stability}
\|\widetilde f_k^{\mathrm{TS}}
  -f_R^{\mathrm{TS}}\{\cdot;\eta(P_0)\}\|_{L^2(P_0)}=o_P(1).
\end{equation}
A sufficient primitive set under~(N1)--(N2) is
\begin{equation}\label{eq:nie-primitives}
\begin{aligned}
\|\widehat\eta_1-\eta_1\|_{L^2(P_X\otimes G_0)}
 +\|\widehat\eta_1-\eta_1\|_{L^2(P_X\otimes G_1)}&=o_P(1),\\
\|\widehat e-e\|_{L^2(P_X)}
 +\|\widetilde r-r\|_{L^2(P_X\otimes G_1)}
 +\|\widehat b_0-b_0\|_{L^2(P_X)}&=o_P(1).
\end{aligned}
\end{equation}

\item \emph{Product-rate conditions.}
\begin{align}
\|\widehat e-e\|_{L^2(P_X)}
\left\|\int\widehat\eta_1\,
  \mathrm d(\widehat G_0-G_0)\right\|_{L^2(P_X)}
&=o_P(n^{-1/2}),\nonumber\\
\|\widehat\eta_1-\eta_1\|_{L^2(P_X\otimes G_1)}
\|\widetilde r-r\|_{L^2(P_X\otimes G_1)}
&=o_P(n^{-1/2}),\nonumber\\
\|\widehat e-e\|_{L^2(P_X)}
\|\widehat\eta_1-\eta_1\|_{L^2(P_X\otimes G_1)}
&=o_P(n^{-1/2}).
\label{eq:nie-product-rates}
\end{align}
\end{enumerate}
Then
\[
T_n=\mathbb P_n[f_R^{\mathrm{TS}}\{\cdot;\eta(P_0)\}]
 +o_{P_0}(n^{-1/2}).
\]
Thus condition~(A) holds at $P_0$ with influence function $f_R^{\mathrm{TS},P_0}-\Psi_{\mathrm{cross}}(P_0)$.  The population map $f_R^{\mathrm{TS}}$ satisfies exact condition~(B) at every law whose truth and admissible working slots satisfy Assumption~\ref{ass:med-slot-domain}, by the one-wrong algebra in Section~\ref{sec:nie-IF}.
\end{theorem}

\begin{proof}
Write $f_0=f_R^{\mathrm{TS}}\{\cdot;\eta(P_0)\}$ and $D_k=\widetilde f_k^{\mathrm{TS}}-f_0$. Conditional on the training sample for fold $k$, $D_k$ is fixed and the validation observations are independent with law $P_0$. Hence
\[
\E\!\left[\left\{\sqrt n\frac{n_k}{n}(\mathbb P_{n,k}-P_0)D_k\right\}^2
\,\middle|\,\text{training sample for fold }k\right]
=\frac{n_k}{n}\Var_{P_0}(D_k)
\le\|D_k\|_{L^2(P_0)}^2=o_P(1).
\]
Conditional Chebyshev and the fixed number of folds control the sum of the centered terms. The primitive conditions in~\eqref{eq:nie-primitives} suffice for this step because, under~(N1)--(N2), for a constant $K$ depending only on the bounds,
\[
\|D_k\|_{L^2(P_0)}
\le K\left\{\|\widehat e-e\|_{L^2(P_X)}
+\|\widetilde r-r\|_{L^2(P_X\otimes G_1)}
+\sum_{a=0}^1\|\widehat\eta_1-\eta_1\|_{L^2(P_X\otimes G_a)}
+\|\widehat b_0-b_0\|_{L^2(P_X)}\right\}.
\]
For the conditional bias, suppress the fold superscript and put $\delta_e=\widehat e-e$, $\delta_\mu=\widehat\eta_1-\eta_1$, and $A_0(X)=\int\widehat\eta_1\,\mathrm d(\widehat G_0-G_0)$. Conditioning on treatment first gives
\[
\E_{P_0}[\widetilde f_k^{\mathrm{TS}}\mid X]-b_0
=-\frac{e}{\widehat e}\int\widetilde r\delta_\mu\,\mathrm dG_1
-\frac{\delta_e}{1-\widehat e}A_0
+\int\delta_\mu\,\mathrm dG_0.
\]
Substitute $\mathrm dG_0=r\,\mathrm dG_1$ and $e/\widehat e=1-\delta_e/\widehat e$ to obtain
\begin{align}
&\E_{P_0}[\widetilde f_k^{\mathrm{TS}}\mid X]-b_0(X)\nonumber\\
&=-\frac{\delta_e}{1-\widehat e}
  \int\widehat\eta_1\,\mathrm d(\widehat G_0-G_0)
 +\int\delta_\mu(r-\widetilde r)\,\mathrm dG_1
 +\frac{\delta_e}{\widehat e}
  \int\delta_\mu\widetilde r\,\mathrm dG_1.
\label{eq:nie-product-bias}
\end{align}
Cauchy--Schwarz, the clipped denominators, and $\widetilde r\le C'$ imply
\[
\begin{aligned}
|P_0D_k|\le K\bigl\{&\|\delta_e\|_{L^2(P_X)}\|A_0\|_{L^2(P_X)}
+\|\delta_\mu\|_{L^2(P_X\otimes G_1)}\|\widetilde r-r\|_{L^2(P_X\otimes G_1)}\\
&+\|\delta_e\|_{L^2(P_X)}\|\delta_\mu\|_{L^2(P_X\otimes G_1)}\bigr\}
=o_P(n^{-1/2}).
\end{aligned}
\]
The third bound uses Jensen's inequality for the integral over $G_1(\cdot\mid X)$. Finally,
\[
T_n-\mathbb P_nf_0
=\sum_{k=1}^{\mathcal K}\frac{n_k}{n}
 \{(\mathbb P_{n,k}-P_0)D_k+P_0D_k\}
=o_P(n^{-1/2}).
\]
The population map satisfies~(B) by the integrable one-wrong identities in Section~\ref{sec:nie-IF}.
\end{proof}

\begin{remark}[Why the score-stability norms include $G_0$]
The mediator-stratum score depends on the integrated regression as well as the density ratio. With $X$ constant, $e=\widehat e=1/2$, $G_0=G_1=(1/2,1/2)$, $\widehat G_0=\widehat G_1=(0.9,0.1)$, and $\widehat\eta_1=\eta_1=(1,0)$, both density ratios equal one, but $b_0=1/2$ and $\widehat b_0=0.9$. The mediator-stratum terms then differ by $-0.8\1\{A=0\}$, with $L^2(P_0)$ norm $0.8/\sqrt2$. The regression function itself has norms $1/\sqrt2$ under $G_0$ and $\sqrt{0.9}$ under $\widehat G_0$. These calculations explain the separate integrated-regression condition and the explicit measure subscripts in~\eqref{eq:nie-primitives}.
\end{remark}

\begin{assumption}[Cross-world attainment on a subclass]
\label{ass:nie-uniform-achievement}
Fix a rich comparison model $\cP$ and let $\cP_{\mathrm{ach}}\subseteq\cP$ satisfy:
\begin{enumerate}[wide=0pt,label={\textup{(U\arabic*)}},itemsep=1pt]
\item One fold scheme, learner family, and clipping rule, all fixed independently of $P$, satisfy the foldwise conditions of Theorem~\ref{thm:nie-achieve} at every $P\in\cP_{\mathrm{ach}}$.
\item The comparison model satisfies Assumptions~\ref{ass:rich} and~\ref{ass:meas-main}.
\item The paths of Assumption~\ref{ass:med-tangent} lie in $\cP$ and are available at every $P\in\cP_{\mathrm{ach}}$.
\item Assumption~\ref{ass:med-slot-domain} holds throughout $\cP$ for every truth and every law-achievable working slot, with the same declared slot constants and one positive slack constant in \eqref{eq:med-interior}.
\end{enumerate}
\end{assumption}

\begin{corollary}[Global membership; mediation equality]
\label{cor:nie-achieve-global}
Under Assumption~\ref{ass:nie-uniform-achievement},
\[
T_n\in\TRL(\Psi_{\mathrm{cross}};\cP_{\mathrm{ach}}\mid\cP).
\]
Moreover, for every $P\in\cP_{\mathrm{ach}}$,
\[
\begin{aligned}
\Vmin(x;P)&=V^{\mathrm{TS}}(x;P)&&P_X\text{-a.e.},\\
\Vlow(P)&=\Var_{P_X}\{b_0(X)\}
+\E_{P_X}\{V^{\mathrm{TS}}(X;P)\}.
\end{aligned}
\]
Both benchmarks are taken over $\IFRL$ on the rich comparison model $\cP$; the learnable subclass need not be closed under gluing.
\end{corollary}

\begin{proof}
Condition (U4) makes the universal-selector slot a single measurable, real-valued rule on the comparison model, and its integrable envelope and the direct one-wrong calculation give condition~(B) on $\cP$. Condition (U1) and the three-product identity \eqref{eq:nie-product-bias} give condition~(A) for one estimator fixed over the subclass.  Conditions (U2)--(U3) permit the score-orthogonality lower bound~\eqref{eq:nie-pythagoras}.  The cross-fitted rule attains that lower bound at the true nuisance values.
\end{proof}

\begin{assumption}[Joint score paths for the full effect]
\label{ass:nie-joint-tangent}
At each law under consideration, write $\mu(x,m)=\E_P[Y\mid X=x,A=1,M=m]$, $e(x)=P(A=1\mid X=x)$, $G_a(\cdot\mid x)=P(M\in\cdot\mid X=x,A=a)$, $b_a(x)=\int\mu(x,m)\,\mathrm dG_a(m\mid x)$, and $r=R(G_0,G_1)$. Define
\begin{align}
S_Y^\Delta&\eqdef\frac{\1\{A=1\}}{e(X)}
 (1-r(X,M))\{Y-\mu(X,M)\},\nonumber\\
S_{M_1}^\Delta&\eqdef\frac{\1\{A=1\}}{e(X)}
 \{\mu(X,M)-b_1(X)\},\nonumber\\
S_{M_0}^\Delta&\eqdef-\frac{\1\{A=0\}}{1-e(X)}
 \{\mu(X,M)-b_0(X)\}.
\label{eq:nie-joint-score-components}
\end{align}
These three centered components belong to $L^2(P)$.  For each $j\in\{Y,M_1,M_0\}$ there are bounded, conditionally centered scores $s_{j,k}$ generating two-sided bounded linear tilts $\mathrm dP_t=(1+ts_{j,k})\mathrm dP$ in $\cP$, which move only, respectively, $Y\mid X,A{=}1,M$, $M\mid X,A{=}1$, and $M\mid X,A{=}0$, and
\[
s_{j,k}\longrightarrow S_j^\Delta\qquad\text{in }L^2(P).
\]
For every bounded Borel $b(X)$, $b(X)s_{j,k}$ is also the score of a bounded linear tilt in the comparison model, preserving $P_X$ and moving only component $j$.  Each path stays in the regression slot space $\mathcal H$ and remains in the declared mediator-law domain for all sufficiently small positive and negative path parameters, as required by Assumption~\ref{ass:med-slot-domain}.
\end{assumption}

\begin{theorem}[Joint identity for the full natural indirect effect]
\label{thm:nie-joint}
Let $\Delta(P)\eqdef\E_P[Y^{1,M^1}]-\E_P[Y^{1,M^0}]$.  Under the same identification conditions, its marginal-integral map is
\[
\psi_\Delta\{x,(\bm\mu,G)\}
=\int\bm\mu(x,m)\,\mathrm dG_1(m\mid x)
-\int\bm\mu(x,m)\,\mathrm dG_0(m\mid x),
\qquad G=(G_0,G_1).
\]
At a truth $P$, suppress law arguments and write
\[
b_a(x)\eqdef\int\bm\mu(x,m)\,\mathrm dG_a(m\mid x),
\qquad a=0,1.
\]
Both components must be built on \emph{one} common slot tuple $(\bm\mu,e,G_0,G_1)$, and the treated-arm rule must be recentred on the slot's own integrated regression,
\begin{equation}\label{eq:nie-F11}
F^{11}\eqdef\frac{\1\{A=1\}}{e}\bigl\{Y-\bar\mu(\bm\mu,G_1)\bigr\}
+\bar\mu(\bm\mu,G_1),
\qquad
\bar\mu(\bm\mu,G_1)\eqdef\int\bm\mu\,\mathrm dG_1 ,
\end{equation}
with $F^{10}$ the cross-world rule built from the same tuple through the canonical selector~\eqref{eq:rn-selector}, and $F^{\Delta}\eqdef F^{11}-F^{10}$.  Under~\eqref{eq:nie-F11}, each of the three declared one-wrong faces $(\bm\mu,e,G)$, with $G=(G_0,G_1)$, is exactly unbiased.

Assume Assumptions~\ref{ass:rich} and~\ref{ass:meas-main}.  Let $f^{\mathrm{TS}}_1\eqdef F^{11}$ and $f^{\mathrm{TS}}_0\eqdef F^{10}$ be those rules, under the slot domain of Assumption~\ref{ass:med-slot-domain}, and assume Assumption~\ref{ass:nie-joint-tangent} for the outcome, $M\mid A{=}1$, and $M\mid A{=}0$ score families.  For any map $f\in\IFRL$ for $\Delta$ on the comparison model $\cP$, fix a measurable auxiliary map $w$ and suppose $f_w^P\in L^2(P)$. Put $g\eqdef f-(f^{\mathrm{TS}}_1-f^{\mathrm{TS}}_0)$.  Then at $P_X$-a.e.\ $x$,
\begin{equation}\label{eq:nie-joint-pythagoras}
\Var_{Q_x^P}\bigl[f^P_w(x,\cdot)\bigr]
=\Var_{Q_x^P}\bigl[f^{\mathrm{TS}}_1-f^{\mathrm{TS}}_0\bigr]
+\E_{Q_x^P}\bigl[g^2\bigr],
\end{equation}
so the displayed variance is a conditional lower bound for every such map.  Equality for a given map holds if and only if $f=f^{\mathrm{TS}}_1-f^{\mathrm{TS}}_0$ $Q_x^P$-a.s.  Exact minimization over $\IFRL$ follows because the common-slot map $F^{11}-F^{10}$ is globally integrable by the same envelope used for $F^{10}$ and satisfies~(B).  Theorem~\ref{thm:nie-joint-achieve} establishes estimator attainment on a learnable subclass.

For the comparison with the separate component variances, put $D_a\eqdef F^{1a}-\E[F^{1a}\mid X]$ at the truth and $V_a(x)\eqdef\Var(D_a\mid X=x)$. At covariate values where $V_1(x)+V_0(x)<\infty$, direct conditioning yields
\begin{equation}\label{eq:nie-joint-covariance}
\Cov(D_1,D_0\mid X=x)
=\frac1{e(x)}\int\sigma_Y^2(x,m)\,\mathrm dG_0(m\mid x)\ge0.
\end{equation}
Consequently, whenever $V_1(x)+V_0(x)>0$,
\begin{equation}\label{eq:nie-joint-ratio-range}
0\le
\frac{\Var(F^{11}-F^{10}\mid X=x)}{V_1(x)+V_0(x)}
\le1.
\end{equation}
The interval is sharp. The variance of the full effect includes this covariance between its two component representations.
\end{theorem}

\begin{proof}
For a working nuisance tuple, put $\bar\mu=\int\mu\,\mathrm dG_1$ and let $b_1^P=\int\mu_P\,\mathrm dG_1^P$. Direct conditioning gives
\[
\E_P(F^{11}\mid X)-b_1^P
=\left(\frac{e_P}{e}-1\right)(b_1^P-\bar\mu).
\]
If only $\mu$ or $G$ is substituted, then $e=e_P$; if only $e$ is substituted, then $\bar\mu=b_1^P$. The displayed bias is therefore zero in all three cases. The same identities for $F^{10}$ were proved in~\eqref{eq:nie-one-wrong-mean}, so $F_R^\Delta=F^{11}-F^{10}$ satisfies~(B). Its global integrability follows by adding the envelopes of the two terms.

At the truth, direct subtraction gives
\[
F_R^{\Delta,P}-(b_1-b_0)
=S_Y^\Delta+S_{M_1}^\Delta+S_{M_0}^\Delta.
\]
These components are in $L^2(P)$ by Assumption~\ref{ass:nie-joint-tangent}; hence so is $g^P=f_w^P-F_R^{\Delta,P}$. Conditional unbiasedness gives $\E[g^P\mid X]=0$. For every bounded Borel $b(X)$ and every approximating score $s_{j,k}$, freeze the working tuple at $\eta(P)$ in the path with score $b(X)s_{j,k}$. The two exact identities~(B) imply
\[
0=\E_{P_t}[g\{W;\eta(P)\}]
=\E_P[g^P]+t\E_P[g^P b(X)s_{j,k}],
\qquad
\E_P[g^P b(X)s_{j,k}]=0.
\]
Furthermore,
\[
|\E_P[g^P b(X)(S_j^\Delta-s_{j,k})]|
\le\|b\|_\infty\|g^P\|_2\|S_j^\Delta-s_{j,k}\|_2\longrightarrow0.
\]
Testing against bounded $b$ yields $\E[g^P S_j^\Delta\mid X]=0$ for each $j\in\{Y,M_1,M_0\}$. Therefore
\[
\Var_P(f_w^P\mid X)
=\Var_P(F_R^{\Delta,P}\mid X)
 +\E_P[(g^P)^2\mid X],
\]
which proves~\eqref{eq:nie-joint-pythagoras} and its equality criterion. The verified map $F_R^\Delta$ has zero excess and finite variance under the joint score assumptions, so it attains the infimum.

For the separate-component comparison, fix a covariate value at which $V_1+V_0<\infty$. At the truth,
\[
D_1=\frac{\1\{A=1\}}e(Y-b_1),\qquad
D_0=\frac{\1\{A=1\}}e r(Y-\mu)
 +\frac{\1\{A=0\}}{1-e}(\mu-b_0).
\]
The treatment indicators have disjoint supports and $\E[Y-\mu\mid X,A=1,M]=0$, so
\[
\begin{aligned}
\Cov(D_1,D_0\mid X)
&=\frac1e\int r\,\E[(Y-b_1)(Y-\mu)\mid X,A=1,M=m]\,\mathrm dG_1(m)\\
&=\frac1e\int r\sigma_Y^2\,\mathrm dG_1
=\frac1e\int\sigma_Y^2\,\mathrm dG_0\ge0.
\end{aligned}
\]
It follows that $0\le\Var(D_1-D_0\mid X)=V_1+V_0-2\Cov(D_1,D_0\mid X)\le V_1+V_0$, giving~\eqref{eq:nie-joint-ratio-range} when the denominator is positive. With $G_0=G_1$, constant regression and positive bounded outcome noise attain ratio zero; nonconstant regression and deterministic outcome attain ratio one.
\end{proof}

\begin{remark}[Recentering, tangent families, and sharp endpoints]
\label{rem:nie-joint-boundaries}
The recentering in~\eqref{eq:nie-F11} is required for the $G_1$ face of condition~(B).  With $X$ constant, $e=\tfrac12$, $\bm\mu=(0,1)$, truth $G_1=(\tfrac12,\tfrac12)$, and substituted $G_1=(0.9,0.1)$, the unrecentered rule $\1\{A=1\}e^{-1}(Y-\bm\mu)+\bar\mu$ has expectation $0.1$ for a target equal to $\tfrac12$.  In~\eqref{eq:nie-F11}, the two integrated regression terms cancel on that face.

All three score families in Assumption~\ref{ass:nie-joint-tangent} are needed for the full-model identity.  If $G_1$ is instead fixed and known, an exact rule may omit its score.  At $X$ constant, $e=\tfrac12$, $G_0=G_1$, $\bm\mu=(0,1)$, and unit conditional variances, the resulting restricted-model variance is $\tfrac12$, whereas the full-model benchmark is $1$.

Finally, with $G_0=G_1$, a constant regression and positive outcome noise attain ratio $0$ in~\eqref{eq:nie-joint-ratio-range}.  A nonconstant regression with deterministic outcome attains ratio $1$; positive outcome noise can approach that endpoint.
\end{remark}

\begin{definition}[Full-effect population rule and cross-fitted estimator]
\label{def:nie-joint-estimator}
Let $F_R^{10}$ be the cross-world population slot map using the Borel selector $R(G_0,G_1)$ in~\eqref{eq:rn-selector}, let $F^{11}$ be \eqref{eq:nie-F11}, and set $F_R^\Delta=F^{11}-F_R^{10}$.  Its declared factorization is
\[
\eta_\Delta=(\bm\mu,e,G),\qquad G=(G_0,G_1).
\]
Use the fixed finite fold scheme of Theorem~\ref{thm:nie-achieve}. For one collection of off-fold learners, suppress the fold superscript inside each validation score and define
\begin{align}
\widehat b_a(x)&=\int\widehat\mu(x,m)\,
                    \mathrm d\widehat G_a(m\mid x),\qquad a=0,1,
                    \nonumber\\
\widetilde F^{11}
&=\frac{\1\{A=1\}}{\widehat e(X)}
       \{Y-\widehat b_1(X)\}+\widehat b_1(X),\nonumber\\
\widetilde F^{10}
&=\frac{\1\{A=1\}}{\widehat e(X)}
       \widetilde r(X,M)\{Y-\widehat\mu(X,M)\}
 +\frac{\1\{A=0\}}{1-\widehat e(X)}
       \{\widehat\mu(X,M)-\widehat b_0(X)\}
 +\widehat b_0(X),\nonumber\\
T_n&=\sum_{k=1}^{\mathcal K}\frac{n_k}{n}\mathbb P_{n,k}[\widetilde F_k^{11}-\widetilde F_k^{10}].
\label{eq:nie-joint-estimator}
\end{align}
Here $\widetilde r$ is the single bounded ratio learner used by the estimator, including any prespecified truncation.
\end{definition}

\begin{assumption}[Full-effect estimation rates]
\label{ass:nie-joint-rates}
Fix $P_0$ and impose (N1)--(N4) of Theorem~\ref{thm:nie-achieve} on the same folds as Definition~\ref{def:nie-joint-estimator}.  In addition, foldwise,
\begin{enumerate}[wide=0pt,label={\textup{(J\arabic*)}},itemsep=1pt]
\item
\[
\|\widehat b_1-b_1\|_{L^2(P_X)}
+\|\widehat e-e\|_{L^2(P_X)}=o_P(1);
\]
\item
\begin{equation}\label{eq:nie-joint-extra-rate}
\|\widehat e-e\|_{L^2(P_X)}
\left\|\int\widehat\mu\,
       \mathrm d(\widehat G_1-G_1)\right\|_{L^2(P_X)}
=o_P(n^{-1/2}).
\end{equation}
\end{enumerate}
\end{assumption}

\begin{lemma}[Full-effect conditional-bias identity]
\label{lem:nie-joint-product-bias}
Write $\delta_e=\widehat e-e$, $\delta_\mu=\widehat\mu-\mu$, and $r=R(G_0,G_1)$.  Conditional on the training sample and $X$,
\begin{align}
&\E_P[\widetilde F^{11}-\widetilde F^{10}\mid X]
   -\{b_1(X)-b_0(X)\}\nonumber\\
&=\frac{\delta_e}{\widehat e}
   \int\widehat\mu\,\mathrm d(\widehat G_1-G_1)
 +\frac{\delta_e}{\widehat e}
   \int\delta_\mu(1-\widetilde r)\,\mathrm dG_1\nonumber\\
&\quad
 +\frac{\delta_e}{1-\widehat e}
   \int\widehat\mu\,\mathrm d(\widehat G_0-G_0)
 +\int\delta_\mu(\widetilde r-r)\,\mathrm dG_1.
\label{eq:nie-joint-product-bias}
\end{align}
For the population slot,
\begin{equation}\label{eq:nie-F11-bias}
\E_P[F^{11}(\bm\mu^\ast,e^\ast,G^\ast)\mid X]-b_1
=\frac{e-e^\ast}{e^\ast}
 \{b_1-\bar\mu(\bm\mu^\ast,G_1^\ast)\}.
\end{equation}
Consequently, $F_R^\Delta$ satisfies condition~(B) on the three declared one-wrong faces of $(\bm\mu,e,G)$.
\end{lemma}

\begin{proof}
Let $A_a(X)=\int\widehat\mu\,\mathrm d(\widehat G_a-G_a)$ for $a=0,1$. Then $\widehat b_a-b_a=A_a+\int\delta_\mu\,\mathrm dG_a$. Conditioning on treatment gives
\[
\E_P[\widetilde F^{11}\mid X]-b_1
=\frac{\delta_e}{\widehat e}(\widehat b_1-b_1)
=\frac{\delta_e}{\widehat e}
 \left(A_1+\int\delta_\mu\,\mathrm dG_1\right).
\]
Subtract the cross-world bias in~\eqref{eq:nie-product-bias}:
\[
\begin{aligned}
&\E_P[\widetilde F^{11}-\widetilde F^{10}\mid X]-(b_1-b_0)\\
&=\frac{\delta_e}{\widehat e}A_1
 +\frac{\delta_e}{\widehat e}\int\delta_\mu(1-\widetilde r)\,\mathrm dG_1
 +\frac{\delta_e}{1-\widehat e}A_0
 +\int\delta_\mu(\widetilde r-r)\,\mathrm dG_1.
\end{aligned}
\]
This is~\eqref{eq:nie-joint-product-bias}. For population arguments, the first conditioning calculation reads
\[
\E_P[F^{11}(\bm\mu^\ast,e^\ast,G^\ast)\mid X]-b_1
=\left(\frac e{e^\ast}-1\right)
 \{b_1-\bar\mu(\bm\mu^\ast,G_1^\ast)\},
\]
which is~\eqref{eq:nie-F11-bias}. Its bias vanishes on each declared one-wrong configuration, as does that of $F_R^{10}$, proving~(B) for their difference.
\end{proof}

The bias identity uses the population change of measure $\mathrm dG_0=r\,\mathrm dG_1$ and is valid for the bounded learner $\widetilde r$ used in the estimator. Condition~(J2) controls its first term; the three products in~(N4) control the remaining terms. The nuisance declaration has three components, $(\bm\mu,e,G)$, with the two mediator laws grouped as $G=(G_0,G_1)$.

\begin{theorem}[One estimator for the full natural indirect effect]
\label{thm:nie-joint-achieve}
Under Assumption~\ref{ass:nie-joint-rates}, the estimator in Definition~\ref{def:nie-joint-estimator} satisfies
\[
T_n=\mathbb P_n[F_R^\Delta\{W;\eta(P_0)\}]
       +o_{P_0}(n^{-1/2}).
\]
Suppose, throughout a subclass $\cP_{\mathrm{ach}}$, that one fold scheme, learner family, clipping rule, and tuple of constants, all fixed independently of the law, satisfy Assumption~\ref{ass:nie-joint-rates}; that Assumption~\ref{ass:med-slot-domain} holds on the rich comparison model $\cP\supseteq\cP_{\mathrm{ach}}$ for every truth and every law-achievable working slot.  Then $T_n\in\TRL(\Delta;\cP_{\mathrm{ach}}\mid\cP)$.  If, in addition, the comparison model supplies Assumptions~\ref{ass:rich} and~\ref{ass:meas-main} and the paths of Assumption~\ref{ass:nie-joint-tangent} at each law in the learnable subclass, then, for every $P\in\cP_{\mathrm{ach}}$,
\[
\Vmin(\cdot;P)=V_\Delta(\cdot;P)\quad\text{in }L^0(P_X),
\]
where the benchmark is defined over $\IFRL$ on $\cP$, and
\begin{equation}\label{eq:nie-joint-floor}
V_\Delta(x)=\frac1{e(x)}\int
 \left[(1-r)^2\sigma_Y^2+(\bm\mu-b_1)^2\right]\,\mathrm dG_1
+\frac1{1-e(x)}\int(\bm\mu-b_0)^2\,\mathrm dG_0.
\end{equation}
\end{theorem}

\begin{proof}
For each fold, the treated-arm score satisfies the identity
\[
\widetilde F^{11}-F^{11,P_0}
=\1\{A=1\}\left(\frac1{\widehat e}-\frac1e\right)(Y-b_1)
 +\left(1-\frac{\1\{A=1\}}{\widehat e}\right)(\widehat b_1-b_1).
\]
Bounded outcomes and clipped denominators therefore imply
\[
\|\widetilde F^{11}-F^{11,P_0}\|_{L^2(P_0)}
\le K\{\|\widehat e-e\|_{L^2(P_X)}
       +\|\widehat b_1-b_1\|_{L^2(P_X)}\}=o_P(1).
\]
Together with~(N3), this proves $L^2(P_0)$ convergence of the full fitted score. The foldwise conditional-variance bound in the proof of Theorem~\ref{thm:nie-achieve} thus makes the centered empirical difference $o_P(n^{-1/2})$.

Put $A_a=\int\widehat\mu\,\mathrm d(\widehat G_a-G_a)$. Lemma~\ref{lem:nie-joint-product-bias}, Cauchy--Schwarz, conditional Jensen, and boundedness of the inverse propensities and $\widetilde r$, together with~(J2) and~(N4), give
\[
\begin{aligned}
&\left|P_0(\widetilde F^{11}-\widetilde F^{10})-\Delta(P_0)\right|\\
&\quad\le K\bigl\{
\|\delta_e\|_{L^2(P_X)}(\|A_1\|_{L^2(P_X)}+\|A_0\|_{L^2(P_X)})
+\|\delta_e\|_{L^2(P_X)}\|\delta_\mu\|_{L^2(P_X\otimes G_1)}\\
&\hspace{5.3em}
+\|\delta_\mu\|_{L^2(P_X\otimes G_1)}
 \|\widetilde r-r\|_{L^2(P_X\otimes G_1)}\bigr\}
=o_P(n^{-1/2})
\end{aligned}
\]
Summing over the fixed folds proves the expansion. The common learner specification gives~(A) at every law of the learnable subclass, while the integrable population identities give~(B) on the comparison model.

Under the joint score-path assumptions, Theorem~\ref{thm:nie-joint} identifies this represented variance with the algebraic floor. To evaluate it, the three components in~\eqref{eq:nie-joint-score-components} are pairwise conditionally orthogonal: the two mediator terms have disjoint treatment supports, and the outcome residual has conditional mean zero given $(X,A,M)$. Hence
\[
\Var_P(F_R^{\Delta,P}\mid X)
=\E_P[(S_Y^\Delta)^2+(S_{M_1}^\Delta)^2+(S_{M_0}^\Delta)^2\mid X],
\]
whose three conditional integrals are exactly~\eqref{eq:nie-joint-floor}.
\end{proof}

\begin{remark}[Roles of the uniform-class conditions]
\label{rem:nie-uniform-roles}
Condition~(U1) specifies one estimation procedure whose expansion holds at each law of the learnable subclass. Condition~(U4) fixes the population nuisance domain and its integrable envelope. Condition~(U3) supplies the model-contained score paths used by the conditional variance identity, including their preservation of the regression space and the mediator dominance bound.
\end{remark}

\begin{remark}[Why score richness cannot be dropped]
\label{rem:nie-tangent-countermodel}
Let $X$ be constant, $e=1/2$, and let $M\mid A=a\sim\operatorname{Bernoulli}(1/2)$ in both arms.  Fix $\eta_1(m)=m$ and the remaining nuisance components, and allow only the shape of $Y\mid A=1,M$ to vary at fixed conditional mean.  The target is identically $1/2$, so $T_n\equiv1/2$ has variance zero and satisfies (A) and (B).  If $Y=M+E$ with $E\in\{-1,1\}$ symmetric, however, $V^{\mathrm{TS}}=5/2$.  The equality fails because the model has no path moving $\eta_1$ or $\eta_3$.  A regression path with $\left.\partial_t\int\eta_1(P_t)(m)\,\mathrm dG_0(m)\right|_{t=0}\ne0$ makes the target vary and excludes the constant estimator.
\end{remark}

\section{The bridge to classical efficiency}\label{ssec:bridge}

This section proves Theorem~\ref{thm:bridge} and Corollary~\ref{cor:bridge-maxarm}. All expectations, closures, and inner products are at the evaluating law $P$. The statistical experiments are generated by the finite-dimensional paths in (G3) and, for the affine max-arm specialization, condition~\textup{(L)}.

\subsection{Conditional projection and the target derivative}

Let $b_j(P)(X)$ be finite conditional branches, $m_P(X)=\max_{j\le K}b_j(P)(X)$, $\Psi(P)=\E_P[m_P(X)]$, and $q=m_P-\Psi(P)\in L^2(P)$. Let $\zeta=(\zeta_1,\ldots,\zeta_K)^\top$ have components in $L^2(P)$ with $\E[\zeta_j\mid X]=0$. Write $\cA^*(X)=\argmax_j b_j(P)(X)$ and $\Sigma=\E[\zeta\zeta^\top\mid X]$. The active covariance block is $S=\Sigma_{\cA^*}$.  In Section~\ref{sec:bridge}, these same objects are denoted by $J$, $q_P$, $\bm\Gamma(X)$, and $\bm\Gamma_{\mathrm{act}}(X)$, respectively, in place of $K$, $q$, $\Sigma$, and $S$ used in this proof.

Define the closed linear space
\begin{equation}\label{eq:sbridge-H}
 \mathcal H=\{s\in L_0^2(P):
 \1\{i,j\in\cA^*(X)\}\E[(\zeta_i-\zeta_j)s\mid X]=0,
 \ i,j\le K\}.
\end{equation}
Closedness follows because each constraint is continuous from $L^2(P)$ to $L^1(P_X)$, by conditional expectation and Cauchy--Schwarz.

For a positive-semidefinite matrix $S$, a measurable minimizing unit-sum vector is
\begin{equation}\label{eq:sbridge-weights}
 \lambda^*(S)=
 \begin{cases}
 \displaystyle\frac{S^+\1}{\1^\top S^+\1},
       &\1\in\operatorname{range}(S),\\[6pt]
 \displaystyle\frac{(I-SS^+)\1}{\|(I-SS^+)\1\|^2},
       &\1\notin\operatorname{range}(S).
 \end{cases}
\end{equation}
Indeed, $S\lambda^*=c(S)\1$, where
\begin{equation}\label{eq:sbridge-c}
 c(S)=\min_{\1^\top\lambda=1}\lambda^\top S\lambda
 =\begin{cases}
 (\1^\top S^+\1)^{-1},&\1\in\operatorname{range}(S),\\
 0,&\1\notin\operatorname{range}(S).
 \end{cases}
\end{equation}
In the first case, decompose an arbitrary feasible vector as $\lambda=\lambda^*+v$, where $\1^\top v=0$; the objective is $c(S)+v^\top Sv$.  In the second, $\lambda^*$ lies in the kernel and attains zero.  The denominators in their respective cases are strictly positive.  Matrix pseudoinversion, rank, and the two cases are Borel measurable; enumerating the finitely many active sets therefore gives a measurable $\lambda^*(X)$.  Extend it by zero to inactive coordinates and put $r^*=\lambda^{*\top}\zeta$ and $D^*=q+r^*$.  Conditional optimality against any active coordinate gives
\begin{equation}\label{eq:sbridge-oracle-L2}
 \E[(r^*)^2\mid X]=c(S)
 \le\sum_{j=1}^K\E[\zeta_j^2\mid X].
\end{equation}
Thus $r^*,D^*\in L^2(P)$, even if $\lambda^*$ is unbounded.

\begin{lemma}[Projection]\label{lem:sbridge-projection}
For active unit-sum weights $\lambda$ such that $D_\lambda=q+\lambda^\top\zeta\in L^2(P)$,
\[
 D^*\in\mathcal H,\qquad
 D_\lambda-D^*\perp\mathcal H,\qquad
 \|D^*\|_2^2=\Var(m_P(X))+\E[c(S)].
\]
\end{lemma}
\begin{proof}
Conditional centering gives $\E[D^*]=0$. For active $i,j$,
\[
\E[(\zeta_i-\zeta_j)D^*\mid X]
=(S\lambda^*)_i-(S\lambda^*)_j=0,
\]
so $D^*\in\mathcal H$. Let $v=\lambda-\lambda^*$ and $j_0=\min\cA^*$. Since $\sum_jv_j=0$,
\[
D_\lambda-D^*=\sum_{j\in\cA^*}v_j(\zeta_j-\zeta_{j_0}).
\]
For $A_m=\{\|v(X)\|\le m\}$ and $s\in\mathcal H$, each coefficient in the restricted sum is bounded, giving $\E[\1_{A_m}(D_\lambda-D^*)s]=0$. Moreover,
\[
\left|\E[\1_{A_m^c}(D_\lambda-D^*)s]\right|
\le\|\1_{A_m^c}(D_\lambda-D^*)\|_2\|s\|_2\longrightarrow0.
\]
Therefore $D_\lambda-D^*\perp\mathcal H$. Finally $\E[r^*\mid X]=0$ and~\eqref{eq:sbridge-oracle-L2} yield
$\|D^*\|_2^2=\E[q^2]+\E[(r^*)^2]=\Var(m_P(X))+\E[c(S)]$.
\end{proof}

Suppose a bounded mean-zero score $s$ defines a feasible path $dP_t=(1+ts)dP$ and, for every $j$,
\begin{equation}\label{eq:sbridge-expansion}
 b_j(P_t)=b_j(P)+t\dot b_j(s)+o_{L^1(P_X)}(|t|),
 \qquad \dot b_j(s)=\E[\zeta_js\mid X].
\end{equation}
For $s\in\mathcal H$, put $\dot b_*(s)=\dot b_{j_0}(s)$ with $j_0=\min\cA^*$, and let
\[
M_t=\max_j\{b_j(P)+t\dot b_j(s)\},\qquad
H_s=\max_j|\dot b_j(s)|.
\]
The active derivatives agree and $H_s\in L^1(P_X)$. Define $\delta_+(X)=\min_{j\notin\cA^*}\{m_P(X)-b_j(P)(X)\}$, with an empty minimum equal to infinity. Finite $K$ gives $\delta_+>0$ pointwise. For either sign of $t$,
\[
0\le M_t-m_P-t\dot b_*(s)
\le2|t|H_s\1\{\delta_+\le2|t|H_s\}.
\]
Dominated convergence makes the right side $o_{L^1}(|t|)$. Since the maximum is Lipschitz in its finite vector of arguments, the remainders in~\eqref{eq:sbridge-expansion} give $m_{P_t}=m_P+t\dot b_*(s)+o_{L^1}(|t|)$. Moreover $dP_{t,X}/dP_X=1+t\E[s\mid X]$ exactly, so
\begin{align}
\Psi(P_t)-\Psi(P)
&=\E[m_{P_t}-m_P]+t\E[m_P\E[s\mid X]]\notag\\*
&\quad+t\E[(m_{P_t}-m_P)\E[s\mid X]]\notag\\*
&=t\E[D_\lambda s]+o(|t|)
=t\E[D^*s]+o(|t|).
\label{eq:sbridge-derivative}
\end{align}
The last cross term is $O(t^2)+o(t^2)$ because $s$ is bounded. Any square-integrable unit-sum active combination $D_\lambda$ represents the common active derivative by conditional expectation and coefficient localization as in Lemma~\ref{lem:sbridge-projection}.

For a feasible score outside $\mathcal H$, the same Lipschitz and dominated-convergence argument gives right and left derivatives
\[
\dot\Psi_+(s)=\E[qs]+\E[\max_{j\in\cA^*}\dot b_j(s)],\qquad
\dot\Psi_-(s)=\E[qs]+\E[\min_{j\in\cA^*}\dot b_j(s)].
\]
Their difference is positive because the active derivatives disagree on a set of positive probability. This proves the derivative characterization, without a margin-rate assumption.

Let $\mathcal S$ and $\mathcal T$ be the score space and its closure specified in (G3). In a finite-dimensional experiment with score vector $s=(s_1,\ldots,s_d)^\top$, put $h_j=(\dot b_j(s_l))_{l=1}^d$ and $H=\max_j\|h_j\|$. The preceding maximum remainder is bounded by
\[
2\|t\|H\1\{\delta_+\le2\|t\|H\}=o_{L^1}(\|t\|).
\]
Thus $\Psi(P_t)=\Psi(P)+t^\top v+o(\|t\|)$ with $v=\E[sD^*]$. Boundedness of $s$ also gives the quadratic-mean expansion
\[
\left\|\sqrt{1+t^\top s}-1-\tfrac12t^\top s\right\|_2=O(\|t\|^2),
\]
and the local log-likelihood ratio, for fixed $h\in\R^d$, is
\[
\log\frac{dP_{h/\sqrt n}^{\otimes n}}{dP^{\otimes n}}
=h^\top n^{-1/2}\sum_{i=1}^ns(O_i)
-\tfrac12h^\top I h+o_P(1),\qquad I=\E[ss^\top].
\]
The finite-experiment convolution variance is $v^\top I^+v=\|\Pi_{\operatorname{span}(s_1,\ldots,s_d)}D^*\|_2^2$. Singular score vectors are reduced to a linearly independent basis. Along increasing finite spans dense in $\mathcal T$, these variances increase to $\|\Pi_{\mathcal T}D^*\|_2^2$, giving the convolution factor for an estimator regular in all the specified experiments by the semiparametric convolution theorem \citep{BKRW1993}. Hence the canonical gradient is $\Pi_{\mathcal T}D^*$, and equals $D^*$ when $D^*\in\mathcal T$. This proves Theorem~\ref{thm:bridge}.

\subsection{Bounded approximation with conditional constraints}

The following lemma supplies the density statement without bounded outcomes, a uniform eigenvalue lower bound, or a bound on oracle coefficients.

\begin{lemma}[Conditional Gram truncation]\label{lem:sbridge-density}
Let $Z$ be a finite-dimensional random vector with $\E[\|Z\|^2\mid X]<\infty$ almost surely.  If $r\in L^2(P)$ and $\E[Zr\mid X]=0$, there are bounded $r_l$ such that $\E[Zr_l\mid X]=0$ and $r_l\to r$ in $L^2(P)$. The conditional equalities are understood after localization when $Zr$ is not globally integrable.
\end{lemma}
\begin{proof}
Define
\[
 C=\E[ZZ^\top\mid X],\quad
 Z_j=Z\1\{\|Z\|\le j\},\quad
 C_j=\E[ZZ^\top\1\{\|Z\|\le j\}\mid X],
\]
and the increasing measurable sets
\[
 E_j=\{\E[\|Z\|^2\mid X]\le j,
       \ \operatorname{rank}(C_j)=\operatorname{rank}(C),
       \ \|C_j^+\|_{\mathrm{op}}\le j\}.
\]
For almost every $x$, $C_j(x)$ increases to $C(x)$ in Loewner order.  Its ranges are nested, and their union spans the range of $C(x)$.  Finite dimension implies eventual equality of the ranks; on that common range the nonzero eigenvalues then converge to positive eigenvalues of $C(x)$.  Thus $E_j\uparrow\cX$ modulo a $P_X$-null set.  Monotonicity of the inverse on the common range also shows that these sets are increasing.

Let $t_n=\max(-n,\min(r,n))$ and $m_n=\E[Zt_n\mid X]$. On $E_j$, $m_n$ lies in $\operatorname{range}(C)=\operatorname{range}(C_j)$, since every vector in the conditional kernel of $C$ is orthogonal to $Z$ conditionally almost surely.  Set
\[
 r_{j,n}=\1_{E_j}\{t_n-Z_j^\top C_j^+m_n\}.
\]
Then $\E[Zr_{j,n}\mid X]=0$ exactly.  Furthermore, $\|m_n\|\le n\sqrt j$ on $E_j$, so $|r_{j,n}|\le n+n j^{5/2}$.  Since $m_n=\E[Z(t_n-r)\mid X]$, the correction has conditional second moment
\[
 \E[(Z_j^\top C_j^+m_n)^2\mid X]
 =m_n^\top C_j^+m_n
 \le j^2\E[(t_n-r)^2\mid X]\quad\text{on }E_j.
\]
It follows that
\[
 \|r_{j,n}-\1_{E_j}r\|_2\le(1+j)\|t_n-r\|_2.
\]
For each fixed $j$, choose $n(j)$ so that the right side is at most $1/j$.  Then $r_{j,n(j)}\to r$ by dominated convergence on the exhausting sets $E_j$.
\end{proof}

Apply the lemma with $Z$ consisting of $1$ and the finitely many vectors $\1\{i,j\in\cA^*(X)\}(\zeta_i-\zeta_j)$. For $s\in\mathcal H$, write $s=a(X)+r$, where $a=\E[s\mid X]$ and $\E[r\mid X]=0$.  Conditional centering of the branch gradients implies that $r$ satisfies all these constraints.  The lemma approximates $r$ by bounded, conditionally centered constrained residuals; centered clipping approximates $a$ by bounded mean-zero functions of $X$.  Their sums show
\begin{equation}\label{eq:sbridge-H-density}
 \overline{\mathcal H\cap L^\infty(P)}^{\,L^2(P)}=\mathcal H.
\end{equation}
If the propensity of a finite-valued variable $A$ is fixed, include the arm indicators $\1\{A=a\}$ in $Z$ instead of $1$.  The same proof yields density in
\begin{equation}\label{eq:sbridge-H0}
 \mathcal H_0=\{s\in\mathcal H:
                   \E[s\mid X,A]=\E[s\mid X]\}.
\end{equation}
Density is a property of the conditional constraints.  To identify a statistical tangent space, the approximating paths must also belong to the model.

\subsection{Exact substitution identities and restricted regularity}

Fix a full-slot map $f\in\IFRL$, an admissible auxiliary section, and its truth evaluation $f_P\in L^2(P)$; put $D_f=f_P-\Psi(P)$.  Let $\mathcal S$ and $\mathcal T$ describe the local experiment above.  Suppose a collection $\mathcal S_B$ of scores in $\mathcal T\cap L^\infty(P)$ has dense linear span in $\mathcal T$, and each of its paths $dP_t=(1+ts)dP$ is feasible, preserves $\kappa$ exactly, changes at most one nuisance component from its value at $P$, and has the branch expansion \eqref{eq:sbridge-expansion}.  The auxiliary section is held fixed at its evaluating-law choice.  By (B),
\[
 \E_{P_t}[f_P]=\Psi(P_t).
\]
The left side is exactly $\E_P[f_P]+t\E_P[f_Ps]$, so \eqref{eq:sbridge-derivative} gives
\begin{equation}\label{eq:sbridge-B-gradient}
 \E[D_fs]=\E[D^*s]\quad(s\in\mathcal S_B).
\end{equation}
Continuity extends this identity to $\mathcal T$.  Consequently, if an estimator has the expansion
\[
 \sqrt n\{\widehat\Psi_n-\Psi(P)\}
 =n^{-1/2}\sum_{i=1}^nD_f(O_i)+o_P(1),
\]
it is regular in the stated bounded-tilt experiments.  To see this directly, joint central limit convergence of $D_f$ and any finite score vector, local asymptotic normality, and Le Cam's third lemma give a shift $\E[D_fs^\top]h$ under $P_{h/\sqrt n}^{\otimes n}$.  Contiguity transfers the remainder. The target shifts by $\E[D^*s^\top]h$, which is the same by \eqref{eq:sbridge-B-gradient}; the centered limit is therefore $N(0,\|D_f\|_2^2)$, independent of $h$.  This proves regularity under the stated dense one-component score condition.

\subsection{Exact paths and conditional sharpness for affine max-arm models}

Consider $b_a(P)(X)=g'_a(X)u_{P,a}(X)+\rho(X,a)$ with $g'_a(X)\ne0$, $\pi_a^P(X)>0$, and finite positive active conditional variances $\sigma_a^2(X;P)$.  In this subsection assume $q\in L^2(P)$ and
\[
 \zeta_a(O)=\frac{\1\{A=a\}g'_a(X)}{\pi_a^P(X)}
       \{h_a(\widetilde W)-u_{P,a}(X)\}\in L^2(P)
 \quad(a\in\cA).
\]
The active covariance matrix is diagonal, with entries $(g'_a)^2\sigma_a^2/\pi_a^P$.  Therefore
\begin{equation}\label{eq:sbridge-harmonic}
 c(X)=\left\{\sum_{a\in\cA^*(X)}
 \frac{\pi_a^P(X)}{(g'_a(X))^2\sigma_a^2(X;P)}\right\}^{-1},
 \qquad \E[(r^*)^2\mid X]=c(X).
\end{equation}

Here is the precise local closure sufficient for the corollary. Call $s=a(X)+r(O)$ admissible if $a$ is bounded and mean zero, $r$ is bounded, and the following three conditions hold; require that for every finite family $s_1,\ldots,s_d$ of admissible scores the model contains $dP_t=(1+\sum_{l=1}^dt_ls_l)dP$ for all $t$ in a neighbourhood of $0\in\R^d$.  The conditions are:
\begin{enumerate}[wide=0pt,label={(\roman*)},itemsep=2pt]
\item $\E[r\mid X,A]=0$;
\item $d_b(X)=\E[\zeta_b r\mid X]$ has a common value on $b\in\cA^*(X)$;
\item $r$ vanishes outside a covariate set on which $\max_b\E[|\zeta_b|\mid X]$ is bounded and the smallest strictly positive inactive gap $\min_{b\notin\cA^*(X)}\{m_P(X)-b_b(P)(X)\}$ is bounded away from zero.  An empty minimum is $+\infty$.
\end{enumerate}
This is local closure for the explicitly described bounded changes, not closure under all measurable kernel selections.  It holds in the unrestricted conditional-law model with the relevant integrability requirements: density ratios are uniformly bounded above and below, and the conditional branch changes below are bounded on their support.  Binding structural restrictions require separate verification.

For such a tilt, the covariate density ratio is $1+ta(X)$ and the propensity is unchanged exactly.  Direct conditional integration gives
\begin{equation}\label{eq:sbridge-exact-affine}
 b_b(P_t)(X)=b_b(P)(X)+\frac{t\,d_b(X)}{1+ta(X)}.
\end{equation}
On the support of $r$, choose finite $M$ and positive $\delta$ such that $\max_b|d_b|\le M$ and every inactive gap is at least $\delta$. Then every gap remains positive whenever
\[
\frac{2|t|M}{1-|t|\|a\|_\infty}<\delta,
\qquad |t|\|a\|_\infty<1.
\]
Equality among active branches is preserved by their common $d_b$; outside this support the branches are unchanged. The active set and propensity therefore remain fixed, and only the mean nuisance component changes. Moreover,
\[
\|b_b(P_t)-b_b(P)-td_b\|_{L^1(P_X)}
\le\frac{t^2\|a\|_\infty\|d_b\|_1}{1-|t|\|a\|_\infty}=O(t^2),
\]
which proves the required branch expansion. If $r=0$, all branch changes vanish and the same conclusions hold directly.

Let $\mathcal S_0$ be the linear space of admissible scores, including $r=0$; it is a linear space because (i)--(iii) are preserved under linear combinations, the support of a combination being the finite union of the supports, with the minimum gap and maximum moment bound.  Finite-dimensional tilts by members of $\mathcal S_0$ are model-feasible on a neighbourhood of the origin by the closure premise, and the exact branch formula above gives their expansion with remainder $o(\|t\|)$.  By construction $\mathcal S_0\subset\mathcal H_0$. To prove density, apply Lemma~\ref{lem:sbridge-density} to the residual of any $s\in\mathcal H_0$, with arm indicators and active contrasts in $Z$.  After its bounded approximation, multiply by the indicators of
\[
 G_j=\{\max_b\E[|\zeta_b|\mid X]\le j,\quad
       \min_{b\notin\cA^*(X)}(m_P-b_b(P))\ge1/j\}.
\]
These sets increase to full measure, preserve every conditional constraint, and give the support property in (iii).  A diagonal choice preserves $L^2$ convergence.  Centered clipping treats the covariate part.  Hence
\begin{equation}\label{eq:sbridge-exact-density}
 \overline{\mathcal S_0}^{\,L^2(P)}=\mathcal H_0.
\end{equation}
In particular $D^*\in\mathcal H_0$ and is in the closure of exactly active-set-preserving, one-component path scores. The projection and regularity assertions of the affine max-arm corollary now follow from the preceding subsections.  No propensity-replacement or inactive-arm-reachability assumption is used.

For completeness, the same paths give the conditional lower bound for every $L^1$ truth evaluation, even without global finite variance. Marginal reweighting and (B) imply
\begin{equation}\label{eq:sbridge-conditional-mean}
 \E[f_P\mid X]=m_P(X).
\end{equation}
Let $v_f(X)=\E[(f_P-m_P(X))^2\mid X]\in[0,\infty]$ and $F_m=\{v_f\le m\}$.  For an arbitrary bounded measurable $\phi(X)$, apply the residual approximation above to $\phi(X)\1_{F_m}r^*$ and multiply the approximants by $\1_{F_m}$.  Call the resulting residual-only scores $r_l$. They obey $r_l\to\phi\1_{F_m}r^*$ in $L^2(P)$, keep the propensity and active set exactly fixed, and change only the mean slot.  From (B) and \eqref{eq:sbridge-derivative},
\[
 \E[(f_P-\Psi(P))r_l]=\E[D^*r_l].
\]
Conditional centering gives $\E[q(X)r_l]=0$. Since $\|\1_{F_m}(f_P-m_P)\|_2\le\sqrt m$,
\[
\left|\E[(f_P-m_P)\{r_l-\phi\1_{F_m}r^*\}]\right|
\le\sqrt m\,\|r_l-\phi\1_{F_m}r^*\|_2\longrightarrow0.
\]
The right side of the score identity converges by $r^*\in L^2(P)$. Thus
\[
 \E[\phi(X)\1_{F_m}(f_P-m_P(X))r^*]
       =\E[\phi(X)\1_{F_m}(r^*)^2].
\]
As $\phi$ was arbitrary, on $F_m$,
\[
 \E[(f_P-m_P(X))r^*\mid X]=c(X).
\]
Conditional Cauchy--Schwarz now gives $\Var(f_P\mid X)c(X)\ge c(X)^2$; where $c(X)>0$, $\Var(f_P\mid X)\ge c(X)$.  Exhausting over $m$ proves this where $v_f$ is finite; where it is infinite the inequality is automatic.  A zero conditional oracle variance, when allowed in a singular extension, needs only nonnegativity. Together with \eqref{eq:sbridge-conditional-mean},
\[
 \Var(f_P)\ge\Var(m_P(X))+\E[c(X)].
\]
If one globally integrable full-slot member realizes $f_P=m_P+r^*$, it attains both inequalities.  Thus the essential infimum of the members' conditional variances equals $c(X)$, and the overall variational infimum equals the classical bound of this local experiment.  This proof does not first remove ancillary between-arm intercepts or inactive-arm residual terms. Full-slot realization is a separate requirement from the evaluating-law $L^2$ property of $D^*$.

\subsection{Nuisance-dependent branches and the scope of the bridge}

The gradients in the bridge are conditional derivatives at $P$; they need not be evaluations of fixed, nuisance-free identifying scores.  For example, let $A$ be a finite instrument, $C$ an observed finite cell, $\pi_z=P(A=z\mid X)$, and $p_{zc}=P(C=c\mid X,A=z)$.  For a branch
\[
 b_j(P)=\alpha_j+\sum_{z,c}B_{jzc}p_{zc},\qquad
 M_{jz}=\sum_cB_{jzc}p_{zc},
\]
the centered truth DR signal is
\begin{equation}\label{eq:sbridge-BP-gradient}
 \zeta_j=
 \sum_z\frac{\1\{A=z\}}{\pi_z(X)}
      \{B_{jzC}(X)-M_{jz}(X)\}.
\end{equation}
For a bounded tilt with score $s$, conditional integration gives
\[
 p_{zc}(P_t)-p_{zc}(P)
 =\frac{t\,\E[(\1\{C=c\}-p_{zc})s\mid X,A=z]}
        {1+t\E[s\mid X,A=z]}.
\]
The bounded denominators imply \eqref{eq:sbridge-expansion} when the branch gradients are in $L^2(P)$, for example under bounded branch coefficients and uniform instrument overlap. The branch-gradient projection therefore applies to conditional Balke--Pearl branches. In the unrestricted observed-law model, sufficiently small bounded tilts are feasible. Restrictions on the cell model determine which of these paths remain feasible and hence the tangent space used for the classical bound. The BP finite-substitution and realization arguments identify its exact variational floor; the projection describes the corresponding feasible local experiment.

Finally, the score-span condition cannot be omitted from an assertion about arbitrary declarations.  In the normal location model $Y\sim N(\theta,1)$, declare both nuisance components to be $\theta$ and take $\kappa$ constant.  One-component compatible comparisons then force the same $\theta$, so (B) only requires truth unbiasedness.  The full-slot map $f(y,e_1,e_2)=y\1\{e_1\ne0\}$ satisfies that requirement.  The Hodges estimator $\overline Y_n\1\{|\overline Y_n|>n^{-1/4}\}$ has the associated asymptotic linear expansion at every fixed $\theta$, but is not regular at zero.  This explains why the bridge uses specific feasible score spaces and a nonvacuous substitution condition.  It also explains why exact finite-substitution robustness can have an additional variance cost in structurally constrained persistent-tie models: their finite identities need not reduce to the first-order constraints of this section.

\section{Numerical verifications}\label{ssec:num-checks}

This section provides the analytical benchmarks and full numerical results for the comparison with the Whitehouse estimator in Section~\ref{sec:numerics}. Oracle calculations isolate the variance predicted by Corollary~\ref{cor:whitehouse-otr-dominance}, and paired cross-fitted implementations evaluate the resulting gain in finite samples. The additional equal-weight estimator isolates the contribution of adaptive weighting, while the Xu--Guo procedure supplies an optimality control. We report variance, bias, and coverage together.

For clarity, all Monte Carlo summaries use the following definitions.  For a method $m$, let $\widehat\psi_{m,r}$ and $\widehat{\rm se}_{m,r}$ denote the estimate and estimated standard error in replication $r=1,\ldots,R$, let $\overline\psi_m=R^{-1}\sum_r\widehat\psi_{m,r}$, and set
\[
 s_m^2=\frac{1}{R-1}\sum_{r=1}^R
   (\widehat\psi_{m,r}-\overline\psi_m)^2.
\]
The reported scaled variance and scaled bias are $ns_m^2$ and $\sqrt n(\overline\psi_m-\psi)$, respectively.  Empirical standard error is $s_m$, mean standard error is $R^{-1}\sum_r\widehat{\rm se}_{m,r}$, and their calibration ratio is the latter divided by $s_m$.  Coverage and mean interval length are
\[
 \frac1R\sum_{r=1}^R
 \1\{|\widehat\psi_{m,r}-\psi|\leq z_{.975}
       \widehat{\rm se}_{m,r}\},
 \qquad
 \frac{2z_{.975}}R\sum_{r=1}^R\widehat{\rm se}_{m,r}.
\]
Writing $\widehat\mu_{4,m}=R^{-1}\sum_r (\widehat\psi_{m,r}-\overline\psi_m)^4$, the displayed Monte Carlo standard error for $ns_m^2$ is the finite-fourth-moment plug-in value
\[
 n\left[\frac{\widehat\mu_{4,m}
 -(R-3)s_m^4/(R-1)}{R}\right]_+^{1/2}.
\]
A variance ratio is $s_m^2/s_{\rm R}^2$. Its percentile interval resamples the matched replication index jointly for method $m$ and its RALU reference and quantifies Monte Carlo uncertainty in the estimated ratio.

\subsection{Binary OTR: analytical benchmarks and implementation}
\label{ssec:otr-efficiency-details}

For the binary design in Section~\ref{sec:numerics}, write $v_a(x)=\sigma_a^2(x)/\pi_a(x)$.  At a tie, the RALU-optimal arm-one weight and residual variance are
\[
\lambda_1^\ast(x)
=\frac{\pi_1(x)/\sigma_1^2(x)}
 {\pi_1(x)/\sigma_1^2(x)+\pi_0(x)/\sigma_0^2(x)},
\qquad
V_{\rm R}(x)=\frac{v_0(x)v_1(x)}{v_0(x)+v_1(x)}.
\]
The equal-face weight is $1/2$, with variance $V_{\rm eq}(x)=\{v_0(x)+v_1(x)\}/4$.  Direct subtraction gives
\begin{equation}\label{eq:supp-binary-equal-gap}
V_{\rm eq}(x)-V_{\rm R}(x)
=\frac{\{v_0(x)-v_1(x)\}^2}{4\{v_0(x)+v_1(x)\}}\ge0.
\end{equation}
This is the binary specialization of Corollary~\ref{cor:whitehouse-otr-dominance}.  It also shows why a strict gap requires positive tie probability and unequal values of $\sigma_a^2/\pi_a$ on that tie set.

The exact integrated values are as follows.

\begin{center}
\begin{tabular}{lrrr}
\toprule
design & RALU variance & equal-face variance & equal/RALU\\
\midrule
unequal information, $q=0.30$ & 1.4737909583 & 1.5706767150 & 1.065739\\
unequal information, $q=0.80$ & 0.6591543777 & 0.8632477027 & 1.309629\\
information equality, $q=0.80$ & 1.2047360000 & 1.2047360000 & 1.000000\\
\bottomrule
\end{tabular}
\end{center}

The three methods are RALU, Whitehouse softmax, and the additional equal-weight estimator. Each is evaluated with true nuisance functions and with cross-fitted estimates. The oracle RALU and equal-weight scores use the true set of optimal treatments, while the corresponding feasible estimators share its tolerance estimate. The Whitehouse estimator uses the softmax value and its derivative in both implementations.

For $u=(u_0,u_1)$, the Whitehouse estimator uses
\[
w_a^\beta(u)=\frac{\exp(\beta u_a)}{\sum_b\exp(\beta u_b)},
\qquad
s_\beta(u)=\sum_a w_a^\beta(u)u_a,
\]
and
\begin{equation}\label{eq:supp-softmax-literal-gradient}
d_a^\beta(u)
=\frac{\partial s_\beta(u)}{\partial u_a}
=w_a^\beta(u)\{1+\beta[u_a-s_\beta(u)]\}.
\end{equation}
The one-step score is
\[
s_\beta\{\widehat u(X)\}
+\sum_{a=0}^1 d_a^\beta\{\widehat u(X)\}
 \frac{\1(A=a)}{\widehat\pi_a(X)}
 \{Y-\widehat u_a(X)\}.
\]
Equation~\eqref{eq:supp-softmax-literal-gradient} is the derivative in \citet[Proposition~3.2, equation~(4)]{Whitehouse2025Softmax}. Its second term is retained, including when a derivative coordinate is negative. Following their Section~5.2, equation~(11), we use margin exponent two and temperature
\[
\beta_n=1.5n^{1/6}\log\log n.
\]

The formal feasible experiment uses the correctly specified, law-specific working models
\[
\mu_a(x)=b_{a0}+b_{a1}(x-q)_+,
\qquad
\pi_1(x)=c_0+c_1x,
\qquad
\sigma_a^2(x)=s_a^2.
\]
The three core estimators share the observations, random fold vector, and cross-fitted outcome and propensity predictions. RALU and the additional equal-weight estimator also share the conditional-variance fits and tolerance active set. Fitted propensities are clipped to $[0.05,0.95]$, fitted conditional variances are bounded below by $0.02$, and the active-set tolerance is $1.2n^{-1/4}$. The learner treats $q$ as known. The comparison is made on the subclass specified in Section~\ref{sec:num-results}, where the true regressions have a common intercept and the positive-gap margin exponent is two.

We now derive the asymptotic representation of the implemented softmax estimator in these Gaussian designs. Write $S_\beta(O;\mu,\pi)$ for the score in~\eqref{eq:num-whitehouse-score}, and let $f_{\rm eq}$ be its equal-weight limit at the true nuisances. With $z=\beta(u_1-u_0)$ and $p(z)=(1+e^{-z})^{-1}$,
\[
s_\beta(u)=u_0+(u_1-u_0)p(z),\qquad
d_1^\beta(u)=p(z)+zp(z)\{1-p(z)\},\qquad d_0^\beta(u)=1-d_1^\beta(u).
\]
Both $p(z)+zp(z)\{1-p(z)\}$ and its derivative are bounded on $\R$. Thus $\sup_u\|\nabla s_\beta(u)\|\le C$ and $\sup_u\|\nabla^2s_\beta(u)\|\le C\beta$, with $C$ independent of $\beta\ge1$.

For a training fold, put $\delta_a=\widehat\mu_a-\mu_a$, $r_{\mu,n}=\max_a\|\delta_a\|_\infty$, and $r_{\pi,n}=\max_a\|\widehat\pi_a-\pi_a\|_\infty$. Below these rates denote the largest error over the fixed finite number of folds. Conditional expectation gives the exact identity
\begin{align*}
&P\{S_{\beta_n}(\widehat\mu,\widehat\pi)-S_{\beta_n}(\mu,\pi)\}\\
&\quad=\E_{P_X}\!\left[s_{\beta_n}(\widehat\mu)-s_{\beta_n}(\mu)
 -\sum_a d_a^{\beta_n}(\widehat\mu)\delta_a
 +\sum_a d_a^{\beta_n}(\widehat\mu)
       \left(1-\frac{\pi_a}{\widehat\pi_a}\right)\delta_a\right].
\end{align*}
Taylor's formula, the derivative bounds, and clipped propensities therefore yield
\begin{align*}
\left|P\{S_{\beta_n}(\widehat\mu,\widehat\pi)-S_{\beta_n}(\mu,\pi)\}\right|
 &\le C\{\beta_n r_{\mu,n}^2+r_{\mu,n}r_{\pi,n}\},\\
\|S_{\beta_n}(\widehat\mu,\widehat\pi)-S_{\beta_n}(\mu,\pi)\|_{L^2(P)}
 &\le C\{\beta_n r_{\mu,n}+r_{\pi,n}\}.
\end{align*}
The second inequality also uses the uniformly bounded conditional second moments in the displayed design. The fixed-dimensional least-squares fits give $r_{\mu,n}+r_{\pi,n}=O_P(n^{-1/2})$.

Let $\Delta=|\mu_1(X)-\mu_0(X)|$. The smoothing error is zero at a tie and is $\Delta/(1+e^{\beta\Delta})$ otherwise. On each design the density of the positive gap satisfies $g_\Delta(t)\le Ct$, so
\[
0\le\Psi(P)-P S_\beta(\mu,\pi)
=\E\!\left[\frac{\Delta}{1+e^{\beta\Delta}}\right]
\le C\int_0^\infty\frac{t^2}{1+e^{\beta t}}\,\mathrm dt
=O(\beta^{-3}).
\]
For $\beta_n=1.5n^{1/6}\log\log n$, both $\beta_n/\sqrt n\to0$ and $\sqrt n\,\beta_n^{-3}\to0$. Moreover, the bounded derivative coefficients converge to uniform weights on the optimal arms and zero on the others. Dominated convergence gives $\|S_{\beta_n}(\mu,\pi)-f_{\rm eq}\|_{L^2(P)}\to0$.

Let $\nu_k=n_k/n$, and let $\mathbb P_{n,k}$ denote the average over held-out fold $k$. With $\widehat S_k=S_{\beta_n}(\widehat\mu^{(-k)},\widehat\pi^{(-k)})$ and $S_n=S_{\beta_n}(\mu,\pi)$, the estimator satisfies
\begin{align*}
T_{n,\rm W}-\Psi(P)
&=(\mathbb P_n-P)f_{\rm eq}
 +\sum_k\nu_k(\mathbb P_{n,k}-P)(\widehat S_k-f_{\rm eq})\\
&\quad+\sum_k\nu_k P(\widehat S_k-S_n)+P S_n-\Psi(P).
\end{align*}
Conditional on the training observations for fold $k$, the variance of its centered empirical term is at most $\|\widehat S_k-f_{\rm eq}\|_{L^2(P)}^2/n_k=o_P(n_k^{-1})$. The bias bounds above control the last two terms. Conditional Chebyshev's inequality and the fixed number of folds thus give $T_{n,\rm W}-\Psi(P)=(\mathbb P_n-P)f_{\rm eq}+o_P(n^{-1/2})$, establishing the equal-weight asymptotic variance used in the comparison.

A separate sensitivity analysis uses equal-width histogram regression with $\lceil n^{1/3}\rceil$ bins.  The correctly specified, oracle, and histogram panels use $n=4{,}000,16{,}000,64{,}000$ with $2{,}000,2{,}000,1{,}000$ replications, respectively.  A paired bootstrap resamples replication indices jointly for each comparator and its RALU reference.

\subsection{Binary OTR: oracle variance comparison}
\label{ssec:otr-oracle-results}

The next table evaluates all three methods with the true nuisance functions. The equal-weight score gives the Whitehouse asymptotic variance in Corollary~\ref{cor:whitehouse-otr-dominance}; the Whitehouse rows use its implemented smoothing sequence. Comparing these rows also shows the finite-sample contribution of smoothing near treatment ties.

\begin{center}
\scriptsize
\setlength{\tabcolsep}{2.5pt}
\begin{tabular}{cccrrrc}
\toprule
$q$ & $n$ & score & $n\Var$ (MCSE) & $\sqrt n$ bias & coverage
& variance/RALU $[\text{paired }95\%\ \mathrm{MC}]$\\
\midrule
0.30 & 4,000 & RALU & 1.371 (0.045) & $0.022$ & 0.953 & 1.000\\
     &       & equal face & 1.469 (0.048) & $0.030$ & 0.957
     & 1.071 $[1.045,1.096]$\\
     &       & Whitehouse & 1.508 (0.049) & $-0.083$ & 0.959
     & 1.099 $[1.075,1.125]$\\
0.30 & 16,000 & RALU & 1.460 (0.046) & $0.008$ & 0.951 & 1.000\\
     &        & equal face & 1.581 (0.050) & $0.015$ & 0.950
     & 1.083 $[1.059,1.107]$\\
     &        & Whitehouse & 1.603 (0.051) & $-0.076$ & 0.949
     & 1.098 $[1.074,1.122]$\\
0.30 & 64,000 & RALU & 1.419 (0.060) & $0.046$ & 0.955 & 1.000\\
     &        & equal face & 1.508 (0.064) & $0.037$ & 0.958
     & 1.063 $[1.031,1.098]$\\
     &        & Whitehouse & 1.522 (0.064) & $-0.040$ & 0.955
     & 1.072 $[1.040,1.106]$\\
\midrule
0.80 & 4,000 & RALU & 0.675 (0.021) & $0.015$ & 0.948 & 1.000\\
     &       & equal face & 0.881 (0.029) & $0.039$ & 0.950
     & 1.304 $[1.249,1.359]$\\
     &       & Whitehouse & 0.916 (0.030) & $-0.176$ & 0.950
     & 1.357 $[1.303,1.416]$\\
0.80 & 16,000 & RALU & 0.668 (0.021) & $-0.002$ & 0.941 & 1.000\\
     &        & equal face & 0.859 (0.027) & $-0.005$ & 0.950
     & 1.286 $[1.231,1.343]$\\
     &        & Whitehouse & 0.898 (0.028) & $-0.244$ & 0.945
     & 1.344 $[1.284,1.403]$\\
0.80 & 64,000 & RALU & 0.639 (0.028) & $0.016$ & 0.948 & 1.000\\
     &        & equal face & 0.846 (0.038) & $0.042$ & 0.951
     & 1.323 $[1.251,1.402]$\\
     &        & Whitehouse & 0.878 (0.039) & $-0.196$ & 0.944
     & 1.374 $[1.301,1.454]$\\
\bottomrule
\end{tabular}
\end{center}

The equal-weight/RALU ratios track the exact benchmarks $1.066$ and $1.310$. The Whitehouse oracle variance is also larger than the RALU variance in both designs, with coverage close to $0.95$. Scaled bias is reported relative to the unsmoothed optimal-treatment value, so that the table includes the effect of smoothing on inference for the target of interest.

\subsection{Binary OTR: feasible comparison under correctly specified common nuisances}
\label{ssec:otr-feasible-results}

The feasible comparison uses the common nuisance fits described above. The Whitehouse rows report the primary comparison in Section~\ref{sec:num-results} over all three sample sizes; the equal-weight rows isolate the change due to variance-adaptive treatment weights.

\begin{center}
\scriptsize
\setlength{\tabcolsep}{3.5pt}
\begin{tabular}{cclrrrl}
\toprule
design & $n$ & score & $n\Var$ (MCSE) & $\sqrt n$ bias & coverage
& variance/RALU $[\text{paired }95\%\ \mathrm{MC}]$\\
\midrule
$q=.30$ & 4,000 & RALU & 1.414 (.047) & $-.115$ & .949 & 1.000\\
& & equal face & 1.516 (.050) & $-.056$ & .950 & 1.073 $[1.052,1.095]$\\
& & Whitehouse & 1.681 (.056) & $-.225$ & .946 & 1.189 $[1.148,1.231]$\\
& 16,000 & RALU & 1.481 (.047) & $-.091$ & .950 & 1.000\\
& & equal face & 1.601 (.051) & $-.047$ & .948 & 1.081 $[1.059,1.104]$\\
& & Whitehouse & 1.679 (.053) & $-.195$ & .946 & 1.134 $[1.100,1.169]$\\
& 64,000 & RALU & 1.425 (.061) & $-.021$ & .951 & 1.000\\
& & equal face & 1.512 (.064) & $-.003$ & .952 & 1.061 $[1.030,1.094]$\\
& & Whitehouse & 1.535 (.065) & $-.121$ & .948 & 1.077 $[1.041,1.116]$\\
\midrule
$q=.80$ & 4,000 & RALU & .887 (.026) & $-.595$ & .776 & 1.000\\
& & equal face & 1.001 (.033) & $-.374$ & .884 & 1.128 $[1.083,1.173]$\\
& & Whitehouse & 1.244 (.044) & $-.429$ & .896 & 1.402 $[1.317,1.497]$\\
& 16,000 & RALU & .878 (.029) & $-.395$ & .866 & 1.000\\
& & equal face & .992 (.031) & $-.269$ & .907 & 1.129 $[1.082,1.179]$\\
& & Whitehouse & 1.021 (.034) & $-.399$ & .921 & 1.163 $[1.096,1.236]$\\
& 64,000 & RALU & .668 (.030) & $-.183$ & .939 & 1.000\\
& & equal face & .879 (.039) & $-.091$ & .939 & 1.315 $[1.245,1.390]$\\
& & Whitehouse & .929 (.042) & $-.294$ & .936 & 1.390 $[1.299,1.483]$\\
\midrule
equality & 4,000 & RALU & 1.315 (.042) & $-.433$ & .908 & 1.000\\
& & equal face & 1.317 (.042) & $-.431$ & .909 & 1.0015 $[.9980,1.0053]$\\
& & Whitehouse & 1.823 (.064) & $-.594$ & .895 & 1.387 $[1.312,1.465]$\\
& 16,000 & RALU & 1.362 (.041) & $-.325$ & .916 & 1.000\\
& & equal face & 1.364 (.041) & $-.325$ & .915 & 1.0014 $[.9994,1.0033]$\\
& & Whitehouse & 1.571 (.051) & $-.480$ & .917 & 1.153 $[1.098,1.211]$\\
& 64,000 & RALU & 1.322 (.056) & $-.112$ & .931 & 1.000\\
& & equal face & 1.323 (.056) & $-.112$ & .931 & 1.0007 $[.9997,1.0018]$\\
& & Whitehouse & 1.506 (.064) & $-.361$ & .921 & 1.139 $[1.099,1.183]$\\
\bottomrule
\end{tabular}
\end{center}

The released summary also reports standard-error calibration, interval length, Wilson coverage intervals, and paired variance gaps.  At $n=64{,}000$, the RALU standard-error calibration ratios are $1.014$, $.972$, and $.946$ in the $q=.30$, $q=.80$, and equality designs.

In the high-tie-mass design, RALU coverage increases from $.776$ at $n=4{,}000$ and $.866$ at $n=16{,}000$ to $.939$ at $n=64{,}000$. At the largest sample size, Whitehouse coverage is $.936$ and its variance is $1.390$ times the RALU variance. These results show the finite-sample approach to the variance ordering together with the improvement in interval calibration.

\subsection{Sensitivity analysis with binned nuisance learners}
\label{ssec:otr-histogram-sensitivity}

The following table replaces the correctly specified nuisance models with binned regressions and evaluates the sensitivity of variance and interval coverage to nuisance estimation.

\begin{center}
\scriptsize
\setlength{\tabcolsep}{3.5pt}
\begin{tabular}{cccrrrc}
\toprule
$q$ & $n$ & score & $n\Var$ (MCSE) & $\sqrt n$ bias & coverage
& variance/RALU $[\text{paired }95\%\ \mathrm{MC}]$\\
\midrule
0.30 & 4,000 & RALU & 1.818 (0.059) & $-0.486$ & 0.912 & 1.000\\
     &       & equal face & 1.786 (0.059) & $-0.320$ & 0.933
     & 0.982 $[0.965,1.000]$\\
     &       & Whitehouse & 2.012 (0.065) & $-0.631$ & 0.907
     & 1.107 $[1.073,1.141]$\\
0.30 & 16,000 & RALU & 1.711 (0.056) & $-0.390$ & 0.925 & 1.000\\
     &        & equal face & 1.729 (0.057) & $-0.260$ & 0.936
     & 1.011 $[0.991,1.031]$\\
     &        & Whitehouse & 1.960 (0.063) & $-0.724$ & 0.903
     & 1.145 $[1.110,1.182]$\\
0.30 & 64,000 & RALU & 1.563 (0.068) & $-0.240$ & 0.943 & 1.000\\
     &        & equal face & 1.590 (0.068) & $-0.143$ & 0.945
     & 1.017 $[0.989,1.045]$\\
     &        & Whitehouse & 1.791 (0.078) & $-0.761$ & 0.913
     & 1.146 $[1.097,1.199]$\\
\midrule
0.80 & 4,000 & RALU & 1.276 (0.039) & $-0.604$ & 0.811 & 1.000\\
     &       & equal face & 1.332 (0.043) & $-0.398$ & 0.875
     & 1.044 $[1.011,1.077]$\\
     &       & Whitehouse & 1.965 (0.060) & $-1.327$ & 0.735
     & 1.540 $[1.475,1.612]$\\
0.80 & 16,000 & RALU & 1.205 (0.038) & $-0.571$ & 0.841 & 1.000\\
     &        & equal face & 1.214 (0.038) & $-0.398$ & 0.894
     & 1.008 $[0.974,1.040]$\\
     &        & Whitehouse & 1.743 (0.056) & $-1.694$ & 0.648
     & 1.447 $[1.374,1.523]$\\
0.80 & 64,000 & RALU & 0.999 (0.049) & $-0.399$ & 0.894 & 1.000\\
     &        & equal face & 1.108 (0.053) & $-0.247$ & 0.922
     & 1.109 $[1.057,1.165]$\\
     &        & Whitehouse & 1.683 (0.082) & $-1.988$ & 0.560
     & 1.684 $[1.559,1.813]$\\
\bottomrule
\end{tabular}
\end{center}

At $q=0.30$ and $n=64{,}000$, the equal-weight/RALU ratio is $1.017$ with paired interval $[0.989,1.045]$. In the $q=0.80$ design, the binned nuisance fits yield appreciable scaled bias and undercoverage: the corresponding RALU and Whitehouse coverages are $.894$ and $.560$. These results show the sensitivity of interval calibration to nuisance estimation in the high-tie-mass design.

\subsection{Binned information-equality control}\label{ssec:otr-equality-control}

In the equality control, $\sigma_0^2/\pi_0=\sigma_1^2/\pi_1=2$ on every tie.  Hence the oracle RALU and equal-face scores are observationwise identical.  Their empirical variance ratio is exactly one at every sample size.  With binned nuisance fits and estimated conditional variances, the feasible equal-weight/RALU ratios are
\[
1.0074\ [0.9994,1.0154],\quad
0.9937\ [0.9886,0.9986],\quad
0.9987\ [0.9937,1.0038]
\]
at $n=4{,}000,16{,}000,64{,}000$.  The corresponding coverage pairs (RALU/equal face) are $0.887/0.890$, $0.894/0.896$, and $0.916/0.920$, in order of increasing $n$. The oracle identity verifies the exact equality condition.  The feasible ratios remain within $0.8\%$ of one and have no consistent direction across sample sizes.  The corresponding control with correctly specified nuisance models is reported in Section~\ref{ssec:otr-feasible-results}.

The Whitehouse estimator has oracle variance ratios $1.027$, $1.038$, and $1.032$ relative to RALU in this control. Smoothing near the boundary contributes to these finite-sample differences. As the temperature diverges, its limiting variance equals the common variance of the RALU and equal-weight scores.

\subsection{Xu and Guo positive control}\label{ssec:xuguo-positive-control}

For the homoscedastic binary model, \citet{XuGuo} set $t_0(x)=\pi_1(x)$.  The RALU optimizer satisfies
\[
t_0^\ast(x)=
\frac{\pi_1(x)/\sigma_1^2(x)}
 {\pi_1(x)/\sigma_1^2(x)+\pi_0(x)/\sigma_0^2(x)}
=\pi_1(x)
\]
when $\sigma_0^2(x)=\sigma_1^2(x)$.  Thus their estimator is a positive control for attainment of the binary RALU bound and an external optimality benchmark.  To distinguish this positive control from the information-equality control, we use the $q=0.80$ outcome and covariate law with $\sigma_0^2=\sigma_1^2=1$ and constant $\pi_1=p=0.75$.  On the tie stratum, Xu--Guo and RALU both have residual variance one, whereas equal-face weighting has residual variance
\[
\frac14\left(\frac1p+\frac1{1-p}\right)
=\frac{1}{4p(1-p)}=\frac43.
\]
The variance of the conditional target is $0.004736$.  Adding the residual contribution gives
\[
 V_{\rm R}=V_{\rm XG}=0.004736+0.8(1)+0.2(1/0.75)
 =1.0714026667,
\]
whereas
\[
 V_{\rm eq}=0.004736+0.8\left(\frac43\right)+0.2(1/0.75)
 =1.3380693333,
 \qquad \frac{V_{\rm eq}}{V_{\rm R}}=1.248895.
\]
Thus Xu--Guo/RALU equality and strict equal-face/RALU inequality are separate analytical statements.

We implement their piecewise-linear rule
\[
d_s\{X;\widehat\tau,h,t_0\}
=\operatorname{clip}\{\widehat\tau(X)/h+t_0(X),0,1\}
\]
and their data-driven bandwidth.  For each outer training fold, write $I_j=I_{j,1}\mathbin{\dot\cup}I_{j,2}$, $j\in\{1,2\}$, where the two inner halves have equal size.  The estimated approximation error for a contrast trained on $I_j$ is
\[
\widehat{\operatorname{EAE}}(\widehat\tau_{I_j})
=\frac{1}{|I_{3-j}|}
\sum_{k=1}^2\sum_{i\in I_{3-j,k}}
\left\{
\widehat\tau_{I_{3-j,3-k}}(X_i)-\widehat\tau_{I_j}(X_i)
\right\}^2.
\]
It determines the fold-specific bandwidth
\[
h_{n,I_j}=C\log(n)n^{1/4}
\left\{\frac{\log n}{Cn}\vee
\widehat{\operatorname{EAE}}(\widehat\tau_{I_j})\right\}^{3/4},
\qquad C=0.05.
\]
The implementation retains the nested split structure of their Algorithms~1 and~2.  The contrast used for smoothing is trained on the opposite outer fold, the auxiliary contrast in the EAE calculation is trained on the opposite inner half, and the nuisances in the AIPW score are trained on the complement of the evaluation inner fold.  These fits are not collapsed into a single ordinary cross-fit.  The table's mean bandwidth averages $h_{n,I_1}$ and $h_{n,I_2}$ over replications.

\begin{center}
\scriptsize
\setlength{\tabcolsep}{3.5pt}
\begin{tabular}{clrrrrl}
\toprule
$n$ & score & $n\Var$ (MCSE) & $\sqrt n$ bias & coverage & mean $h$
& variance/RALU $[\text{paired }95\%\ \mathrm{MC}]$\\
\midrule
4,000 & RALU & 1.179 (.045) & $-.224$ & .934 & & 1.000\\
& equal face & 1.532 (.056) & $-.444$ & .920 & & 1.299 $[1.246,1.351]$\\
& Whitehouse & 2.402 (.092) & $-.727$ & .875 & & 2.037 $[1.901,2.184]$\\
& Xu--Guo & 1.420 (.048) & $-.085$ & .930 & .3053 & 1.204 $[1.151,1.257]$\\
16,000 & RALU & 1.139 (.039) & $-.235$ & .930 & & 1.000\\
& equal face & 1.499 (.051) & $-.429$ & .919 & & 1.316 $[1.266,1.368]$\\
& Whitehouse & 1.830 (.065) & $-.629$ & .900 & & 1.607 $[1.507,1.716]$\\
& Xu--Guo & 1.209 (.040) & $-.066$ & .947 & .1994 & 1.062 $[1.023,1.102]$\\
64,000 & RALU & 1.022 (.044) & $-.088$ & .952 & & 1.000\\
& equal face & 1.282 (.056) & $-.191$ & .949 & & 1.254 $[1.183,1.326]$\\
& Whitehouse & 1.543 (.071) & $-.486$ & .924 & & 1.509 $[1.400,1.631]$\\
& Xu--Guo & 1.126 (.050) & $-.011$ & .951 & .1256 & 1.101 $[1.058,1.147]$\\
256,000 & RALU & 1.028 (.059) & $-.002$ & .962 & & 1.000\\
& equal face & 1.248 (.071) & $-.038$ & .960 & & 1.214 $[1.118,1.320]$\\
& Whitehouse & 1.395 (.080) & $-.314$ & .942 & & 1.358 $[1.233,1.492]$\\
& Xu--Guo & 1.107 (.067) & $+.047$ & .962 & .0772 & 1.077 $[1.022,1.138]$\\
\bottomrule
\end{tabular}
\end{center}

The selected Xu--Guo bandwidth decreases from $.3053$ at $n=4{,}000$ to $.0772$ at $n=256{,}000$. Its empirical variance ratios relative to RALU are $1.101\,[1.058,1.147]$ and $1.077\,[1.022,1.138]$ at the two largest sample sizes. The analytical identity $t_0^\ast=\lambda_1^\ast$ gives their common asymptotic variance, while the table quantifies the remaining finite-sample difference. Equal weighting has the larger limiting ratio $1.248895$.

\subsection{Three-arm all-tie control}\label{ssec:three-arm-control}

To verify that the gap is not a binary-treatment artifact, consider a constant covariate, $\cA=\{1,2,3\}$, three exactly tied arms, treatment probabilities $(0.6,0.3,0.1)$, and unit conditional variances.  The RALU weights are proportional to the propensities and the exact residual variance is
\[
V_{\rm R}=\left(\sum_{a=1}^3\pi_a\right)^{-1}=1.
\]
Equal active-arm weights give
\[
V_{\rm eq}=\frac{1}{9}\sum_{a=1}^3\frac{1}{\pi_a}=\frac53,
\qquad
\frac{V_{\rm eq}}{V_{\rm R}}=\frac53.
\]
At an exact all-arm tie, the equal-weight score is the Whitehouse limiting score. With $5{,}000$ observations and $5{,}000$ oracle replications, the empirical scaled variances are $1.0194$ for RALU and $1.6664$ for equal weighting, with coverages $0.9508$ and $0.9488$. The empirical variance ratio is $1.6347\,[1.5813,1.6917]$, consistent with the exact ratio $5/3$.

\subsection{Effect of hard tie breaking}
\label{ssec:hard-argmax-diagnostic}

Hard selection chooses one treatment even at an exact tie. The following Gaussian calculation illustrates the variance contribution of this choice in a two-fold construction, using the conditional variances from the $q=0.30$ design.

With two folds, fold one is evaluated using the arm chosen by fold two and conversely.  Let $\{Z_{r,a}:r\in\{1,2\},a\in\{0,1\}\}$ be mutually independent, with $Z_{r,a}\sim\cN\{0,v_a(x)\}$, and define $S_r=\1(Z_{r,1}>Z_{r,0})$.  The within-fold independence follows from joint Gaussianity and the zero covariance of the centered arm scores, whose supports are disjoint.  The cross-selected residual in this Gaussian experiment is
\[
H_x=\{Z_{1,S_2}+Z_{2,S_1}\}/\sqrt2.
\]
Put $v=v_0(x)+v_1(x)$. Since $Z_{r,1}-Z_{r,0}\sim\cN(0,v)$, $P(S_r=1)=1/2$, and Gaussian regression gives
\[
\E[Z_{r,0}\1\{S_r=1\}]=-\frac{v_0(x)}{\sqrt{2\pi v}},\qquad
\E[Z_{r,1}\1\{S_r=1\}]=\frac{v_1(x)}{\sqrt{2\pi v}}.
\]
The corresponding moments with $S_r=0$ have the opposite signs. Independence between folds then yields
\[
\E[Z_{1,S_2}]=0,\qquad
\E[Z_{1,S_2}^2]=\frac v2,\qquad
\E[Z_{1,S_2}Z_{2,S_1}]
=\sum_{a,b=0}^1\E[Z_{1,a}\1\{S_1=b\}]\E[Z_{2,b}\1\{S_2=a\}]
=\frac v{2\pi}.
\]
Consequently
\begin{equation}\label{eq:hard-tie-var}
\Var(H_x)=\frac{v_0(x)+v_1(x)}{2}
\left(1+\frac1\pi\right).
\end{equation}
Integrating~\eqref{eq:hard-tie-var} over the tie set and adding the strict region gives
\[
1.4737909583+0.4366931744=1.9104841327.
\]
This integrated Gaussian benchmark is $29.6\%$ above the RALU floor. The increase comes from selecting one arm at a tie and from the dependence induced by using each fold to choose the treatment evaluated in the other fold.

\subsection{The product form of the nuisance remainder}

Identity~\eqref{eq:R2-product-pointwise} states that on the event $\{\widehat\kappa=\kappa(P)\}$ the conditional bias of the weighted family is a pure product of the two componentwise errors.  We evaluate the conditional bias by simulation ($8\times10^{6}$ draws) at a two-arm law with tied active set, and separately in the strict region, in four configurations.

\begin{center}
\begin{tabular}{lcccc}
\toprule
& \multicolumn{2}{c}{tie point} & \multicolumn{2}{c}{strict region}\\
\cmidrule(lr){2-3}\cmidrule(lr){4-5}
configuration & simulated & formula & simulated & formula\\
\midrule
correct nuisances        & $\phantom{-}0.00009$ & $0$ & $0.00052$ & $0$\\
outcome model wrong only & $-0.00053$ & $0$ & $0.00029$ & $0$\\
propensity wrong only    & $-0.00014$ & $0$ & $0.00019$ & $0$\\
both wrong               & $\phantom{-}0.08152$ & $0.08158$ & $0.07745$ & $0.07800$\\
\bottomrule
\end{tabular}
\end{center}

The first three rows agree with the zero bias implied by Definition~\ref{def:ralu}(B). When both nuisance components are misspecified, the fourth row agrees with the product bias in~\eqref{eq:R2-product-pointwise}.

\subsection{Tolerance-active-set recovery}

Lemma~\ref{lem:face-recovery} predicts, on the sup-norm event, that no true active arm is dropped, that the faces disagree only on a shell of probability $O(\tau^{\gamma})$, and that the integrated face bias is $O(\tau^{1+\gamma})$.  At a law with tie mass $0.30$ and margin index $\gamma=2$:

\begin{center}
\begin{tabular}{ccccc}
\toprule
$\tau$ & false exclusion & face disagreement & ratio & $\E|\text{face bias}|$\\
\midrule
$0.200$ & $0$ & $0.02324$ & \text{n/a} & $4.01\times10^{-3}$\\
$0.100$ & $0$ & $0.00566$ & $4.1$ & $4.87\times10^{-4}$\\
$0.050$ & $0$ & $0.00142$ & $4.0$ & $6.1\times10^{-5}$\\
$0.025$ & $0$ & $0.00036$ & $3.9$ & $8\times10^{-6}$\\
\bottomrule
\end{tabular}
\end{center}

In the table, halving $\tau$ reduces the disagreement probability by approximately four and the integrated bias by approximately eight, in agreement with the powers $\tau^\gamma$ and $\tau^{1+\gamma}$ for $\gamma=2$. False exclusion is zero, as guaranteed on the sup-norm event of Lemma~\ref{lem:face-recovery}.  Let $\Delta_{\mathrm{ord}}(x)$ denote the ordinary best-to-second-best gap, defined to be zero when the best arm is tied.  On the same law, $P_X\{\Delta_{\mathrm{ord}}(X)\le t\}\ge0.30$ for every $t>0$; hence a condition $P_X\{\Delta_{\mathrm{ord}}(X)\le t\}\le Ct^\gamma$ cannot hold for any $\gamma>0$.  The positive-gap condition used here excludes the tie set from its event.

\subsection{Balke--Pearl}

The eight upper-bound expressions of~\eqref{eq:bp-U-explicit} are generated as the dual vertices described there and validated against a direct solution of the response-type linear program on $1{,}624$ feasible random laws, using the fixed invocation \texttt{bp\_lp.py --n-random 1624 --seed 0}, with maximum discrepancy below $5\times10^{-15}$.  The recorded maximum discrepancy for that invocation is $2.22\times10^{-16}$. Deterministic checks include the perfect-compliance law with $A=Z$ and $Y=A$, for which the bounds are $[1,1]$ as they must be, and two laws with an unidentified effect, for which they are $[0,1]$.  The kink and singular examples of Sections~\ref{sec:bp-numerical} and~\ref{sec:bp-degenerate} are computed from the same coefficient matrix, with the reported floors confirmed by simulation.

\section{Additional details for the Card--NLSYM application}
\label{ssec:card-details}

We analyze the National Longitudinal Survey of Young Men data introduced by \citet{Card1995}, using the processed replication files deposited by \citet{Wang2017NLSData} and the covariate-assisted Balke--Pearl setup of \citet{levis2025covariate}.  The analysis includes all $3{,}010$ observations.  The instrument is proximity to a four-year college, the exposure is education beyond high school, and the binary outcome is $\1\{\text{1976 hourly wage}>537.5\text{ cents}\}$.  The full-cohort median of $537.5$ cents was fixed when constructing the analysis outcome, before nuisance fitting and cross-fitting.  We adjust for age, parental education, southern and metropolitan residence, race, IQ, and the corresponding missingness indicators.  Survey weights are normalized to sum to the sample size.  The weighted first stage is $0.117901$, with a lower $95\%$ Wald endpoint of $0.073475$.

The primary analysis uses twenty deterministic, cell-stratified five-fold partitions.  Within each instrument level, the four conditional probabilities of $(D,Y)$ are estimated by survey-weighted ridge multinomial logistic regression; the instrument propensity is estimated by survey-weighted ridge logistic regression.  Penalties are selected by training-only three-fold cross-validation.  The full conditional covariance matrix of the eight Balke--Pearl scores is then calculated from the fitted coherent cell probabilities, retaining every off-diagonal entry.

The estimated active face uses $\tau_n=0.25n^{-1/4}=0.0337519$.  RALU and the equal-face comparator share the observations, folds, nuisance estimates, and estimated active face, and differ only in their weights on that face.  The RALU weights are produced by the stabilized selector of Section~\ref{sec:bp-bound} with $\ell^1$ cap $L=5$ and guard fallback.

In the primary weighted analysis the cap activation fraction and the projection-or-fallback fraction are zero on all $20$ splits at both endpoints, so the primary weights are the unconstrained covariance-optimal weights. The largest recorded lower-endpoint $\ell^1$ norm is $3.638015$.  Over the full tolerance grid of the weighted analysis the cap affects at most $0.39458\%$ of the weighted observations, and in the boosting companion's primary tolerance at most $0.15121\%$; these quantities are recorded in the archived tolerance summaries.

The survey normalization has the usual H\'ajek ratio interpretation. For a frozen complete score $F$, let $P_s$ denote the sampling law of independent records and let $w$ be the record's survey weight.  If $0<\E_{P_s}w<\infty$ and $\E_{P_s}|wF|<\infty$, the ratio target and its influence representation are
\[
\theta_w(F)=\frac{\E_{P_s}[wF]}{\E_{P_s}w},
\qquad D_{w,F}=\frac{w\{F-\theta_w(F)\}}{\E_{P_s}w};
\]
when in addition $\E_{P_s}[D_{w,F}^2]<\infty$, the ratio estimator is asymptotically linear with influence function $D_{w,F}$. The variance comparison concerns the complete fitted scores with the same survey normalization for both methods.  The archived quantities use the sample-size scale $\widehat V=n\widehat{\mathrm{SE}}^{\,2}$ with $n=3010$; the normalization therefore cancels in their variance ratio.  When $w$ is included in the conditioning variables, its conditional $w^2$ factor is common to all active-face weights; the conditional optimization is then unchanged by that factor.

Across the twenty partitions, the componentwise median estimates and estimated standard errors are:
\begin{center}
\small
\begin{tabular}{@{}lrrrrrr@{}}
\toprule
endpoint
& \multicolumn{2}{c}{RALU}
& \multicolumn{2}{c}{equal face}
& \multicolumn{2}{c}{hard branch}\\
\cmidrule(lr){2-3}\cmidrule(lr){4-5}\cmidrule(lr){6-7}
& estimate & SE & estimate & SE & estimate & SE\\
\midrule
lower & $-0.396026$ & $0.020253$ & $-0.403418$ & $0.021676$
      & $-0.400673$ & $0.023454$\\
upper & $\phantom{-}0.523302$ & $0.021296$
      & $\phantom{-}0.528294$ & $0.023321$
      & $\phantom{-}0.525698$ & $0.025716$\\
\bottomrule
\end{tabular}
\end{center}
The median complete-score variance ratios $\widehat V_{\rm eq}/\widehat V_{\rm R}$ are $1.148601$ for the lower endpoint and $1.197891$ for the upper endpoint.  Equivalently, RALU reduces the estimated variance by $12.94\%$ and $16.52\%$, respectively.  All twenty partitions favor RALU for both endpoints. The splitwise reductions range from $7.75\%$ to $22.16\%$ for the lower endpoint and from $13.12\%$ to $23.11\%$ for the upper endpoint.

The lower-endpoint conclusion is stable across the prespecified implementation variants:
\begin{center}
\small
\begin{tabular}{@{}lrrr@{}}
\toprule
analysis & primary median ratio & grid minimum & favorable comparisons\\
\midrule
weighted ridge   & $1.148601$ & $1.026157$ & $60/60$\\
unweighted ridge & $1.163915$ & $1.040759$ & $60/60$\\
weighted boosting& $1.077519$ & $1.015398$ & $15/15$\\
\bottomrule
\end{tabular}
\end{center}
For the weighted ridge analysis, the median variance reductions at tolerance multipliers $0.125$, $0.25$, and $0.50$ are $6.21\%$, $12.94\%$, and $26.41\%$.  Thus every split favors RALU at every prespecified tolerance.  The primary active face contains multiple branches for a median $53.17\%$ of weighted observations at the lower endpoint and $66.65\%$ at the upper endpoint.

We also reproduced the published five-fold Super Learner analysis of \citet{levis2025covariate}, using \texttt{SL.glm}, \texttt{SL.rpart}, \texttt{SL.ranger}, and \texttt{SL.polymars} \citep{vanDerLaanEtAl2007SuperLearner}.  It returns direct point bounds $(-0.417,0.529)$ and the outward Wald interval $(-0.468,0.580)$, closely reproducing the published analysis.  The replication archive contains the frozen aggregate tables, exact split definitions, software versions, and code for all reported comparisons.

\section{Scope and open extensions}\label{ssec:limitations}

The conditional-variance reduction is an influence-function result.  Its specialized closed forms and attainment statements use additional structure, whose roles are summarized here.

The harmonic formula for max-arm problems requires the arm index to partition the residual observation space.  When several candidate scores use the same observation, their covariance must be retained and the relevant program is the quadratic form based on the full score covariance matrix.  The Balke--Pearl application illustrates this shared-score geometry.  The finite-arm theory does not cover continuous actions because point evaluation at a data-dependent action generally lacks an $L^2$ representer under an absolutely continuous treatment law.  Continuous-action policy learning can instead be developed under additional structure on the action dependence, such as the known-feature semiparametric value model of \citet{DemirerSyrgkanis2019}.

The moment reductions require a real-valued identifying score and the fixed-mean perturbations stated in the richness assumptions. Conditional quantiles, conditional densities, and models with binding shape constraints require a different characterization, typically in the appropriate tangent space rather than through first moments alone. Cross-covariate smoothness constraints also prevent arbitrary measurable kernel replacement.  In such a constrained model, localization and the variance program must be re-derived using constraint-preserving perturbations.  Equality with the nonparametric bound then requires a separate analysis of the constraint-preserving perturbations and their variance restrictions.

On a strict max-arm branch, the exact perturbation equation forces the active link to be affine on the bridge-reachable interval.  Along a tie-preserving path, the same equation yields the linear program of Theorem~\ref{thm:finite-path} and need not force armwise affineness. The present paper proves attainment for the affine max-arm class and for the explicit applications treated above.  Constructing one globally measurable estimator that realizes the exact perturbation program at every law remains open in the general nonlinear tied case.

For tolerance-based active-set recovery, the admissible tolerance window couples the nuisance sup-norm rate to the near-tie margin index; positive tie mass itself is allowed.  Estimated auxiliary variances and covariances may be used under the convergence conditions in the corresponding attainment theorems.

The mediation results require treatment overlap and the stated one-sided bounded dominance of the control-mediator law by the treated-mediator law.  Extensions to unbounded Radon--Nikodym transport weights would require matched moment and product-rate conditions.  The shared-score attainment arguments use bounded estimated weights and convergence of their score evaluations. The stable-rank and separated-kernel results verify these conditions through spectral bounds; the integrated criterion in Corollary~\ref{cor:bp-integrated-gap} also covers sequences with shrinking gaps when its remainder conditions hold. At an unstable rank boundary, an oracle weight may have no path-independent continuous extension, and general attainment across such boundaries remains open.

The variational benchmark is defined over globally integrable score maps satisfying exact multiple robustness on the comparison model. One estimator attains it on the separate subclass where its learning conditions hold.  Section~\ref{ssec:bridge} identifies a classical interpretation on specified feasible local experiments; it does not assert regularity along arbitrary tie-breaking alternatives.  In constrained persistent-tie models, finite substitution identities can impose a positive variance cost beyond first-order regularity. A local-minimax analysis would instead optimize worst-case risk over shrinking neighborhoods and need not yield the same value as the fixed-law RALU bound.

The main open directions are measurable attainment of the general exact perturbation program, endogenous realization of covariance paths near unstable rank changes, closed forms for identifying objects beyond conditional means, and localization within smooth models subject to nontrivial structural constraints.

\bibliographystyle{abbrvnat}
\let\doi\relax
\bibliography{references}
\end{document}